\documentclass[a4paper,10pt,reqno]{amsart}
\numberwithin{equation}{section}
\usepackage{latexsym,amssymb}
\usepackage{hyperref}
\usepackage{amsmath,amsthm}
\usepackage{amsfonts}
\usepackage[american]{babel}
 \usepackage[T1]{fontenc}
 \usepackage[utf8]{inputenc}
\usepackage{mathrsfs}
\usepackage{psfrag}
\usepackage{epsfig,inputenc}
\usepackage{graphpap,latexsym,epsf}
\usepackage{color}
\usepackage{amssymb,eucal,paralist,color,enumerate}
\usepackage{graphicx}

\usepackage{tgheros}

\usepackage[eulergreek]{sansmath}

\newcommand{\bbM}{\mathbb{M}}

\newcommand{\R}{\mathbb{R}}
\newcommand{\N}{\mathbb{N}}
\newcommand{\C}{\mathbb{C}}

\newcommand{\Om}{\Omega}

\mathchardef\emptyset="001F

\numberwithin{equation}{section}

\newtheorem{theorem}{Theorem}[section]

\newtheorem{lemma}[theorem]{Lemma}
\newtheorem{remark}[theorem]{Remark}

\newtheorem{definition}[theorem]{Definition}
\newtheorem{proposition}[theorem]{Proposition}
\newtheorem{problem}[theorem]{Problem}

\newtheorem{notation}[theorem]{Notation}

\newcommand{\disp}{\displaystyle}

\newcommand{\eps}{\varepsilon}

\newcommand{\weakto}{\rightharpoonup} 

\newcommand{\aein}{\text{a.e.\ in }}

\newcommand{\down}{\downarrow}
\newcommand{\weaksto}{\overset{*}{\rightharpoonup}}

\newcommand{\AC}{\mathrm{AC}}

\newcommand{\dualoperator}

 \def\calE{{\mathcal E}} \def\calF{{\mathcal F}}
 \def\calH{{\mathcal H}} \def\calI{{\mathcal I}}
 \def\calK{{\mathcal K}} \def\calL{{\mathcal L}}
\def\calM{{\mathcal M}} \def\calN{{\mathcal N}} 
\def\calP{{\mathcal P}} \def\calQ{{\mathcal Q}} \def\calR{{\mathcal R}}
\def\calS{{\mathcal S}}  \def\calU{{\mathcal U}}
\def\calV{{\mathcal V}} \def\calW{{\mathcal W}} 
 
 \def\rmn{{\mathrm n}}

  \def\rmC{{\mathrm C}}
\def\rmD{{\mathrm D}} \def\rmE{{\mathrm E}}

\def\rmM{{\mathrm M}}  
  \def\rmR{{\mathrm R}}

\def\FG{\mathbf}

\def\bfD{{\FG D}}

 \def\bfQ{{\FG Q}}

\def\dd{\;\!\mathrm{d}} 

\newcommand{\pairing}[4]{ \sideset{_{ #1 }}{_{ #2 }}  {\mathop{\langle #3 , #4
\rangle}}}

\newcommand{\nchi}{{\raise.2ex\hbox{$\chi$}}}

\definecolor{ddcyan}{rgb}{0,0.1,0.9}
\definecolor{ddmagenta}{rgb}{0.8,0,0.8}
\definecolor{dmagenta}{rgb}{0.9,0,0.7}
\definecolor{orange}{rgb}{0.6,0.2,0}
\definecolor{vgreen}{rgb}{0.1,0.5,0.2}
\definecolor{dred}{rgb}{.8,0,0}
\definecolor{Turk}{rgb}{0,0.5,0.6}

\newcommand{\piecewiseConstant}[2]{\overline{#1}_{\kern-1pt#2}}
\newcommand{\pwc}{\piecewiseConstant}

\newcommand{\upiecewiseConstant}[2]{\underline{#1}_{\kern-1pt#2}}

\newcommand{\upwc}{\upiecewiseConstant}
\newcommand{\piecewiseLinear}[2]{{#1}_{\kern-1pt#2}}
\newcommand{\pwl}{\piecewiseLinear}
\newcommand{\pwwll}[2]{\widehat{#1}_{\kern-1pt#2}}

\newcommand{\piecewiseVariational}[2]{\tilde{#1}_{\kern-1pt#2}}

\newcommand{\ds}[3]{{#1}_{#2}^{#3}}
\newcommand{\dsd}[3]{\dot{#1}_{#2}^{#3}}

\newcommand{\DDD}[3]{\begin{array}[t]{c}#1\vspace*{-1em}\\_{#2}\vspace*{-.5em}\\_{#3}\end{array}}
\newcommand{\DDDn}[2]{\begin{array}[t]{c}#1\vspace*{-1em}\\_{#2}\end{array}}
\newcommand{\ddd}[3]{\DDD{\begin{array}[t]{c}\underbrace{#1}\vspace*{.6em}\end{array}}{\text{\footnotesize #2}}{\text{\footnotesize #3}}}
\newcommand{\dddn}[2]{\DDDn{\begin{array}[t]{c}\underbrace{#1}\vspace*{.6em}\end{array}}{\text{\footnotesize #2}}}

\newcommand{\foraa}{\text{for a.a.}}

\newcommand{\norm}[2]{\| #1\|_{#2}}

\newcommand{\BV}{\mathrm{BV}}
\newcommand{\BD}{\mathrm{BD}}

\newcommand{\Dir}{\mathrm{Dir}}
\newcommand{\Neu}{\mathrm{Neu}}

 \def\trait #1 #2 #3 {\vrule width #1pt height #2pt depth #3pt}

 \def\fin{\hfill
         \trait .3 5 0
         \trait 5 .3 0
         \kern-5pt
         \trait 5 5 -4.7
         \trait 0.3 5 0
 \medskip}
 
\newcommand{\QED}{\mbox{}\hfill\rule{5pt}{5pt}\medskip\par}

\newcommand{\bbC}{\mathbb{C}}

\newcommand{\bbD}{\mathbb{D}}

\newcommand{\mt}{\bbM}

\newcommand{\sym}{\mathrm{sym}}
\newcommand{\dev}{\mathrm{D}}

\newcommand{\dip}[2]{\mathrm{H}(#1,#2)}
\newcommand{\dipname}{\mathrm{H}}

\newcommand{\did}[1]{\mathrm{R}(#1)}

\newcommand{\didname}{\mathrm{R}}

\newcommand{\Gdir}{\Gamma_{\Dir}}
\newcommand{\Gneu}{\Gamma_{\Neu}}

\newcommand{\As}{A_{\mathrm{m}}}
\newcommand{\Hs}{H^{\mathrm{m}}}
\newcommand{\ass}{a_{\mathrm{m}}}

\newcommand{\calHt}[1]{\mathcal{H}^{\mathrm{tot}}_{#1}}
\newcommand{\calRt}{\mathcal{R}^{\mathrm{tot}}}

\newcommand{\RRR}{\color{black}}

\definecolor{violet}{rgb}{0.5,0.3,0.8}

\newcommand{\RRRN}{\color{black}}
\newcommand{\RCZ}{\color{black}}
\newcommand{\GGG}{\color{black}} 

\newcommand{\disv}[2]{\calV_{#1,#2}}
\newcommand{\dish}[2]{\calH_{#1,#2}}
\newcommand{\disvo}[1]{\calV_{#1}}
\newcommand{\disho}[1]{\calH_{#1}}
\newcommand{\densh}[2]{H_{#1,#2}}
\newcommand{\psin}[1]{\Psi_{#1}}
\newcommand{\psie}[2]{\Psi_{#1,#2}}

\newcommand{\Mnn}{{\mathbb{M}^{n\times n}_{\mathrm{sym}}}}
\newcommand{\MD}{{\mathbb M}^{n{\times}n}_{\mathrm{D}}}
\newcommand{\Sym}{\mathrm{Sym}}
\newcommand{\Lin}{\mathrm{Lin}}
\newcommand{\Mb}{\mathrm{M}_{\mathrm{b}}}
\newcommand{\MbD}{{\Mb(\Omega \cup \Gdir; \MD)}}
\newcommand{\Lnn}{{L^2(\Omega; \Mnn)}}
\newcommand{\Linftyn}{{L^\infty(\Omega; \MD)}}

\newcommand{\men}{^\mu_{\varepsilon,\nu}}
\newcommand{\MredVito}{ \mathcal{N}_{\varepsilon,\nu}^{\mu, \mathrm{red}}}
\newcommand{\MredVitok}{ \mathcal{N}_{\varepsilon_k,\nu}^{\mu, \mathrm{red}}}
\newcommand{\MredVitokk}{ \mathcal{N}_{\varepsilon_k,\nu_k}^{\mu, \mathrm{red}}}
\newcommand{\MredVitokkk}{ \mathcal{N}_{\varepsilon_k,\nu_k}^{\mu_k, \mathrm{red}}}
\newcommand{\MVito}{\mathcal{N}_{\varepsilon,\nu}^\mu}
\newcommand{\MeVitoname}{ \calN_{\eps,\nu}^\mu}

\newcommand{\DVito}{\mathcal{D}}
\newcommand{\DVitos}{\mathcal{D}_\nu^{*,\mu}}
\newcommand{\DVitosk}{\mathcal{D}_{\nu_k}^{*,\mu}}
\newcommand{\DVitoskk}{\mathcal{D}_{\nu_k}^{*,\mu_k}}

\newcommand{\DVitosred}{\mathcal{D}^{*,\mu}}
\newcommand{\DVitosredzero}{\mathcal{D}^{*,0}}

\newcommand{\diver}{\mathrm{div}}
\newcommand{\Diver}{\widehat{\mathrm{Div}}}
\newcommand{\hn}{\mathscr{H}^{n-1}}

\newcommand{\ol}{\overline}
\newcommand{\ole}{\overline{e}_\tau}
\newcommand{\olu}{\overline{u}_\tau}
\newcommand{\olp}{\overline{p}_\tau}
\newcommand{\olz}{\overline{z}_\tau}

\newcommand{\dutau}{{u}'_\tau}

\newcommand{\dpt}{{p}'_\tau}
\newcommand{\dzt}{{{z}'_\tau}}
\newcommand{\tk}{_\tau^k}
\newcommand{\tm}{_\tau^{k-1}}

\newcommand{\sig}[1]{\mathrm{E}(#1)}

\newcommand{\sft}{\mathsf{t}}
\newcommand{\sfu}{\mathsf{u}}
\newcommand{\sfz}{\mathsf{z}}
\newcommand{\sfp}{\mathsf{p}}
\newcommand{\sfe}{\mathsf{e}}
\newcommand{\serifsigma}{{\sansmath \sigma}}
\newcommand{\sfq}{\mathsf{q}}
\newcommand{\Me}[4]{ \calM_{\eps,\nu}^\mu(#1,#2,#3,#4)}

\newcommand{\Mename}{ \calM_{\eps,\nu}^\mu}

\newcommand{\Mredname}{\calM_{\eps,\nu}^{\mu, \mathrm{red}}}

\newcommand{\Mli}[4]{ \calM_{0,\nu}^\mu(#1,#2,#3,#4)}

\newcommand{\Mliname}{ \calM_{0,\nu}^\mu }

\newcommand{\Mliredname}{\calM_{0,\nu}^{\mu,\mathrm{red}}}
\newcommand{\oMli}[4]{\calM_{0,0}^0(#1,#2,#3,#4)}

\newcommand{\oMliname}{{\calM}_{0,0}^0}

\newcommand{\oMliredname}{{\calM}_{0,0}^{0,\mathrm{red}}}

\newcommand{\Mlizero}[4]{ \calM_{0,0}^\mu(#1,#2,#3,#4)}
\newcommand{\Mlieczero}[8]{\calM_{0,0}^\mu(#1,#2,#3,#4,#5,#6,#7, #8)}
\newcommand{\Mlinamezero}{\calM_{0,0}^\mu}
\newcommand{\Mliredzero}[8]{\calM_{0,0}^{\mu,\mathrm{red}}(#1,#2,#3,#4,#5,#6,#7, #8)}
\newcommand{\Mlirednamezero}{\calM_{0,0}^{\mu,\mathrm{red}}}
\newcommand{\Hsm}{H_{-}^{\mathrm{m}}(\Omega)}
\newcommand{\tilded}{\widetilde{d}}
\newcommand{\dLtwo}{d_{L^2}}
\newcommand{\enen}[1]{\calE_{#1}}
\newcommand{\newmu}{{ \mu}}

\newcommand{\Qpp}{\bfQ_{\scriptstyle \tiny \mathrm{PP}}}
\newcommand{\Hpp}{\mathcal{H}_{\scriptstyle \tiny \mathrm{PP}}}

\begin{document}
\title[Balanced Viscosity solutions to a coupled system for damage and plasticity]{Balanced Viscosity solutions to a rate-independent coupled elasto-plastic damage system}

\author{Vito Crismale \and Riccarda Rossi}

\address{V.\ Crismale, CMAP, \'Ecole Polytechnique, 91128 Palaiseau Cedex, France}
\email{vito.crismale@polytechnique.edu}

\address{R.\ Rossi, DIMI, Universit\`a degli studi di Brescia,
via Branze 38, 25133 Brescia, Italy}
\email{riccarda.rossi\,@\,unibs.it}

\thanks{The authors 
 are very grateful to 
 Giuliano Lazzaroni for several interesting discussions and suggestions.  V.C.\ is supported by the Marie Sk\l odowska-Curie Standard European Fellowship No.\ 793018. He also acknowledge the financial support from the Investissement d’avenir project, reference ANR-11-LABX-0056-LMH, LabEx LMH, and from the Laboratory Ypatia. R.R.\ has been partially supported by the  Gruppo Nazionale per  l'Analisi Matematica, la
  Probabilit\`a  e le loro Applicazioni (GNAMPA)
of the Istituto Nazionale di Alta Matematica (INdAM)}



\begin{abstract}
A rate-independent model coupling small strain associative elasto-plasticity and damage is studied via a \emph{vanishing-viscosity} analysis with respect to all the variables describing the system.
This extends the analysis performed  for the same system in \cite{Crismale-Lazzaroni}, where a  vanishing-viscosity regularization involving only the damage variable was set forth.
In the present work, an additional approximation featuring vanishing plastic hardening is introduced  in order to deal with the vanishing viscosity in the plastic variable. 
Different regimes are considered, leading to  different notions of Balanced Viscosity solutions for the perfectly plastic damage system, and for its version  with hardening.
\end{abstract}
\maketitle

\medskip

\section{Introduction}
In this paper we address the analysis of a rate-independent system coupling small-strain associative elasto-plasticity and damage. We construct weak solutions for the related initial-boundary value problem via a 
 \emph{vanishing-viscosity}  regularization  that affects  all the  variables describing the system.
Before entering into the details of this procedure, let us briefly illustrate  the rate-independent model we are interested in.
\par
In a time interval $[0,T]$, for a bounded open Lipschitz domain $\Omega\subset \R^n$, $n\geq 2$, and  time-dependent volume  and surface forces $f$ and $g$,  we consider a PDE system coupling the evolution of the  \emph{displacement} 
$u: (0,T) \times \Omega \to \R^n$, of the \emph{elastic} and \emph{plastic strains} $e: (0,T)\times \Omega \to \Mnn$ and $p: (0,T)\times \Omega \to \MD$, and  of a \emph{damage variable} $z: (0,T)\times \Omega \to  [0,1]$ that assesses the soundness of the material: for $z(t,x)=1$ ($z(t,x)=0$, respectively) the material is in the undamaged (fully damaged, resp.) state,  at the time $t\in (0,T)$ and ``locally'' around the point $x\in \Omega$. In fact, the PDE system consists of 
 \begin{subequations}
 \label{RIS-intro}
 \begin{itemize}
\item[-] the momentum balance
\begin{align}
\label{mom-balance-intro}
 - \mathrm{div}\,\sigma = f  \quad \text{ in } \Omega \times (0,T)\,, \qquad \quad  \sigma \rmn =g \text{ on } \Gneu \times (0,T), 
 \end{align}
(with $\Gneu $ the Neumann part of the boundary $\partial \Omega$), where  the \emph{stress} tensor is  given by 
\begin{equation}
\label{stress-intro}
 \sigma =  \mathbb{C}(z)e  
 \quad   \text{ in } \Omega \times (0,T), 
   \end{equation}
 and the kinematic admissibility condition for the \emph{strain} $ \rmE(u)  = \frac{\nabla u + \nabla u^T}{2}$ reads
 \begin{equation}
 \label{kam-intro}
    \rmE(u)  = e+p   \quad  \text{ in } \Omega \times (0,T)\,;
 \end{equation} 
 \item[-] the flow rule for the damage variable $z$
 \begin{align}
\label{flow-rule-dam-intro}
\partial\did{\dot{z}}  + \As (z)+ W'(z) \ni - \tfrac12\bbC'(z)e : e  \quad  \text{ in } \Omega \times (0,T),
\end{align}
where, above and in \eqref{flow-rule-plast-intro}, the symbol $\partial$ denotes the convex analysis subdifferential of  the density of dissipation potential
\[
\didname:\R \to [0,+\infty]  \  \text{ defined by } \   \did{\eta}:= \left\{ \begin{array}{ll}
|\eta| & \text{ if } \eta\leq 0,
\\
+\infty & \text{ otherwise},
\end{array}
\right.
\]
encompassing the unidirectionality in the evolution of damage,  $\As$ is 
 the  $\mathrm{m}$-Laplacian operator, with $\mathrm{m}>\tfrac d2$,
 and $W$ is a suitable nonlinear, possibly nonsmooth, function; 
  \item[-] the flow rule for the plastic tensor 
 \begin{align}
&
\label{flow-rule-plast-intro}
\partial_{\dot{p}}  \dip{z}{\dot{p}}  \ni \sigma_{\mathrm{D}}  \quad \text{ in } \Omega \times (0,T),
\end{align}
with $\sigma_\dev$ the deviatoric part of the stress tensor $\sigma$ and $\dipname(z, \cdot)$ the density of plastic dissipation potential; $\dipname(z, \cdot)$ is  the support function of the \emph{constraint set} $K(z)$.
The PDE system is supplemented with initial conditions and 
the boundary conditions 
\begin{equation}
 \label{viscous-bound-cond}
 u=w  \text{ on } \Gdir \times (0,T), \quad \qquad \partial_n z =0 \text{ on } \partial\Omega \times (0,T), 
  \end{equation}
with  $\Gdir$  the Dirichlet part of the boundary $\partial \Omega$.
 \end{itemize}
 \end{subequations}
Let us highlight that the damage variable  $z$ influences both the Hooke tensor $\bbC$, which determines the elastic stiffness of the material, and the constraint set $K$ for the deviatoric part of the stress, which is such that the material undergoes plastic deformations only if $\sigma_\dev$ reaches the boundary $\partial K$. By our choice of the dissipation potential $\rmR$, the variable $z$ is forced to decrease in time: it is then usual to assume that $[0,1]\ni z \mapsto\C(z)$ is non-increasing and that $[0,1]\ni z \mapsto K(z)$ is non-decreasing, with respect to the natural ordering for positive definite tensors and to the  inclusion of sets (cf.\ Section~\ref{s:2} for the precise assumptions).

The elasto-plastic damage model \eqref{RIS-intro}, which   reduces to the Prandtl-Reuss model for  perfect plasticity (cf.\ e.g.\ \cite{DMDSMo06QEPL, Sol09, FraGia2012, Sol14}) if no dependence on damage is assumed, was first proposed and studied  in \cite{AMV14, AMV15}. Subsequently,
in  \cite{Crismale} (with refinements in \cite{CO17}, see also \cite{CO19}), the existence of Energetic solutions  \`a la Mielke-Theil (cf.\ \cite{MieThe99MMRI, MieThe04RIHM}) was proved. We recall that this weak solvability concept for rate-independent processes,  
also known as \emph{quasistatic evolution} (cf.\ e.g.\ \cite{DM-Toa2002}),
consists of \emph{(i)} a   \emph{global stability} condition, which prescribes that  at each process  time the current configuration minimizes the sum of the  total internal energy and  the dissipation potential; \emph{(ii)} an Energy-Dissipation balance featuring the variation of the   internal energy between the current and  the initial times, the total dissipated energy, and the work of the external loadings.
Thus,  the Energetic formulation is  derivative-free and hence very flexible and suitable for limit passage procedures. 
 In the framework of Energetic-type solution concepts,
the study of models coupling damage and plasticity indeed
seems to have attracted some attention over the last years: in this respect, we may e.g.\ quote \cite{Cri17} for a damage model coupled with  strain-gradient plasticity, as well as  \cite{BonRocRosTho16, Rossi-Thomas-SIAM, RouVal17} for plasticity with hardening, \cite{Roub-Valdman2016} accounting also for damage healing, \cite{MelScaZem19} for finite-strain plasticity with damage,  and \cite{DavRouSte19} for perfect plasticity and damage in viscoelastic solids in a dynamical setting.
\par
System  \eqref{RIS-intro} has however  been  analyzed also from a perspective different from that of Energetic solutions. 
Indeed, despite their manifold advantages, Energetic solutions have a catch: when the energy functional driving the system is nonconvex, Energetic solutions as functions of time may have  ``too early''  and  ``too long''  jumps  between energy wells,
cf.\ e.g.\ 
\cite[Ex.\ 6.3]{KnMiZa07?ILMC},
 \cite[Ex.\ 6.1]{MRS09}, and the full characterization of Energetic solutions to $1$-dimensional rate-independent systems from \cite{RosSav12}. Essentially,
this
due to  the rigidity of the global stability condition that involves  the global, rather than the local, energy landscape. 
 These considerations have motivated the quest of alternative weak solvability  notions for rate-independent systems. In this paper,
 we focus on 
notions obtained by a 
 \emph{vanishing-viscosity} approximation  of the original rate-independent process.
\par
The vanishing-viscosity approach stems from the idea that 
rate-independent processes  originate in the limit of systems governed by  
two  time scales: the  inner scale of the system and the time scale of the external loadings.
The latter scale  is considerably slower  than the former, but it is dominant, and from its viewpoint viscous dissipation is negligible.
But viscosity is expected to re-enter into the picture in the description of the system behavior at jumps, 
which should be indeed considered as
\emph{viscous transitions}  between metastable states, cf.\ \cite{EfeMie06RILS}. 
Thus, one  selects those solutions to the original rate-independent system that arise as limits 
of solutions to the viscously regularized system. What is more, following an idea from  \cite{EfeMie06RILS},
in order to capture the viscous transition path between two jump points one reparameterizes the viscous trajectories
and performs the vanishing-viscosity analysis for curves in an extended phase space that
also comprises the rescaling function. For this, it is
crucial to control the length (or a  ``generalized length'')  of the viscous curves, uniformly w.r.t.\ the viscosity parameter.
This limit procedure then leads to reparameterized solutions (functions of
an ``artificial'' time variable $s\in [0,S]$) of the original rate-independent system, such that the reparameterized 
state variable(s) is  (are) coupled with a rescaling function $\sft:[0,S]\to [0,T]$ that takes values in the original time interval.
In this way,
 equations for the paths connecting the the left and right limits (stable states, themselves) of the system at a jump point 
 may be derived; the (possibly viscous) 
 path followed by the (reparameterized) limit solution at a jump  point is also accounted for 
   in a suitable Energy-Dissipation balance. 
Furthermore, the solution concept obtained by vanishing viscosity
 is  supplemented by a first-order, \emph{local} stability condition, which holds
  in the ``artificial'' time  intervals corresponding to those  in which  the system does not jump in the original (fast) time-scale.
\par
Moving from the pioneering  \cite{EfeMie06RILS}, in 
\cite{MRS09, MRS12,  MRS13} (cf.\ also \cite{Negri14}) this idea has been formalized in an abstract setting,
codifying  the properties of these ``vanishing-viscosity solutions''  in
the notion of 
\emph{Balanced Viscosity} (hereafter often shortened as $\BV$) solution to a rate-independent system. In parallel,
the vanishing-viscosity technique has been developed and refined in various concrete applications, 
 ranging from plasticity  (cf., e.g., \cite{DalDesSol11, BabFraMor12, FrSt2013}), 
	 to damage,  fracture, and fatigue  (see for instance \cite{KnMiZa08ILMC, Lazzaroni-Toader, KRZ13, Almi17, CL17fra, ACO2018, ALL19}).
\par
For the present elasto-plastic damage system \eqref{RIS-intro}, the vanishing-viscosity approach was
first addressed in
\cite{Crismale-Lazzaroni}. There, $\BV$ solutions to system \eqref{RIS-intro} were constructed by passing to the limit in the viscously
 regularized system
 featuring
  viscosity \emph{only in the flow rule for the damage variable} $z$. 
  Namely, the momentum balance \eqref{mom-balance-intro} (with \eqref{stress-intro} \& \eqref{kam-intro}) and the plastic flow rule
  were coupled with the 
    \emph{rate-dependent} subdifferential inclusion 
\begin{equation*}\label{0709191140}
\partial\did{\dot{z}} +\eps \dot{z} + \As (z) + W'(z) \ni - \tfrac12\bbC'(z)e : e  \quad  \text{ in } \Omega \times (0,T),
\end{equation*}
(with $0<\eps \ll 1$), in place of \eqref{flow-rule-dam-intro}.
Accordingly, the dissipation potential governing  \eqref{RIS-intro} was augmented by a viscosity contribution 
featuring the $L^2$-norm for the damage rate $\dot z$. Actually, in
\cite{Crismale-Lazzaroni} the authors succeeded in deriving 
estimates  (uniform w.r.t.\ the viscosity parameter $\varepsilon$) for the length  of the viscous solutions $(z_\eps)_\eps$ 
in the  $\Hs$-norm, even ($\Hs(\Omega)$ being the reference space for the damage variable).
Relying on these bounds and on the reparameterization procedure above described, 
they obtained a notion of 
$\BV$ solution such that  only viscosity in $z$ (possibly) enters in the description of the transition path followed by the 
system at jumps. Accordingly, this is reflected in the 
 Energy-Dissipation balanced satisfied by $\BV $ solutions.
\par
Nonetheless, jumps in the other variables are not excluded during jumps for $z$, and 
the  ``reduced''  vanishing-viscosity approach carried out in \cite{Crismale-Lazzaroni}
does not provide information on the (possibly) viscous trajectories  followed by those variables at jumps.
Furthermore, specific examples in simplified situations from
 \cite[Section 5]{AMV14} show that,   where damage nucleates, one could expect a close interaction
  between the damage and the plastic variable, which, in turn, is  intrinsically related to the displacement via the 
  kinematic admissibility condition \eqref{kam-intro}.
  In this connection, we may also quote 
   \cite{DMOrlToa16CalcVar}, where the limit passage from elasto-plastic damage to elasto-plastic fracture static models was 
   justified 
    in the antiplane  setting.
\par
These considerations have motivated us to develop a 
 \emph{``full'' vanishing-viscosity approach} to system \eqref{RIS-intro}. Namely,
 we have approximated  \eqref{RIS-intro}
 by a viscously regularized system featuring a viscosity contribution 
  for the plastic and the displacement variables, besides the damage variable and, correspondingly, obtained a
   notion of 
  Balanced Viscosity solution for \eqref{RIS-intro}  different from the one in \cite{Crismale-Lazzaroni}.
%

%
 \subsection*{The ``full'' vanishing-viscosity approach}
Upon viscously regularizing 
all variables $u$, $z$, and $p$, 
the scenario 
 turns out to be more complicated than the one in \cite{Crismale-Lazzaroni}
  from an analytical point of view.  The first challenge is 
  related to the derivation of (uniform, w.r.t.\ the viscosity parameter) estimates for the  length (in a suitable sense)
   of the viscous solutions.
Adding just a $L^2$-viscous regularization for the plastic variable $p$ (and, consequently,
 a $L^2$-viscous regularization for the total strain  $\rmE(u)$,  tightly related to $p$ via \eqref{kam-intro}),
 does not lead 
to any \emph{a priori} length estimate for $p$ with respect to  the norm of its reference space,
 that is the space $\Mb(\Omega;\MD)$ of bounded Radon measures with values in $\MD$. 
 \par
On the one hand, this could be due to the fact that the usual techniques for proving
 \emph{a priori} estimates, based on testing the viscously regularized equations
  with the time derivatives of the corresponding variables, 
  seem suitable to get good length estimates only in Hilbert spaces. Now, 
   estimates for $p$ in Hilbert spaces contained in $\Mb(\Omega;\MD)$, such as $L^2(\Omega;\MD)$, 
  would be unnatural and incompatible with the concentration effects (in space) that one would see in the limiting, perfectly plastic, evolution.
\par
On the other hand,  adding  \emph{directly}, to the plastic flow rule,
 a viscous regularization
  that features the $L^2$-norm, stronger than the one in  the reference space $\Mb(\Omega;\MD)$,
   does not seem to be the right procedure from a heuristic point of view. Indeed, the idea associated with
    the vanishing-viscosity approach
is to let the system explore the energy landscape around the starting configuration 
and choose an arrival configuration that is preferable from an energetic viewpoint, but close enough in terms of the viscosity norm
(this   becomes more evident on  the level of the time discretization of the viscous system, 
    cf.\ \eqref{def_time_incr_min_problem_eps}).
When  viscosity vanishes, the evolution still keeps track of this procedure. 
Therefore,
in this respect it is reasonable to take, for the viscous regularization, 
a norm that is not stronger than the reference norm. In this way, the system  is free to detect the updated configuration in all the reference space.
 This is the case, for instance,  of the $L^2$-viscous regularization for $z$, whose reference space is $\Hs(\Omega)$,
used for the damage flow rule here and in \cite{KRZ13, Crismale-Lazzaroni, Neg16}.
\par
In order to mimic this approach also for the variable $p$ (and, consequently, for $u$), 
\begin{compactitem}
\item[-]
we have  introduced a further hardening regularizing term to the plastic flow rule,
 tuned by a parameter $\mu>0$.
 In this way,
  the reference space for  the plastic strain $p$ becomes $L^2(\Omega;\MD)$; 
  \item[-]
  we have  address a viscous regularization for $p$ such that the viscosity parameter $\varepsilon$ is modulated  by an additional
parameter $\nu$, with $\nu \leq \mu$.
\end{compactitem}
All in all, 
we consider the following  \emph{rate-dependent} system for damage coupled with viscoplasticity, 
 featuring the three parameters $\eps,\, \nu,\, \mu>0$: 
 \begin{subequations}
 \label{RD-intro}
\begin{itemize}
\item[-] the viscous \RRRN (albeit \emph{quasi-static}, as inertial forces are neglected),   momentum balance
\begin{align}
\label{VISC-mom-balance-intro}
 - \mathrm{div}(\eps\nu  \mathbb{D}   \rmE{(\dot{u})}   + \sigma ) = f  \quad \text{ in } \Omega \times (0,T), \qquad \quad  (\eps\nu  \mathbb{D}  \rmE{(\dot{u})}  + \sigma) \rmn = g \quad\text{ in }\Gneu \times (0,T) 
 \end{align}
 (with $\bbD$  a fixed positive-definite fourth-order tensor),
coupled with the expression for $\sigma$ from \eqref{stress-intro} and
 the kinematic admissibility condition \eqref{kam-intro};
 \item[-] the rate-dependent damage flow rule for $z$ 
 \begin{align}
\label{flow-rule-dam-intro-RD}
&
\partial\did{\dot{z}} +\eps \dot{z} + \As (z) + W'(z) \ni - \tfrac12\bbC'(z)e : e  \quad  \text{ in } \Omega \times (0,T);
\end{align}
  \item[-] the viscous flow rule for the plastic tensor 
 \begin{align}
&
\label{flow-rule-plast-intro-RD}
\partial_{\dot{p}}  \dip{z}{\dot{p}} + \eps\nu  \dot{p} +\newmu  p \ni   \sigma_{\mathrm{D}}  \quad \text{ in } \Omega \times (0,T).
\end{align}
 \end{itemize}
  \end{subequations} 
The system is supplemented with the boundary conditions \eqref{viscous-bound-cond}.
We highlight that viscosity for the $u$ variable has been encompassed  in the stress
 tensor
 (in accord with \emph{Kelvin-Voigt rheology}) through the term $\sig{\dot u}$. In fact, the other possible choice, $\dot e$, would 
 not have preserved the \emph{gradient structure} of the system, which is crucial for our analysis. 
 \par
 Let us emphasize that, for the rate-dependent system with hardening (i.e.\ with fixed $\eps, \nu, \mu>0$)
  both the reference space and the viscosity space for $p$ are  $L^2(\Omega;\MD)$. Furthermore,  the
 choice $\nu\leq \mu$ (one could take $\nu \leq C \mu$ as well) guarantees that we do not lose the desired ``order''
  between viscosity and reference norm for $p$ as $\nu$, $\mu$ vanish. 
  This has enabled us to derive  \emph{a priori} estimates
  for the viscous solutions 
   that are uniform not only w.r.t.\
  $\varepsilon$, but also w.r.t.\ $\mu$ (and $\nu$).
  \par 
  We shall refer to $\nu$ as a \emph{rate} parameter. Indeed, for fixed  $\nu>0$  
  and $\eps \down 0$, 
the displacement and the plastic rate converge to equilibrium and rate-independent
  evolution, respectively, at the same rate at which the damage parameter
  converges to rate-independent evolution. When $\eps\down0$ and  $\nu \down 0$
  \emph{simultaneously}, relaxation to equilibrium and rate-independent behavior occurs at a faster rate for 
  $u$ and $p$ than for $z$. 
  The vanishing-viscosity analysis 
  then acquires a \emph{multi-rate} character. Balanced Viscosity to \emph{multi-rate} systems have been explored in an abstract, albeit finite-dimensional setting, 
  in \cite{MRS14} (cf.\ the forthcoming \cite{MieRosBVMR} for the extension to the infinite-dimensional setup). 
  \subsection*{Our results}
  In what follows, we shall address \emph{three different problems}.
  \par
  First of all, we shall carry out 
the vanishing-viscosity analysis of \eqref{RD-intro} as $\eps \down 0$, with $\mu >0$
 fixed.
  This will lead to  the existence of (two different types of)
  Balanced Viscosity solutions to a  rate-independent system for damage and plasticity \emph{with hardening}, consisting of 
  (\ref{mom-balance-intro}, \ref{stress-intro}, \ref{flow-rule-dam-intro}, \ref{kam-intro}, \ref{viscous-bound-cond}) coupled with 
  \begin{equation}
  \label{pl-hard-intro}
  \partial_{\dot{p}}  \dip{z}{\dot{p}} + \newmu p  \ni \sigma_{\mathrm{D}}  \quad \text{ in } \Omega \times (0,T).
  \end{equation}
 In fact, we  shall consider two cases in the vanishing-viscosity analysis  as  $\eps \down 0$,
 with $\mu>0$ fixed:
  \begin{enumerate}
\item first, we shall keep the rate parameter $\nu>0$ fixed, so that  $(u,z,p)$ relax to equilibrium (for $u$) and
rate-independent evolution (for $z$ and $p$) with the same rate. In this way, we shall prove the existence of 
  \emph{$\BV$ solutions to the rate-independent system with hardening} (\ref{mom-balance-intro}, \ref{stress-intro}, \ref{flow-rule-dam-intro}, \ref{kam-intro}, \ref{viscous-bound-cond} \ref{pl-hard-intro}), see Definition~\ref{def:BV-solution-hardening} and Theorem~\ref{teo:rep-sol-exist};
\item  second, we shall let $\nu \down 0$ together with $\eps\down0$, so that  $u$ and $p$ relax to equilibrium and rate-independent 
evolution faster than $z$, relaxing to rate-independent evolution.
In this way, we  shall obtain \emph{$\BV$ solutions to the multi-rate system with hardening} (\ref{mom-balance-intro}, \ref{stress-intro}, \ref{flow-rule-dam-intro}, \ref{kam-intro}, \ref{viscous-bound-cond} \ref{pl-hard-intro}), see Definition~\ref{def:multi-rate-BV-solution-hardening} and Theorem~\ref{teo:rep-sol-exist-2}. 
\end{enumerate}
\par
Balanced Viscosity solutions to the rate-independent system with hardening arising from the ``full'' vanishing-viscosity approach are parameterized curves 
$(\sft,\sfu,\sfz,\sfp)$ 
defined on an ``artificial'' time interval  (with $\sft$ the rescaling function)  that satisfy a suitable (scalar) Energy-Dissipation balance encoding all information on the evolution of the system. This is in   accord with the notion that has been codified, in an abstract (finite-dimensional) setup, in  \cite{MRS14}. More in general, this solution concept stems from a variational approach to gradient flows and general gradient systems; indeed, it is in the same spirit as  the notion  of 
\emph{curve of maximal slope} \cite{AGS08}. The Energy-Dissipation balance characterizing (parameterized) $\BV$ solutions features
a \emph{vanishing-viscosity contact potential}, namely
 a functional $\calM = \calM (\sft,\sfq,\sft',\sfq') $ (hereafter, we shall often use $\sfq$ as a place-holder for the triple $(\sfu,\sfz,\sfp)$), whose expression (and notation) depends on the different regimes considered.
\par
In all cases, $\mathcal{M}$ encodes the possible onset of viscous behavior of the system at jumps. Indeed, in the ``artificial'' time, jumps  occur at instants at which  the rescaled slow time variable $\sft$ is frozen, i.e.\ $\sft'=0$. Now,   (only) at the jump instants the system may not satisfy (a weak version of the) first order stability conditions in the variables $u$, $p$, $z$,  and for this it dissipates energy in a way that is described by the specific expression of $\calM$ for $\sft'=0$. In  particular, 
\begin{compactitem}
\item[\emph{(i)}] for the $\BV$ solutions obtained  via vanishing viscosity with $\nu>0$ fixed, the  contact potential 
$\calM(\sft,\sfq,0,\sfq')$ features
 a term with the (viscous) $H^1{\times}L^2{\times}L^2$-norm of \emph{the full triple} $(\sfu', \sfp', \sfz')$. While referring to 
Section \ref{s:6-nu-fixed} for more comments, here we highlight  that the expression of $\calM$ reflects the fact that, at a jump, the system may be 
switch to a regime where viscous dissipation in the three variables intervenes ``in the same way''. This mirrors the fact that the variables $u,\,z,\,p$ relax to static equilibrium and rate-independent evolution with the same rate;
\item[\emph{(ii)}] for the $\BV$ solutions obtained  in the limit   as $\eps,\,\nu \down 0$ jointly, in the expression of $\calM(\sft,\sfq,0,\sfq')$ 
two distinct terms account for the roles of the rates $(\sfu',\sfp')$ and of $\sfz$'. A careful analysis,   carried out in Section 
\ref{s:6-nu-vanishes}, in particular shows that, at a jump, $\sfz$ is frozen until   $\sfu,\, \sfp$ have reached the elastic equilibrium/attained the local stability condition, respectively. This reflects the fact that  $u,\, p$ relax to equilibrium/rate-independent behavior faster than $z$, hence the  multi-rate  character of the evolution.
\end{compactitem}
The above considerations can be easily inferred from the
 PDE characterization of (parameterized)  $\BV$ solutions that we provide in Propositions~\ref{prop:charact-BV} \& \ref{prop:diff-charact-BV-zero}; we also refer to Remarks~\ref{rmk:mech-beh} \& \ref{rmk:mech-interp} for further comments and for a comparison between the two notions of solutions for the system with hardening. 
\par
After the discussion of plasticity with fixed hardening, 
\begin{itemize}
\item[(3)]
we  shall consider the case when also $\mu$ vanishes and thus address the asymptotic analysis of
system \eqref{RD-intro} as 
the parameters $\varepsilon$, $\nu$, $\mu \down 0$ \emph{simultaneously}. 
With our main result, Theorem~\ref{teo:exparBVsol}, we shall prove that, after a suitable reparameterization, viscous solutions converge to 
Balanced Viscosity solutions for the perfectly plastic  system \eqref{RIS-intro} that differ from the ones obtained in \cite{Crismale-Lazzaroni}  in this respect:
the description of the trajectories during jumps may possibly involve viscosity in \emph{all} the variables $u$, $p$, $z$.
Since $\eps $ and $\nu$ vanish jointly, the system
has again a multi-rate character.
\end{itemize}
However, in the perfectly plastic case
the situation is more complex than for the case with hardening. Indeed, for 
perfect plasticity the reference function space for  (the rescaled plastic strain) $\sfp$ is $\Mb(\Omega;\MD)$ instead of $L^2(\Omega;\MD)$, 
while the viscous dissipation that (possibly) intervenes at jumps features the $L^2$-norm of $\dot{\sfp}$. 
  In particular,
 at jumps the expression  of the  contact  potential $\calM$ guarantees that $\sfp$ is in $L^2(\Omega;\MD)$ and $\sfu$ is in $H^1(\Omega;\R^n)$, which is reminiscent of the approximation through plastic hardening.
The change in  the functional framework  occurring at the jump regime has important consequences for the 
analysis. On the one hand, we have to exploit density arguments and equivalent characterizations of the stability conditions to pass from the 
$L^2(\Omega;\MD)$-framework  to the $\Mb(\Omega;\MD)$-setting. On the other hand, a suitable reparameterization and abstract tools are needed to reveal more spatial regularity for $\sfu$ and $\sfp$ along jumps, in the spirit of \cite[Subsection~7.1]{MRS13} (cf.\ Section \ref{s:van-hard} for more details).  Another interesting point is that the present approximation through plasticity with hardening completely alleviates the need for a classical Kohn-Temam duality between stress and plastic strain, so we can use only the duality in \cite{FraGia2012} and therefore we do not have to impose more regularity on $\Omega$ or more regularity on the external loading (cf.\ Remark~\ref{rem:3009191901}).  
\par
A natural question is whether the passage to the limit in Section~\ref{s:van-hard} as $\varepsilon$, $\nu$, $\mu$ tend to zero simultaneously commutes with the passage to the limit as $\mu\down 0$ in the ``multi-rate'' $\BV$ solutions obtained in Section \ref{s:6-nu-vanishes}.
In other words, one wonders if the simultaneous limit passage as $\varepsilon$, $\nu$, $\mu$  vanish in system \eqref{RD-intro} results in the same notion of solution as 
the one obtained by first letting
 $\varepsilon, \, \nu \down 0$  with $\mu >0$ fixed and 
 then letting $\mu\down0$. 
 It will be shown that the two approaches lead to the same evolution in a forthcoming paper, where we shall also develop a more thorough comparison with the 
 notion of solution obtained in \cite{Crismale-Lazzaroni}.
 \paragraph{\bf Plan of the paper.} In Section \ref{s:2} we  fix all the standing assumptions on the constitutive functions and on the problem data, and prove some preliminary results.
  Section \ref{s:grsyst} focuses on
  the gradient structure that underlies the rate-dependent system \eqref{RD-intro}, and that is at the core of its vanishing-viscosity analysis.
 Based on this structure, we set out to prove the existence of solutions to  \eqref{RD-intro} by passing to the limit in a carefully devised time-discretization scheme.
 A series of  priori estimates on the time-discrete solutions are proved in  Section~\ref{s:4}. Such bounds serve as a basis both for the existence proof 
 for the viscous problem, and for its vanishing-viscosity analysis. Indeed, in Proposition~\ref{prop:enh-energy-est} we obtain  estimates for the total variation  of the discrete solutions that are uniform w.r.t.\ the viscosity parameter $\varepsilon$ and w.r.t.\ $\nu$ and, in some cases,
 $\mu$ as well. For this, the condition $\nu \leq \mu$ plays a crucial role. Such bounds will lead to the estimates on the lengths of the curves needed for the arc-length repameterizations and the vanishing-viscosity limit passages.
We then derive the existence of solutions for the viscous system \eqref{RD-intro} in Section~\ref{s:exist-viscous}. This is the common ground for the subsequent analysis as either some or all parameters vanish.
The limit passages in \eqref{RD-intro} with $\mu>0$ fixed are carried out in Section~\ref{s:van-visc}: in particular, Section~\ref{s:6-nu-fixed} focuses on the analysis as  $\eps\down0$ with fixed $\nu>0$, while the limit as $\eps,\ \nu\down 0$ is discussed in Section~\ref{s:6-nu-vanishes}. 
The limit passage as $\eps,\nu,\mu\down 0$ is performed in 
Section~\ref{s:van-hard}.  Therein,  we  do not reparameterize the  viscous solutions by their ``classical'' arclength but by an \emph{Energy-Dissipation} arc-length that 
somehow encompasses the onset, for the limiting $\BV$ solutions, of rate-dependent behavior and additional spatial regularity during jumps.

\section{Setup for the rate-dependent and rate-independent systems}
\label{s:2}
In this section we   establish the setup and the assumptions
 on the constitutive functions and on the problem data, 
 both for the rate-dependent system
\eqref{RD-intro} and for its rate-independent limits. Namely, we  shall  propose a framework of conditions suiting 
\begin{compactitem}
\item[-]
  \emph{both} 
 the rate-independent process with hardening, i.e.\ that obtained by taking the vanishing-viscosity limit 
 of \eqref{RD-intro} as $\eps \down0$, 
  (and, possibly, $\nu\down0$ in the \emph{multi-rate} case), 
 with $\mu>0$ fixed,   \emph{and} 
 \item[-]
 the  rate-independent process for perfect plasticity and damage (i.e., that obtained in the further limit passage as $\mu\down 0$).
\end{compactitem}
Further definitions and auxiliary results for the perfectly plastic damage system will be 
expounded in Section~\ref{s:van-hard}. 
\par
First of all, let us fix some notation that will be used throughout the paper. 
\begin{notation}[General notation  and preliminaries] 
\label{not:1.1}
\upshape
 Given a Banach space $X$, we shall denote by 
$\pairing{}{X}{\cdot}{\cdot}$   the duality pairing between $X^*$ and $X$ (and, for simplicity, also between $(X^n)^*$ and $X^n$). We  will   just write $\pairing{}{}{\cdot}{\cdot}$ for the inner Euclidean product in $\R^n$. 
  Analogously, we  shall indicate  by $\| \cdot \|_{X}$ the norm in $X$ and  often use the same symbol for the norm in $X^n$, as well,  and just write $|\cdot| $ for the Euclidean norm in $\R^m$, $m\geq1$. 
We  shall  denote by  
 $B_r(0)$  the open ball of radius $r$, centered at $0$, in the Euclidean space $X=\R^m$. 
\par
We  shall denote  by $\Mnn$ the space of the symmetric  $(n{\times}n)$-matrices, and by $\MD$ the subspace of the deviatoric matrices   with null trace. In fact, 
$\Mnn = \MD \oplus \R I$ ($I$ denoting the identity matrix), since every $\eta \in \Mnn$ can be written as 
$
\eta = \eta_\dev+ \frac{\mathrm{tr}(\eta)}n I
$
with $\eta_\dev$ the orthogonal projection of $\eta$ into $\MD$. We  refer  to $\eta_\dev$ as the deviatoric part of $\eta$.  
 We  write  for  $\Sym(\MD;\MD)$  the set of symmetric endomorphisms on $\MD$.
\par
We  shall often use  the short-hand notation $\| \cdot \|_{L^p}$, $1\leq p<+\infty$, for the $L^p$-norm on the space $L^p(O;\R^m)$, with $O$ a measurable subset of $\R^n$,  and analogously we  write  $\| \cdot \|_{H^1}$.  We  shall denote  by $\Mb(O;\R^m)$ the space of  bounded Radon measures on $O$ with values in $\R^m$. 
\par
As already mentioned in the Introduction, as in \cite{KRZ13,Crismale-Lazzaroni} the mechanical energy shall encompass   a gradient regularizing contribution  for the damage variable, featuring the bilinear form
\begin{align}
\label{bilinear-s_a}
\ass : \Hs(\Omega) \times\Hs(\Omega) \to \R \quad 
\ass(z_1,z_2): =
 \int_\Omega \int_\Omega\frac{\big(\nabla z_1(x) - \nabla
   z_1(y)\big)\cdot \big(\nabla z_2(x) - \nabla
   z_2(y)\big)}{|x-y|^{n  + 2 (\mathrm{m} - 1)}}\dd  x \dd y \text{ with } \mathrm{m} > \frac n2\,.
\end{align}
We  shall  denote by 
$\As: \Hs(\Omega) \to \Hs(\Omega)^*$ the associated operator, viz.
\[
 \langle \As(z), w \rangle_{\Hs(\Omega)}  := \ass(z,w) \quad \text{for
every $z,\,w \in \Hs(\Omega)$.}
\]
We recall that $\Hs(\Omega)$ is a Hilbert space with the inner product $\pairing{}{\Hs(\Omega)}{z_1}{z_2}: = \int_\Omega z_1 z_2 \dd x + \ass(z_1,z_2)$. Since we assume $\mathrm{m}>\tfrac n2$, we have the compact embedding $\Hs(\Omega)\Subset \mathrm{C}(\overline\Omega)$. 
\par
Whenever working with a real function $v$ defined on a space-time cylinder $\Omega\times (0,T)$ and differentiable w.r.t.\ time a.e.\ on $\Omega\times (0,T)$, we  shall  denote by $\dot{v}:\Omega\times (0,T)\to\R$ its (almost everywhere defined) partial time derivative. However, as soon as we consider $v$ as a (Bochner) function from $(0,T)$ with values in a suitable Lebesgue/Sobolev space $X$
 (with the Radon-Nikod\'ym property)  and $v$ is in  the space $\AC ([0,T];X)$,  we  shall  denote by $v':(0,T) \to X$ its  (almost everywhere defined)  time derivative.
\par
 Finally, we shall use the symbols
$c,\,c',\, C,\,C'$, etc., whose meaning may vary even within the same   line,   to denote various positive constants depending only on
known quantities. Furthermore, the symbols $I_i$,  $i = 0, 1,... $,
will be used as place-holders for several integral terms (or sums of integral terms) appearing in
the various estimates: we warn the reader that we  shall  not be
self-consistent with the numbering, so that, for instance, the
symbol $I_1$ will occur several times with different meanings.  
\end{notation}
Let us recall some basic facts about the space  $\BD(\Omega)$ of  \emph{functions of bounded deformations},   defined by
\begin{equation}
\BD(\Omega): = \{ u \in L^1(\Omega;\R^n)\, : \ \sig{u} \in  \Mb(\Omega;\Mnn)  \},
\end{equation}
 where $\Mb(\Omega;\Mnn) $ is   the space of  bounded Radon measures on $\Omega$ with values in $\Mnn$,  with norm $\| \lambda\|_{\Mb(\Omega;\Mnn) }: = |\lambda|(\Omega)$ and  $|\lambda|$ the variation of the measure. 
Recall that, by the Riesz representation theorem, 
 $\Mb(\Omega;\Mnn) $ can be identified with the dual of the space $\mathrm{C}_0(\Omega;\Mnn )$. The space $ \BD(\Omega)$ is endowed with the graph norm
\[
\| u \|_{\BD(\Omega)}: = \| u \|_{L^1(\Omega;\R^n)}+ \| \sig{u}\|_{\Mb(\Omega;\Mnn) },
\]
which makes it a Banach space. It turns out that  $\BD(\Omega)$ is the dual of a normed space, cf.\ \cite{Temam-Strang80}.
\par
In addition to the strong convergence induced by $\norm{\cdot}{\BD(\Omega)}$,  the duality 
from \cite{Temam-Strang80} 
defines  a notion of weak$^*$ convergence on $\BD(\Omega)$: 
a sequence $(u_k)_k $ converges weakly$^*$ to $u$ in $\BD(\Omega)$ if $u_k\weakto u$ in $L^1(\Omega;\R^n)$ and $\sig{u_k}\weaksto \sig{u}$ in $ \Mb(\Omega;\Mnn)$.  The space $\BD(\Omega)$ is contained in $L^{n/(n{-}1)}(\Omega; \R^n)$;  every bounded sequence in $\BD(\Omega)$ has  a weakly$^*$ converging subsequence and, furthermore, a subsequence converging weakly in $L^{n/(n{-}1)}(\Omega;\R^n)$ and strongly in $L^{p}(\Omega;\R^n)$  for every $1\leq p <\frac n{n-1}$. 
\par
 Finally, we recall that for  every $u \in \BD(\Omega)$  the trace $u|_{\partial\Omega}$ is well defined as an element in $L^1(\partial\Omega;\R^n)$, and that (cf.\ \cite[Prop.\ 2.4, Rmk.\ 2.5]{Temam83}) a Poincar\'e-type inequality holds:
\begin{equation}
\label{PoincareBD}
\exists\, C>0  \ \ \forall\, u \in \BD(\Omega)\, : \ \ \norm{u}{L^1(\Omega;\R^n)} \leq C \left( \norm{u}{L^1(\Gamma_\Dir;\R^n)} + \norm{\sig u}{\rmM(\Omega;\Mnn)}\right)\,.
\end{equation}

\subsection{Assumptions and preliminary results}
\label{ss:2.1}
 \paragraph{\bf The reference configuration}
%
Let $\Omega \subset \R^n$,  $n\in \{2,3\}$,   be a bounded  Lipschitz  domain.  
 The minimal assumption for our analysis is that $\Omega$ is a \emph{geometrically admissible multiphase domain} in the sense of \cite[Subsection~1.2]{FraGia2012} with only \emph{one phase}, that is $i=1$ therein, where $(\Omega_i)_i$ is a partition corresponding to\ the phases. Referring still to \cite{FraGia2012}, this corresponds to   assuming  that $\partial|_{\partial\Omega} \Gdir$ is \emph{admissible} in the sense of \cite[(6.20)]{FraGia2012}. As observed in \cite[Theorem~6.5]{FraGia2012}, a sufficient condition  for this  is the so-called \emph{Kohn-Temam condition}, that we recall below and assume throughout the paper: 
\begin{equation}
\label{Omega-s2}
\tag{2.$\Omega$}
\begin{gathered}
\partial \Omega = \Gdir \cup
\Gneu \cup \Sigma \quad \text{ with $\Gdir, \,\Gneu, \, \Sigma$ pairwise disjoint,}
\\
\text{
 $\Gdir$ and $\Gneu$ relatively open in $\partial\Omega$, and $ \partial\Gdir = \partial \Gneu = \Sigma$ their relative boundary in $\partial\Omega$,}
  \\
  \text{ with $\Sigma$ of class $\mathrm{C}^2$ and $\calH^{n-1}(\Sigma)=0$, and with $\partial\Omega$ Lipschitz and of class $\mathrm{C}^2$ in a neighborhood of $\Sigma$.}
  \end{gathered}
  \end{equation} 
We  shall  work with the  spaces 
\[
H_\Dir^1(\Omega;\R^n) := \left\{ u\in H^1(\Omega;\R^n) \colon u=0 \ \text{on}\ \Gdir \right\}
\]
 and 
\begin{equation}\label{2809191333}
\widetilde{\Sigma}(\Omega):= \{ \sigma \in \Lnn \colon \mathrm{div}(\sigma)  \in L^2(\Omega;\R^n) \}\,.
\end{equation}
For $\sigma \in \widetilde{\Sigma}(\Omega)$ one may define the distribution $[\sigma \rmn]$ on $\partial \Omega$ by
\begin{equation}\label{2809192054}
\langle [\sigma \rmn] , \psi \rangle_{\partial\Omega}:= \langle \mathrm{div}(\sigma) , \psi \rangle_{L^2} + \langle \sigma , \rmE(\psi) \rangle_{L^2} 
\end{equation}
for $\psi \in H^1(\Omega;\R^n)$. It is known (see e.g.\ \cite[Theorem~1.2]{Kohn-Temam83} or \cite[(2.24)]{DMDSMo06QEPL}) that $[\sigma \rmn] \in H^{-1/2}(\partial \Omega; \R^n)$ and that if $\sigma \in \mathrm{C}^0(\Omega;\Mnn)$  the distribution  $[\sigma \rmn] $ coincides with $ \sigma \rmn$,  that is 
the pointwise product matrix-normal vector in $\partial \Omega$. 
 With  each $\sigma \in \widetilde{\Sigma}(\Omega)$ we associate
an elliptic operator 
 in  $H_\Dir^1(\Omega;\R^n)^*$  denoted by $-\mathrm{Div}(\sigma)$ and defined  by 
%
\begin{equation}
\label{div-Gdir}
\langle -\mathrm{Div} (\sigma) , v \rangle_{ H_\Dir^1(\Omega;\R^n) }  : = \langle - \mathrm{div}(\sigma),   v \rangle_{L^2(\Omega; \R^n )}  + \langle [\sigma \rmn], v \rangle_ {H^{1/2}(\partial\Omega; \R^n)} = \int_\Omega \sigma : \sig v \dd x 
\end{equation}
for all $v \in H_\Dir^1(\Omega;\R^n)$, where the equality above is an integration by parts formula based on  the  Divergence Theorem. 
\paragraph{\bf The elasticity and viscosity tensors}
We assume that the elastic tensor $\C : [0,+\infty) \to   \Lin(\Mnn;\Mnn)  $
 fulfills  the following conditions
\begin{align}
&
\tag{$2.\mathbb{C}_1$}
 \C \in \mathrm{C}^{1,1}([0,+\infty);  \Lin(\Mnn;\Mnn))\,,
 \label{spd}
\\
&
\tag{$2.\mathbb{C}_2$}
 z  \mapsto \C(z) \xi : \xi \ \text{is nondecreasing for every}\ \xi \in \Mnn\,,   \label{C3} \\ 
&
\tag{$2.\mathbb{C}_3$}
\exists\, \gamma_1,\, \gamma_2>0  \ \ \forall\, z \in [0,+\infty)  \ \forall\, \xi \in \Mnn \, : \quad  \gamma_1 |\xi |^2  \leq \C(z) \xi : \xi \leq \gamma_2 |\xi |^2  \,, \label{C2}
\\
&
\tag{$2.\mathbb{C}_4$}
\label{deviat}
\C(z)\xi:=\C_{\mathrm{D}}(z)\xi_{\mathrm{D}}+\kappa(z)(\text{tr}\,\xi)I \quad \text{with  }\C_{\mathrm{D}} \in L^\infty (0,1; \Sym(\MD;\MD))\,, \ \kappa \in L^\infty (0,1)\,.
\end{align}
Again, observe that \eqref{deviat}  is relevant for the perfectly plastic damage system, only.
Even in that context, \eqref{deviat}  is not needed for the analysis, but it is just assumed for mechanical reasons, since purely volumetric deformations do not
affect plastic behavior. 
\par
We introduce the stored elastic energy
$\calQ: L^2(\Omega;\Mnn) \times \rmC^0(\overline\Omega) \to \R$ 
\begin{equation}
\label{elastic-energy}
\calQ (z,e) := \frac12 \int_\Omega \bbC(z) e: e \dd x\,.
\end{equation}
\par
As for the viscosity tensor $\mathbb{D}$, we require that 
\begin{align}
\label{visc-tensors-1} 
\tag{$2.\bbD_1$}
&
  \bbD
  \in \mathrm{C}^0(\overline{\Omega};  \Sym(\MD;\MD))\,,  \text{ and } 
  \\
&
\label{visc-tensors-2} 
\tag{$2.\bbD_2$}
 \ \exists\,  \delta_1,\, \delta_2>0  \  \forall x \in \Omega \   \ \forall A \in \mt_\sym^{d\times d} \, : \quad \ \delta_1 |A|^2 \leq  \bbD(x) A : A \leq  \delta_2 |A|^2,
  \end{align}
  For later use, we introduce the dissipation potential 
  \begin{equation}
  \label{visc-dissip-u}
  \disv 2{\nu}(v) : =   \frac{\nu}{2}  \int_\Omega \bbD \sig{v} : \sig{v} \dd x\,.
  \end{equation}
   Throughout the paper, we shall use that $\bbD$ induces an equivalent (by a Korn-Poincaré-type inequality)
   Hilbert norm on $H_\Dir^1(\Omega;\R^n)$, namely
\begin{equation}
\label{norma-equivalente}
\| u \|_{H^1, \bbD} : = \left(  \int_{\Omega} \bbD \sig{u} : \sig{u} \dd x \right)^{1/2} \quad\text{with }
\| u \|_{H^1, \bbD} \leq K_{\bbD} \| \sig{u}\|_{L^2} \qquad \text{for } u \in  H_\Dir^1(\Omega;\R^n),
\end{equation}
and the
 ``dual norm''
\begin{equation}
\label{norma-duale}
\| \eta \|_{(H^1, \bbD)^*}: = \left( \int_\Omega \bbD^{-1} \xi : \xi \right)^{1/2}  \text{ for all } \eta \in H_\Dir^1(\Omega;\R^n)^* \text{ with } \eta = \mathrm{Div}(\xi)  \text{ for some } \xi \in \widetilde{\Sigma}(\Omega)\,. 
\end{equation} 
\paragraph{\bf The overall mechanical energy.}
Besides the elastic energy $\calQ$ from \eqref{elastic-energy}
and the regularizing, nonlocal gradient contribution featuring the bilinear form $\ass$, the mechanical energy functional shall feature a further term acting on the damage variable $z$, with density $W$ satisfying 
\begin{align}
&
\tag{$2.W_1$}
W\in \rmC^{2}((0,+\infty); \R^+)\cap  \rmC^0([0,+\infty); \R^+{\cup}\{+\infty\})\,,\label{D1}\\
&
\tag{$2.W_2$}
s^{2n} W(s)\rightarrow +\infty \text{ as }s\rightarrow 0^+\,. \label{D2}
\end{align}
 In particular, it follows from \eqref{D2} that 
$W(z)\uparrow +\infty$ if $z\down 0$.  
Clearly, these requirements on $W$ force  $z$ to be strictly  positive (cf.\ also the upcoming Remark \ref{rmk3.2});  consequently, the material never reaches the most damaged state at any point.
We also have the contribution of  a time-dependent loading   $F:   [0,T]  \to  H^1(\Omega;\R^n)^* $,   specified in \eqref{body-force} below,
which subsumes  the volume and the surface forces $f$ and $g$. 
All in all, the  energy functional driving   the rate-dependent and rate-independent systems  \emph{with hardening}  
is
 $\enen{\mu} \colon [0,T] \times H_\Dir^1(\Omega;\R^n) \times \Hs(\Omega) \times L^2(\Omega;\MD) \to \R  \cup \{+\infty\} $,  defined  for $\mu > 0$  by 
  \begin{equation}
 \label{RIS-ene}
 \begin{aligned}
\enen{\mu} (t,u,z,p): =  \calQ( z,\sig{u{+}w(t)} {-}p)   + \int_\Omega \left(  W(z) {+} \GGG \frac{\mu}{2} |p|^2 \right) \,\mathrm{d}x 
+ \frac12 \ass(z,z) - \langle F(t), u+w(t) \rangle_{H^1(\Omega;\R^n)}
\end{aligned}
\end{equation}
with $w$ the time-dependent Dirichlet loading specified in \eqref{dir-load} ahead. 
\begin{remark}
\upshape
The structure of $\enen{\mu}$ reflects the way we impose the  time-dependent  Dirichlet boundary condition on $u$, 
cf.\ \eqref{viscous-bound-cond}.
 Indeed, $\enen{\mu} (t,\cdot, p,z)$  depends  
 on the current displacement at time $t$, given by $u{+}w(t)$, with
 $ u\in H_\Dir^1(\Omega;\R^n) $ and 
  $u{+}w(t)=w(t)$ on $\Gdir$.  
Notice that the dissipation potential  
$\disv 2\nu$
 from \eqref{visc-dissip-u} will instead be calculated at $u'$ (and not at $u'{+}w'(t)$), see
   \eqref{visc-prob} and 
 \eqref{overall-dissip}.  
These choices allow us to observe that system \eqref{RD-intro} has a gradient structure, 
 cf.\ \eqref{abs-gen-gradsyst} ahead. 
It would be possible to slightly change the problem by considering a dissipation potential featuring
$\disv 2\nu(u'{+}w'(t))$.  However, in that case we would loose the  gradient structure of the viscous system
 illustrated in  the upcoming Section \ref{s:grsyst}.  This structure  is 
at the heart of our vanishing-viscosity analysis. That is why, we have decided not to pursue this path. 
\end{remark}
\paragraph{\bf{The plastic  dissipation potential and the overall  plastic dissipation functional}} 
The plastic dissipation potential reflects the constraint that the admissible stresses belong to given constraint sets.  In turn, such sets
depend on the damage variable $z$: this, and the $z$-dependence of the matrix $\mathbb{C}(z)$ of elastic coefficients, provides a strong coupling between the plastic and the damage
flow rules. 
More precisely, in a softening framework, along the footsteps of 
\cite{Crismale-Lazzaroni} we require that the constraint sets $(K(z))_{z\in [0,+\infty)}$   fulfill
\begin{align}
&
\label{Ksets-1}
\tag{$2.K.1$}
K(z) \subset \MD \text{ is closed and convex for all }  z\in [0,+\infty),
\\
&
\label{Ksets-2}
\tag{$2.K.2$}
\exists\,  0<\bar{r} <  \bar{R} \quad   \forall\, 0\leq z_1\leq z_2  \, : \qquad B_{\bar r} (0) \subset K(z_1)\subset K(z_2) \subset   B_{\bar{R}} (0),  
\\
&
\label{Ksets-3}
\tag{$2.K.3$}
\exists\, C_K>0 \quad \forall\, z_1,\, z_2 \in   [0,+\infty)\, :   \qquad d_{\mathscr{H}} (K(z_1),K(z_2)) \leq C_K  |z_1{-}z_2|, 
\end{align}
with 
$d_{\mathscr{H}}$ the Hausdorff distance between two subsets of $\MD$, defined by
\[
 d_{\mathscr{H}} (K_1,K_2): =  \max \left( \sup_{x\in K_1} \mathrm{dist}(x,K_2), \, \sup_{x\in K_2} \mathrm{dist}(x,K_1) \right),
\]
and  $\mathrm{dist}(x,K_i): = \min_{y\in K_i} |x-y|$,  $i=1,2$.
We now introduce the support function  $H:[0,+\infty)  \times  \MD \to [0,+\infty)$  defined by
\begin{equation}\label{0307191050}
H(z,\pi): = \sup_{\sigma \in K(z)} \sigma : \pi \qquad \text{for all } (z,\pi) \in [0,+\infty)\times  \MD \,.
\end{equation}
It was shown in \cite[Lemma 2.1]{Crismale-Lazzaroni} that,  thanks to 
\eqref{Ksets-1}--\eqref{Ksets-3}, $H$ enjoys the following properties:
\begin{subequations}
\label{propsH}
\begin{align}
&
\label{propsH-1}
H \text{ is continuous},
\\
&
\label{propsH-2}
0\leq H(z_2,\pi) - H(z_1,\pi)  \qquad \text{for all } 0 \leq z_1\leq z_2   \text{ and all } \pi \in \MD \text{ with } |\pi|=1\,,
\\
&
\label{propsH-2+1/2}
\exists\, C_K>0 \ \forall\, z_1,\, z_2 \in [0,+\infty) \ \forall\,  \pi \in \MD 
\qquad |H(z_2,\pi) - H(z_1,\pi) |\leq  C_K |\pi| |z_2{-}z_1| \,, 
\\
&
\label{propsH-3}
\pi \mapsto H(z,\pi) \text{ is convex and $1$-positively homogeneous for all } z \in [0,1]\,,
\\
&
\label{propsH-4}
\bar{r} |\pi| \leq H(z,\pi) \leq \bar{R}|\pi|\,.
\end{align}
\end{subequations}
%
%
%
As observed in  \cite{Crismale-Lazzaroni}, properties   \eqref{Ksets-1}--\eqref{Ksets-3} are satisfied by 
constraint sets  in the  ``multiplicative form''   $K(z)=V(z)K(1)$, with $V \in \rmC^{1,1}([0,+\infty))$ non-decreasing and such that $\overline{m}\leq V(z)\leq \overline{M}$ for all $z\in [0,+\infty)$ and some $\overline{m},\, \overline{M}>0$. 
\par
The \emph{plastic dissipation  potential} $\mathcal{H} \colon \rmC^0(\overline{\Omega}; [0,+\infty))  \times  L^1(\Omega;\MD) 
\to \mathbb{R}$ is defined by
\begin{equation}
\label{plast-diss-pot-visc}
\mathcal{H}(z,\pi):= \int_{ \Omega} H(z(x),\pi(x))\dd x \,.
\end{equation}  
Clearly, 
it follows from 
\eqref{propsH-1}--\eqref{propsH-4} that 
\begin{subequations}
\label{props-calH}
\begin{align}
&
\label{props-calH-1}
\pi \mapsto \mathcal{H}(z,\pi) \,\,\text{is convex and positively one-homogeneous for every  } z \in \rmC^0(\overline{\Omega};[0,+\infty)),
\\
&
\label{props-calH-2}
\bar{r} \|\pi\|_1\leq \mathcal{H}(z, \pi) \leq \bar{R} \|\pi\|_1 \quad \text{for all } z \in \rmC^0(\overline{\Omega};[0,+\infty)) \text{ and } \pi \in  L^1(\Omega;\MD),
\\
&
\label{props-calH-3}
 0\leq \mathcal{H}(z_2, \pi) - \mathcal{H}(z_1, \pi)  \quad
\text{for all }  z_1\leq z_2 \in  \rmC^0(\overline{\Omega};[0,+\infty)) \quad \text{for all } \pi \in  L^1(\Omega;\MD)\,,
\\
&
\label{props-calH-3+1/2}
\begin{aligned}
&
\left| \mathcal{H}(z_2, \pi) - \mathcal{H}(z_1, \pi) \right| \leq C_K \|z_1{-}z_2\|_{L^\infty(\Omega)} \| \pi \|_{1}  \quad 
\text{for all }  z_1, z_2 \in  \rmC^0(\overline{\Omega};[0,+\infty)),\,  \pi \in  L^1(\Omega;\MD) \,.
 \end{aligned} 
\end{align}
\end{subequations}
 Let us introduce the set
\begin{equation}\label{2909191029}
\widetilde{\mathcal{K}}_z(\Omega):= \{ \sigma \in \widetilde{\Sigma}(\Omega) \colon \sigma_\dev(x) \in K(z(x)) \text{ for a.e.}\ x \in \Omega\} \,.
\end{equation}
 By standardly approximating   (in the  $L^1$-norm) $\pi$ by  piecewise constant functions, we show that if $z\mapsto K(z)$ is constant,
namely $K(z) \equiv \ol K\subset \MD$,  then 
 \begin{equation}\label{2909191033}
 \calH(z, \pi)=\sup_{\sigma \in \widetilde{K}_z(\Omega)} \langle \sigma_D , \pi \rangle_{L^1}\,. 
 \end{equation} 
 For  a general map   $z\mapsto K(z)$ the argument in \cite[Theorem~3.6 and Corollary~3.8]{Sol09} shows that \eqref{2909191033} still holds.

The convex analysis subdifferential
$\partial_\pi \calH \colon \rmC^0(\overline{\Omega}; [0,+\infty)) \times L^1(\Omega;\MD)\rightrightarrows L^\infty(\Omega;\MD)$, given  by
\[
\omega \in \partial_\pi \calH(z, \pi) \quad \text{ if and only if } \quad  \calH(z,\varrho)- \calH(z,\pi) \geq \int_\Omega \omega (\varrho - \pi) \dd x  \quad \text{for all } \varrho \in L^1(\Omega;\MD)
\]
fulfills
\begin{equation}
\label{subdiff-calH}
\omega \in \partial_\pi \calH(z, \pi) \quad \text{ if and only if } \quad  \omega(x)\in \partial H(z(x),\pi(x)) \quad \foraa\, x \in \Omega\,.
\end{equation} 
 %
 \par
The rate-dependent system \eqref{RD-intro}  with the viscously regularized plastic flow rule 
\eqref{flow-rule-plast-intro-RD}  features 
the dissipation potential $ \calHt\nu:    \rmC^0(\overline{\Omega}; [0,+\infty))  \times L^2(\Omega;\MD) \to [0,+\infty)$ defined by
 \begin{equation}
 \label{visc-diss-pot-plast}
 \calHt\nu(z,\pi): =
  \calH(z,\pi) +\dish 2\nu(\pi)  \qquad \text{with } \dish 2\nu(\pi): =  \frac{\nu}2 \| \pi \|_{L^2(\Omega)}^2\,. 
 %
 \end{equation}
By the sum rule for 
 convex analysis subdifferentials 
 (cf.\ e.g.\ \cite[Corollary IV.6]{Aubin-Ekeland}), the subdifferential
 $\partial_\pi \calH \colon  \rmC^0(\overline{\Omega}; [0,+\infty)) \times   L^2(\Omega;\MD) \rightrightarrows L^2(\Omega;\MD) $ is given by 
 \begin{equation}
 \label{sum-rule-calH}
\partial_\pi\calHt\nu(z,\pi) = \partial_\pi \calH(z,\pi) +  \{ \nu\pi \}  \qquad \text{for all } \pi \in L^2(\Omega;\MD) \text{ and for all } z \in \rmC^0(\overline{\Omega}; [0,+\infty))\,.
 \end{equation} 
 \RRR
 \paragraph{\bf The damage dissipation potential.}
 We consider the damage dissipation density $ \mathrm{R} : \R \to [0,+\infty] $ defined by
\[ 
  \mathrm{R}(\zeta) : = P(\zeta) + I_{(-\infty,0]}(\zeta) 
 \qquad \text{with } 
 P(\zeta): = -\kappa \zeta, \quad \text{and} \quad
 I_{(-\infty,0]} \text{ the indicator function of } (-\infty,0]\,, 
  \]
so that
 \[
 \mathrm{R}(\zeta): = \begin{cases}
 -\kappa \zeta & \text{if } \zeta \leq 0,
 \\
 +\infty & \text{otherwise.}
 \end{cases}
 \]
 With $\mathrm{R}$ 
 we associate the dissipation  potential
$
 \calR : L^1(\Omega) \to [0,+\infty] $ defined by $ \calR(\zeta) : = \int_{\Omega}\mathrm{R}(\zeta(x)) \dd x$. 
 In fact, since the flow rule for the damage variable will be posed in $\Hs(\Omega)^*$ (cf.\ \eqref{1509172300} ahead),
 it will be convenient to consider the restriction of  $\calR$ to the space $\Hs(\Omega)$
 which, 
with a slight abuse of notation, we  shall  denote by the same symbol, namely
 \begin{equation}
 \label{dam-diss-pot}
 \calR: \Hs(\Omega) \to [0,+\infty], \quad \calR(\zeta) = \int_\Omega \mathrm{R}(\zeta(x)) \dd x  = \calP(\zeta) +\calI(\zeta) \ \text{with } \begin{cases}
 \calP(\zeta): = \int_\Omega P(\zeta(x)) \dd x, 
 \\
 \calI(\zeta): = \int_\Omega I_{(-\infty,0]}(\zeta(x)) \dd x 
 \end{cases}
 \end{equation}
 \par
 The viscously regularized damage flow rule \eqref{flow-rule-dam-intro-RD} in fact features the dissipation potential
 \begin{equation}
 \label{visc-diss-pot-dam}
 \calRt:  \Hs(\Omega) \to [0,+\infty], \qquad \calRt(\zeta) =   \calR(\zeta) +
 \calR_2(\zeta) \qquad \text{with } \calR_2(\zeta) := 
  \frac12 \|\zeta\|_{L^2(\Omega)}^2\,.
 \end{equation}
 We shall denote by $\partial \calR : \Hs(\Omega) \rightrightarrows  \Hs(\Omega)^*$ and $\partial \calRt : \Hs(\Omega) \rightrightarrows  \Hs(\Omega)^*$ the  subdifferentials of $\calR$ and $\calRt$
 in the sense of convex analysis. 
 Observe that $\mathrm{dom}(\partial\calR) = \mathrm{dom}(\partial\calRt) = \Hsm$. We shall provide explicit formulae for both subdifferentials in Lemma \ref{l:expl-subd} at the end of this section. 

 \paragraph{\bf  The initial data, the  body forces, and  the Dirichlet loading.}
 We shall consider initial data
\begin{equation}
\label{init-data} 
u_0 \in  H_\Dir^1(\Omega;\R^n),  \qquad z_0 \in \Hs(\Omega) \text{ with } W(z_0) \in L^1(\Omega) \text{ and } z_0 \leq 1 \text{ in } \ol\Omega,  \qquad p_0 \in L^2(\Omega;\MD). 
\end{equation}
 The assumptions that we require on the  volume and surface  forces depend on the type of plasticity considered. In the analysis of  systems with hardening 
we may assume less regularity on the body forces, while the study of the perfectly-plastic damage  system
\eqref{RIS-intro}
 hinges on further regularity  and on a \emph{uniform safe-load condition}. 
\par
Hence, for the analysis of systems with hardening in Sections~\ref{s:exist-viscous} and \ref{s:van-visc}, the conditions assumed on the volume force $f$ and the assigned traction $g$ are 
 \begin{subequations} \label{hyp-data}
 \begin{equation}\label{1808191033}
 f\in H^1(0,T;  L^2(\Omega;\R^n) ), \qquad g \in H^1(0,T; H^{1/2} (\Gamma_\Neu;\R^n)^*);
 \end{equation}
 to shorten notation,  we  shall  often incorporate the forces $f$ and  $g$
  into the induced  total load, namely the function
$F\colon [0,T] \to H^1(\Omega;\R^n)^*$ defined at $t\in (0,T)$ by 
\begin{equation}
\label{body-force}
\langle F(t), v \rangle_{H^1(\Omega;\R^n)}: = \langle f(t), v\rangle_{ L^2(\Omega;\R^n)} + \langle g (t), v \rangle_{H^{1/2}(\Gamma_\Neu;\R^n)} 
\end{equation} 
for all $v \in H^1(\Omega;\R^n)$.
\par
 Conversely, for the treatment of the perfectly-plastic damage system of Section~\ref{s:van-hard}, we require that    
 \begin{equation}
\label{forces-u}
f \in H^1(0,T;  L^n(\Omega;\R^n) ) \,, \quad g \in H^1(0,T; L^\infty(\Gneu;\R^n))\,,
\end{equation}
 so that $F$ turns out to take values in $ \BD(\Omega)^*$, defining 
\[
\langle F(t), v \rangle_{\BD(\Omega)}: = \langle f(t), v\rangle_{ L^{n/(n{-}1)}(\Omega;\R^n)} + \langle g (t), v \rangle_{L^1(\Gamma_\Neu;\R^n)}
\]
for all $v \in \BD(\Omega)$ (recall the first properties of $\BD(\Omega)$ in Section~\ref{s:2} and the fact that $L^\infty(\Omega; \R^n)=L^1(\Omega; \R^n)^*$,   since $\Omega$ has finite measure).   Both for the analysis of the system with hardening and of the perfectly plastic one,  
 we  shall assume a \emph{uniform safe load condition}, namely that  there exists 
\begin{equation}\label{2909191106}
\rho \in H^1(0,T; \Lnn) \quad\text{ with }\quad\rho_\dev \in H^1 (0,T;\Linftyn)
\end{equation} and there exists $\alpha>0$ such that  for every $t\in[0,T]$  (recall \eqref{2809192054})
\begin{equation}\label{2809192200}
-\mathrm{div}(\varrho(t))=f(t) \text{ a.e.\ on }\Omega\,,\quad\qquad [\varrho(t) \rmn]= g(t) \text{ on } \Gneu\,.
\end{equation}
\begin{equation}
\label{safe-load}
\rho_\dev(t,x) +\xi \in K  \qquad \text{for a.a.}\ x\in\Omega \ \text{and for every}\ \xi\in\Mnn \ \text{s.t.}\ |\xi|\le\alpha\,.
\end{equation} 
 Assumption \eqref{forces-u}
 will be crucial in the derivation of \emph{a priori} uniform estimates with respect to the parameter $\mu$ in Proposition \ref{prop:energy-est-SL},  while with \eqref{1808191033} the estimates 
 would 
depend on $\mu >0$,  
cf.\ also Remark \ref{sta-in-rmk} ahead. 
Combining  \eqref{1808191033} with  \eqref{2909191106}--\eqref{safe-load} gives  $-\mathrm{Div}(\varrho(t))=F(t)$ 
in $H_\Dir^1(\Omega,\R^n)^*$, while if \eqref{forces-u} holds  then    \eqref{2909191106}--\eqref{safe-load}  yield  
$-\Diver(\varrho(t))=F(t)$ for all $t \in [0,T]$ (where the operator  $-\Diver$  will be introduced in \eqref{2307191723}). 
%
\end{subequations}
 For later use,   we notice that, thanks to \eqref{2909191033} it is easy to deduce that for all $t \in [0,T]$
\begin{equation}\label{controllo-SL}
\mathcal{H}( z,p)   
- \int_\Omega \rho_\dev(t) p \dd x 
\ge \alpha \|p\|_{L^1(\Omega)}  \,.
\end{equation}

\par
As for the time-dependent Dirichlet loading $w$, we shall require that
\begin{equation}
\label{dir-load} 
 w \in H^1(0,T;H^1(\R^n;\R^n)).  
\end{equation}
%
\begin{remark}\upshape
\label{rmk:on-w}
In fact, the analysis of the \emph{rate-independent} system for damage and plasticity, with or without hardening, would just require 
 $w\in \AC ([0,T];H^1(\R^n;\R^n))$  so that, upon taking the vanishing-viscosity limit as $\eps \down 0$ of system \eqref{RD-intro}, we could approximate a loading   $w\in \AC ([0,T];H^1(\R^n;\R^n))$  with a sequence $(w_\eps)_\eps \subset H^1(0,T;H^1(\R^n;\R^n))$.  
 The same applies to the time regularity of the forces.  However, 
to avoid overburdening the exposition
we have preferred not to pursue this path. 
\end{remark} 
\begin{remark}[Rewriting the driving energy functional]
\label{rmk:rewriting}
\upshape
By the safe-load condition  \eqref{2809192200} and the integration by parts formula in \eqref{div-Gdir} applied to $u \in H^1_\Dir(\Omega;\R^n)$, 
$\enen{\mu}$  rewrites as
\[
\begin{aligned}
\enen{\mu} (t,u,z,p)   =  \calQ(z,  e(t)) + \int_\Omega \left(  W(z) {+} \GGG \frac{\mu}{2}  |p|^2 \right) \,\mathrm{d}x + \frac12 \ass(z,z)  -\int_\Omega \rho(t) \rmE(u) \dd x  - \langle F(t), w(t) \rangle_{H^1(\Omega; \R^n)} 
\end{aligned}
\]
where we have highlighted the \emph{elastic part} of the strain tensor $\sig{u{+}w(t)}$,
\begin{equation}
\label{elastic-part}
 e(t)  : =\sig{u{+}w(t)} -p\,.
\end{equation}
We now introduce  the functional
\begin{equation}
\label{def-Fnu}
 \calF_\mu (t,u,z,p)  : =   \calQ(z,e(t)) + \int_\Omega \left(  W(z) {+} \GGG \frac{\mu}{2}  |p|^2 \right) \,\mathrm{d}x + \frac12 \ass(z,z) -\int_\Omega \rho(t) (e(t)  - \rmE(w(t))  \dd x   - \langle F(t), w(t) \rangle_{H^1(\Omega; \R^n)}  \,.
\end{equation}
Then,
taking into account  that $\int_\Omega (\rho{-}\rho_\dev) p \dd x =0$, we have
\begin{equation}
\label{serve-for-later}
\enen{\mu} (t,u,z,p)  
 = \calF_\mu (t,u,z,p)   - \int_\Omega \rho_\dev(t) p \dd x\,.
\end{equation}
\end{remark} 
\par
In the following result we clarify 
the expression of the subdifferentials $\partial \calR$ and $\partial \calRt$;
 these basic facts will be useful, for instance,  in the proof of Lemma \ref{lemma:sup}. 
\begin{lemma}
\label{l:expl-subd}
We have the following representation formula for
 the subdifferential 
 $ \partial\calI\colon \Hs(\Omega) \rightrightarrows \Hs(\Omega)^*$ of the functional 
 $\calI$ from \eqref{dam-diss-pot}:   for all $\zeta \in \Hsm: = \{ v\in \Hs(\Omega)\, :  \ v \leq 0 \text{ in }\Omega\}$
\begin{equation}
\label{repre-sub-I}
 \chi  \in  \partial\calI (\zeta) \qquad \text{if and only if} \qquad \langle \chi, w-\zeta \rangle_{\Hs(\Omega)}\leq 0 \qquad \text{for all } w \in \Hsm\,.
\end{equation}
Moreover, for all $\zeta \in \Hsm$ there holds
\begin{align}
& 
\label{sum-rule-1}
\partial \calR(\zeta) =\partial\calP(\zeta) + \partial\calI(\zeta) =  -\kappa+\partial\calI(\zeta)  
\\
&
\label{sum-rule-2}
 \partial \calRt(\zeta) =  \partial \calR(\zeta)  + \{  \zeta\}
\end{align}
where 
$-\kappa$ stands for the functional $\Hs(\Omega) \ni \zeta \mapsto \int_\Omega (-\kappa) \zeta(x) \dd x$, and 
we simply write $\zeta$, in place of $J(\zeta)$ ($J: L^1(\Omega) \to \Hs(\Omega)^*$ denoting the Riesz mapping). 
\end{lemma}
\begin{proof}
Formula \eqref{repre-sub-I} is in fact the definition of $\partial\calI (\zeta)$, whereas \eqref{sum-rule-1}
and \eqref{sum-rule-2} follow
 the sum rule for convex subdifferentials, cf.\ e.g.\ \cite[Cor.\ IV.6]{Aubin-Ekeland}.
\end{proof} 
\section{The gradient structure of the viscous system}
 \label{s:grsyst}
%
In this section we are going to  establish the functional setup
  in which the  (Cauchy problem for the) \emph{rate-dependent} system 
   with hardening (i.e., with $\mu>0$)  
  \eqref{RD-intro} is formulated and, accordingly, specify the notion of solution we are interested in.
  This will enable us to unveil the 
gradient structure underlying system \eqref{RD-intro}, which will have a twofold outcome:
\begin{compactenum}
\item Exploiting this structure we  shall  show that \eqref{RD-intro} can be equivalently reformulated in terms of an Energy-Dissipation inequality, which is
in turn equivalent to an Energy-Dissipation \emph{balance}.
This observation   will simplify  the proof of existence of viscous solutions, carried out in Section~\ref{s:exist-viscous}.
\item 
The Energy-Dissipation \emph{balance} will be at the core of the vanishing-viscosity analysis (with $\nu>0$ fixed) performed in Section~\ref{s:van-visc},
 as well as of the vanishing-hardening analysis carried out in Section~\ref{s:van-hard}. 
\end{compactenum} 
\noindent
 Throughout this section, we  shall  tacitly assume that the  constitutive functions of the model and the problem data comply with the conditions 
 listed in Section~\ref{s:2}, and therefore we  shall  never explicitly invoke them in the statement of the various results. 
System 
 \eqref{RD-intro} involves
 the rescaled dissipation potentials $
 \disv \eps\nu: H^1(\Omega;\R^n) \to [0,+\infty)$,  $\dish\eps\nu:  \rmC^0(\overline{\Omega}; [0,+\infty))  \times L^2(\Omega;\MD)\to [0,+\infty)$, and $\calR_\eps: \Hs(\Omega) \to [0,+\infty]$  defined by
  \begin{equation} \label{dissipazioni-riscalate}
  \begin{gathered}
   \disv\eps\nu(v): = \frac1{\eps} \disv 2\nu(\eps v), \quad 
   \dish\eps\nu(z,\pi) : = \frac1\eps \calHt\nu(z,\eps\pi) = \calH(z,\pi) + \frac1\eps \dish2\nu(z,\eps\pi) , 
   \\
   \calR_\eps(\zeta): =  \frac1\eps \calRt(\eps\zeta) = \calR(\zeta)+  \frac1\eps \calR_2(\eps\zeta) 
   \end{gathered}
   \end{equation}
   with   $\disv 2\nu$, $\dish 2\nu$, and   $\calR_2$,  from \eqref{visc-dissip-u}, 
\eqref{visc-diss-pot-plast}, and \eqref{visc-diss-pot-dam}, respectively. 
 With   $\densh\eps\nu$  we shall denote the density of the integral functional $\dish\eps\nu$. 
 We are now in a position to provide 
the variational formulation of (the Cauchy problem for) system  \eqref{RD-intro}.  Since we  shall  treat
$u$, $z$, and $p$ as  (Bochner) functions from $(0,T)$ with values in their respective  Lebesgue/Sobolev spaces,
as specified in Notation \ref{not:1.1} we shall denote by $u'$, $z'$, and $p'$ their 
(almost everywhere defined)  time derivatives. 
 \begin{problem}
 \label{prob:visc}
 Find  a triple $(u,z,p)$ with
 \begin{equation}
 \label{reg-of-sols}
  u \in  H^1(0,T;H_\Dir^1(\Omega;\R^n)),    \quad z \in H^1(0,T;\Hs(\Omega)) \text{ with } W(z) \in L^\infty (0,T;L^1(\Omega)), \quad p \in  H^1(0,T;L^2(\Omega;\MD)),
 \end{equation}
 such that,    with  
$ e(t):=\sig{u(t)+w(t)} -p(t) $ and $ \sigma(t) : = \mathbb{C}(z(t)) e(t)$,
  there holds
  \begin{subequations}
 \label{visc-prob}
  \begin{align}
&
\label{1509172145-cont}
-\mathrm{Div}\big(\varepsilon \nu  \mathbb{D}\sig{u'(t)} {+}\sigma(t)\big) =  F(t)
&& \text{in } H_\Dir^1(\Omega;\R^n)^*  && \foraa\, t \in (0,T),
\\
&
\label{1509172300}
\partial \mathcal{R}_\eps(z'(t)) + \As(z(t)) + W'(z(t)) 
\ni -\frac12 \mathbb{C}'(z)  e(t) : e (t) 
 && \text{in } 
\Hs(\Omega)^* && \foraa\, t \in (0,T),
\\
&
\label{1509172259}
\partial_{\pi} \densh{\eps}\nu(z(t), p'(t)) 
+\newmu p(t) \ni  \big(\sigma(t)\big)_\dev  && \text{a.e.\ in  } \Omega && \foraa\, t \in (0,T).
\end{align}
\end{subequations}
joint with the initial conditions 
\begin{equation}
\label{initial-conditions-visc}
u(0)=u_0 \text{ in }  H_\Dir^1(\Omega;\R^n),   \qquad z(0) = z_0 \text{ in } \Hs(\Omega), \qquad p(0) =p_0 \text{ in } L^2(\Omega).
\end{equation}
 \end{problem}
 \begin{remark}
 \upshape
 \label{rmk3.2}
 A few observations on formulation \eqref{visc-prob} are in order:
 \begin{enumerate}
 \item    As shown in  \cite[Lemma 3.3]{Crismale-Lazzaroni},  from the requirement $W(z) \in L^\infty(0,T; L^1(\Omega))$
  we deduce the strict positivity property
 \begin{equation}
 \label{strict-posivity}
 \exists\, m_0>0 \quad \forall\, (x,t) \in \overline{\Omega}\times [0,T]\, : \qquad z(x,t) \geq m_0.
 \end{equation}
 \item In view of \eqref{strict-posivity} and of \eqref{D1}, we have that 
$W'(z) \in \mathrm{C}^0(\ol\Omega \times [0,T])$.  The term  featuring in  \eqref{1509172300} has to be understood as 
 the image of $W'(z(t)) \in \mathrm{C}^0(\ol\Omega)$  under the Riesz mapping with values in $\Hs(\Omega)^*$.
 \item By the monotonicity of  $t\mapsto z(x,t)$ and the requirement that $z_0 \leq 1$ in $\ol\Omega$, we immediately infer that
 $z(x,t)\leq 1$ for all $(x,t)\in\ol\Omega \times [0,T]$ which, combined with \eqref{strict-posivity}, is consistent with the physical meaning of the
  damage variable.
  \end{enumerate}
 \end{remark}
The requirement $W(z) \in L^\infty(0,T; L^1(\Omega))$ which, as shown by Remark \ref{rmk3.2}, has  an important impact 
 on the properties of the solution component $z$, is  in turn consistent with the gradient structure of system \eqref{visc-prob}
 with respect to the driving energy $\enen{\mu}$ from \eqref{RIS-ene}.
To reveal this structure, it will be convenient to introduce the following notation for the triple $(u,z,p)$ of
  state variables and the associated state space
\begin{equation}
\label{q-notation}
q: =(u,z,p) \in \bfQ: = H^1_{\Dir} (\Omega;\R^n)\times \Hs(\Omega) \times L^2(\Omega;\MD)\,.
\end{equation}
With slight abuse of notation, we  shall  write both $\enen{\mu} (t,u,z,p)$ and $\enen{\mu} (t,q)$.
The result below fixes some crucial properties of $\enen{\mu}$,
 that will be at the core of the interpretation of 
\eqref{RD-intro} as  a gradient system.
We shall explore further properties of $\enen{\mu}$ in Lemma \ref{l:coercivity-Enu} ahead. 
\begin{lemma}
\label{l:Fr-diff-Enu}
 For every  $\mu>0$
  the proper domain of 
$\enen{\mu} : [0,T] \times \bfQ   \to  \R  \cup \{+\infty\} $ is 
\[
\bfD_T := [0,T] \times \bfD \qquad \text{ with } \bfD=\{(u,z,p) \in \bfQ\, : \ z>0 \text{ in } \ol\Omega\}. 
\]
For all $t\in [0,T]$, the functional $q\mapsto \enen{\mu} (t,q)$ is Fr\'echet differentiable on $\bfD$, with Fr\'echet differential 
\begin{equation}
\label{Fr-diff-q}
\begin{aligned}
\rmD_q  \enen{\mu} (t,q) &  = (\rmD_u \enen{\mu}(t,u,z,p), \rmD_z \enen{\mu}(t,u,z,p), \rmD_p \enen{\mu}(t,u,z,p)) \\ &  = \left(-\mathrm{Div} (\sigma(t)) -F(t), \As(z)+W'(z)+\tfrac12 \mathbb{C}'(z)  e(t) :  e(t) , 
\newmu p - \sigma_\dev(t))\right) \in \bfQ^*.
\end{aligned}
\end{equation}
Furthermore, for all $q\in Q$ the function $t\mapsto \enen{\mu} (t,q)$ is in $\AC([0,T])$, with 
\begin{equation}
\label{part-t-q} 
\partial_t \enen{\mu} (t,q) = \int_\Omega  \sigma(t) : \sig{w'(t)} \dd x -\langle F'(t), u+w(t) \rangle_{H^1(\Omega;\R^n)} -\langle F(t), w'(t) \rangle_{H^1(\Omega)} \quad \foraa\, t \in (0,T).
\end{equation}
Finally,   the following chain-rule property holds:   for all  $q \in H^1(0,T;\bfQ)$ with $\sup_{t\in [0,T]} |\enen{\mu} (t,q(t))|<+\infty$, 
\begin{equation}
\label{chain-rule}
\begin{gathered}
 \quad \text{the mapping } t \mapsto  
 \enen{\mu} (t,q(t)) \text{ is in } \AC ([0,T]), \text{ and}
 \\
 \frac{\dd}{\dd t}  \enen{\mu} (t,q(t)) = \pairing{}{\bfQ}{\rmD_q \enen{\mu} (t,q(t))}{q'(t)} + \partial_t  \enen{\mu} (t,q(t)) \qquad \foraa\, t \in (0,T).
\end{gathered}
\end{equation}
\end{lemma}
\begin{proof}
First of all, \eqref{Fr-diff-q} gives the G\^ateaux differential of $\enen{\mu}(t,\cdot)$:  we  shall  just check the formula for $\rmD_u \enen{\mu}(t,u,z,p)$
by observing that, since $\enen{\mu}(t,\cdot,z,p)$ is convex, we have that $\eta = \rmD_u \enen{\mu}(t,u,z,p)$ if and only if it holds that  $\enen{\mu}(t,v,z,p) - \enen{\mu}(t,u,z,p)\geq \langle \eta, v-u \rangle_{H_\Dir^1(\Omega;\R^n)} $, or, equivalently, that 
\begin{equation}
\label{ultimately}
  \frac12 \int_\Omega \bbC(z) (\sig{v{+}w(t)}{-}p): (\sig{v{+}w(t)}{-}p) \dd x - 
 \frac12 \int_\Omega  \sigma(t) \colon e(t)\dd x  - \langle F(t), v{-}u \rangle_{H_\Dir^1}
 \geq  \langle \eta, v{-}u \rangle_{H_\Dir^1} 
\end{equation}
for all $v \in H_\Dir^1(\Omega;\R^n)$  (using the short-hand notation  $\langle \cdot, \cdot \rangle_{H_\Dir^1}
$).  Ultimately, \eqref{ultimately}  holds true if  and only if  
\[
\langle \eta, \tilde v \rangle_{H_\Dir^1(\Omega;\R^n)} = \int_\Omega \sigma(t) :\sig{\tilde v} \dd x  - \langle F(t), \tilde v \rangle_{H_\Dir^1(\Omega;\R^n)}
\]
for all $\tilde v \in H_\Dir^1(\Omega;\R^n)$.
 In order to check the Fr\'echet differentiability, it is enough to prove the continuity property
\begin{equation}
\label{cont-Gat-diff}
\left( q_n=(u_n,z_n,p_n)\to q = (u,z,p) \text{ in } \bfQ \right) \ \Longrightarrow \ \left(\rmD_q  \enen{\mu} (t,q_n) \to \rmD_q  \enen{\mu} (t,q) \text{ in } \bfQ^*\right)\,.
\end{equation}
For this, we observe that $z_n\to z$ in $\Hs(\Omega)$ implies $z_n\to z$ in $\mathrm{C}^0(\ol\Omega)$ and, thus, $\mathbb{C}(z_n) \to \mathbb{C}(z)$
and $\mathbb{C}'(z_n) \to \mathbb{C}'(z)$
 in $L^\infty(\Omega;\Lin(\Mnn;\Mnn))$. Therefore, we have  $\mathbb{C}(z_n) e_n(t) \to \mathbb{C}(z) e(t)$ in $L^2(\Omega;\Mnn)$, which gives $\rmD_u \enen{\mu} (t,u_n,z,p) \to \rmD_u \enen{\mu} (t,u,z,p)$. 
We also find that $\mathbb{C}'(z_n) e_n(t){:} e_n(t) \to \mathbb{C}'(z) e(t){:} e(t)$ in $L^1(\Omega)$, hence  we have the convergence $ \rmD_z \enen{\mu}(t,u_n,z_n,p_n) \to  \rmD_z \enen{\mu}(t,u,z,p)$ in $\Hs(\Omega)^*$. 
We easily have  $ \rmD_p \enen{\mu}(t,u_n,z_n,p_n) \to \rmD_p \enen{\mu}(t,u,z,p)$ in $L^2(\Omega;\MD)$, which concludes the proof of \eqref{cont-Gat-diff}. 
\par
 By standard arguments we  conclude  \eqref{part-t-q} and  \eqref{chain-rule}. This finishes the proof.
 \end{proof}
Let us now introduce the overall dissipation potential $  \psin\nu: \bfQ \times \bfQ \to [0,+\infty]$
\begin{equation}
\label{overall-dissip}
\begin{gathered}
  \psin\nu(q, q') 
  : = \disv 2\nu(u') + \calRt(z')+ \calHt\nu(z,p')  \\  \text{ and its rescaled version } \psie\eps\nu( q, q') : = \frac1{\eps} \psin\nu(q,\eps q') = \disv \eps\nu(u')+\calR_\eps(z')+ \dish\eps\nu(z,p') \,. 
\end{gathered}
\end{equation}
Taking into account \eqref{Fr-diff-q}, it is then a standard matter to reformulate
Problem~\ref{prob:visc} in these terms: \emph{find $q\in H^1(0,T;\bfQ)$ with $\sup_{t\in (0,T)}|\enen{\mu}(t,q(t))|<+\infty$ solving the generalized gradient system}
\begin{equation}
\label{abs-gen-gradsyst}
\partial_{q'} \psie\eps\nu(q(t), q'(t)) + \rmD_q\enen{\mu} (t,q(t)) \ni 0 \qquad \text{in } \bfQ^*, \quad \foraa\, t \in (0,T).
\end{equation}
\par
This reformulation allows us to easily obtain the Energy-Dissipation balance underlying system \eqref{visc-prob},  which is in fact \emph{equivalent} to \eqref{abs-gen-gradsyst}. 
   Indeed, 
arguing as in
\cite{MRS2013} (this observation is however at the core of the variational approach to gradient flows, cf.\ \cite{AGS08}), 
we observe that \eqref{abs-gen-gradsyst}, namely $- \rmD_q\enen{\mu} (t,q) \in \partial_{q'} \psie\eps\nu(q,q') $,  
is  equivalent, by standard convex analysis results, to the identity
\begin{equation}
\label{FM-bal}
\psie\eps\nu(q(t), q'(t)) + \psie\eps\nu^*(q(t),{-} \rmD_q\enen{\mu} (t,q(t))) = \pairing{}{\bfQ}{{-} \rmD_q\enen{\mu} (t,q(t))}{q'(t)} \quad \foraa\, t \in (0,T),
\end{equation}
with $\psie\eps\nu^*: \bfQ \times  \bfQ^* \to [0,+\infty]$, $\psie\eps\nu^*(q,\xi): = \sup_{v\in \bfQ}(\pairing{}{\bfQ}{\xi}v{-}\psie\eps\nu(q, v))$ the Fenchel-Moreau conjugate of $\psie\eps\nu(q,\cdot)$. By the definition of $\psie\eps\nu^*$, the $\geq$ estimate in \eqref{FM-bal} is automatically verified. Therefore, \eqref{FM-bal} is in fact equivalent to the $\leq$ estimate
\begin{equation}
\label{FM-ineq}
\psie\eps\nu(q(t),q'(t)) {+} \psie\eps\nu^*(q(t),{-} \rmD_q\enen{\mu} (t,q(t)))   {\leq} \pairing{}{\bfQ}{{-} \rmD_q\enen{\mu} (t,q(t))}{q'(t)} 
 {=} {-} \frac{\dd}{\dd t}  \enen{\mu} (t,q(t))  + \partial_t  \enen{\mu} (t,q(t))
\end{equation}
 for a.a.\ $t \in (0,T)$,  where the latter identity follows from the chain rule \eqref{chain-rule}. In fact, it is immediate to check that \eqref{abs-gen-gradsyst} is equivalent to  the
integrated versions of \eqref{FM-bal} and  of \eqref{FM-ineq}. The latter reads  
\begin{equation}
\label{integr-FM-ineq}
\int_0^t \left( \psie\eps\nu(q(r),q'(r)) + \psie\eps\nu^*(q(r),{-} \rmD_q\enen{\mu} (r,q(r))) \right) \dd r +   \enen{\mu} (t,q(t)) \leq \enen{\mu} (0,q(0))+\int_0^t \partial_t  \enen{\mu} (r,q(r)) \dd r 
\end{equation}
for all $ t \in [0,T]$. 
Observe that for $\xi = (\eta, { \chi},\omega) \in \bfQ^*$  we have
%
\begin{equation}
\label{Psi-eps-star}
\begin{aligned}
&
\psie\eps\nu^*(q,\xi)   = \disv\eps\nu^*(\eta) + \calR_\eps^*( { \chi}) + \dish\eps\nu^*(z,\omega)  
\\
& \text{with } \quad \begin{cases}
\disp \disv\eps\nu^*(\eta) =   \frac1{2\eps\nu}  \int_\Omega \bbD^{-1} \tau \colon  \tau \dd x \quad \text{for }  \eta = {-}\mathrm{Div}  (\tau)
\quad  \text{and } \tau \in  \widetilde{\Sigma}(\Omega) 
 \\[1em]
\disp  \calR_\eps^*( { \chi})  = 
  \frac1{2\eps}\tilded_{L^2(\Omega)}^2( { \chi}, \partial \calR(0))
  := \frac1{2\eps} \min_{\gamma \in \partial \calR(0)}   \mathfrak{f}_2 ( { \chi}{-}\gamma) \,, 
    \\[1em]
\disp  \dish\eps\nu^*(z,\omega)  = 
    \frac1{2\eps\nu} \dLtwo^2(\omega,\partial_\pi \calH(z,0)) := \frac1{2\eps\nu} \min_{\rho \in \partial_\pi \calH(z,0)} \| \omega -\rho\|_{L^2(\Omega)}^2 \,,
    \end{cases}
    \end{aligned}
\end{equation}
 where   $\widetilde{\Sigma}(\Omega)$  is from    \eqref{2809191333},  
\[
\mathfrak{f}_2: \Hs(\Omega)^* \to [0,+\infty] \text{ is defined by } \quad \mathfrak{f}_2(\beta): = \begin{cases}
\|\beta\|_{L^2(\Omega)}^2 & \text{if } \beta \in L^2(\Omega) \,,
\\
+\infty & \text{if } \beta \in \Hs(\Omega)^* \setminus L^2(\Omega) \,,
\end{cases}
\]
 and observe that the $\min$ in the definition of $\tilded_{L^2(\Omega)}( { \chi}, \partial \calR(0))$ is attained as soon as $\tilded_{L^2(\Omega)}$ is finite. 
 Indeed, we have calculated
\[
\begin{aligned}
\disv\eps\nu^*({-}\mathrm{Div}(\tau))  & = \sup_{v \in H_{\Dir}^1(\Omega;\R^n)} 
\left( \langle {-}\mathrm{Div}(\tau), v \rangle_{H^1(\Omega;\R^n)}  - \disv\eps\nu(v)\right)
\\
 & =
 \sup_{v \in H_{\Dir}^1(\Omega;\R^n)}  \left( \int_\Omega  \tau  \colon \sig{v} \dd x -  \frac{\eps\nu}2\int_\Omega \bbD \sig{v} \colon \sig{v} \dd x\right)  
= 
\frac1{2\eps\nu} \int_\Omega \bbD^{-1}  \tau  \colon  \tau  \dd x
\end{aligned}
\]
whereas the formulae for $\calR_\eps^*$ and $\dish\eps\nu^*$ follow from the $\inf$-$\sup$ convolution formula, cf.\ e.g.\ \cite[Thm.\ 3.3.4.1]{IofTih79TEPe}.
Therefore, we may calculate explictly the second contribution to the left-hand side of \eqref{integr-FM-ineq}. Indeed,  recalling that, by  \eqref{2809192200}, 
we have $F(t) = -\mathrm{Div}(\rho(t))$, we find that 
\[
\begin{aligned}
\disv\eps\nu^*({-} \rmD_u\enen{\mu} (r,u(r), z(r), p(r)) ) & = \disv\eps\nu^* (\mathrm{Div}(\sigma(r)) {+}F(r))   = \disv\eps\nu^*(\mathrm{Div}(\sigma(r){-}\rho(r)) )\\ &  = \frac1{2\eps\nu} \int_\Omega \bbD^{-1}(\sigma(r){-}\rho(r) ) \colon (\sigma(r){-}\rho(r) ) \dd x\,,
\end{aligned}
\]
  (where  $\sigma(r) = \bbC(z(r)) e(r)$). 
All in all, 
we arrive at the following result, which will play a key role
 for the analysis of 
the rate-dependent 
system \eqref{RD-intro}, 
 since it provides a characterization of solutions to
the viscous Problem \ref{prob:visc}. 
\begin{proposition}
\label{prop:charact}
The following properties are equivalent for a triple  $q=(u,z,p)\in H^1(0,T;\bfQ)$ fulfilling the initial conditions \eqref{initial-conditions-visc}: 
\begin{enumerate}
\item \GGG $q$  is a solution of \GGG Problem~\ref{prob:visc}; 
\item \GGG $q$  fulfills
 the 
  \emph{Energy-Dissipation upper estimate}  
  \begin{equation}
 \label{enineq-plast}
 \begin{aligned}
&  \enen{\mu} (t,\GGG q(t) ) + \int_0^t \left( \disv\eps\nu(u'(r)) {+} \calR_\eps(z'(r)) {+} \dish\eps\nu(z(r), p'(r))\right) \dd r 
\\
  &\quad 
  +  \int_0^t \Big[ \disv\eps\nu^*\Big(\mathrm{Div}(\sigma(r)){+} F(r)\Big)  {+}   \calR_\eps^*\Big({-}\As(z(r)) {-}W'(z(r)){-}\tfrac12 \bbC'(z(r)) e(r) : e(r)\Big) \\  
  &\qquad \qquad \qquad 
  {+}  \dish\eps\nu^*\Big(z(r), -\newmu p(r) + \sigma_\dev(r)\Big)
   \Big] \dd r 
   \leq
  \enen{\mu} (0,\GGG q_0 ) + \int_0^t 
\partial_t \enen{\mu} (r,q(r))  \dd r; 
  \end{aligned}
 \end{equation}
 \item \GGG $q$  fulfills \eqref{enineq-plast} as an
  \emph{Energy-Dissipation balance}, integrated on any interval $[s,t]\subset [0,T]$.  
 \end{enumerate}
 \end{proposition} 
 With the upcoming result we exhibit a further characterization of solutions to the viscous system
 that will be useful 
 for the vanishing-viscosity analyses carried out in Sections 
 \ref{s:van-visc} and \ref{s:van-hard},
  borrowing an idea from \cite{Crismale-Lazzaroni}. 
  Proposition \ref{Vito-Bs}
 indeed
  shows that the 
  \emph{Energy-Dissipation balance}   \eqref{enineq-plast} can be rewritten
    in terms of
  the
 functionals
  \begin{equation}\label{0307191952}
  \begin{aligned}
 & \MVito({t}, {q}, {q}'):= \calR({z}')+\calH({z},{p}') + \MredVito({t}, {q}, {q}')\,, \quad \text{ where }  \quad \MredVito({t}, {q}, {q}'):=\DVito_\nu({q}')\, \DVitos({t},{q}) \quad \text{ with }
   \\
   & 
   \DVito_\nu({q}'): = 
    \sqrt{ \nu \| {u}'(t)\|^2_{ H^1, \bbD }   {+} 
\|{z}'(t)\|_{L^2}^2
{+} \nu\|{p}'(t)\|_{L^2}^2} 
\\
&
 \DVitos({t},{q}): =   \sqrt{
\frac1{\nu}\, \|{-}\mathrm{D}_u \calE_\mu ({t},{q} )\|^2_{( H^1, \bbD )^*}
+ \tilded_{L^2} ({-}\mathrm{D}_z \calE_\mu  ({t},{q}) ,\partial\calR(0))^2
+ \frac1\nu \, \dLtwo ({-}\mathrm{D}_p \calE_\mu  ({t},{q}) ,\partial_\pi \calH( {z} ,0))^2
}
\end{aligned}
  \end{equation}
 \begin{proposition}
 \label{Vito-Bs}
 Along a solution $q\in H^1(0,T;\bfQ)$ there holds $\foraa\, r\in (0,T)$
 \begin{equation}
 \label{0307192010}
 \begin{aligned}
 &
  \disv\eps\nu(u'(r))  {+} \calR_\eps(z'(r)) {+} \dish\eps\nu(z(r), p'(r)){+}
 \disv\eps\nu^*\Big(\mathrm{Div}(\sigma(r)){+} F(r)\Big)  
 \\ &\qquad
 {+}   \calR_\eps^*\Big({-}\As(z(r)) {-}W'(z(r)){-}\tfrac12 \bbC'(z(r)) e(r) : e(r)\Big)  
  {+}  \dish\eps\nu^*\Big(z(r), -\newmu p(r) + \sigma_\dev(r)\Big)  
  \\
  & = \MVito(r, q(r), q'(r)) 
  =  \calR({z}'(r))+\calH({z(r)},{p}'(r)) +   \varepsilon \Big( \nu \| u'(r)\|^2_{ H^1, \bbD }   {+} 
\|z'(r)\|_{L^2}^2
{+} \nu\|p'(r)\|_{L^2}^2\Big) \,.
\end{aligned}
 \end{equation}
  In particular,  a curve  $q\in H^1(0,T;\bfQ)$ is a solution to the Cauchy problem 
 \ref{prob:visc} if and only if it satisfies  for every $t\in [0,T]$ 
 the Energy-Dissipation balance 
  \begin{equation}\label{0307191911}
   \begin{aligned}
  \enen{\mu} (t, q(t)) + \int_0^t \MVito(r, q(r), q'(r)) \dd r 
  = \enen{\mu} (0,q_0 ) + \int_0^t 
\partial_t \enen{\mu} (r,q(r))  \dd r\,.
  \end{aligned}
  \end{equation} 
 \end{proposition}
%
 \begin{proof}
First, we have that for a.e.\   $r \in (0,T)$
\begin{equation}
\label{Vito-10Lug2019}
\begin{aligned}
&
 \MVito(r, q(r), q'(r))  =  \calR(z'(r)) + \calH(z(r), p'(r)) +    \DVito_\nu(q'(r))  \,  \DVitos(r,q(r))
\\
&\leq\calR(z'(r)) + \calH(z(r), p'(r)) + \frac\eps2  \DVito_\nu^2(q'(r)) 
 + \frac{1}{2\eps}( \DVitos(r,q(r)))^2
\leq \pairing{}{\bfQ}{{-} \rmD_q\enen{\mu} (r,q(r))}{q'(r)} \,,
\end{aligned}
 \end{equation}
 by the Cauchy inequality and \eqref{FM-bal}, which holds along the solutions.
\par
Let us now prove the converse inequality. Consider a measurable selection $r\mapsto \gamma(r) \in   \partial \calR(0) $ fulfilling 
\[
 \tilded_{L^2(\Omega)}( {-}\rmD_z \calE_\mu(r, q(r)), \partial \calR(0)) = \| {-}\rmD_z \calE_\mu(r, q(r)){-}\gamma(r)\|_{L^2(\Omega)}
\]  
 (observe that the existence of $\gamma$ is guaranteed by the fact that $ \tilded_{L^2(\Omega)}( {-}\rmD_z \calE_\mu(r, q(r)), \partial \calR(0))<+\infty$ for almost all $r\in (0,T)$,  and 
  that 
  $r \mapsto \rmD_z \calE_\mu(r, q(r))$ is measurable).   
Analogously, let 
$r\mapsto \rho(r) \in \partial_\pi \calH(z(r),0)$ fulfill
\[
\dLtwo({-}\rmD_p \calE_\mu(r, q(r)),\partial_\pi \calH(z(r),0))  =  \| {-}\rmD_z \calE_\mu(r, q(r)) -\rho(r)\|_{L^2(\Omega)}\,.
\]
Then, we have (using shorter notation for the duality pairings between $H_\Dir^1(\Omega;\R^n)$ and $H_\Dir^1(\Omega;\R^n)^*$, 
$\Hs(\Omega)$ and $\Hs(\Omega)^*$, and for the scalar product in $L^2(\Omega;\MD)$) 
\begin{equation}
\label{DVito-added}
\begin{aligned}
  & \pairing{}{\bfQ}{{-}\rmD_q \calE_\mu(r, q(r))}{q'(r)}  =  \pairing{}{H^1}{{-}\rmD_u \calE_\mu(r, q(r))}{u'(r)} +
  \pairing{}{\Hs}{{-}\rmD_z \calE_\mu(r, q(r))}{z'(r)} + \pairing{}{L^2}{{-}\rmD_p \calE_\mu(r, q(r))}{p'(r)} 
  \\
  & \leq \| u'(r)\|_{ H^1, \bbD }  \|{-}\mathrm{D}_u  \calE_\mu (r, q(r))\|_{( H^1, \bbD )^*}  +  \pairing{}{\Hs}{{-}\rmD_z \calE_\mu(r, q(r)){-}\gamma(r)}{z'(r)} 
  +\pairing{}{\Hs}{\gamma(r)}{z'(r)}
  \\
  & \qquad   +  \pairing{}{L^2}{{-}\rmD_p \calE_\mu(r, q(r)){-}\rho(r)}{p'(r)} 
  +\pairing{}{L^2}{\rho(r)}{p'(r)}  
 \\
 &  \stackrel{(1)}{\leq}   \| u'(r)\|_{ H^1, \bbD }  \|{-}\mathrm{D}_u  \calE_\mu(r, q(r))  \|_{( H^1, \bbD )^*}  +  \|z'(r)\|_{L^2} \tilded_{L^2} ({-}\mathrm{D}_z \calE_\mu(r, q(r)),\partial\calR(0)) + \calR(z'(r))
 \\
 &  \qquad  + \|p'(r)\|_{L^2}  d_{L^2}({-}\rmD_p \calE_\mu(r, q(r)),\partial_\pi \calH(z(r),0))+\calH(z(r), p'(r))
 \\
 & 
  \stackrel{(2)}{\leq}   \calR(z'(r)) + \calH(z(r), p'(r)) +    \DVito_\nu(q'(r))  \,  \DVitos(r,q(r))  =  \MVito(r, q(r), q'(r)), 
  \end{aligned}
 \end{equation}
 where (1) follows from the very definition of $\partial\calR(0)$ and $\partial \calH(z(r),0)$ \GGG combined with the fact that  $\calR(0)=\calH(z(r),0)=0$,
 and (2) from the Cauchy-Schwarz inequality. 
 \par
 Then, all inequalities in \eqref{Vito-10Lug2019} are equalities whence, in particular, 
 we conclude that for a.a.\ $r \in (0,t)$
 \begin{equation}
 \label{DVITOequalsSTAR}
\DVitos(r, q(r)) = \varepsilon   \DVito_\nu(q'(r)) \quad  \text{and} \quad   \MVito(r,q(r), q'(r))=\pairing{}{\bfQ}{{-} \rmD_q\enen{\mu} (r,q(r))}{q'(r)} \,. 
 \end{equation}
%
 This shows \eqref{0307192010} and concludes the proof.
 \end{proof} 
 
We now study the semicontinuity properties of the  distance-type functionals  introduced in \eqref{Psi-eps-star}, that also enter the definition of 
$\DVitos$. 
We  shall make use of the norms     $\|\cdot\|_{(H^1, \bbD)}$, $\| \cdot \|_{(H^1, \bbD)^*}$ from \eqref{norma-equivalente}, \eqref{norma-duale},   and refer to the space  $\Hsm$  and  the functional $\kappa:\Hs(\Omega) \to \R $   introduced in Lemma \ref{l:expl-subd}. 
\begin{lemma}\label{lemma:sup}
 Let $\mu>0 $ be fixed. 
For any $({t}, {q})=({t},({u}, {z}, {p})) \in [0,T]\times \bfQ$
there holds
\begin{subequations}\label{eqs:1106192255}
\begin{align}
\|\rmD_u \calE_\mu({t}, {q}) \|_{(H^1, \bbD)^*}
&=  \sup_{\substack{\eta_u \in H_\Dir^1(\Omega;\R^n) \\  \|\eta_u\|_{(H^1, \bbD)}\le 1}}   \langle -\mathrm{Div}(  \sigma(t)  ) -F({t}), \eta_u \rangle_{ H_\Dir^1(\Omega;\R^n)}
\,, \label{1106192256} \\
\tilded_{L^2} ({-}\rmD_z \calE_\mu ({t},{q}),\partial\calR(0))^2 &= \sup_{\substack{ \eta_z \in \Hsm \\ \|\eta_z\|_{L^2} \le 1}}  \langle \As({z})+W'({z})+\tfrac12 \mathbb{C}'({z}) {e(t)}: {e(t)} + \kappa, -\eta_z \rangle_{\Hs(\Omega)}\,,\label{1106192257}\\
\dLtwo ({-}\rmD_p \calE_\mu({t}, {q}),\partial_\pi \calH( {z} ,0)) &=  \sup_{\substack{\eta_p \in L^2(\Omega;\MD) \\  \|\eta_p\|_{L^2}\le1}} \left(  \langle  {\sigma}_\dev (t)  - \mu {p}, \eta_p \rangle_{L^2(\Omega;\MD)} -\calH({z},\eta_p) \right) \,.\label{1106192258}
\end{align}
\end{subequations}
 Hence,  for all  $(t_k,q_k)_k, (t,q) \in [0,T]\times \bfQ$ with $t_k \to t$ and $q_k\weakto q$ in $\bfQ$ 
we have that
\begin{subequations}\label{eqs:1106191955}
\begin{align}
\| \mathrm{D}_u \calE_\mu (t,q )\|_{(H^1, \bbD)^*} &\leq \liminf_{k \to 0} \| \mathrm{D}_u \calE_{\mu} (t_k, q_k )\|_{(H^1, \bbD)^*} \,,\label{1106191956} \\
\tilded_{L^2} ({-}\mathrm{D}_z \calE_\mu (t,q ),\partial\calR(0)) & \leq \liminf_{k \to 0} \tilded_{L^2} ({-}\mathrm{D}_z \calE_{\mu} (t_k,q_k ),\partial\calR(0))\,,\label{1106191957}\\
\dLtwo ({-}\mathrm{D}_p \calE_\mu ( t,q ),\partial_\pi \calH( z ,0)) & \leq \liminf_{k \to 0} \dLtwo ({-}\mathrm{D}_p \calE_{\mu} ( t_k,q_k ),\partial_\pi \calH( z_k,0))\,. \label{1106191958}
\end{align}
\end{subequations}

\end{lemma}
%
%
\begin{proof}
$\vartriangleright $ \eqref{eqs:1106192255}: 
  The well-known fact  that
$\|{-}\phi_u\|_{(H^1, \bbD)^*}
= \sup \Big\{ \langle -\phi_u, \eta_u \rangle_{ (H^1, \bbD) } 
\colon 
\|\eta_u\|_{(H^1, \bbD)}\le 1  \Big\}$  yields \eqref{1106192256}.  
 As for \eqref{1106192257},  one has
\[
 \tilded_{L^2(\Omega)} ({-}\phi_z,\partial\calR(0)) 
= \sup \Big\{ \langle -\phi_z + \kappa, \eta_z \rangle_{\Hs(\Omega)}
\colon
\eta_z \in \Hsm \,, \ \|\eta_z\|_{L^2(\Om)} \le 1  \Big\} \,.
\]
This follows from \cite[Remark 4.3 and Lemma 4.4]{Crismale-Lazzaroni}; in fact, the set $G$ from
 \cite{Crismale-Lazzaroni} equals the set $-\partial\calI(\zeta)=-\partial \calR(\zeta)-\kappa$ in the notation of the present paper, see Lemma \ref{l:expl-subd}.
Finally,
we have 
\[
\begin{aligned}
& \sup \Big\{ \langle -\phi_p, \eta_p \rangle_{L^2(\Omega;\MD)}-\calH(\sfz,\eta_p)
\colon
\|\eta_p\|_{L^2(\Omega;\MD)}\le1 \Big\} \\
& = \sup_{\eta_p \in L^2(\Omega;\MD)} \Big\{ \langle -\phi_p, \eta_p \rangle_{L^2(\Omega;\MD)}-\calH(\sfz,\eta_p) -  I_{B_{L^2}}(\eta_p)  \Big\}
\\
& = \left( \calH(\sfz,\cdot) +  I_{B_{L^2}}  \right)^*( {-}\phi_p) 
 = \min_{\eta_p \in L^2(\Omega;\MD)}
\left\{
\calH(\sfz,\cdot)^*(\eta_p) + \| {-}\phi_p-\eta_p\|_{L^2} 
\right\}
\\
& =  \min_{\eta_p \in L^2(\Omega;\MD)}
\left\{
 \|{-}\phi_p-\eta_p\|_{L^2} +I_{\partial_\pi \calH( \sfz ,0)}(\eta_p) 
\right\}
 = 
 \dLtwo ({-}\phi_p,\partial_\pi \calH( \sfz ,0)) \,,
\end{aligned}
\]
  where $I_{B_{L^2}}$ is the indicator function of the closed unit ball in 
$L^2(\Omega;\MD)$ (namely,
 $I_{B_{L^2}}(\eta_p)=0$ if $\|\eta_p\|_{L^2}\le 1$ and $I_{B_{L^2}}(\eta_p)=+\infty$ otherwise). 
 Hence 
 \eqref{1106191958}
 follows,  recalling \eqref{Fr-diff-q}.
\\
$\vartriangleright $ \eqref{eqs:1106191955}:
In order to show  the lower semicontinuity properties \eqref{eqs:1106191955}, we notice that,
 for fixed $\eta_u \in H_\Dir^1(\Omega;\R^n)$ and $\eta_p \in L^2(\Omega;\MD)$, 
  the functions $({t}, {q}) \mapsto \langle-\mathrm{Div}(\sigma(t)) -F({t}),\eta_u\rangle$ and   $({t}, {q}) \mapsto \langle {\sigma}_\dev(t) - \mu {p},\eta_p \rangle-\calH({z},\eta_p)$ in \eqref{1106192256} and \eqref{1106192258} (here we abbreviate the notation for the duality products) are continuous  with respect to the convergence of ${t}$ and  the weak convergence in $\bfQ$. For this, we rely on assumptions ($2.\mathbb{C}$), on \eqref{propsH-2+1/2},  and also on the fact that 
if ${q}_k \weakto {q}$ in $\mathrm{Q}$ then $z_k\to z $ in $\rmC^0(\overline\Omega)$. 
 
Moreover, for fixed  $\eta_z \in \Hsm$  the function $({t}, {q}) \mapsto \langle\As({z})+W'({z})+\tfrac12 \mathbb{C}'({z}) {e}: {e} + \kappa, -\eta_z \rangle$ is semicontinuous with respect to the convergence of ${t}$ and the weak convergence in $\bfQ$: the  contribution  $({t}, {q}) \mapsto \langle\As({z})+W'({z}) + \kappa, -\eta_z \rangle$ is continuous, recalling \eqref{bilinear-s_a}, ($2.\mathbb{C}$), ($2.W$), while $({t}, {q}) \mapsto\langle \mathbb{C}'({z}) {e}: {e}, -\eta_z \rangle$ is lower semicontinuous, since $-\eta_z \geq 0$ (cf.\ also \cite[(4.48) and (4.52)]{Crismale-Lazzaroni}).

Therefore we get \eqref{eqs:1106191955} since, by \eqref{eqs:1106192255}, we are  taking   supremums of lower semicontinuous functions. 
\end{proof}


 \section{Time discretization}
 \label{s:4}
\RRRN  In this section we discretize the rate-dependent system \eqref{RD-intro} and, again exploiting its underlying 
 gradient structure, we derive a series of estimates on the discrete solutions that are uniform w.r.t.\ the discretization parameter $\tau$, as well as  the parameters  $\eps$, $\nu$, and $\newmu$.  Therefore, 
 \begin{compactitem}
 \item[-]  we shall use these estimates to pass to the limit in the discretization scheme, for  $\eps$, $\nu$, and $\newmu$  fixed, and construct a solution to 
 Problem \ref{prob:visc} in Section \ref{s:exist-viscous};
 \item[-] since the viscous solutions to  system \eqref{RD-intro} thus obtained will enjoy estimates uniform w.r.t.\ $\eps$  and $\nu$,  we shall resort to them in the vanishing-viscosity   analyses as $\eps\down 0$ and $\eps,\nu \down 0$,
  for $\newmu>0$ fixed, carried out in Section~\ref{s:van-visc};
 \item[-] \GGG the estimates  that are  also uniform w.r.t.\ $\newmu>0$, will be inherited by the
  viscous solutions. Therefore, we  shall  exploit them to perform  the \emph{joint} vanishing-viscosity and  vanishing-hardening analysis in Section~\ref{s:van-hard},  as well. 
 \end{compactitem}
 %
    Throughout this section as well, we shall omit
     to explicitly invoke the conditions listed in Section 
 \ref{s:2} in the various statements,
  with the exception of Proposition 
 \ref{prop:enh-energy-est}, which 
 needs further conditions, in addition to those of  Section~\ref{s:2}. 
  Recall that in what follows we  shall denote by the symbols $c,\,c',\, C,\,C'$ various positive constants depending only on
known quantities.
 \par
  We construct time-discrete solutions to the Cauchy problem for the  rate-dependent  system for damage and plasticity   \eqref{RD-intro} by solving  the following
 time incremental minimization problems:
  for fixed  $\eps, \, \nu,\, \mu>0$,  we consider a  uniform  partition $\{0=t^0_\tau<\ldots<t^N_\tau=T\}$ of the time
interval $[0,T]$ with fineness $\tau = 
t_\tau^{k+1} - t_\tau^k =T/N$.  
We  shall   use the notation 
$
\eta^k_\tau:=\eta(t^k_\tau)$  for  $\eta\in\{w,F\}$.
 The elements  $(\ds q\tau k)_{0\leq k\leq N}= (\ds u\tau k, \ds z\tau{k}, \ds p \tau k)_{0\leq k\leq N}$ are determined by
\[
\ds u\tau 0 := u_0, \qquad 
 \ds z\tau 0: =z_0, \qquad  \ds p\tau 0: =p_0, 
\]
and, for $ k\in \{1, \ldots, N\},$
by solving the time-incremental problems 
\begin{align}
\label{def_time_incr_min_problem_eps} 
\begin{aligned}
\ds q\tau k  \in &\   \mathrm{Argmin} \Big\{   \tau \psie\eps\nu \left( q, \frac{q-\ds q\tau{k-1}}{\tau} \right) + \enen{\mu}(\ds t\tau{k},q)\, : \ q \in \bfQ\Big\}
\\   = &\  \mathrm{Argmin} \Big\{ 
 \frac\eps{2\tau} \left( \int_{\Omega}\nu \bbD (\sig u  - \sig {u_\tau^{k-1}}) : (\sig u  - \sig {u_\tau^{k-1}})   \dd x
  + \| z- \ds z\tau {k-1}\|_{L^2}^2  +  \GGG \nu   \| p- \ds p\tau {k-1}\|_{L^2}^2  \right)   \\
  &\qquad \qquad  \qquad \qquad
+\mathcal{R}(z-\ds z\tau{k-1})  + \mathcal{H}( z  , p-\ds p\tau{k-1})  
\\
&\qquad \qquad \qquad \qquad 
+\enen{\mu} (\ds t\tau k, u, p,  z) \colon 
u \in H_\Dir^1(\Omega;\R^n) \,, \ z \in \Hs(\Omega) \,, \ p \in 
L^2(\Omega;\MD)\Big \}\,.
\end{aligned}
\end{align}
Notice that, to shorten  notation,  we omit to write the dependence of the minimizers $(\ds q\tau k )_{k=1}^N$ on the positive parameters  $\eps$, $\nu$ and  $\mu$. 
\par 
\begin{remark}
\upshape
Taking into account that  $\calR(z{-}\ds z\tau{k-1}) = \calP(z{-}\ds z\tau{k-1}) + \calI(z{-}\ds z\tau{k-1})$ with $\calP$ and $\calI$ from 
\eqref{dam-diss-pot}, it is immediate to check that the minimum problem 
 \eqref{def_time_incr_min_problem_eps} reformulates as
 \[
 \begin{aligned}
\ds q\tau k  \in   & \mathrm{Argmin} \Big\{ 
 \frac\eps{2\tau} \left( \int_{\Omega}\nu \bbD (\sig u  - \sig {u_\tau^{k-1}}) : (\sig u  - \sig {u_\tau^{k-1}})   \dd x
  + \| z- \ds z\tau {k-1}\|_{L^2}^2  +  \GGG \nu   \| p- \ds p\tau {k-1}\|_{L^2}^2  \right)   \\
  & \quad\qquad \qquad
  -\int_\Omega \kappa   z \dd x 
 + \mathcal{H}( z  , p-\ds p\tau{k-1})   
+\enen{\mu} (\ds t\tau k, u, p,  z) \colon 
 (u, z, p) \in \bfQ,  \ z \leq \ds z\tau {k-1} \text{ in } \Omega \Big \}\,.
\end{aligned}
\]
Observe that, upon setting  $\nu=\mu=0$  the above  problem does coincide with the time-incremental minimization scheme  used to construct solutions to the viscous system in \cite{Crismale-Lazzaroni}. 
\end{remark}
The existence of a minimizing triple  for
\eqref{def_time_incr_min_problem_eps} 
 relies on the coercivity properties of the functional $\enen{\mu}$, specified in  Lemma \ref{l:coercivity-Enu} below.  Let us highlight that the coercivity estimates below are  \emph{uniform} w.r.t.\ the  hardening parameter  $\mu\in [0,1] $, and  in particular they are valid also for $\newmu=0$. This  will have a key role in the derivation of a priori estimates on the viscous solutions uniform w.r.t. $\newmu $   as well. 
\begin{lemma}
\label{l:coercivity-Enu}
There exist constants $c_E,\, C_E>0$ such that  for all $\mu \in [0,1]$ and  $(t,u,z,p) \in [0,T] \times \bfQ$ 
\begin{equation}
\label{coercivity}
\begin{aligned}
& \enen{\mu}(t,u,z,p)+\calH (z,p) + \|z\|_{L^2(\Omega)} 
\\
&\geq c_E \left( 
\| e(t)  \|_{\Lnn}
{+}\|z\|_{H^{\mathrm{m}}(\Omega)} {+} \mu^{1/2} \|p\|_{L^2(\Omega;\MD)}  {+} \mu^{1/2}\|u\|_{H^1(\Omega;\R^n)} +\|p\|_{L^1(\Omega;\MD)}\right)-C_E\,.
\end{aligned}
\end{equation}
\end{lemma}
\begin{proof}
 In the following lines, we  shall   use that $\calE_\mu$ rewrites as 
$\calE_\mu(t,u,z,p) = \calF_\mu (t,u,z,p)   - \int_\Omega \rho_\dev(t) p \dd x$, cf.\
 \eqref{serve-for-later}. 
Now, taking into account \eqref{C2}, the positivity of $W$,  and we easily have that 
\begin{equation}
\label{est-Fnu}
\calF_\mu (t,u,z,p) \geq \frac{\gamma_1}2 \| e(t) \|_{L^2}^2  +\GGG \frac{\mu}{2} \|p\|_{L^2}^2 +  \frac12 \ass(z,z)  - \frac1{2\gamma_1} \|\rho(t)\|_{L^2}^2 
\geq  \frac{\gamma_1}2 \| e(t) \|_{L^2}^2  + \GGG \frac{\mu}{2}  \|p\|_{L^2}^2  + \frac12 \ass(z,z)  -C_\rho,
\end{equation}
 By \eqref{controllo-SL}, we deduce that 
\[
\enen{\mu} (t,u,z,p) + \mathcal{H}( z,p)  \geq c\left(  \| e(t) \|_{L^2}^2  {+}\mu\|p\|_{L^2}^2 {+}   \ass(z,z) 
{+} \|p\|_{L^1} \right) -C,
\]
and \eqref{coercivity} easily follows  by a Korn-Poincaré  inequality for $u \in H_\Dir^1(\Omega;\R^n)$. 
\end{proof}
By virtue of Lemma \ref{l:coercivity-Enu} and 
the Direct Method of Calculus of Variations, 
problem \eqref{def_time_incr_min_problem_eps}
does admit a solution $(\ds q\tau k)_{0\leq k\leq N}= (\ds u\tau k, \ds z\tau{k}, \ds p \tau k)_{0\leq k\leq N}$.
Moreover, we set 
\begin{equation}
\label{discrete-place-holders}
e^k_\tau:=\sig{u^k_\tau+w^k_\tau}-p^k_\tau \qquad \text{ and } \qquad \sigma^k_\tau:=\bbC(z^k_\tau)e^k_\tau.
\end{equation}  
\par
For  $\eta \in \{q, u, e,z,p, \sigma,  w, F \}$, we  shall  use the short-hand notation 
\begin{equation}\label{1808191307}
\dsd \eta\tau k : = \frac{\ds \eta\tau k -\ds\eta\tau{k-1}}{\tau} \qquad \text{for } k \in \{0,\ldots,N\}\,.
\end{equation}
 In addition, 
 the following piecewise constant and piecewise
linear interpolation functions will be used
\[
\pwc \eta\tau (t) := \ds \eta \tau k   \,\,\text{for }t\in
(\ds t\tau{k-1}, \ds t \tau k],\,\,\, 
\upwc \eta\tau(t): =  \ds \eta \tau {k-1} \,\,\text{for }t\in
[\ds t\tau{k-1}, \ds t \tau k),\,\,\,
\pwl \eta\tau(t): =\ds \eta\tau{k-1} + \frac{t - \ds t\tau{k-1}}{\tau} (\ds \eta\tau{k}{-} \ds \eta\tau{k-1} )\,\,\text{for }t\in [\ds t\tau{k-1}, \ds t \tau k]
\]
with  \GGG  $\pwc \eta\tau(0):=\eta_0$, $\pwl \eta\tau(T):=\eta_\tau^k$.    
%
 Furthermore, we shall use the
notation
\[
\begin{array}{llll}
\pwc t\tau(r)&= \ds t\tau k  &\quad&\text{for } r\in (\ds t\tau{k-1}, \ds t \tau k],
\\
\upwc {t}\tau(r)&=  \ds t\tau {k-1} &\quad&\text{for } r\in [\ds t\tau{k-1}, \ds t \tau k).
\end{array}
\]
\par
Relying on the sum rule  from 
\cite[Prop.\ 1.107]{Mord-book-I}, we see 
 that
the minimizers $(\ds q\tau k )_{k=1}^N$ for \eqref{def_time_incr_min_problem_eps} 
satisfy the Euler-Lagrange equation
\begin{equation}
\label{abstract-subdiff}
\partial_{q'} \psie\eps\nu\left(\ds q{\tau}{k}, \frac{\ds q{\tau}{k}-\ds q{\tau}{k-1}}\tau \right)
+\tau \, \partial_{q} \psie\eps\nu\left(\ds q{\tau}{k}, \frac{\ds q{\tau}{k}-\ds q{\tau}{k-1}}\tau \right)
\ni
- \rmD_q \enen{\mu}(\ds t\tau k, \ds q\tau k)  \qquad \text{in } \bfQ^*, \quad \text{for } k=1,\ldots, N,
\end{equation}
where, with a slight abuse of notation, we have denoted by $\partial_q \psie\eps\nu$ the \emph{Fr\'echet subdifferential} of $q\mapsto \psie\eps\nu(q,q')$, i.e.\ the multivalued operator 
$\partial_q \psie\eps\nu: \bfQ \times \bfQ \rightrightarrows \bfQ^*$
 defined  by 
\[
{ \xi} \in \partial_q \psie\eps\nu(q,q') \quad  \text{ if and only if } \quad  \lim_{w\to q} \frac{\psie\eps\nu(w,q') - \psie\eps\nu(q,q') -
 \pairing{}{\bfQ}{{ \xi}}{w-q}}{\|w-q\|_{\bfQ}} \geq 0\,.
\]
Now, $\partial_q  \psie\eps\nu$ in fact reduces to the Fr\'echet subdifferential  
$\partial_z \calH: \mathrm{C}^0(\overline\Omega) \times L^1(\Omega;\MD) \rightrightarrows \mathrm{M}(\Omega)$. Hence,   the term 
  $\tau \partial_q \psie\eps\nu(\ds q\tau k,\dsd q\tau k)$ in  \eqref{abstract-subdiff} leads to the contribution
 $\tau\partial_z \mathcal{H}(\ds z\tau k, \dsd \pi\tau k)\in \mathrm{M}(\Omega) \subset \Hs(\Omega)^*$ that features in the discrete flow rule for the damage variable, cf.\ \eqref{DEL-2} below. 
Taking into account Lemma \ref{l:Fr-diff-Enu}, \eqref{abstract-subdiff} in fact   translates into 
the system, for all $k\in \{1,\ldots, N\}$,
\begin{subequations}
\label{discrete-Euler-Lagrange}
\begin{align}
\label{DEL-1}
&
- \mathrm{Div} \big(\varepsilon\nu \mathbb{D} \sig{\dsd u\tau k} {+}\ds\sigma\tau k\big) = \ds F\tau k
&& \text{in } H_\Dir^1(\Omega;\R^n)^*\,, 
\\
&
\label{DEL-2}
\partial \mathcal{R}_\eps(\dsd z\tau k)+ 
\As(\ds z \tau k) + W'(\ds z\tau k)   + \tau\, \partial_z \densh\eps\nu (\ds z \tau k, \dot p_\tau^k) 
\ni -\frac12 \bbC'(\ds z\tau k) \ds e\tau k : \ds e\tau k 
&& \text{in } 
\Hs(\Omega)^*\,,
\\
&
\label{DEL-3}
\partial_\pi \densh \eps\nu(\ds z\tau{k}, \dsd p\tau k)+\mu \ds p\tau k \ni  \big(\ds \sigma \tau k)_\dev && \text{a.e.\ in  } \Omega.
\end{align}
\end{subequations} 
For later use, let us rewrite system \eqref{discrete-Euler-Lagrange} in terms  of the piecewise constant and linear interpolants of the discrete solutions, also taking into account the structure formulae 
\eqref{sum-rule-calH} and \eqref{sum-rule-2} :  we have 
\begin{subequations}
\label{EL-interpolants}
\begin{align}
\label{EL-interp-1}
&
- \mathrm{Div} \big(\varepsilon\nu \mathbb{D} \sig{\pwl u\tau'} {+}\pwc \sigma \tau \big) = \pwc F\tau
&& \text{in } H_\Dir^1(\Omega;\R^n)^*\,, \   
\\
&
\label{EL-interp-2}
 \pwc\chi\tau  {+}\eps \pwl z\tau'{+}
\As(\pwc z\tau) {+} W'(\pwc z\tau)  {+} \tau \pwc\lambda\tau 
= {-}\frac12 \bbC'(\pwc z\tau) \pwc e\tau : \pwc e\tau
&& \text{in } 
\Hs(\Omega)^*  \ \text{with } 
 \pwc\chi\tau   \in \partial \mathcal{R}(\pwl z\tau'), \, 
\pwc\lambda\tau \in  \partial_z \densh\eps\nu (\pwc z\tau, \pwl p\tau')\,,
\\
&					
\label{EL-interp-3}
\pwc\omega\tau   +\eps\nu \pwl p\tau' + \mu \pwc p\tau = (\pwc\sigma\tau)_\dev
 && \text{a.e.\ in  } \Omega  \quad \text{with } \pwc\omega\tau  \in  \partial_\pi H(\pwc z\tau, \pwl p\tau')\,
\end{align}
\end{subequations} 
 almost everywhere in $(0,T)$. 
%
\par
Proposition \ref{prop:energy-est-SL} below collects the first set of a priori estimates for the discrete solutions.   Essentially, these estimates are obtained 
from the basic energy estimate following from choosing the competitor $q = \ds q\tau{k-1}$ in the minimum problem \eqref{def_time_incr_min_problem_eps}, which leads to 
\begin{equation}
\label{disc-en-est}
\enen{\mu} (\ds t{\tau}k, \ds q\tau k)+
\tau\psie\eps\nu \left( \ds q{\tau}{k}, \frac{\ds q\tau k -\ds q\tau{k-1}}{\tau}\right)
\leq \enen{\mu} (\ds t{\tau}{k-1}, \ds q\tau {k-1}) + \int_{\ds t\tau{k-1}}^{\ds t \tau k}
\partial_t \enen{\mu} (s,\ds q{\tau}{k-1}) \dd s\,.
\end{equation}
Let us mention in advance that, \GGG in  Proposition \ref{prop:discr-UEDE} ahead,  we shall derive a finer discrete Energy-Dissipation inequality, 
which will be the starting point for the limit passage as $\tau \downarrow 0$.
\begin{proposition}[Basic energy estimates]
\label{prop:energy-est-SL}
There exists a constant $C_1>0$, \GGG independent of $\eps,\,\mu,\,\nu, \, \tau>0$,  such that  the following estimates hold:
\begin{subequations}
\label{basic-en-est-SL}
\begin{align}
&
\begin{aligned}
&
\sup_{t\in [0,T]} \Big( \|\ole(t)\|_{\Lnn} + \|\olp(t)\|_{L^1(\Omega;\MD)}+
 \| \olu(t)\|_{\mathrm{BD}(\Omega)} + \| \olz(t)\|_{\Hs(\Omega)} 
\\
&
\qquad \qquad  + \int_\Omega W(\olz(t))\,\mathrm{d}x  + \sqrt{\newmu}\|\olp(t)\|_{L^2(\Omega;\MD)} + \sqrt{\newmu}\|\olu(t)\|_{H^1(\Omega;\R^n)} \Big)
  \leq C_1\,,
  \end{aligned}
   \label{1709172242-SL}
\\
&
\int_0^T   \left(  \|\dpt(s)\|_{L^1(\Omega;\MD)}+\|\dzt(s)\|_{L^1(\Omega)}\right) \dd s \leq C_1\,, 
 \label{ulteriore-stima-SL}
 \\
&
\varepsilon\int_0^T \bigg(\nu \|\dutau(s)\|_{H^1(\Omega;\R^n)}^2+\nu  \|\dpt(s)\|_{L^2(\Omega;\MD)}^2+\|\dzt(s)\|_{L^2(\Omega)}^2\bigg) \dd s \leq C_1\,. \label{terza-stima-SL}
\end{align}
\end{subequations}
Therefore,
 there exists $m_0>0$, independent of  $\eps,\,\nu, \, \newmu, \, \tau>0$ such that
\begin{equation}\label{2009171038-SL}
\olz(x,t)\geq m_0\,,\qquad z_\tau(x,t) \geq m_0 \quad \text{for all } (x,t) \in [0,T]\times \ol\Omega\,.
\end{equation}
\end{proposition}
\begin{proof}
It is immediate to check that  the time-incremental minimization problem 
 \eqref{def_time_incr_min_problem_eps} is equivalent to 
\begin{align*}
\begin{aligned}
\ds q\tau k  \in    \mathrm{Argmin} \Big\{   & 
 \frac\eps{2\tau} \left( \int_{\Omega}  \nu \, \bbD (\sig u  - \sig {u_\tau^{k-1}}) : (\sig u  - \sig {u_\tau^{k-1}})   \dd x
  + \| z- \ds z\tau {k-1}\|_{L^2}^2  +  \GGG \nu   \| p- \ds p\tau {k-1}\|_{L^2}^2  \right)  \\
  &  \
+\mathcal{R}(z-\ds z\tau{k-1})  + \mathcal{H}( z_\tau^{k}  , p-\ds p\tau{k-1})   
- \int_\Omega (\rho_\tau^{k}(t))_\dev ( p{-}\ds p\tau{k-1}) \dd x 
+\calF_\newmu (\ds t\tau k, u, p,  z) \colon  (u,z,p) \in \bfQ 
\Big \}\,,
\end{aligned}
\end{align*}
with $\calF_\mu$ from \eqref{def-Fnu}. 
%
Then, considering the analogue of  estimate 
%
%
%
\eqref{disc-en-est} and summing it  up with respect to the index $k=1,\dots,j$, with $j$ arbitrary in $ \{1,\ldots, N\}$, we find 
\begin{equation}
\label{Giuliano}
\calF_\newmu (\ds t{\tau}j, \ds q\tau j)+
\sum_{k=1}^{j} \left[ \tau\psie\eps\nu \left( \ds q{\tau}{k}, \frac{\ds q\tau k -\ds q\tau{k-1}}{\tau}\right)
- \int_\Omega (\rho_\tau^{k}(t))_\dev (\ds p\tau{k}-\ds p\tau{k-1}) \dd x  \right]
\leq \calF_\newmu (0, \ds q\tau 0) + \sum_{k=1}^{j} \int_{\ds t\tau{k-1}}^{ \ds t \tau k }
\partial_t \calF_\newmu (s,\ds q{\tau}{k-1}) \dd s.
\end{equation}
On the one hand, again thanks to \eqref{controllo-SL} we have that 
\[
\begin{aligned}
\tau\psie\eps\nu \left( \ds q{\tau}{k}, \frac{\ds q\tau k -\ds q\tau{k-1}}{\tau}\right)
- \int_\Omega (\rho_\tau^{k}(t))_\dev (\ds p\tau{k}-\ds p\tau{k-1}) \dd x &  \geq \tau\widetilde{\Psi}_{\eps,\nu} \left(  \frac{\ds q\tau k -\ds q\tau{k-1}}{\tau}\right) 
 \end{aligned}
\]
  with
 $\widetilde\Psi_{\eps,\nu}(  q') : =  \disv\eps\nu(u')+\calR_\eps(z')+ \alpha \|p'\|_{L^1}$. 
 
On the other hand,
 since 
  $\partial_t \calF_\newmu(t,q) = \int_\Omega \sigma(t) : \sig{w'(t)} \dd x - \int_\Omega \rho'(t) (e(t)-\rmE(w(t)) \dd x - \partial_t (\langle F(t), w(t) \rangle_{H^1})$,     we easily find also in view of \eqref{C2},  of \eqref{2909191106}-\eqref{2809192200}, and  of \eqref{dir-load}, 
 that
\[
|\partial_t \calF_\newmu(t,q)| \leq \mathcal{L}(t) \|e\|_{\Lnn}  + \tilde{\mathcal{L}}  (t) 
  \qquad \text{with } 
  \begin{cases}
  \mathcal{L}(t): = 
C \left(  \|w'(t)\|_{H^1} + \|\varrho'(t)\|_{L^2}\right) \in L^1(0,T)\,,
\\
  \tilde{\mathcal{L}} (t): =C'   \|F'(t)\|_{(H^1)^*}  \in L^1(0,T).
  \end{cases}
\]
From \eqref{Giuliano} we then gather  that  
\begin{equation}
\label{discrete-G-Lemma}
\begin{aligned}
\calF_\newmu & (\ds t{\tau}j, \ds q\tau j) + 
\sum_{k=1}^{j}  \tau  \widetilde\Psi_{\eps,\nu}(  \ds q\tau{k},\dsd q\tau k)
 \leq 
\calF_\newmu (0, \ds q\tau 0) +  \sum_{k=1}^{j}  \| \ds e\tau{k-1}\|_{\Lnn}  \int_{\ds t\tau{k-1}}^{ \ds t \tau k } 
\calL(s) \dd s  + \int_0^T  \tilde{\mathcal{L}} (t) \dd t  
\\
&\stackrel{(1)}{\leq} 
 \calF_1 (0, \ds q\tau 0)  + 
 \| \ds e\tau 0\| \|\mathcal{L}\|_{L^1(0,T)} + 
  \sum_{k=1}^{j}  \left( \calF_\newmu (\ds t{\tau}{k-1}, \ds q\tau {k-1}) +C_\rho +\frac1{2\gamma_1} \right)   \int_{\ds t\tau{k-1}}^{ \ds t \tau k } 
\calL(s) \dd s   + \int_0^T  \tilde{\mathcal{L}} (t) \dd t   
\\
&\stackrel{(2)}{\leq}  C +  \frac2{\gamma_1}  \sum_{k=0}^{j-1} \left( \calF_\newmu (\ds t{\tau}{k}, \ds q\tau k) {+}C_\rho\right)  \int_{\ds t\tau{k}}^{ \ds t \tau {k+1} }  
\calL(s) \dd s   + \int_0^T  \tilde{\mathcal{L}} (t) \dd t,   
\end{aligned}
\end{equation}
where {\footnotesize (1)} \&  {\footnotesize (2)} follow from the fact that, by
 \eqref{init-data}   and $\newmu \in [0,1]$,   it holds
$
\calF_\newmu (0, \ds q\tau 0) \le \calF_1 (0, \ds q\tau 0) \le C$  uniformly in  $\mu$ and $\tau >0$,
as well as from estimate \eqref{est-Fnu}.
We are now in a position to apply  a version of the discrete Gronwall Lemma  (cf., e.g., Lemma \ref{l:discrG1} ahead),
 to conclude that 
\[
\calF_\newmu (\ds t{\tau}j, \ds q\tau j) +C_\rho  \leq C' \exp \left( \frac2{\gamma_1}   \sum_{k=0}^{j-1}  \int_{\ds t\tau{k}}^{ \ds t \tau {k+1} } 
\calL(s) \dd s\right) \leq C,
\]
where the latter estimate follows from  \eqref{2909191106}-\eqref{2809192200} 
and \eqref{dir-load}. 
All in all, from \eqref{discrete-G-Lemma} we conclude that
\[
\exists\, C>0 \quad  \forall\,  \eps, \, \nu, \, \newmu,  \, \tau>0 \quad  \forall\, j \in \{1,\ldots, N\}, \qquad 
|\calF_\newmu (\ds t{\tau}j, \ds q\tau j)|+
\sum_{k=1}^{j}  \tau  \widetilde\Psi{\eps,\nu}(  \ds q\tau{k},\dsd q\tau k) \leq C.
\]
 In particular, we find that $\| \ds p\tau j\|_{L^1(\Omega)} \leq C$.
Then,
  recalling that $\enen{\mu}(t,q) = \calF_\nu(t,q)-\int_\Omega \rho_\dev(t) p \dd x $, that $\rho_\dev \in L^\infty(0,T;\Linftyn)$, and using \eqref{props-calH-2}
   it is immediate to check that 
 \[
 \exists\, C>0 \quad \forall\,  \eps ,\,\nu,\, \newmu, \, \tau >0\, : \qquad 
\sup_{t\in [0,T]}\left| 
\enen{\mu} (\pwc t\tau(t), \pwc q\tau (t))\right| + \int_0^{T}  \psie\eps\nu ( \pwc q{\tau}(s),   q_\tau' (s) )  \dd s \leq C\,.
\]
  Then, estimates \eqref{ulteriore-stima-SL} and \eqref{terza-stima-SL} immediately follow, while \eqref{1709172242-SL} ensues on account of the coercivity property \eqref{coercivity}.
  Let us additionally mention that the estimates for $\ole$ and $\olp$ entail a bound for $\sig{\olu}$ in $L^\infty(0,T;L^1(\Omega;\Mnn))$, which then yields the bound for $\olu$ in $L^\infty (0,T;  \mathrm{BD}(\Omega)) $  via the Poincar\'e type inequality \eqref{PoincareBD}.
Property \eqref{2009171038-SL}  can be deduced from  the fact that $\sup_{t\in [0,T]} \int_\Omega W(\olz(t))\,\mathrm{d}x \leq C_1$ (cf.\ \eqref{1709172242-SL})  
arguing as in \cite[Lemma 3.3]{Crismale-Lazzaroni}, cf.\  also  Remark \ref{rmk3.2}. 
\end{proof}

The following step is the derivation of \emph{enhanced} a priori estimates for the discrete
solutions
$(q_\tau)_\tau = (u_\tau, z_\tau, p_\tau)_\tau$,  which  are uniform with respect to the  discretization parameter $\tau>0$. 
Recall that, with Proposition~\ref{prop:energy-est-SL} we have obtained  for $(q_\tau)_\tau$  an a priori estimate in $H^1(0,T;\bfQ)$ that blows up as  $\varepsilon$, $\nu \down 0$; 
such estimate will be used to conclude the existence of viscous solutions to system \eqref{RD-intro} for $\eps$, $\nu$, and $\newmu>0$ fixed.  

Now, with   Proposition \ref{prop:enh-energy-est} below we prove a set of \emph{enhanced}  a priori estimates, \emph{uniform}  in 
$\tau$, $\nu$, $\newmu$,  and blowing up as $\eps\down 0$, for $\dot{e}_\tau$, $\dot{z}_\tau$ and $\sqrt{\newmu} \dot{p}_\tau$, $\sqrt{\newmu} \dot{u}_\tau$ (cf.\ \eqref{2009170048} below): such estimates will ensure the existence of solutions to the viscous system with higher temporal regularity than that  guaranteed  by  Proposition~\ref{prop:energy-est-SL}.  What is more, 
 we obtain a set of a priori estimates, also \emph{uniform in $\eps$},
 for the triple $(\sqrt{\newmu} \dot{u}_\tau, \dot{z}_\tau, \sqrt{\newmu} \dot{p}_\tau)$ in $W^{1,1}(0,T;\bfQ)$ and for $e_\tau$ in $W^{1,1}(0,T;L^2(\Omega;\Mnn))$). 
 Such bounds   will be at the basis of the vanishing-viscosity analyses carried out in Section~\ref{s:van-visc},  as well as
 of the vanishing-hardening  limit passage   in Section~\ref{s:van-hard}. 
 Let us mention in advance that all  of  these estimates  shall hold
 under the further condition that 
 $\nu\leq \mu$, which is consistent both with  
  \begin{compactitem}
 \item[-] the situation   in which
 the hardening parameter $\mu$ is kept fixed, the viscosity parameter $\eps $ vanishes,
 and either $\nu$ is kept fixed (cf.\ Section~\ref{s:6-nu-fixed}), or $\nu$ vanishes along with $\eps$ (cf.\ Section~\ref{s:6-nu-vanishes});
 \item[-] and with the case where we perform  joint  vanishing-viscosity and  vanishing-hardening analysis for the 
 viscous solutions, 
 cf.\ 
 Section~\ref{s:van-hard}.  
 \end{compactitem}
\par
We 
prove  Proposition~\ref{prop:enh-energy-est}   under the following additional conditions on the initial data  
$q_0=(u_0,z_0,p_0)$:
\begin{equation}
\label{EL-initial}
\begin{aligned}
&\rmD_q  \enen{\mu} (0,q_0)   = (\rmD_u \enen{\mu}(0,u_0, z_0, p_0), \rmD_z \enen{\mu}(0,u_0, z_0, p_0), \rmD_p \enen{\mu}(0,u_0, z_0, p_0)) 
\\ &  = \left(-\mathrm{Div}( \sigma_0 ) -F(0), \As(z_0)+W'(z_0)+\tfrac12 \mathbb{C}'(z_0) e_0: e_0, 
\newmu p_0 - (\sigma_0)_\dev \right) \in L^2(\Omega;\R^n{\times} \R {\times} \MD).
\end{aligned}
\end{equation}  
\begin{proposition}[Enhanced a priori estimates]
\label{prop:enh-energy-est}
 Under the assumptions of Section \ref{s:2},
suppose in addition that the initial data 
$(u_0,z_0,p_0)$ fulfill conditions \eqref{EL-initial}. 
Then,  for $\frac{\tau}{\varepsilon}$ small enough, we have that 
\begin{subequations}
\label{enhanced-discr-est}
\begin{enumerate}
\item there exists a constant $C_2^\eps>0$,  independent of $\tau$, $\nu$, $\mu>0$,  with $C_2^\eps \uparrow +\infty $
 as $\eps \downarrow 0$, such that for all $\tau,\, \nu,\, \newmu>0$    with \underline{$\nu \leq \newmu$} 
there holds  
\begin{equation}
\label{2009170048}
\begin{aligned}
&
 \sqrt{\newmu} \| \dot{u}_\tau  \|_{L^\infty(0,T;H^1(\Omega;\R^n))} + \| \dot{z}_\tau \|_{L^\infty(0,T;L^2(\Omega))} +
 \sqrt{\newmu} \| \dot{p}_\tau  \|_{L^\infty(0,T;L^2(\Omega;\MD))}  \leq C_2^\eps,
 \\
 & 
   \| \dot{e}_\tau  \|_{L^2(0,T;L^2(\Omega;\Mnn))} + \| \dot{z}_\tau \|_{L^2(0,T;\Hs(\Omega))} 
\leq C_2^\eps
\end{aligned}
\end{equation}
\item there exists a constant $C_2>0$,  independent of $\varepsilon$, $\tau$, $\nu$, $\mu>0$,  such that for all   $\tau,\,  \eps,\, \nu,\, \newmu>0$ with \underline{$\nu \leq \newmu$} 
there holds 
\begin{equation}
 \label{2009170054}
   \| \dot{e}_\tau  \|_{L^1(0,T;L^2(\Omega;\Mnn))} + \| \dot{z}_\tau \|_{L^1(0,T;\Hs(\Omega))} + \sqrt{\newmu}  \|\dot{p}_\tau  
 \|_{L^1(0,T;L^2(\Omega;\MD))} + \sqrt{\newmu}  \|\dot{u}_\tau  
 \|_{L^1(0,T;H^1(\Omega;\R^n))}  \leq C_2\,.
 \end{equation}
\end{enumerate} 
\end{subequations}

\end{proposition}
\noindent
 As we  shall  see in  Remark \ref{sta-in-rmk} later on,
 assuming only \eqref{1808191033} in place of  \eqref{forces-u}, 
 estimates \eqref{enhanced-discr-est} hold for two constants $C_2^{\varepsilon,\mu}$ and $C_2^\mu$ depending also on $\mu>0$.  
%
\paragraph{\bf Outline of the proof.} 
Our argument will be split in the following steps:  
\begin{enumerate}
\item
The first step basically corresponds to  ``differentiating w.r.t.\ time''   each of the  discrete Euler-Lagrange equations/subdifferential inclusions 
satisfied by the discrete solutions, 
and testing them by 
$\dsd u\tau k$, $\dsd z\tau k$, $\dsd p\tau k$, respectively. 
In practice, we shall do so with the discrete equations for $\ds u\tau k$ and $\ds p\tau k$ (i.e., \eqref{DEL-1} and \eqref{DEL-3}), while, instead of working with the discrete flow rule \eqref{DEL-2} for $z$ (and dealing with the Fr\'echet subdifferential term therein), we  shall  resort to 
\eqref{GV-1} \& \eqref{GV-2} below, which are a key consequence of the minimum problem \eqref{def_time_incr_min_problem_eps}. 
We  will add up the resulting relations and perform suitable calculations.
\item Next, we perform a suitable estimate of $\|\dsd p\tau k  \|_{L^1(\Omega;\MD)}$. The key role of this calculation is commented upon in Remark \ref{2109170023} ahead. 
\item We  shall  rearrange the estimate  obtained  in Steps  $1$--$2$. 
\item
The tasks in Steps $1$--$3$   are addressed   by working with the discrete Euler-Lagrange system \eqref{discrete-Euler-Lagrange} for $k \in \{2,\ldots,N_\tau\}$. In this step, we  shall  separately treat the case $k=1$. 
\item We  shall  apply the Gronwall Lemma  \ref{l:discrG1/2}
to get estimates \eqref{2009170048}, blowing up as $\eps\down 0$;
\item We  shall  apply the Gronwall-type  Lemma  \ref{l:discrG2}
to get estimates \eqref{2009170054},  uniform w.r.t.\ $\eps, \, \nu, \,  \newmu>0$.
\end{enumerate}
\par
 We  shall  also use the following result.
\begin{lemma} \cite[Lemma 3.4]{Crismale-Lazzaroni}
The minimizers $(\ds q\tau k)_{k=1}^{N_\tau}$ of \eqref{def_time_incr_min_problem_eps} satisfy for all $k\in \{1,\ldots, N_\tau \}$
\begin{align}
&
\label{GV-1}
\calR(\zeta) + 
\eps \int_\Omega \dsd z\tau k \zeta \dd x + \ass(\ds z\tau k,\zeta) +\int_\Omega \left( W'(\ds z\tau k) {+} \frac12 \bbC'(\ds z\tau k) \ds e\tau k \ds e\tau k \right)\zeta \geq 0 \qquad \text{for all } \zeta \in \Hs(\Omega),
\\
& 
\label{GV-2}
\calR(\dsd z\tau k) + 
\eps \|\dsd z\tau k\|_{L^2}^2  + \ass(\ds z\tau k,\dsd z\tau k) +\int_\Omega \left( W'(\ds z\tau k) {+} \frac12 \bbC'(\ds z\tau k) \ds e\tau k \ds e\tau k \right)\dsd z\tau k \leq C_K \tau \| \dsd z\tau k\|_{L^\infty} 
 \| \dsd p\tau k\|_{L^1} 
\end{align}
with $C_K$ from \eqref{propsH-2+1/2}. 
\end{lemma} 
 \noindent \textbf{Proof of Proposition \ref{prop:enh-energy-est}:} 
\par
\noindent
{\bf Step $1$:} 
For $k \in \{2,\ldots, N_\tau\}$,  let us subtract \eqref{DEL-1}  at  step $k-1$ from  \eqref{DEL-1}  at  step $k$. Testing the resulting relation by $\dsd u\tau k$, we obtain
\begin{equation}
\label{step1-u}
\dddn{\int_\Omega \eps \nu \bbD \sig{\dsd u\tau k{-} \dsd u\tau{k-1}} :\sig{\dsd u\tau k} \dd x}{$\doteq I_1$} + \dddn{\int_\Omega (\ds \sigma\tau k {-} \ds \sigma \tau{k-1}) : \sig{\dsd u\tau k} \dd x}{$\doteq I_2$}  = 
\dddn{\langle \ds F\tau{k} {-} \ds F\tau{k-1}, \dsd u\tau k \rangle_{H^1(\Omega;\R^n)}}{$\doteq I_3$}
\end{equation}
 Since $\int_{\Omega} \bbD 
\sig{u_1}:\left(\sig{u_1}{-}\sig{u_2}\right) \dd x \geq \|u_1\|_{ H^1, \bbD } (\|u_1\|_{ H^1, \bbD }- \|u_2\|_{ H^1, \bbD }) \geq  \frac12\| u_1 \|^2_{ H^1, \bbD }  - \frac12 \| u_2 \|^2_{ H^1, \bbD }$,  we have 
\[
I_1 \geq \eps\nu \|\dsd u\tau k \|_{ H^1, \bbD } (\| \dsd u\tau k \|_{ H^1, \bbD }{-} \| \dsd u\tau {k-1} \|_{ H^1, \bbD }).
\]
As for $I_2$, we use that 
$\sig{\dsd u \tau k} = \dsd e\tau k + \dsd p\tau k - \sig{\dsd w\tau k}$
and that $\ds\sigma\tau k = \bbC(\ds z\tau k) \ds e\tau k$ (cf.\ \eqref{discrete-place-holders}), so that 
\begin{equation}
\label{later-used-for-u}
\ds\sigma\tau k - \ds\sigma\tau {k-1}= \bbC(\ds z\tau k) (\ds e\tau k {-} \ds e\tau{k-1}) + \left(  \bbC(\ds z\tau k){-}  \bbC(\ds z\tau {k-1}) \right)  \ds e\tau{k-1}.
\end{equation}
Therefore, 
\[
\begin{aligned}
I_2  & = \dddn{\int_\Omega \bbC(\ds z\tau k) (\ds e\tau k{-} \ds e\tau {k-1}): \dsd e\tau k \dd x}{$\doteq I_{2,1}$} +
   \dddn{\int_\Omega \left( \bbC(\ds z\tau k){-} \bbC(\ds z\tau {k-1}) \right) \ds e\tau {k-1}: \dsd e\tau k \dd x}{$\doteq I_{2,2}$}
\\ & \quad    + \dddn{\int_\Omega (\ds \sigma \tau k {-} \ds \sigma \tau {k-1} ) \dsd p\tau k \dd x }{$\doteq I_{2,3}$} -  
    \dddn{\int_\Omega (\ds \sigma \tau k {-} \ds \sigma \tau {k-1} ) \sig{ \dsd w\tau k} \dd x }{$\doteq I_{2,4}$}. 
    \end{aligned}
\]
Now, we have that 
\[
I_{2,1} = \tau \int_\Omega \bbC(\ds z\tau k)  \dsd e\tau k :  \dsd e\tau k \dd x \geq \gamma_1 \tau \| \dsd e\tau k \|_{L^2}^2
\]
by \eqref{C2},
whereas, since the mapping $z \mapsto \bbC(z)$ is Lipschitz continuous, 
\[
\left| I_{2,2} \right| \leq  C\tau \| \dsd z\tau k\|_{L^\infty} \| \ds e\tau{k-1}\|_{L^2}  \| \dsd e\tau k\|_{L^2}\leq C  
\tau \| \dsd z\tau k\|_{L^\infty}  \| \dsd e\tau k\|_{L^2},
\]
where the last estimate follows from the previously obtained \eqref{1709172242-SL}.
While the term $I_{2,3}$ will be canceled in the next lines, again relying on \eqref{later-used-for-u} and the Lipschitz continuity of $\bbC$, 
 we estimate 
\[
\begin{aligned}
\left| I_{2,4} \right|  & \leq \tau \| \bbC(\ds z\tau k)\|_{L^\infty} \|\dsd e\tau k\|_{L^2} \| \sig{\dsd w\tau k}\|_{L^2} + C \tau \| \dsd z \tau k\|_{L^\infty}  \| \ds e\tau{k-1}\|_{L^2}\| \sig{\dsd w\tau k}\|_{L^2} 
\\ & 
\leq 
C\tau \left(  \|\dsd e\tau k\|_{L^2} \| \sig{\dsd w\tau k}\|_{L^2}{+}  \| \dsd z \tau k\|_{L^\infty} \| \sig{\dsd w\tau k}\|_{L^2}  \right),
\end{aligned}
\]
where the latter estimate again follows from  \eqref{1709172242-SL}.
Finally, recalling that $F\in  H^1(0,T;  \BD(\Omega)^*)$, we may estimate 
\begin{equation}\label{1808191206}
\begin{aligned}
\left| I_3 \right|  & \leq \tau \| \dsd F\tau k\|_{ \BD(\Omega)^*}  \| \dsd u\tau k \|_{ \BD(\Omega)}
\leq  C\tau \| \dsd F\tau k\|_{ \BD(\Omega)^*}  \| \sig{\dsd u\tau k}\|_{L^1}
\\ &  \leq  C \tau \| \dsd F\tau k\|_{ \BD(\Omega)^*} \left( \| \sig{\dsd w\tau k}\|_{L^1}{+}  \| \dsd e\tau k\|_{L^1}{+} 
\| \dsd p\tau k\|_{L^1}\right) 
\end{aligned}
\end{equation}
 where the second estimate follows from Poincar\'e's inequality for $\BD(\Omega)$, cf.\ \eqref{PoincareBD},   and the very last one   follows from 
 the fact that $\sig{\dsd u \tau k} = \dsd e\tau k + \dsd p\tau k - \sig{\dsd w\tau k}$.  
All in all, combining the above calculations with \eqref{step1-u}, we conclude that
\begin{equation}
\label{conclu-est-u}
\begin{aligned}
&
 \eps \nu \|\dsd u\tau k \|_{ H^1, \bbD } (\| \dsd u\tau k \|_{ H^1, \bbD }{-} \| \dsd u\tau {k-1} \|_{ H^1, \bbD })
 + \gamma_1 \tau \| \dsd e\tau k \|_{L^2}^2 +\int_\Omega (\ds \sigma \tau k {-} \ds \sigma \tau {k-1} )_\dev \dsd p\tau k \dd x 
 \\ & \leq
 C\tau \Big(  \| \dsd z\tau k\|_{L^\infty}  \| \dsd e\tau k\|_{L^2} {+}  \|\dsd e\tau k\|_{L^2} \| \sig{\dsd w\tau k}\|_{L^2}{+}  \| \dsd z \tau k\|_{L^\infty} \| \sig{\dsd w\tau k}\|_{L^2} {+}  \| \dsd F\tau k\|_{ \BD^*} \| \sig{\dsd w\tau k}\|_{L^1}
 \\
 & \qquad \qquad 
 {+}  \| \dsd F\tau k\|_{ \BD^*}  \| \dsd e\tau k\|_{L^1}  {+}  \| \dsd F\tau k\|_{ \BD^*}   \| \dsd p\tau k\|_{L^1} \Big)\,.
 \end{aligned}
\end{equation}
\par
Let us now consider estimate \eqref{GV-2} at step $k$ and subtract from it \eqref{GV-1} at step $k-1$
 (recall that $k \in \{2,\ldots,N_\tau\}$), 
 with the test function $\beta: = \dsd z\tau k$. We thus obtain
\begin{equation}
\label{step1-z}
\begin{aligned}
&
 \dddn{\calR(\dsd z\tau k) - \calR(\dsd z\tau k)}{$=0$} + 
\varepsilon \int_\Omega (\dot z\tk {-} \dot z\tm) \dot z\tk  \dd x + \ass(z\tk -z\tm, \dot z\tk)
\\
& 
\leq   \dddn{\int_\Omega \big[W'(z\tm)-W'(z\tk) \big] \dot z\tk \dd x}{$\doteq I_4$}
+\dddn{\frac{1}{2}\int_\Omega  \big[\mathbb{C}'(z\tm)- \mathbb{C}'(z\tk) \big] e\tk: e\tk \, \dot z\tk  \dd x}{$\doteq I_5$}
 \\
 & \qquad 
- \dddn{\frac{1}{2}\int_\Omega  \left(\bbC'(\ds z\tau {k-1}) \ds e\tau k : \ds e\tau k {-} \bbC'(\ds z\tau {k-1}) \ds e\tau {k-1} : \ds e\tau {k-1}  \right) \dot z\tk  \dd x}{$\doteq I_6$}
  + C_K \tau \|\dsd z\tau k\|_{L^\infty} \| \dsd p \tau k \|_{L^1}\,. 
  \end{aligned}
\end{equation} 
Now, recall that, by \eqref{2009171038-SL}, $0<m_0 \leq \dsd z\tau k \leq 1$ for  all $k \in \{0,\ldots, N_\tau\}$. Since the restriction of $W'$ to $[m_0,1]$ is Lipschitz continuous, we conclude that 
\[
\left| I_4 \right| \leq C \int_\Omega |\ds z\tau k {-} \ds z\tau{k-1}||\dot z\tk | \dd x \leq C\tau \| \dsd z\tau k\|_{L^2}^2\,;
\]
by the Lipschitz continuity of $\bbC'$ we have that
\[
\left| I_5 \right| \leq C \int_\Omega |\ds z\tau k {-} \ds z\tau{k-1}| |\ds e\tau k|^2 |\dot z\tk | \dd x  \leq C\tau \| \dsd z\tau k\|_{L^\infty}^2 \| \ds e\tau k\|_{L^2}^2 \leq  C\tau \| \dsd z\tau k\|_{L^\infty}^2 \,,
\]
the latter estimate due to   \eqref{1709172242-SL}; finally, 
\[
\left| I_6 \right| \leq C\int_\Omega |\ds e\tau k{+} \ds e\tau{k-1}||\ds e\tau k{-} \ds e\tau{k-1}| \dot z\tk|  \dd x  \leq C \tau \| \dsd e\tau k \|_{L^2}  \| \dsd z\tau k\|_{L^\infty} \,,
\]
where we have used that $\|\bbC (\ds z\tau k )\|_{L^\infty} \leq C$, and again the previously proved \eqref{1709172242-SL}.
Inserting the above estimates into \eqref{step1-z} leads to 
\begin{equation}
\label{conclu-est-z}
\begin{aligned}
\eps \| \dsd z\tau k\|_{L^2} \left(  \| \dsd z\tau k\|_{L^2}{-}  \| \dsd z\tau {k-1}\|_{L^2} \right) + \tau \ass (\dsd z\tau k, \dsd z\tau k) \leq C\tau  \| \dsd z\tau k\|_{L^\infty} \left( \| \dsd z\tau k\|_{L^\infty}{+}  \| \dsd e\tau k \|_{L^2} + \| \dsd p\tau k \|_{L^1}   \right). 
\end{aligned}
\end{equation} 
\par
Prior to working with  \eqref{DEL-3}, 
let us specify that it reformulates as 
\begin{equation}
\label{with-omega-k}
\omega\tk + \eps\nu \dot p\tk +\newmu \ds p\tau k =  \big(\ds \sigma \tau k)_\dev \qquad \text{for some } \omega\tk \in \partial_\pi H(\ds z\tau{k}, \dsd p\tau k)  \qquad \aein\, \Omega
\end{equation}
(cf.\ \eqref{EL-interp-3}). 
 We subtract  \eqref{with-omega-k}, written at  step $k-1$,  from  \eqref{with-omega-k}  at  step $k$, and test  the resulting relation by $\dsd p\tau k$.  This leads to 
\begin{equation}
\label{step1-p}
\begin{aligned}
& \dddn{\int_\Omega (\ds \omega\tau k {-} \ds \omega\tau{k-1}) \dsd p\tau k \dd x}{$\doteq I_7$} +\varepsilon \nu \int_\Omega (\dot  p\tk {-} \dot p\tm) \dot p\tk  \dd x + \newmu\int_\Omega  (\ds p\tau k {-} \ds p\tau{k-1})\dsd p\tau k \dd x 
= \int_\Omega (\ds \sigma \tau k {-} \ds \sigma \tau {k-1} )_\dev \dsd p\tau k \dd x \,.
\end{aligned}
\end{equation}
From the $1$-homogeneity of $H$ and the fact that $\omega\tk \in \partial_\pi H(\ds z\tau{k}, \dsd p\tau k)$ and $\ds\omega\tau{k-1} \in \partial_\pi H(\ds z\tau{k-1}, \dsd p\tau {k-1})$  a.e.\ in $\Omega$, it follows that 
\[
\int_\Omega \ds \omega\tau k \dsd p\tau k \dd x = \mathcal{H}(\ds z\tau {k}, \dsd p\tau k), \qquad \int_\Omega \ds \omega\tau {k-1} \dsd p\tau k \dd x \leq  \mathcal{H}(\ds z\tau {k-1}, \dsd p\tau k)\,.
\]
Therefore,  by \eqref{props-calH-3} we conclude that
\[
\left| I_7 \right| \leq \left|  \mathcal{H}(\ds z\tau {k}, \dsd p\tau k) {-}   \mathcal{H}(\ds z\tau {k-1}, \dsd p\tau k) \right| \leq C_K' \tau \|  \dsd z\tau{k}\|_{L^\infty} \| \dsd p \tau {k-1}\|_{L^1} \,. 
\]
All in all, from \eqref{step1-p} we infer that 
\begin{equation}
\label{conclu-est-p}
\varepsilon \nu \| \dsd p\tau k\|_{L^2} \left( \| \dsd p\tau k\|_{L^2} {-}  \| \dsd p\tau {k-1}\|_{L^2} \right)   +  \newmu \tau  \|\dot p\tk\|_{L^2}^2 \leq   \int_\Omega (\ds \sigma \tau k {-} \ds \sigma \tau {k-1} )_\dev \dsd p\tau k \dd x + C  \tau \|  \dsd z\tau{k}\|_{L^\infty} \| \dsd p \tau {k-1}\|_{L^1}\,. 
\end{equation}
\par
Summing up \eqref{conclu-est-u}, \eqref{conclu-est-z},  and \eqref{conclu-est-p}, 
adding $\tau \| \dsd z\tau k  \|_{L^2}^2$  to both sides of the inequality, and observing the cancellation of one term, 
we conclude that 
\begin{equation}
\label{intermediate-conclusion}
\begin{aligned}
&
 \eps \nu \|\dsd u\tau k \|_{ H^1, \bbD } (\| \dsd u\tau k \|_{ H^1, \bbD }{-} \| \dsd u\tau {k-1} \|_{ H^1, \bbD }) +
 \eps \| \dsd z\tau k\|_{L^2} \left(  \| \dsd z\tau k\|_{L^2}{-}  \| \dsd z\tau {k-1}\|_{L^2} \right) 
\\
& 
\hspace{1.5em}+
\varepsilon \nu \| \dsd p\tau k\|_{L^2} \left( \| \dsd p\tau k\|_{L^2} {-}  \| \dsd p\tau {k-1}\|_{L^2} \right)  +   \bar{\zeta}   \tau  \left( \| \dsd e\tau k \|_{L^2}^2  {+} \| \dsd z\tau k \|_{\Hs}^2 {+} \newmu  \|\dot p\tk\|_{L^2}^2 \right)
\\
& \leq  
 C\tau \Big( \| \dsd z\tau k\|_{L^\infty}  \| \dsd e\tau k\|_{L^2} {+}  \|\dsd e\tau k\|_{L^2} \| \sig{\dsd w\tau k}\|_{L^2}{+}  \| \dsd z \tau k\|_{L^\infty} \| \sig{\dsd w\tau k}\|_{L^2} {+}  \| \dsd F\tau k\|_{ \BD^*} \| \sig{\dsd w\tau k}\|_{L^1}
 \\
 & \qquad \qquad 
 {+}  \| \dsd F\tau k\|_{ \BD^*}  \| \dsd e\tau k\|_{L^1}  {+}  \| \dsd F\tau k\|_{ \BD^*}   \| \dsd p\tau k\|_{L^1} {+} \| \dsd z\tau k\|_{L^\infty}^2 + \| \dsd z\tau {k}\|_{L^\infty}  \| \dsd p\tau k\|_{L^1} 
 \Big)\,. 
\end{aligned}
\end{equation}
with  $\bar{\zeta} = \min\{ \gamma_1, 1 \}$. 
\par
\noindent
{\bf Step $2$:}  Let us now estimate 
$\| \dsd p\tau k \|_{L^1}$  for $k \in \{2,\ldots,N_\tau\}$.   We observe that 
\begin{equation}
\label{crucial-est-Vito}
\begin{aligned}
\alpha \| \dsd p\tau k \|_{L^1}  &  \stackrel{(1)}{\leq}
\mathcal{H}( \ds z\tau {k}, \dsd p\tau k)   
- \int_\Omega (\ds \rho\tau k)_\dev  \dsd p\tau k \dd x 
\\ & 
 \stackrel{(2)}{=} \mathcal{H}( \ds z\tau {k}, \dsd p\tau k)    + \int_\Omega \ds \rho\tau k : (\dsd e\tau k{-} \sig{\dsd w\tau k} ) \dd x  -  \int_\Omega \ds \rho\tau k : \sig{\dsd u\tau k} \dd x 
 \\ & 
 \stackrel{(3)}{=} \mathcal{H}( \ds z\tau {k}, \dsd p\tau k)    + \int_\Omega \ds \rho\tau k : (\dsd e\tau k{-} \sig{\dsd w\tau k} ) \dd x  -   \langle \ds F\tau{k}, \dsd u\tau k \rangle_{\BD(\Omega)} 
 \\ & 
 \begin{aligned}
 \stackrel{(4)}{=}   -   \eps\nu \| \dsd p\tau k \|_{L^2}^2 - \newmu \int_\Omega \ds p\tau k \dsd p\tau k \dd x +\int_\Omega (\ds \sigma\tau k)_\dev  \dsd p\tau k \dd x  & + \int_\Omega \ds \rho\tau k : (\dsd e\tau k{-} \sig{\dsd w\tau k} ) \dd x 
 \\ &  - \int_\Omega \ds \sigma\tau k : \sig{\dsd u\tau k} \dd x    - \eps \nu \int_\Omega \bbD \sig{\dsd u\tau k}: \GGG \sig{\dsd u\tau k}  \dd x 
 \end{aligned}
 \\ & 
 \stackrel{(5)}{\leq} - \newmu \int_\Omega \ds p\tau k \dsd p\tau k \dd x 
+ \int_\Omega (\ds \rho\tau k{-} \ds \sigma\tau k) : (\dsd e\tau k{-} \sig{\dsd w\tau k} ) \dd x 
 \\ & 
\leq \sqrt{\newmu} \| \ds p\tau k \|_{L^2}\sqrt{\newmu} \| \dsd p\tau k \|_{L^2}
 + \| \ds \rho\tau k{-} \ds \sigma\tau k\|_{L^2} \|\dsd e\tau k{-} \sig{\dsd w\tau k} \|_{L^2} 
  \stackrel{(6)}{\leq}  C \left( \|\dsd e\tau k\|_{L^2} {+} \| \sig{\dsd w\tau k} \|_{L^2}  +\sqrt{\newmu} \| \dsd p\tau k \|_{L^2}  \right) 
\end{aligned}
\end{equation}
where {\footnotesize (1)} follows from  \eqref{controllo-SL},  {\footnotesize (2)} is due to the fact that $\dsd p\tau k = \sig{\dsd u\tau k +\dsd w\tau k} - \dsd e\tau k$,  {\footnotesize (3)} follows from  the integration by parts formula \eqref{div-Gdir} observing that  $\dsd u\tau k \in H_\Dir^1(\Omega;\R^n)$ and that $ \ds F\tau{k} = -\mathrm{Div}(\ds \varrho\tau{k})$ by \eqref{2809192200}, 
  {\footnotesize (4)} ensues from testing \eqref{DEL-1} by $\dsd u\tau k$ and \eqref{DEL-3} by $\dsd p\tau k$,  {\footnotesize (5)} from  the fact that $-   \eps\nu \| \dsd p\tau k \|_{L^2}^2 \leq 0$ and 
$- \eps \nu \int_\Omega \bbD \sig{\dsd u\tau k}: \GGG \sig{\dsd u\tau k}  \dd x \leq 0$, and again from $\sig{\dsd u\tau k} =   \dsd e\tau k +\dsd p\tau k - \sig{\dsd w\tau k}, $ and {\footnotesize (6)} is due to the fact that $\rho \in L^\infty (0,T;\Mnn)$ and to the previously obtained estimates for $\overline{\sigma}_\tau$  and $\sqrt{\newmu} \olp$ in $L^\infty (0,T;L^2(\Omega;\Mnn))$, cf.\ \eqref{1709172242-SL}.
\par
In view of \eqref{crucial-est-Vito}, estimate \eqref{intermediate-conclusion} rewrites as  
\begin{equation}
\label{intermediate-conclusion-2}
\begin{aligned}
&
 \eps \nu \|\dsd u\tau k \|_{ H^1, \bbD } (\| \dsd u\tau k \|_{ H^1, \bbD }{-} \| \dsd u\tau {k-1} \|_{ H^1, \bbD }) +
 \eps \| \dsd z\tau k\|_{L^2} \left(  \| \dsd z\tau k\|_{L^2}{-}  \| \dsd z\tau {k-1}\|_{L^2} \right) +
\varepsilon \nu \| \dsd p\tau k\|_{L^2} \left( \| \dsd p\tau k\|_{L^2} {-}  \| \dsd p\tau {k-1}\|_{L^2} \right) 
\\
& 
+{  \bar{\zeta} }\tau \left( \| \dsd e\tau k \|_{L^2}^2  {+} \| \dsd z\tau k \|_{\Hs}^2 {+} \newmu  \|\dot p\tk\|_{L^2}^2 \right)
\\
& \leq  
 C\tau   \|\dsd e\tau k\|_{L^2} \| \sig{\dsd w\tau k}\|_{L^2} 
 +C\tau 
  \| \dsd z\tau k\|_{L^\infty} ( \| \dsd e\tau k\|_{L^2}{+}  \| \sig{\dsd w\tau k}\|_{L^2} {+}\| \dsd z\tau k\|_{L^\infty}) 
   \\ & \qquad 
  +C\tau ( \| \dsd F\tau k\|_{ \BD^*}{+}  \| \dsd z\tau {k}\|_{L^\infty})   \left(  \| \sig{\dsd w\tau k}\|_{L^2}
{+}   \| \dsd e\tau k\|_{L^2}  {+} \sqrt{\newmu}  \| \dsd p\tau k\|_{L^2} \right)\,. 
\end{aligned}
\end{equation}
\par
\noindent
{\bf Step $3$:} Let us introduce the vector
\[
v_k : = (\sqrt{\nu} \| \dsd u\tau k \|_{ H^1, \bbD }, \| \dsd z\tau k \|_{L^2},  \sqrt{\nu} \| \dsd p\tau k \|_{L^2}).
\] 
Then, observe that  the first three  terms on the left-hand side of \eqref{intermediate-conclusion-2} rewrite as 
$
 \eps \langle v_k, v_{k}{-} v_{k-1} \rangle$.
For the fourth term  we have the estimate 
\[
\begin{aligned}
{  \bar{\zeta} } \tau \left( \| \dsd e\tau k \|_{L^2}^2  {+} \| \dsd z\tau k \|_{\Hs}^2 {+} \newmu  \|\dot p\tk\|_{L^2}^2 \right)
 & \stackrel{(1)}{\geq} c \tau  \left( \| \dsd e\tau k \|_{L^2}^2  {+} \| \dsd z\tau k \|_{\Hs}^2 {+} \newmu  \|\dot p\tk\|_{L^2}^2 {+} \newmu \| \sig{\dsd u\tau k}\|_{L^2}^2  \right) - C\tau \| \sig{\dsd w\tau k}\|_{L^2}^2 
\\
&  \stackrel{(2)}{\geq}   \tilde{\zeta}   \tau  \left( \| \dsd e\tau k \|_{L^2}^2  {+} \| \dsd z\tau k \|_{\Hs}^2 {+} \newmu  \|\dot p\tk\|_{L^2}^2 {+} \nu \| \dsd u\tau k\|^2_{ H^1, \bbD }  \right) - C\tau \| \sig{\dsd w\tau k}\|_{L^2}^2  
\,,
\end{aligned}
\]
where for {\footnotesize (1)}  we have  used 
that 
$
 \newmu \| \sig{\dsd u\tau k}\|_{L^2}^2 \leq 3\newmu \| \dsd e\tau k \|_{L^2}^2 +  3\newmu \| \dsd p\tau k \|_{L^2}^2 + 3\newmu \|  \sig{\dsd w\tau k} \|_{L^2}^2, 
$
while {\footnotesize (2)} ensues from  \eqref{norma-equivalente} 
 and from the fact that $ \newmu \| \sig{\dsd u\tau k}\|_{L^2}^2 \geq \nu  \| \sig{\dsd u\tau k}\|_{L^2}^2$ (since, by assumption, $\nu\leq \newmu$), 
 with the  constant    $\tilde{\zeta}$   fulfilling 
$\tilde{\zeta}(3K_{\bbD}^2  +1) \leq  \bar{\zeta}$ with $K_{\bbD}$ from   \eqref{norma-equivalente}. 

As for the right-hand side of   \eqref{intermediate-conclusion-2}, we  shall  crucially use the   compact embedding of $\Hs(\Omega)$ into $L^\infty(\Omega)$,  which ensures that 
\begin{equation}\label{1809170146} 
\forall\, \delta >0 \ \exists\, C_\delta >0 \ \forall\, \zeta \in \Hs(\Omega)\, : \qquad 
\|\zeta\|_{L^\infty}^2 \leq \delta \|\zeta \|_{\Hs}^2 + C_\delta  \|\zeta\|_{L^1}^2\,.
\end{equation}
Therefore, also by Young's inequality we have the following estimate
\[
  \| \dsd z\tau k\|_{L^\infty} \left(  \| \dsd e\tau k\|_{L^2} {+}   \| \sig{\dsd w\tau k}\|_{L^2} {+} \| \dsd z\tau k\|_{L^\infty} \right) \leq 
  \delta ( \| \dsd z\tau k\|_{\Hs}^2 +  \| \dsd e\tau k\|_{L^2}^2  ) + C_\delta (\| \sig{\dsd w\tau k}\|_{L^2}^2 {+}  \|\dsd z\tau k \|_{L^1}^2)
\]
for some suitable constant $\delta>0$ to be specified later on. 
\par
All in all, from \eqref{intermediate-conclusion-2} we deduce 
\begin{equation}
\label{VitoHelp}
\begin{aligned}
\eps A_k (A_k {-}A_{k-1}) +  \tilde{\zeta} \tau B_k^2 &  \leq C\tau 
\left( 1 {+}  C_k^2  \right) +C\tau    \|\dsd z\tau k \|_{L^1}^2
\\ & \quad
+C\tau \| \dsd z\tau {k}\|_{L^\infty}   \left( \| \sig{\dsd w\tau k}\|_{L^2} {+}  \| \dsd e\tau k\|_{L^2}  {+} \sqrt{\newmu}  \| \dsd p\tau k\|_{L^2} \right) 
+C \delta \tau B_k^2,
\end{aligned}
\end{equation}
where we have used  the place-holders $A_k$, $B_k$,  and $C_k$  defined by 
\[
\begin{aligned}
A_k^2 : = |v_k|^2=  \nu\| \dsd u\tau k \|^2_{ H^1, \bbD } +  \| \dsd z\tau k \|_{L^2}^2 +  & \nu \| \dsd p\tau k \|_{L^2}^2,
\qquad 
B_k^2 : =  \| \dsd e\tau k \|_{L^2}^2  + \| \dsd z\tau k \|_{\Hs}^2 + \newmu  \|\dot p\tk\|_{L^2}^2  + \newmu \| \dsd u\tau k\|^2_{ H^1, \bbD }\,,
\\
&
 C_k^2:= \| \sig{\dsd w\tau k}\|_{L^2}^2  {+} \| \dsd F\tau k\|_{ \BD^*}^2 
\end{aligned}
\]
and estimated  
\begin{equation}\label{1808191210}
\begin{aligned}
&
C\tau   \|\dsd e\tau k\|_{L^2} \| \sig{\dsd w\tau k}\|_{L^2} \leq \delta \tau B_k^2 + C\tau  \| \sig{\dsd w\tau k}\|_{L^2}^2,
\\
 & C\tau  \| \dsd F\tau k\|_{ \BD^*}   \left( \| \sig{\dsd w\tau k}\|_{L^2} {+}  \| \dsd e\tau k\|_{L^2}  {+} \sqrt{\newmu}  \| \dsd p\tau k\|_{L^2} \right) 
\leq  \delta \tau B_k^2 + C\tau \| \dsd F\tau k\|_{ \BD^*}^2  + C\tau \| \sig{\dsd w\tau k}\|_{L^2}^2   
\end{aligned}
\end{equation}
via Young's inequality. Therefore, choosing $\delta>0$ in \eqref{VitoHelp} small enough in such a way as to absorb the term 
$C \delta \tau B_k^2$ on the left-hand side, 
we arrive at
\begin{equation}
\label{VitoHelpAiuto}
\begin{aligned}
\eps A_k (A_k {-}A_{k-1}) + \frac{ \tilde{\zeta} }2 \tau B_k^2 &  \leq C\tau  (1+ C_k^2) 
+C\tau    \|\dsd z\tau k \|_{L^1}^2 
+C\tau \| \dsd z\tau {k}\|_{L^\infty}   \left(  \| \dsd e\tau k\|_{L^2}  {+} \sqrt{\newmu}  \| \dsd p\tau k\|_{L^2} \right) \,.
\end{aligned}
\end{equation}
 Again relying on \eqref{1809170146}  we estimate the last term on the right-hand side of \eqref{VitoHelpAiuto} by
\[
C\tau \| \dsd z\tau {k}\|_{L^\infty}   \left(  \| \dsd e\tau k\|_{L^2}  {+} \sqrt{\mu}  \| \dsd p\tau k\|_{L^2} \right)\leq C\tau\delta  \| \dsd z\tau {k}\|_{\Hs} B_k + C_\delta \tau   \| \dsd z\tau {k}\|_{L^1}B_k  \leq C\tau\delta   B_k^2 + C\tau \delta' B_k^2 + C \tau   \| \dsd z\tau {k}\|_{L^1}^2.
\]
Choosing the constants $\delta$ and $\delta'$ such that $C\tau (\delta{+}\delta')\leq\frac{\tilde\mu}4 \tau  $ and using that  
$  \| \dsd z\tau {k}\|_{L^1}\leq A_k$, 
we obtain that 
\begin{equation}
\label{VitoHelpAiuto-unifr-prelim}
\begin{aligned}
\eps A_k (A_k {-}A_{k-1}) + \frac{ \tilde{\zeta} }4 \tau B_k^2   \leq C\tau   (1+ C_k^2) 
+C\tau    \|\dsd z\tau k \|_{L^1}^2 \leq C\tau   (1+ C_k^2)  +C\tau A_k  \|\dsd z\tau k \|_{L^1}   \qquad \forall\, k \in \{2,\ldots, N_\tau\}\,.
\end{aligned}
\end{equation}
{\bf Step $4$:}
 Let us now address the case $k=1$.
To start with,
let us set $\ds u\tau{-1}: = u_0$, $\ds z\tau{-1}: = z_0$, and $\ds p\tau{-1}: = p_0$, so that 
\begin{equation}
\label{dati-nulli}
\text{
$\dsd u\tau 0 = \tfrac{\ds u\tau 0 - \ds u\tau{-1}}{\tau} =0$ and, analogously, 
$\dsd z\tau 0 =0$ and $\dsd p\tau 0 =0$.}
\end{equation}
We test  \eqref{DEL-1}, with $k=1$, by $\dsd u\tau 1$. With an easy algrebraic  manipulation we  obtain
\begin{equation}
\label{test-at-k=1}
\begin{aligned}
\int_\Omega \eps \nu \bbD \sig{\dsd u\tau 1}\sig{\dsd u\tau 1} \dd x + \int_\Omega (\ds \sigma\tau 1 {-} \ds \sigma \tau{0}) : \sig{\dsd u\tau 1} \dd x    = 
\pairing{}{H^1}{\ds F\tau{1} {-} \ds F\tau{0}}{\dsd u\tau 1}
- \int_\Omega \ds \sigma \tau 0 : \sig{\dsd u\tau 1} \dd x 
+\pairing{}{H^1}{ \ds F\tau{0}}{\dsd u\tau 1}\,.
\end{aligned}
\end{equation}
Repeating the very same calculations as throughout \ref{later-used-for-u} and the subsequent formulae, we arrive at (cf.\ \eqref{conclu-est-u})
\begin{equation}
\label{conclu-est-u-1}
\begin{aligned}
&
 \eps \nu \|\dsd u\tau 1 \|_{ H^1, \bbD } (\| \dsd u\tau 1 \|_{ H^1, \bbD }{-} \| \dsd u\tau {0} \|_{ H^1, \bbD })
 + \gamma_1 \tau \| \dsd e\tau 1 \|_{L^2}^2 +\int_\Omega (\ds \sigma \tau 1 {-} \ds \sigma \tau {0} )_\dev \dsd p\tau 1 \dd x 
 \\ & \leq
 C\tau \Big(  \| \dsd z\tau 1\|_{L^\infty}  \| \dsd e\tau 1\|_{L^2} {+}  \|\dsd e\tau 1\|_{L^2} \| \sig{\dsd w\tau 1}\|_{L^2}{+}  \| \dsd z \tau 1\|_{L^\infty} \| \sig{\dsd w\tau 1}\|_{L^2} {+}  \| \dsd F\tau 1\|_{ \BD^*} \| \sig{\dsd w\tau 1}\|_{L^1}
 \\
 & \qquad \qquad 
 {+}  \| \dsd F\tau 1\|_{ \BD^*}  \| \dsd e\tau 1\|_{L^1}  {+}  \| \dsd F\tau 1\|_{ \BD^*}   \| \dsd p\tau 1\|_{L^1}  \Big) + \left|\int_\Omega (F(0)+\mathrm{Div}(\sigma_0)) \dsd u\tau 1  \dd x \right|,
 \end{aligned}
\end{equation}
where we have used that, by definition, $ \dsd u\tau {0}$, and exploited
the fact that $F(0)+\mathrm{Div}(\sigma_0) \in L^2(\Omega;\R^n)$ by 
 \eqref{EL-initial} to rewrite the last two terms on the right-hand side of \eqref{test-at-k=1}. 
\par  We now write \eqref{GV-2} for $k=1$. With algebraic manipulations, also taking into account that 
$\dsd z\tau 0 =0 $ and using  \eqref{EL-initial}, we arrive at 
\begin{equation}
\label{step1-z-bis}
\begin{aligned}
&
\calR(\dsd z \tau 1) + 
\varepsilon \int_\Omega (\dsd z\tau 1 {-} \dsd z\tau 0) \dsd z\tau 1  \dd x + \ass(\ds z\tau 1 -\ds z\tau0,\dsd z\tau 1)
\\
& 
\leq  \int_\Omega \big[W'(\ds z\tau 1)-W'(\ds z\tau 0) \big] \dsd z\tau1 \dd x
+\frac{1}{2}\int_\Omega  \big[\mathbb{C}'(\ds z\tau 0)- \mathbb{C}'(\ds z\tau1) \big] \ds e\tau1: \ds e\tau 1 \, \dsd z\tau1  \dd x -  \frac{1}{2}\int_\Omega  \left(\bbC'(\ds z\tau0) \ds e\tau 1 : \ds e\tau 1 {-} \bbC'(\ds z\tau 0) \ds e\tau {0} : \ds e\tau {0}  \right) \dsd z\tau1  \dd x
\\
& \quad 
 + C_K \tau \|\dsd z\tau 1\|_{L^\infty} \| \dsd p \tau 1 \|_{L^1}
+\left| \int_\Omega \left( \As \ds z\tau 0 {+} W'(\ds z\tau 0) {+}\tfrac12 \bbC'(\ds z\tau 0) \ds e\tau 0: \ds e\tau 0 \right) \dsd z\tau 1 \dd x  \right|\,.
  \end{aligned}
\end{equation} 
With the same calculations as throughout \eqref{step1-z}--\eqref{conclu-est-z}
we  conclude that 
\begin{equation}
\label{conclu-est-z-1}
\begin{aligned}
&
\eps \| \dsd z\tau 1\|_{L^2} \left(  \| \dsd z\tau 1\|_{L^2}{-}  \| \dsd z\tau {0}\|_{L^2} \right) + \tau \ass (\dsd z\tau 1, \dsd z\tau 1)
\\
 & 
 \leq C\tau  \| \dsd z\tau1\|_{L^\infty} \left( \| \dsd z\tau 1\|_{L^\infty}{+}  \| \dsd e\tau 1 \|_{L^2} +   \| \dsd p\tau 1 \|_{L^1}  \right) +\left| \int_\Omega \left( \As \ds z\tau 0 {+} W'(\ds z\tau 0) {+}\tfrac12 \bbC'(\ds z\tau 0) \ds e\tau 0: \ds e\tau 0 \right) \dsd z\tau 1 \dd x  \right|\,.
\end{aligned}
\end{equation} 
\par
Finally, we test \eqref{with-omega-k}, written for $k=1$, with $\dsd p\tau 1$. Taking into account that, by construction, $\dsd p\tau 0=0$, this leads to
\[
\int_\Omega  \ds \omega\tau 1 \dsd p\tau 1 \dd x
 +\varepsilon \nu \int_\Omega (\dsd p\tau1 {-} \dsd p\tau 0) \dsd p\tau 1  \dd x + \newmu\int_\Omega  (\ds p\tau 1 {-} \ds p\tau{0})\dsd p\tau 1 \dd x 
= \int_\Omega (\ds \sigma \tau 1 {-} \ds \sigma \tau {0} )_\dev \dsd p\tau1 \dd x +\int_\Omega (\ds \sigma\tau 0 {-} \newmu \ds p\tau 0) \dsd p\tau 1 \dd x \,.
\]
With the same   computations  as for \eqref{conclu-est-p}, we obtain
\begin{equation}
\label{conclu-est-p-1}
\calH(\dsd z\tau 1, \dsd p\tau 1) +
\varepsilon \nu \| \dsd p\tau 1\|_{L^2} \left( \| \dsd p\tau 1\|_{L^2} {-}  \| \dsd p\tau {0}\|_{L^2} \right)   +  \newmu \tau  \|\dsd p\tau 1\|_{L^2}^2 \leq   \int_\Omega (\ds \sigma \tau 1 {-} \ds \sigma \tau {0} )_\dev \dsd p\tau 1 \dd x  +\left| \int_\Omega (\ds \sigma\tau 0 {-} \newmu \ds p\tau 0) \dsd p\tau 1 \dd x\right|\,. 
\end{equation}
\par
We add up \eqref{conclu-est-u-1}, \eqref{conclu-est-z-1}, and \eqref{conclu-est-p-1} . The very same calculations as throughout Steps $2$ and $3$ lead to
\[
\eps A_1 (A_1 {-}A_{0}) +   \frac{ \tilde{\zeta} }4 \tau B_1^2    \leq C\tau  (1+ C_1^2)  +C\tau A_1  \|\dsd z\tau 1 \|_{L^1} + F_1\,.
\]
Here, the term $F_1$ subsumes the very last contributions on the right-hand sides of \eqref{conclu-est-u-1}, \eqref{conclu-est-z-1}, and \eqref{conclu-est-p-1}: in fact, for later use we introduce the   place-holder
\[
F_k: = \left|\int_\Omega \mathrm{D}_u \enen{\mu}(0,u_0,z_0,p_0) \dsd u \tau1 \dd x  \right|+ \left|\int_\Omega \mathrm{D}_z \enen{\mu}(0,u_0,z_0,p_0) \dsd z \tau1 \dd x  \right|+ \left|\int_\Omega \mathrm{D}_p \calE	_\nu(0,u_0,z_0,p_0) \dsd p \tau1 \dd x  \right|\,.
\]
Then, \eqref{VitoHelpAiuto-unifr-prelim}
extends to the index $k=1$, and we ultimately get
the relation
\begin{equation}
\label{VitoHelpAiuto-unifr}
\begin{aligned}
\eps A_k (A_k {-}A_{k-1}) +  \frac{ \tilde{\zeta} }4 \tau B_k^2     \leq C\tau 
+C\tau    \|\dsd z\tau k \|_{L^1}^2 \leq C\tau  (1+ C_k^2)  +C\tau A_k  \|\dsd z\tau k \|_{L^1} + \delta_{1,k} F_1   \qquad \forall\, k \in \{1,\ldots, N_\tau\}\,.
\end{aligned}
\end{equation}
 \par\noindent
{\bf Step $5$:}  From \eqref{VitoHelpAiuto-unifr} we infer
\[
\frac12 A_k^2 - \frac12 A_{k-1}^2 +  \frac{ \tilde{\zeta} }{4\eps} \tau   B_k^2 \leq  \frac{C}\eps \tau \left( 1{+}  A_k^2  {+} C_k^2  \right)  + \frac{\delta_{1,k}}\eps F_1   
   \qquad \forall\, k \in \{1,\ldots, N_\tau\}\,. 
\]
so that, adding up the above  relations  we  obtain  (recall $A_0=0$  by \eqref{dati-nulli}) 
\begin{equation}
\label{2Gr1}
 A_k^2 +  \sum_{j=1}^{k}  \frac{ \tilde{\zeta} }\eps \tau  B_j^2  \leq    \frac1{\eps}F_1 +  \frac{2CT}\eps  +  \frac{2C}\eps  \sum_{j=1}^{k} \tau   C_j^2   +  \frac{2C}\eps  \sum_{j=1}^{k} \tau   A_j^2 \,.
\end{equation}
We are now in a position to apply 
Lemma \ref{l:discrG1/2}  in Appendix A,  with the choices $a_k: = A_k^2$,   $\Lambda:=     \frac{2CT}\eps   +  \frac{2C}\eps  \sum_{j=1}^{k} \tau   C_j^2  +\frac1{\eps}F_1 $, and  $b = \frac{2C \tau}{\varepsilon}$:  hence we need to assume, e.g.,  $\tau/\varepsilon<1/(4C)$, so that $b<1$.  Then,  \eqref{discrG1/2} gives  (notice that  $  \sum_{j=1}^{k} \tau   C_j^2  \leq C'$  in view of \eqref{forces-u} and \eqref{dir-load}) 
\begin{equation}
\label{A-k-infty-eps}
\begin{aligned}
\sup_{k=1,\ldots, N_\tau }A_k^2  & \leq \frac1{1-\tau} \left(A_0^2 + \frac{2CT}\eps  +\frac{F_1}{\eps}   +  \frac{2C}\eps  \sum_{j=1}^{k} \tau   C_j^2  \right)  \exp \left( \frac{b}{1-b} k \right)  \\ &  \stackrel{(*)}{\leq} 2 \left(A_0^2 + \frac{2CT}\eps 
 +  \frac{2C}\eps  \sum_{j=1}^{k} \tau   C_j^2  +\frac1{\eps}F_1 \right) 
 \exp\left(\tfrac {4CT}\eps \right)   \doteq S_\eps^1 \qquad \text{ with } S_\eps^1 \uparrow +\infty \text{ as } \eps \down 0.
\end{aligned}
\end{equation}
where estimate $(*)$ is true for, say, $\tau \in [0,1/2]$.   Plugging the above estimate into \eqref{2Gr1}  we obtain
\begin{equation}
\label{B-k-1-eps} 
 \sum_{k=1}^{N_\tau}  \tau  B_k^2  \leq  S_\eps^2 \qquad \text{with } S_\eps^2 \uparrow +\infty \text{ as } \eps \down 0.
\end{equation}
Clearly, \eqref{A-k-infty-eps} and \eqref{B-k-1-eps} give estimates \eqref{2009170048}.
\par
\noindent 
\RCZ \textbf{Step $6$:}  Using that $B_k \geq A_k$, from \eqref{VitoHelpAiuto-unifr} we deduce
\begin{equation}
\label{VitoHelpAiuto-unifr-BIS}
\begin{aligned}
 A_k (A_k {-}A_{k-1}) + \frac{ \tilde{\zeta} \tau}{4 \eps} A_k^2 +  \frac{ \tilde{\zeta} \tau}{4 \eps}  B_k^2   \leq C\frac{\tau}\eps   (1+ C_k^2)    +\frac{\delta_{1,k}}{\eps}F_1  
+ C\frac\tau\eps A_k  \|\dsd z\tau k \|_{L^1}  \qquad \forall\, k \in \{1,\ldots, N_\tau\}\,.
\end{aligned}
\end{equation}
Hence, we are in   position to apply   the forthcoming  Lemma \ref{l:discrG2} with the choices $a_k:=A_k$, $M_k: = B_k$, $\gamma: =\frac{ \tilde{\zeta} \tau}{4 \eps} $,  $c_k:=C_k$,  and $R_k: = \frac{4C}{ \tilde{\zeta} }  \|\dsd z\tau k \|_{L^1}$
 and suitable choices for the constants $c$ and $\rho$  (notice that $A_0=0$ by construction, $R_k \leq \ol c A_k$, and we now have to take $\tau/\varepsilon< 1/(2 \ol c)$).  From \eqref{app-est-discrete}, along with $\sum_{k=1}^{N_\tau} \tau   C_k^2 \leq C'$
 (by \eqref{forces-u} and \eqref{dir-load})
and  \eqref{ulteriore-stima-SL},  we infer 
\begin{equation}
\label{B-k-1-eps-final} 
\exists\, S_2>0 \ \ \forall\, \tau,  \eps, \nu>0 \,: \qquad 
 \sum_{k=1}^{N_\tau}  \tau  B_k  \leq  S_2\,.
\end{equation}
Then, estimate \eqref{2009170054} ensues.   This concludes the proof. 
\QED
\begin{remark}
\label{sta-in-rmk}
\upshape
In the case  in which  only \eqref{1808191033} holds  in place of \eqref{forces-u},  we have only $F\in H^1(0,T; (H^1(\Omega;\R^n)^*)$, so that in place of \eqref{1808191206} we may infer only (with shorter notation for the norms)
\[
\begin{aligned}
\left| I_3 \right|    \leq  C \tau \| \dsd F\tau k\|_{ (H^1)^*} \left( \| \sig{\dsd w\tau k}\|_{L^2}{+}  \| \dsd e\tau k\|_{L^2}{+} 
\| \dsd p\tau k\|_{L^2}\right) \,.
\end{aligned}
\]
This   affects  the second inequality in \eqref{1808191210}, which is now replaced by
\[
C\tau  \| \dsd F\tau k\|_{ (H^1)^*}   \left( \| \sig{\dsd w\tau k}\|_{L^2} {+}  \| \dsd e\tau k\|_{L^2}  {+}  \| \dsd p\tau k\|_{L^2} \right) 
\leq  \delta \tau \left( \| \sig{\dsd w\tau k}\|_{L^2}^2 {+}  \| \dsd e\tau k\|_{L^2}^2  {+}  \| \dsd p\tau k\|_{L^2}^2 \right)  + C\tau \| \dsd F\tau k\|_{ (H^1)^*}^2 
\]
Now, since $\left( \| \sig{\dsd w\tau k}\|_{L^2}^2 {+}  \| \dsd e\tau k\|_{L^2}^2  {+}  \| \dsd p\tau k\|_{L^2}^2 \right)$ equals  $\sqrt{\mu}^{-1} B_k^2$, we may only control
\[
C\tau  \| \dsd F\tau k\|_{ (H^1)^*}   \left( \| \sig{\dsd w\tau k}\|_{L^2} {+}  \| \dsd e\tau k\|_{L^2}  {+}  \| \dsd p\tau k\|_{L^2} \right) 
\leq  \delta \tau B_k^2  + C_\mu \tau \| \dsd F\tau k\|_{ (H^1)^*}^2 \,,
\]
where $C_\mu$ depends  on $\mu$,  too, and blows up as 
$\mu\down 0$.  We could argue in the very same way for the rest of the proof, but the constant affect also the other estimates such as \eqref{VitoHelpAiuto}, that now have to contain constants depending also on $\mu$  on  the right-hand side. 
Thus, we end up proving  \eqref{enhanced-discr-est} with  a  constant depending also on $\mu$.
\end{remark}
\begin{remark}\label{2109170023}
\upshape
Estimate \eqref{crucial-est-Vito}, giving
\begin{equation}\label{1207191432}
\|\dot{p}\tk\|_{L^1(\Omega;\MD)} \leq C \left( \|\dsd e\tau k\|_{L^2} {+} \| \sig{\dsd w\tau k} \|_{L^2}  +\sqrt{\newmu} \| \dsd p\tau k \|_{L^2}  \right) \,,
\end{equation}
 is fundamental since it allows us
to estimate  $\| \dsd p\tau k\|_{L^1} $ by means of  the term $B_k$ and of $\|\sig{\dsd w\tau k}\|_{L^2}$
In this way, the terms 
containing $\| \dsd p\tau k\|_{L^1} $ can be partly absorbed into the left-hand side. 
If we did not resort to estimate \eqref{crucial-est-Vito}, we would have to deal with the term $C\tau \| \dsd p\tau k\|_{L^1}^2$ on the right-hand side of  \eqref{VitoHelpAiuto-unifr}, which would  be controlled only by considering constants depending on $\mu$, as explained in 
 Remark \ref{sta-in-rmk}
 above. 
\end{remark}
With the last result of this section, we prove a  \emph{discrete  version of the 
Energy-Dissipation upper estimate \eqref{enineq-plast}}. 
Estimate
\eqref{discr-UEDE} below 
shall play a crucial role in existence proof for viscous solutions (i.e.,  with $\eps,\,\nu>0$ fixed). Indeed,  it will be sufficient to pass to the limit in
 the Energy-Dissipation balance 
 \eqref{discr-UEDE}  as $\tau \down 0$ (showing that the 
remainder term
on its right-hand side tends to zero as $\tau \down 0$)
to obtain \eqref{enineq-plast} for a limit triple $(u,z,p)$. In turn,  by 
  Proposition~\ref{prop:charact}  the validity of  \eqref{enineq-plast} is equivalent to the fact that $(u,z,p) $
 solve Problem \ref{prob:visc}.
\begin{proposition}[Discrete Energy-Dissipation upper estimate]
\label{prop:discr-UEDE}
The piecewise constant and linear interpolants of the discrete solutions $(\ds u\tau k,\ds z\tau k,\ds p\tau k)_{k=1}^N$ fulfill
\begin{equation}
\label{discr-UEDE}
\begin{aligned}
& \enen{\mu} (t, \pwl u\tau(t), \pwl z\tau(t), \pwl p\tau(t)) +
\int_{\upwc t\tau(s)}^{\pwc t\tau(t)}  \left( \disv\eps\nu(u_\tau'(r)) {+} \calR_\eps(z_\tau'(r)) {+} \dish\eps\nu(\pwc z\tau(r),  p_\tau'(r))\right) \dd r 
\\
  &\quad 
  + \int_{\upwc t\tau(s)}^{\pwc t\tau(t)}    \Big( \disv\eps\nu^*(\mathrm{Div}( \pwc \sigma\tau(r))  {+}\pwc F\tau(r)) {+}   \calR_\eps^*\left({-}\As(\pwc z\tau(r)) {-}W'(\pwc z\tau(r)){-}\tfrac12 \bbC'(\pwc z\tau(r)) \pwc e\tau(r) : \pwc e\tau (r) {-}\tau\pwc\lambda\tau(r)\right)  \\  
  &\qquad \qquad \qquad 
  {+}  \dish\eps\nu^*(\upwc z\tau(r), -\newmu\pwc p\tau(r) + (\pwc\sigma\tau(r))_\dev)
   \Big) \dd r 
   \\
   &
   \leq
\enen{\mu} (s, \pwl u\tau(s), \pwl z\tau(s), \pwl p\tau(s)) +
\int_{\upwc t\tau(s)}^{\pwc t\tau(t)} \int_\Omega \bbC(\pwl z\tau (r)) (\sig{\pwl u\tau(r){+}w(r)}{-} \pwl p\tau (r)): \sig{w'(r)} \dd x  \dd r
  \\
  & \quad 
-  \int_{\upwc t\tau(s)}^{\pwc t\tau(t)}  \langle F'(r), \pwl u{\tau}(r) {+}w(r)  \rangle_{H^1(\Omega;\R^n)} \dd r 
 -  \int_{\upwc t\tau(s)}^{\pwc t\tau(t)}  \langle F(r), w'(r) \rangle_{H^1(\Omega;\R^n)}  \dd r   
+\mathfrak{R}_\tau(s,t)
\end{aligned}
\end{equation} 
where $\pwc \lambda \tau$ is a selection in $\partial_z \dish\eps\nu(\pwc z\tau, \pwl p\tau')$ fulfiling the Euler-Lagrange equation  \eqref{EL-interp-2},  
  the remainder term  is  given by
\begin{equation}
\label{remainder}
\begin{aligned}
\mathfrak{R}_\tau(s,t):= C_3 \int_{\upwc t\tau(s)}^{\pwc t\tau(t)}   & \left( \|\pwc u\tau{-} \pwl u\tau \|_{H^1(\Omega)} {+} \|  \pwc z\tau{-}  \pwl z\tau\|_{\Hs(\Omega)} {+}  \| \pwc p\tau{-}  \pwl p\tau\|_{L^2(\Omega)}   {+} \| \pwc w\tau{-} w \|_{H^1(\Omega)} \right) 
\\
& \qquad \times 
\left(  \| u_\tau' \|_{H^1(\Omega)}{+}  \| z_\tau' \|_{\Hs(\Omega)} {+}\| p_\tau' \|_{L^2(\Omega)}  \right) \dd r
\end{aligned}
\end{equation}
and the constant $C_3$,  uniform w.r.t.\ $\eps,\,\nu, \, \mu, \, \tau$,   only depends on the constant $C_1$ from \eqref{basic-en-est-SL}.  
\end{proposition}
\begin{proof}
With the very same calculations as in the proof of \cite[Lemma 6.1]{KRZ2},
also based on the  convex analysis arguments leading
to \eqref{FM-bal},
  from the 
Euler-Lagrange equation \eqref{abstract-subdiff}  we deduce that the interpolants
$\pwc q\tau$, $\upwc q\tau$, and $q_\tau$ fulfill 
\begin{equation}
\label{discr-UEDE-q}
\begin{aligned}
&
 \enen{\mu} (t,\pwl q\tau(t)) +
 \int_{\upwc t\tau(s)}^{\pwc t\tau(t)} \left( \psie\eps\nu (\pwc q\tau(r),q_\tau'(r)) {+}
\psie\eps\nu^* (\pwc q\tau(r),{-}\rmD_q \enen{\mu} (\pwc t\tau(r), \pwc q\tau(r)))  {-}\tau \partial_q \psie\eps\nu  (\pwc q\tau(r),q_\tau'(r))\right)  \dd r = 
\\
& 
+ \enen{\mu} (s, \pwl q\tau(s)) +  \int_{\upwc t\tau(s)}^{\pwc t\tau(t)}
\partial_t \enen{\mu} (r,q_\tau(r)) \dd r - \ddd{ \int_{\upwc t\tau(s)}^{\pwc t\tau(t)} 
\pairing{}{\bfQ}{\rmD_q \enen{\mu} (\pwc t\tau(r), \pwc q\tau(r)) {-}  \rmD_q \enen{\mu} (r, \pwl q\tau(r))}{q_\tau'(r)} \dd r\,}{$\doteq R_\tau(s,t)$}.
\end{aligned}
\end{equation} 
Then,
taking into account \eqref{EL-interpolants}, 
 it is immediate to check that the left-hand side of \eqref{discr-UEDE-q}
translates into the left-hand side of \eqref{discr-UEDE}. Analogously, 
taking into account the explicit calculation
\eqref{part-t-q}  of $\partial_t \calE_\mu$,  we see that the  first two terms on the right-hand side of  \eqref{discr-UEDE-q} correspond to the first four terms on the right-hand side of \eqref{discr-UEDE}. 
We now estimate the remainder term $R_\tau(s,t)$
 as follows. First of all, we observe that
\[
|R_\tau(s,t)| \leq |R_\tau^1(s,t)| +| R_\tau^2(s,t) |+ |R_\tau^3(s,t)|.
\]
Then,
\begin{equation}
\label{est1-R}
\begin{aligned}
|R_\tau^1(s,t)|  &  = \left|  \int_{\upwc t\tau(s)}^{\pwc t\tau(t)} 
\pairing{}{H^1(\Omega)}{\mathrm{Div}(\bbC(\pwc z\tau)\pwc e\tau) {-} \mathrm{Div}(\bbC(\pwl z\tau)(\sig{\pwl u\tau + w} {-} \pwl p\tau )}{u_\tau'}\right| \dd r 
\\
&
\leq   \int_{\upwc t\tau(s)}^{\pwc t\tau(t)}  \left| \int_\Omega (\bbC(\pwc z\tau){-}\bbC(\pwl z\tau) )\pwc e\tau : \sig{u_\tau'} \dd x  \right|  \dd r 
+ \int_{\pwc t\tau(s)}^{\pwc t\tau(t)} \left|   \int_\Omega  \bbC(\pwl z\tau)(\pwc e\tau{-} (\sig{\pwl u\tau + w} {-} \pwl p\tau )) : \sig{u_\tau'} \dd x \right| \dd r  
\\ & 
\stackrel{(1)}{\leq}   C  \int_{\upwc t\tau(s)}^{\pwc t\tau(t)} \left( \| \pwc z\tau{-} \pwl z\tau \|_\infty \| \pwc e\tau\|_{L^2}  \|\sig{u_\tau'}\|_{L^2} {+} \| \pwl z\tau \|_{L^\infty} \|\pwc e\tau{-} (\sig{\pwl u\tau + w} {-} \pwl p\tau )\|_{L^2} \|\sig{u_\tau'}\|_{L^2}  \right) \dd r  
\\
&
\stackrel{(2)}{\leq} C   \int_{\upwc t\tau(s)}^{\pwc t\tau(t)} \left(   \| \pwc u\tau{-} \pwl u\tau \|_{H^1} {+} \| \pwc z\tau{-} \pwl z\tau \|_{\Hs)} {+}   \| \pwc p\tau{-} \pwl p\tau \|_{L^2} {+}   \| \pwc w\tau - w\|_{H^1}\right)\|\sig{u_\tau'}\|_{L^2} \dd r \,, 
\end{aligned}
\end{equation}
where (1) ensues from \eqref{spd},
 estimate \eqref{1709172242-SL}
and the continuous embedding $\Hs(\Omega) \subset \mathrm{C}(\ol\Omega)$.   For (2) we 
have again used the latter embedding along with the identity 
$
\pwc e\tau = \sig{\pwc u \tau {+}\pwc w\tau} - \pwc p\tau.
$
Secondly,  
\begin{equation}
\label{est2-R}
\begin{aligned}
|R_\tau^2(s,t)|  &  = \Big|  \int_{\upwc t\tau(s)}^{\pwc t\tau(t)} 
\Big( \ass(\pwc z \tau{-}\pwl z \tau, z_\tau'){+} \int_\Omega  (W'(\pwc z\tau){-}W'(\pwl z\tau))z_\tau'  \dd x
\\
& \qquad \qquad \qquad 
 {+}
\frac12   \int_\Omega  \left( \bbC'(\pwc z\tau) \pwc e\tau : \pwc e\tau {-}  \bbC'(\pwl z\tau)  \RCZ (\sig{\pwl u\tau + w} {-} \pwl p\tau ):(\sig{\pwl u\tau + w} {-} \pwl p\tau )   \Big) z_\tau' \dd x 
\right)
\dd r \Big| 
\\
&  \stackrel{(3)}{\leq} C \int_{\upwc t\tau(s)}^{\pwc t\tau(t)} 
\Big( \| \pwc z\tau{-} \pwl z\tau \|_{\Hs}\|z_\tau' \|_{\Hs}{+}
 \| \pwc z\tau{-} \pwl z\tau \|_{2}\|z_\tau' \|_{L^2}
 \\ & \qquad \qquad \qquad 
  {+} \| \pwc u\tau{-} \pwl u\tau \|_{H^1} \|z_\tau' \|_{\Hs} {+}   \| \pwc p\tau{-} \pwl p\tau \|_{L^2}  \|z_\tau' \|_{\Hs}   {+} \| \pwc w\tau{-}  w  \|_{H^1} \|z_\tau' \|_{\Hs} \Big) \dd r   \,,
\end{aligned}
\end{equation} 
where for (3) we have used that, since $\pwc z\tau,\, \pwl z\tau \in [m_0,1]$ by property \eqref{strict-posivity} and $W$ is of class $\mathrm{C}^2$ on $[m_0,1]$, 
it is possible to estimate 
$\|W'(\pwc z\tau){-}W'(\pwl z\tau)\|_{L^2} \leq C  \| \pwc z\tau{-} \pwl z\tau \|_{2}$. We have also 
 estimated 
\[
\begin{aligned}
&
 \left\|\left( \bbC'(\pwc z\tau) \pwc e\tau: \pwc e\tau  {-}  \bbC'(\pwl z\tau) (\sig{\pwl u\tau + w} {-} \pwl p\tau ):(\sig{\pwl u\tau + w} {-} \pwl p\tau )    \right) z_\tau' \right \|_{L^1}
 \\
 &  \leq C  \| \pwc z\tau{-} \pwl z\tau \|_{L^\infty} \| \pwc e\tau  \|_{L^2}^2\|z_\tau' \|_{L^\infty} + 
 \|z_\tau \|_\infty\| \pwc e\tau {+} (\sig{\pwl u\tau + w} {-} \pwl p\tau) \|_{L^2} \| \pwc e\tau {-} (\sig{\pwl u\tau + w} {-} \pwl p\tau) \|_{L^2}\|z_\tau' \|_{L^\infty} 
 \\
 &
 \leq C  \| \pwc z\tau{-} \pwl z\tau \|_{\infty} \|z_\tau' \|_{L^\infty}  +  \| \pwc e\tau {-} (\sig{\pwl u\tau + w} {-} \pwl p\tau)\ \|_{L^2}\|z_\tau' \|_{L^\infty} 
 \end{aligned}
\]
thanks to \eqref{spd} and  estimate  \eqref{1709172242-SL};  subsequently,  we have  estimated  $ \| \pwc e\tau {-} (\sig{\pwl u\tau + w} {-} \pwl p\tau) \|_{L^2}$ as 
we did for \eqref{est1-R}.    All in all, this leads to \eqref{est2-R}. 
Thirdly, we see that 
\begin{equation}
\label{est3-R}
\begin{aligned}
|R_\tau^3(s,t)|  &  = \left|  \int_{\upwc t\tau(s)}^{\pwc t\tau(t)}  \int_\Omega  
\left( \newmu \pwc p\tau{-} \newmu \pwl p\tau{+} (\bbC(\pwl z\tau) (\sig{\pwl u\tau + w}  {-} \pwl p\tau))_\dev  {-}(\pwc \sigma\tau)_\dev \right) p_\tau'  \dd x \dd r \right|
\\
&  \stackrel{(4)}{\leq}   C \int_{\pwc t\tau(s)}^{\pwc t\tau(t)} \left( \|\pwc u\tau- \pwl u\tau \|_{H^1(\Omega)} {+} \|  \pwc z\tau{-}  \pwl z\tau\|_{\Hs(\Omega)} {+}  \| \pwc p\tau{-}  \pwl p\tau\|_{L^2}  {+} \| \pwc w\tau{-}  w  \|_{H^1}  \right) \| p_\tau' \|_{L^2} \dd r\,.
 \end{aligned}
\end{equation}
Here, (4) is  due to \eqref{spd} and the  previously obtained   estimates   \eqref{1709172242-SL}, which also enter into the  estimate 
\[
\begin{aligned}
&
\| (\bbC(\pwl z\tau)   (\sig{\pwl u\tau + w} {-} \pwl p\tau )  {-} \bbC(\pwc z\tau) \pwc e\tau)_\dev  p_\tau'\|_1
\\
 & \leq \| \bbC(\pwl z\tau) {-} \bbC(\pwc z\tau) \|_\infty  \|\pwc e\tau\|_{L^2} \|p_\tau'\|_{L^2}
+ \| \bbC(\pwc z\tau) \|_\infty  \|\pwc e\tau-   (\sig{\pwl u\tau + w} {-} \pwl p\tau )  \|_{L^2} \|p_\tau'\|_{L^2}
\\ & \leq C \left(  \|  \pwc z\tau{-}  \pwl z\tau\|_{\infty} {+}  \|\pwc u\tau- \pwl u\tau \|_{H^1(\Omega)} {+}  \|\pwc p\tau- \pwl p\tau \|_{L^2}  {+} \| \pwc w\tau{-} w \|_{H^1} \right)     \|p_\tau'\|_{L^2}\,. 
\end{aligned}
\]
Combining \eqref{est1-R}--\eqref{est3-R} with \eqref{discr-UEDE-q}, we conclude the proof.
\end{proof}
\section{Existence of solutions to the viscous problem}
\label{s:exist-viscous}
In this section we address the existence of solutions to Problem \ref{prob:visc} 
\underline{for fixed $\eps>0$, $\nu>0$, and   $\mu>0$}.   
Besides the standing assumptions  from Section \ref{ss:2.1}, which we omit to explicitly recall, 
 our existence result, Thm.\ \ref{th:1} below, will require
 conditions \eqref{EL-initial} on the initial data $(u_0,z_0,p_0)$.
 We notice that to prove the 
  \emph{sole} 
 existence of solutions for
 Problem \ref{prob:visc}  it  would be  enough to assume \eqref{1808191033} in place of \eqref{forces-u} 
 since for $\mu>0$ fixed  estimates \eqref{enhanced-discr-est} with  constants depending on $\mu$ 
  (cf.\ also Remark \ref{sta-in-rmk}) 
 are enough. The situation is similar for the  vanishing-viscosity  analysis carried out in  Section~\ref{s:van-visc},
 throughout which we shall keep 
  the hardening parameter $\mu$ fixed.  However, 
   condition \eqref{forces-u} ensures  that 
 the solutions we exhibit in Theorem~\ref{th:1} (in fact, that \emph{all} 
 viscous solutions arising from the time-discretization procedure set up in Section~\ref{s:4}) enjoy 
 the upcoming
 estimates \eqref{uniform-est-cont} uniformly w.r.t.\  the parameters $\eps$,   $\nu$, \emph{and} $\newmu$.
    This  will be at the basis 
   of the 
   vanishing-hardening  analysis  carried out in  Section~\ref{s:van-hard}.
 \begin{theorem}
 \label{th:1}
 Under the assumptions in Section~\ref{s:2},
  and \eqref{EL-initial}
   as well, 
 Problem \ref{prob:visc} admits a solution triple $(u,z,p)$ enjoying the additional regularity and summability properties
 \begin{equation}
 \label{additional-reg&sum}
 u \hspace{-0.1em}\in W^{1,\infty} (0,T; H_\Dir^1(\Omega;\R^n)), \ z \hspace{-0.1em} \in  \GGG H^{1,} (0,T;  H^m(\Omega)  )  {\cap} W^{1,\infty} (0,T; \GGG L^2(\Omega)  ), \ p \hspace{-0.1em} \in W^{1,\infty}(0,T;L^2(\Omega;\MD)).
 \end{equation}
 \par
\noindent Moreover, the triple $(u,z,p)$  fulfills 
\GGG \begin{equation}
\label{bound-diss+conjs}
      \int_0^T \MVito(r, q(r), q'(r)) \dd r  \leq C_4
\end{equation}
for  a constant $C_4>0$ independent of $\eps$, $\mu$,  $\nu>0$. 
Additionally,  we have the following bounds \emph{uniformly} w.r.t.\ all parameters $\eps,\,\nu$ and $\mu$ provided that
 \underline{$\nu\leq\newmu$}  (recall that   $e: =  \sig{u{+}w} -p $): 
\begin{equation}
\label{uniform-est-cont}
\begin{aligned}
&
 \| e \|_{W^{1,1}(0,T;L^2(\Omega;\Mnn))} + \| z\|_{W^{1,1}(0,T;\Hs(\Omega))} + \sqrt{\newmu}  \|p
 \|_{W^{1,1}(0,T;L^2(\Omega;\MD))} \\ & + \sqrt{\newmu}  \| u  
 \|_{W^{1,1}(0,T;H^1(\Omega;\R^n))}  + \| p\|_{W^{1,1}(0,T;L^1(\Omega;\MD))} \leq C_5\,.
 \end{aligned}
\end{equation}
 \end{theorem}
 \begin{proof}
By virtue of   Proposition \ref{prop:charact},  it is sufficient to show 
that the piecewise constant and linear interpolants 
 of the discrete solutions constructed in Section~\ref{s:4} converge to a triple  $(u,z,p)$
 fulfiling  the initial conditions \eqref{initial-conditions-visc} and the Energy-Dissipation upper estimate
\eqref{enineq-plast}.  For this, we  shall  take the limit of 
the discrete Energy-Dissipation inequality \eqref{discr-UEDE}, using that, thanks to \eqref{forces-u} and \eqref{dir-load},
\begin{equation}
\label{converg-Fw}
\pwc F{\tau_k} \to F \quad \text{in } H^1(0,T;\BD(\Omega)^*)\,,\qquad
\pwc w{\tau_k} \to w \quad \text{in } H^1(0,T;H^1(\Omega;\R^n))\,.
\end{equation} 
Estimates
\eqref{uniform-est-cont}
will be inherited by $(u,z,p)$
 and $e$ from the analogous bounds for the approximate solutions via lower semicontinuity arguments.
Accordingly, the proof is split in three steps. 
\par
\noindent
\textbf{Step $1$: Compactness.} 
Let us consider a null sequence $\tau_k \down 0$ and, accordingly, the 
discrete solutions
$(\pwc u{\tau_k}, \pwl u{\tau_k}, \pwc z{\tau_k}, \pwl z{\tau_k}, \pwc p{\tau_k}, \pwl p{\tau_k})_k$, along with $(\pwc e{\tau_k}, \pwl e{\tau_k})_k$. 
It follows from estimates \eqref{basic-en-est-SL} and  \eqref{2009170048}, 
combined with standard weak compactness arguments and Aubin-Lions type compactness results (cf., e.g., \cite{Simon87}), that there exists a triple $(u,z,p)$ fulfiling \eqref{additional-reg&sum}, such that the following convergences hold
\begin{subequations}
\label{converg-uzpe}
\begin{align}
&
\label{converg-u}
 \pwl u{\tau_k} \weaksto u && \text{in } W^{1,\infty}(0,T;H^1(\Omega;\R^n)), &&  \pwl u{\tau_k} \to u && \text{in } \mathrm{C}^0([0,T]; Y), 
 \\
 &
 \label{converg-z}
 \pwl z{\tau_k} \weaksto z && \text{in } \GGG H^1(0,T;\Hs(\Omega)),  &&  \pwl z{\tau_k} \to z && \text{in }   \mathrm{C}^0([0,T]; Z),
\\
&
\label{converg-p}
 \pwl p{\tau_k} \weaksto p && \text{in } W^{1,\infty}(0,T;L^2(\Omega;\MD)), &&  \pwl p{\tau_k} \to p && \text{in }   \mathrm{C}^0([0,T]; W)
\end{align}
\end{subequations}
for any Banach spaces $Y$, $Z$, and $W$ such that $H_\Dir^1(\Omega;\R^n) \Subset Y$, $\Hs(\Omega) \Subset Z$ (in particular, for $Z = \mathrm{C}^0(\overline\Omega)$), and $L^2(\Omega;\MD)\Subset W$.  Hence,  we 
have that 
\begin{equation}
\label{weak-converg-uzpe}
 \pwl u{\tau_k}(t) \weakto u(t) \text{ in } H^1(\Omega;\R^n), 
 \quad  \pwl z{\tau_k}(t) \weakto z(t) \text{ in } \Hs(\Omega),  \quad  \pwl p{\tau_k}(t) \weakto p(t) \text{ in } L^2(\Omega;\MD)  \ \text{for all } t \in [0,T]. 
\end{equation}
Furthermore,  it follows from  estimates \eqref{2009170048}
 that  
\begin{subequations}
\label{stability-properties}
\begin{equation}
\label{stabilty-uzp}
\begin{aligned}
&
\| \pwc u{\tau_k} {-}  \pwl u{\tau_k} \|_{L^\infty(0,T;H^1(\Omega;\R^n))} \leq \tau_k \| \pwl u{\tau_k}' \|_{L^\infty(0,T;H^1(\Omega;\R^n))} \to 0 \qquad \text{as } k \to +\infty,
\\
&
\| \pwc z{\tau_k} {-}  \pwl z{\tau_k} \|_{L^\infty(0,T;\Hs(\Omega))} \leq \tau_k \| \pwl z{\tau_k}' \|_{\GGG L^2(0,T;\Hs(\Omega)) } \to 0 \qquad \text{as } k \to +\infty,
\\
&
\| \pwc p{\tau_k} {-}  \pwl p{\tau_k} \|_{L^\infty(0,T;L^2(\Omega;\MD))} \leq \tau_k \| \pwl p{\tau_k}' \|_{L^\infty(0,T;L^2(\Omega;\MD))} \to 0 \qquad \text{as } k \to +\infty,
\end{aligned}
\end{equation}
and we have the very same estimates for $\upwc u{\tau_k}$, $\upwc z{\tau_k}$, and  $\upwc p{\tau_k}$. 
Therefore, the pointwise  convergences \eqref{weak-converg-uzpe} hold for the sequences $\pwc u{\tau_k}$, $\upwc u{\tau_k}$, $\pwc z{\tau_k}$, $\upwc z{\tau_k}$ $\pwc p{\tau_k}$, and $\upwc p{\tau_k}$,  as well. 
Since $w\in W^{1,\infty} (0,T; H^1(\Omega;\R^n))$, it is not difficult to check that, likewise, 
\begin{equation}
\label{stabilty-w}
\| \pwc w{\tau_k} {-}  \pwl w{\tau_k} \|_{L^\infty(0,T;H^1(\Omega;\R^n))} \leq \tau_k \|w'\|_{L^\infty (0,T; H^1(\Omega;\R^n))} \to 0 \qquad \text{as } k \to +\infty\,.
\end{equation}
\end{subequations}
As a consequence of \eqref{converg-Fw}, \eqref{converg-uzpe}, and  \eqref{stabilty-uzp}, we also have that 
\begin{equation}
\label{converg-e}
\begin{aligned}
&
\pwc e{\tau_k} = \sig{\pwc u{\tau_k}{+} \pwc w{\tau_k}} -\pwc p{\tau_k} \weaksto e: = \sig{u{+}w} -p 
&& \qquad 
 \text{in }L^\infty(0,T;L^2(\Omega;\Mnn))\,,
 \\
 &
\pwc e{\tau_k}(t) \weakto e(t)  &&   \qquad \text{in }L^2(\Omega;\Mnn) \quad \text{for all } t \in [0,T] \,. 
 \end{aligned}
 \end{equation}
  Then,  it turns out  that $\pwc \sigma{\tau_k}\weaksto \sigma $ in $L^\infty(0,T;L^2(\Omega;\Mnn))$, since
\begin{equation}
\label{converg-sigma}
(\pwc \sigma{\tau_k} {-} \sigma) = (\bbC(\pwc z{\tau_k}) {-} \bbC(z)) \pwc e{\tau_k} + \bbC(z) (\pwc e{\tau_k}{-}e) \weaksto 0 \text{ in } L^\infty(0,T;L^2(\Omega;\Mnn)),
\end{equation}
since $\|\bbC(\pwc z{\tau_k}) {-} \bbC(z)\|_{L^\infty (0,T;L^\infty(\Omega;\Mnn))} \to 0 $ by the Lipschitz continuity of $\bbC$, 
 combined with convergences \eqref{converg-z} and \eqref{stabilty-uzp}.  	 
Finally, let us observe that, by the Lipschitz continuity of the functional $z\mapsto \calH(z,\pi)$ (cf.\ 
\eqref{props-calH-3+1/2}), and \cite[Prop.\ 1.85]{Mord-book-I}, the function $\pwc \lambda\tau$ featuring in the argument of $\calR_\eps^*$ on  the left-hand side of \eqref{discr-UEDE} fulfills
\[
\|\pwc \lambda{\tau}\|_{L^\infty(0,T;\mathrm{M}(\Omega))} \leq C_K \|\pwl p\tau' \|_{L^\infty(0,T;L^1(\Omega;\MD))},
\]
so that,  by virtue of estimate  \eqref{2009170048},
\begin{equation}
\label{ci-vuole}
\tau_k \pwc \lambda{\tau_k} \to 0 \quad \text{ in } L^\infty(0,T;\Hs(\Omega)^*) \quad \text{ as } k \to+\infty.
\end{equation} 
\par
Since $\eta_{\tau_k}(0) = \eta_0$ for $\eta \in \{ u,z,p\}$, it follows from convergences \eqref{weak-converg-uzpe} that the triple $(u,z,p)$ complies with the initial conditions \eqref{initial-conditions-visc}. 
\par
\noindent
\textbf{Step $2$: Limit passage in \eqref{discr-UEDE}}. Since we aim at \eqref{enineq-plast}, it is sufficient to take the limit of \eqref{discr-UEDE} written for $s=0$ and $t=T$.
We start by discussing the limit passage on the left-hand side of \eqref{discr-UEDE}.
Relying on  the convergences from Step $1$,    easily  check that 
\[
\liminf_{k\to+\infty}  \enen{\mu} (\pwc t{\tau_k}(T), \pwc u{\tau_k}(T), \pwc z{\tau_k}(T), \pwc p{\tau_k}(T))  \geq  \enen{\mu} (T, u(T),z(T), p(T))\,. 
\]
 In view of convergences \eqref{converg-uzpe}, we immediately have 
\[
\lim_{k\to+\infty}\int_{0}^{T} \left( \disv\eps\nu(u_{\tau_k}'(r)) {+} \calR_\eps(z_{\tau_k}'(r)) \right) \dd r \geq \int_0^T \left( \disv\eps\nu(u'(r)) {+} \calR_\eps(z'(r)) \right) \dd r. 
\]
It follows from \eqref{props-calH-3+1/2} that 
\[
\int_0^T \left| \calH(\upwc z{\tau_k}(t), \pwl p{\tau_k}'(t)) {-}  \calH(z(t), \pwl p{\tau_k}'(t))\right|  \dd t \leq \| \upwc z{\tau_k}{-}z\|_{L^\infty(0,T;L^\infty(\Omega))} \|  \pwl p{\tau_k}'\|_{L^1(0,T; L^1(\Omega;\MD))} \to 0 \qquad \text{as } k \to+\infty,
\]
since $ \upwc z{\tau_k}\to z$ in $L^\infty(0,T;L^\infty(\Omega))$ by \eqref{converg-z} and \eqref{stability-properties}.
On the other hand, by \eqref{converg-p} we have 
\[
\liminf_{k\to+\infty} \int_0^T  \calH(z(t), p_{\tau_k}'(t)) \dd t \geq \int_0^T \calH(z(t), p'(t)) \dd t \,.
\]
Therefore, 
\[
\begin{aligned}
\liminf_{k\to+\infty} \int_0^T \dish\eps\nu(\upwc z{\tau_k}(t), \pwl p{\tau_k}'(t)) \dd t   & \geq
\liminf_{k\to+\infty} \int_0^T \calH(\upwc z{\tau_k}(t), \pwl p{\tau_k}'(t)) \dd t + \liminf_{k\to+\infty} \frac{\eps\nu}{2} \int_0^T \| \pwl p{\tau_k}'(t)\|_{L^2(\Omega;\MD)}^2 \dd t \\ &  \geq \int_0^T \calH(z(t), p'(t)) \dd t + \frac{\eps\nu}{2} \int_0^T \| p'(t)\|_{L^2(\Omega;\MD)}^2 \dd t  = \int_0^T \dish\eps\nu(z(t), p'(t)) \dd t\,.
\end{aligned}
\]
By \eqref{converg-sigma}, 
$\mathrm{Div}(\bbC(\pwc z{\tau_k})\pwc e{\tau_k})  =\mathrm{Div}(\pwc \sigma{\tau_k}) \weaksto \mathrm{Div} (\sigma) = \mathrm{Div} (\bbC(z) e) $ in $L^\infty (0,T; H^1(\Omega;\R^n)^*)$. Therefore, also in view of \eqref{converg-Fw} and by the convexity of $\disv\eps\nu^*$, we find that
\[
\liminf_{k\to+\infty}
\int_0^T  \disv\eps\nu^*(\mathrm{Div}(\bbC(\pwc z{\tau_k}(r))\pwc e{\tau_k}(r)) {+}\pwc F{\tau_k}(r)) \dd r  \geq \int_0^T \disv\eps\nu^* ( \mathrm{Div} (\bbC(z(r)) e(r)) {+} F(r)) \dd r\,.
\]
 We now observe that 
\[
\begin{aligned}
&
\liminf_{k\to+\infty}
\int_0^T 
\calR_\eps^*({-}\As(\pwc z{\tau_k}(r)) {-}W'(\pwc z{\tau_k}(r)){-}\tfrac12 \bbC'(\pwc z{\tau_k}(r)) \pwc e{\tau_k}(r) : \pwc e{\tau_k} (r){-}\tau_k \pwc \lambda{\tau_k}(r)) \dd r 
\\
& \stackrel{(1)}{\geq}
\int_0^T 
 \liminf_{k\to+\infty}
\calR_\eps^*({-}\As(\pwc z{\tau_k}(r)) {-}W'(\pwc z{\tau_k}(r)){-}\tfrac12 \bbC'(\pwc z{\tau_k}(r)) \pwc e{\tau_k}(r) : \pwc e{\tau_k} (r){-}\tau_k \pwc \lambda{\tau_k}(r)) \dd r 
\\
& \stackrel{(2)}{\geq}
\int_0^T 
\calR_\eps^*({-}\As (z(r)) {-}W'(z(r)){-}\tfrac12 \bbC'(z(r)) e(r): e(r) ) \dd r, 
\end{aligned}
\]
where {\footnotesize (1)} follows from the Fatou Lemma, and {\footnotesize (2)} 
from the fact that 
 \[
 \begin{aligned}
 &
  \liminf_{k\to+\infty} \calR_\eps^*({-}\As(\pwc z{\tau_k}(r)) {-}W'(\pwc z{\tau_k}(r)){-}\tfrac12 \bbC'(\pwc z{\tau_k}(r)) \pwc e{\tau_k}(r) : \pwc e{\tau_k} (r){-}\tau_k \pwc \lambda{\tau_k}(r)) \\ & \geq  \sup_{\zeta \in \Hsm}    \liminf_{k\to+\infty} 
 \left(\pairing{}{\Hs(\Omega)}{\As(\pwc z{\tau_k}(r)) {+}W'(\pwc z{\tau_k}(r)){+}\tfrac12 \bbC'(\pwc z{\tau_k}(r)) \pwc e{\tau_k}(r) : \pwc e{\tau_k} (r){+}\tau_k \pwc \lambda{\tau_k}(r)}{{-}\zeta} -\calR_\eps(\zeta) \right)
 \\
 & \geq \sup_{\zeta \in \Hsm}  \left(\pairing{}{\Hs(\Omega)}{\As (z(r)) {+}W'(z(r)){+}\tfrac12 \bbC'(z(r)) e(r): e(r) }{{-}\zeta} -\calR_\eps(\zeta) \right)  \qquad \text{for all } r \in [0,T]
 \end{aligned}
 \]
 by   \eqref{weak-converg-uzpe},   \eqref{converg-e}, and \eqref{ci-vuole}. 
In the end, we have that 
\[
\liminf_{k\to+\infty}
\int_0^T 
  \dish\eps\nu^*(\upwc z{\tau_k}(r), -\newmu \pwc p{\tau_k}(r) + (\pwc\sigma{\tau_k}(r))_\dev)
    \dd r \geq \int_0^T \dish\eps\nu^* (z(r), -\newmu p(r) + (\sigma(r))_\dev) \dd r  
   \]
   by
   convergences \eqref{converg-uzpe} and \eqref{converg-sigma},  and 
     a version of the Ioffe Theorem, cf.\ e.g.\ \cite[Thm.\ 21]{Valadier90}.  The latter result applies since
   \begin{compactenum}
   \item
    the mapping
   $ \Hs(\Omega)\ni z \mapsto \calH^*(z, \omega) $ is lower semicontinuous for all $\omega \in L^2(\Omega;\MD)$ (as 
   $\calH^*(z, \omega) = \sup_{\pi \in L^2(\Omega;\MD) } (\int_\Omega \omega \pi \dd x - \calH(z,\pi))$ and the maps $z\mapsto  - \calH(z,\pi)$ are continuous by \eqref{props-calH-3+1/2}), and thus $z \in  \Hs(\Omega) \mapsto \dish\eps\nu^*(z, \omega) $ is also lower semicontinuous;
   \item the mapping 
   $ L^2(\Omega;\MD) \ni \GGG \pi \mapsto   \dish\eps\nu^*(z, \pi)$ is convex. 
   \end{compactenum}
   \par
      As for the right-hand side of \eqref{discr-UEDE},
      clearly we have 
      $\enen{\mu}(0, \pwc u{\tau_k}(0), \pwc z{\tau_k}(0), \pwc p{\tau_k}(0)) = \enen{\mu} (0,u_0,z_0,p_0)$ for all $k\in \N$. The power terms
   converge, too, as we have 
 \[
 \begin{aligned}
 &
\int_0^T  \int_\Omega \bbC(\pwl z{\tau_k} (r)) (\sig{\pwl u{\tau_k}(r){+}w(r)}{-} \pwl p{\tau_k} (r)): \sig{w'(r)} \dd x  \dd r \stackrel{(1)}\to 
\int_0^T \int_\Omega \bbC(z(r)) e(r) : \GGG \sig{w'(r)}  \dd x \dd r \qquad \text{and}
\\
&
 -
  \int_{0}^{T}\pairing{}{H^1(\Omega;\R^n)}{F'(r)}{\pwl u{\tau_k}(r){+}w(r)}  \dd r  \stackrel{(2)}\to   -
  \int_{0}^{T}\pairing{}{H^1(\Omega;\R^n)}{F'(r)}{u(r){+}w(r)}  \dd r \,,
  \end{aligned}
\]
with {\footnotesize (1)}    due to convergences 
\eqref{converg-uzpe} and 
to the fact that, by the Lipschitz continuity of $\bbC$, $\bbC(\pwl z{\tau_k}) \to \bbC(z)$ in $L^\infty(0,T;L^\infty(\Omega))$,  and {\footnotesize (2)} again due to \eqref{converg-uzpe}. 
 In a completely analogous way the last-but-one term on the right-hand side of  \eqref{discr-UEDE} passes to the limit. 
Finally, we estimate the remainder term $\mathfrak{R}_\tau(0,T)$ from \eqref{remainder} via 
\begin{equation}
\label{remainder-estimates}
\begin{aligned}
\mathfrak{R}_\tau(0,T)  & \leq   C_3     \left( \|\pwc u{\tau_k}{-} \pwl u\tau \|_{L^\infty(0,T;H^1(\Omega))} {+} \|  \pwc z\tau{-}  \pwl z\tau\|_{L^\infty(0,T;\Hs(\Omega))} {+}  \| \pwc p\tau{-}  \pwl p\tau\|_{L^\infty(0,T;L^2(\Omega))}    {+} \| \pwc w\tau{-} w \|_{L^\infty(0,T;H^1(\Omega))} \right) 
\\
& \qquad \qquad \qquad  \times 
\int_0^T \left(  \| u_{\tau_k}' \|_{H^1(\Omega)}{+}  \| z_{\tau_k}' \|_{\Hs(\Omega)} {+}\| p_{\tau_k}' \|_{L^2(\Omega)}  \right) \dd r
\\
&
 \stackrel{(3)}\leq C \tau_k \to 0 \qquad \text{as } k\to+\infty,
\end{aligned}
\end{equation}
with {\footnotesize (3)} due to \eqref{stability-properties}  and the fact that $w\in W^{1,\infty} (0,T;H^1(\Omega;\R^n))$, and estimates \eqref{enhanced-discr-est}.
This concludes the proof of the Energy-Dissipation upper estimate \eqref{enineq-plast}.
\par\noindent
\textbf{Step $3$: proof of  \eqref{uniform-est-cont}.} Estimates \eqref{uniform-est-cont} follow from the analogous bounds \eqref{basic-en-est-SL} and \eqref{2009170054} for the discrete solutions via lower semicontinuity arguments, based on convergences 
 \eqref{converg-uzpe} and \eqref{converg-e}. 
 \GGG We conclude observing that \eqref{bound-diss+conjs} follows from
  the reformulation of the Energy-Dissipation balance as 
  \eqref{0307191911},  and the fact that $\int_0^t\partial_t \calE_\mu$ is 
   uniformly
  bounded w.r.t.\ $\eps,\, \nu, \mu>0$,  in view of the assumptions on $F$ and $w$ and of the previously proven  \eqref{uniform-est-cont}.     
\end{proof} 
\section{The vanishing-viscosity limit with fixed hardening parameter}
 \label{s:van-visc}
 This section focuses on the limit passage in the viscous system \eqref{RD-intro} as $\eps \down 0$ 
  and, possibly, $\nu\down 0$, while  
 the hardening parameter \emph{$\newmu>0$ is kept fixed}. 
In fact we shall distinguish the two cases:
\begin{enumerate}
\item  $\eps \down 0$ and $\nu>0$ is kept fixed, addressed in Section~\ref{s:6-nu-fixed}, in which the vanishing-viscosity analysis will lead to the existence of 
 Balanced Viscosity solutions to the \emph{rate-independent} system 
 for damage and plasticity \emph{with hardening}, cf.\  Theorem \ref{teo:rep-sol-exist}  ahead; 
\item $\eps,\, \nu\down 0$, addressed in Section~\ref{s:6-nu-vanishes}, in which we shall obtain  Balanced Viscosity solutions to the \emph{multi-rate} system 
 for damage and plasticity \emph{with hardening},  cf.\ Theorem \ref{teo:rep-sol-exist-2}  later on. 
\end{enumerate}
Hereafter, we shall work under the standing assumptions from Section~\ref{s:2} and omit to explicitly invoke them in the
various results, with the exception of Theorems\  \ref{teo:rep-sol-exist} and  \ref{teo:rep-sol-exist-2}.  Notice that in this section,  in which  the hardening parameter $\mu$ is fixed, we could use just the assumption \eqref{1808191033} on the external forces, which is however not enough to obtain estimates independent of $\mu$ for the solutions,   cf.\ Remark \ref{sta-in-rmk}.  
 %
 \begin{notation}
 \label{not-6.1}
 \upshape 
 We shall denote by  $(q_{\eps,\nu})_{\eps,\nu} = (u_{\eps,\nu},z_{\eps,\nu},p_{\eps,\nu})_{\eps,\nu}$ a family of solutions to Problem \ref{prob:visc} for $\mu>0$ fixed,
with initial and external data independent of $\eps $ and $\nu$ and satisfying the conditions listed in \GGG Section~\ref{s:2}.
 
\end{notation}
 Prior to distinguishing the case in which $\nu>0$  is fixed from that in which $\nu\down 0$, let us establish the common ground for the vanishing-viscosity analysis.  
Following the well-established reparameterization technique pioneered in \cite{EfeMie06RILS}, 
we  shall  suitably reparameterize the viscous solutions $(q_{\eps,\nu})_{\eps,\nu} $,  observe that 
the rescaled functions $(\sfq_{\eps,\nu})_{\eps,\nu} $ comply with a reparameterized version of  the Energy-Dissipation balance corresponding to \eqref{enineq-plast}, and pass to the limit in it  as $\eps\down 0$ and $\nu>0$ is fixed (see Section~\ref{s:6-nu-fixed}), and as $\eps,\, \nu \down 0$ (see Section~\ref{s:6-nu-vanishes}).  
\paragraph{\bf Rescaling.}
 We introduce the arclength function  $s_{\eps,\nu}: [0,T]\to [0,S_{\eps,\nu}]$ (with $S_\eps := s_{\eps,\nu}(T)$) defined by
 \begin{equation}\label{def:arclength}
 s_{\eps,\nu}(t): = \int_0^t \left(1+\|q_{\eps,\nu}'(\tau)\|_{\bfQ} \right) \dd \tau
  = \int_0^t \left(1 {+}  \|u_{\eps,\nu}'(\tau)\|_{H^1(\Omega;\R^n)} {+} \|z_{\eps,\nu}'(\tau)\|_{\Hs(\Omega)} {+} \|p_{\eps,\nu}'(\tau)\|_{L^2(\Omega;\MD)}  \right) \dd \tau \,.
 \end{equation}
 It follows from estimate \eqref{uniform-est-cont} that $\sup_{\eps,\nu} S_{\eps,\nu} <+\infty$. We now define
 \begin{subequations}
  \label{rescaled-solutions}
 \begin{equation}
 \begin{aligned}
 &
 \sft_{\eps,\nu}: [0,S_{\eps,\nu}] \to [0,T],  &&  \sft_{\eps,\nu}: = s_{\eps,\nu}^{-1}, & \qquad 
 &
 \sfu_{\eps,\nu}: [0,S_{\eps,\nu}] \to H^1(\Omega;\R^n),  &&  \sfu_{\eps,\nu}: = u_{\eps,\nu} \circ \sft_{\eps,\nu}
 \\  
  &
 \sfz_{\eps,\nu}: [0,S_{\eps,\nu}] \to \Hs(\Omega), && \sfz_{\eps,\nu}: = z_{\eps,\nu} \circ \sft_{\eps,\nu}, & \qquad 
  &
 \sfp_{\eps,\nu}: [0,S_{\eps,\nu}] \to L^2(\Omega;\MD),  && \sfp_{\eps,\nu}: = p_{\eps,\nu} \circ \sft_{\eps,\nu}
 \end{aligned}
 \end{equation}
 and set
 $
 \sfq_{\eps,\nu}: = ( \sfu_{\eps,\nu},  \sfz_{\eps,\nu},  \sfp_{\eps,\nu}). 
 $
 In what follows,  with slight abuse of notation  
  we shall often 
write $\calE(t,q)$ in place of $\calE(t,u,z,p)$. 
  We also introduce 
  \begin{equation}
 \begin{aligned}
& \sfe_{\eps,\nu}: [0,S_{\eps,\nu}] \to L^2(\Omega;\Mnn), &&  \sfe_{\eps,\nu} : = e_{\eps,\nu} \circ \sft_{\eps,\nu} = \sig{\sfu_{\eps,\nu} {+} (w{\circ}\sft_{\eps,\nu})} - \sfp_{\eps,\nu},   \\
& \serifsigma_{\eps,\nu}: [0,S_{\eps,\nu}] \to L^2(\Omega;\Mnn), && \serifsigma_{\eps,\nu}: = \sigma_{\eps,\nu} \circ \sft_{\eps,\nu}
= \mathbb{C}(\sfz_{\eps,\nu})\sfe_{\eps,\nu} 
 \end{aligned}
  \end{equation}
 \end{subequations}
 \paragraph{\bf The parameterized Energy-Dissipation balance.}
 A straightforward calculation on the Energy-Dissipation balance corresponding to
 \eqref{enineq-plast}
 yields that the reparameterized viscous solutions $(\sfu_{\eps,\nu}, \sfz_{\eps,\nu}, \sfp_{\eps,\nu})$,  along with the rescaling functions $\sft_{\eps,\nu}$, fulfill
 \begin{equation}
 \label{immediate-reparam}
 \begin{aligned}
     & \calE (\sft_{\eps,\nu}(s_2),\sfq_{\eps,\nu}(s_2)) +
      \int_{s_1}^{s_2}
      \sft_{\eps,\nu}'\Big[ \calV_{\eps,\nu}\left(\frac{\sfu_{\eps,\nu}'}{\sft_{\eps,\nu}'}\right) {+} \calR_{\eps,\nu}\left(\frac{\sfz_{\eps,\nu}'}{\sft_{\eps,\nu}'}\right) {+} \calH_{\eps,\nu}\left(\sfz_{\eps,\nu}, \frac{\sfp_{\eps,\nu}'}{\sft_{\eps,\nu}'}\right)\Big] \dd \tau
     \\ & \quad   +
       \int_{s_1}^{s_2}   \sft_{\eps,\nu}' \Big[  \calV_{\eps,\nu}^*\Big(\mathrm{Div}\big(\serifsigma_{\eps,\nu}{+} \rho(\sft_{\eps,\nu})\big)\Big) {+}   \calR_{\eps}^*\Big({-}\As (\sfz_{\eps,\nu}) {-}W'(\sfz_{\eps,\nu}){-}\tfrac12 \bbC'(\sfz_{\eps,\nu}) \sfe_{\eps,\nu} : \sfe_{\eps,\nu}\Big) 
       \\
       & \qquad \qquad  \qquad \qquad 
  {+}  \calH_{\eps,\nu}^*\Big(\sfz_{\eps,\nu}, -\nu \sfp_{\eps,\nu} + (\serifsigma_{\eps,\nu})_\dev\Big)\Big]
  \dd \tau
   \\&
    =     
     \calE (\sft_{\eps,\nu}(s_1),\sfq_{\eps,\nu}(s_1)
    +\int_{s_1}^{s_2} \partial_t     \calE 
    (\sft_{\eps,\nu},\sfq_{\eps,\nu})\,  \sft_{\eps,\nu}' \dd \tau
   \end{aligned}
 \end{equation}
for all $0\leq s_1\leq s_2 \leq S_{\eps,\nu}$, where we have used that $F \circ \sft_{\eps,\nu} = -   \mathrm{Div} (\rho {\circ} \sft_{\eps,\nu})$ by 
 condition \eqref{2809192200}. 
\par
Let us now introduce a functional 
$\Mename = \Me tq{t'}{q'}$ 
subsuming the terms featuring in the integrals on the left-hand side of \eqref{immediate-reparam}. In order to motivate our definition of $\Mename$, cf.\ \eqref{Me-def} below,
we recall the definitions \eqref{dissipazioni-riscalate}
 of the functionals $\calV_{\eps,\nu}$, $\calR_\eps$, and $\calH_{\eps,\nu}$,
 so that 
  \[
  \begin{aligned}
  &
\calV_{\eps,\nu} \left( \frac{\sfu_{\eps,\nu}'}{\sft_{\eps,\nu}'}  \right) 
= \frac{\eps\nu}{2(\sft_{\eps,\nu}')^2} \| \sfu_{\eps,\nu}' \|^2_{ H^1, \bbD }\,, 
\\
&
\calR_\eps \left( \frac{\sfz_{\eps,\nu}'}{\sft_{\eps,\nu}'}  \right)
 =
 \frac1{\sft_{\eps,\nu}'}\calR({\sfz_{\eps,\nu}'})  +
 \frac{\eps}{2(\sft_{\eps,\nu}')^2} \| \sfz_{\eps,\nu}' \|_{L^2}^2 \,, 
 \\
 &
 \calH_{\eps,\nu} \left( \sfz_{\eps,\nu}, \frac{\sfp_{\eps,\nu}'}{\sft_{\eps,\nu}'}  \right) 
 =
  \frac1{\sft_{\eps,\nu}'} \calH(\sfz_{\eps,\nu}, \sfp_{\eps,\nu}')  +
 \frac{\eps\nu}{2(\sft_{\eps,\nu}')^2} \| \sfp_{\eps,\nu}' \|_{L^2}^2\,.
 \end{aligned}
 \]
 Moreover, we take into account 
 the expressions \eqref{Psi-eps-star} of the conjugates,
 and the fact that the arguments of $\calV_{\eps,\nu}^*$, $\calR_\eps^*$ and $\calH_{\eps,\nu}^*$
 in \eqref{immediate-reparam}
 involve the derivatives $-\mathrm{D}_x \calE (\sft_{\eps,\nu},\sfu_{\eps,\nu},\sfz_{\eps,\nu},\sfp_{\eps,\nu})$ for $x =u$, $x=z$, and $x=p$, respectively.
 In particular, in view of \eqref{norma-duale} we have 
 \begin{equation} \label{v-star}
\begin{split}
 \calV_{\eps,\nu}^*(\mathrm{Div}(\serifsigma_{\eps,\nu}{+} \rho(\sft_{\eps,\nu}))) 
 = \frac1{2\eps\nu} \| -\mathrm{D}_u \calE  (\sft_{\eps,\nu},\sfu_{\eps,\nu},\sfz_{\eps,\nu},\sfp_{\eps,\nu})\|^2_{( H^1, \bbD )^*}\,.
\end{split}
 \end{equation}
All in all, 
 the functional
 $\Mename: [0,+\infty) \times \bfQ \times (0,+\infty) \times \bfQ \to [0,+\infty]$
  encompassing the integrands on the left-hand side of \eqref{immediate-reparam} reads 
\begin{equation}
\label{Me-def}
\begin{aligned}
&
\Me tq{t'}{q'} 
: = \calR(z') + \calH(z,p') + \Mredname (t,q,t',q')
\qquad
\text{ with } \Mredname \text{ defined by }
\\
&
\begin{aligned}
\Mredname (t,q,t',q')
:=\ &  \frac{\eps}{2t'} \DVito_\nu^2(q') + \frac{t'}{2\eps} (\DVitos(t,q))^2\,,
\end{aligned}
\end{aligned}
\end{equation}
 with the functionals 
$\DVito_\nu$ and $\DVitos$ from \eqref{0307191952}, namely
\begin{equation}
\label{new-form-D-Dvito}
\begin{aligned}
  & 
   \DVito_\nu({q}'): = 
    \sqrt{ \nu \| {u}'(t)\|^2_{ H^1, \bbD }   {+} 
\|{z}'(t)\|_{L^2}^2
{+} \nu\|{p}'(t)\|_{L^2}^2} \,..
\\
&
 \DVitos({t},{q}): =   \sqrt{
\frac1{\nu}\, \|{-}\mathrm{D}_u \calE_\mu ({t},{q} )\|^2_{( H^1, \bbD )^*} 
+ \tilded_{L^2} ({-}\mathrm{D}_z \calE_\mu  ({t},{q}) ,\partial\calR(0))^2
+ \frac1\nu \, \dLtwo ({-}\mathrm{D}_p \calE_\mu  ({t},{q}) ,\partial_\pi \calH( {z} ,0))^2
}\,.
\end{aligned}
\end{equation}

Therefore,
the rescaled solutions $(\sft_{\eps,\nu},\sfq_{\eps,\nu})_{\eps,\nu}$ satisfy
 \begin{itemize}
 \item[-] the \emph{parameterized} Energy-Dissipation balance for every $s_1,s_2\in[0,S_{\eps,\nu}]$, 
\begin{equation}
\label{rescaled-enid-eps-OLD}
 \begin{aligned}
  & \calE (\sft_{\eps,\nu}(s_2),\sfq_{\eps,\nu}(s_2)) + \int_{s_1}^{s_2} 
   \Me{\sft_{\eps,\nu}(\tau)}{\sfq_{\eps,\nu}(\tau)}{\sft_{\eps,\nu}'(\tau)}{\sfq_{\eps,\nu}'(\tau)} \dd \tau  
   \\
    & =  
      \calE (\sft_{\eps,\nu}(s_1),\sfq_{\eps,\nu}(s_1)) 
    +\int_{s_1}^{s_2} \partial_t \calE (\sft_{\eps,\nu}(\tau), \sfq_{\eps,\nu}(\tau)) \,\sft_{\eps,\nu}'(\tau) \dd \tau
   \end{aligned}
   \end{equation}
   (which rephrases  \eqref{immediate-reparam});
   \item[-] the \emph{normalization} condition 
\begin{equation}
   \label{normalization}
\sft_{\eps,\nu}'(s) + \|\sfq_{\eps,\nu}'(\tau)\|_{\bfQ} =
  \sft_{\eps,\nu}'(s) + \|\sfu_{\eps,\nu}'(s)\|_{H^1} +  \|\sfz_{\eps,\nu}'(s)\|_{\Hs} +  \|\sfp_{\eps,\nu}'(s)\|_{L^2} \equiv 1 \qquad \foraa\, s \in (0,S_{\eps,\nu})\,.
   \end{equation}
 \end{itemize}
\par
Finally, it follows from 
 \eqref{0307191911} 
 that the  reparameterized viscous solutions $(\sfu_{\eps,\nu}, \sfz_{\eps,\nu}, \sfp_{\eps,\nu})$
 fulfill for all $0\leq s_1\leq s_2 \leq S_{\eps,\nu}$
  \begin{equation}
\label{rescaled-enid-eps}
 \begin{aligned}
      \calE_{\GGG \mu } (\sft_{\eps,\nu}(s_2),\sfq_{\eps,\nu}(s_2))
       \GGG +  \int_{s_1}^{s_2} \MVito(\sft_{\varepsilon,\nu}(\tau), \sfq_{\varepsilon,\nu}(\tau), \sfq'_{\varepsilon,\nu}(\tau))  \dd \tau 
    =     
     \calE_{\GGG \mu } (\sft_{\eps,\nu}(s_1),\sfq_{\eps,\nu}(s_1)
    +\int_{s_1}^{s_2} \partial_t     \calE 
    (\sft_{\eps,\nu},\sfq_{\eps,\nu})\,  \sft_{\eps,\nu}' \dd \tau\,.
   \end{aligned}
 \end{equation}
 Indeed, it will be in  \eqref{rescaled-enid-eps} that we  shall  perform the vanishing-viscosity limit passages.
With the very same arguments as in the proof of 
 Theorem \ref{th:1} (cf.\ \eqref{bound-diss+conjs}), 
 it is immediate to deduce from \eqref{rescaled-enid-eps} that 
 \begin{equation}
 \label{est-Menu}
 \exists\, C>0 \ \forall\, \eps,\nu>0 \, : \quad  \GGG \int_{0}^{S} \MVito(\sft_{\varepsilon,\nu}(\tau), \sfq_{\varepsilon,\nu}(\tau), \sfq'_{\varepsilon,\nu}(\tau))  \dd \tau  \leq C.
 \end{equation}
\subsection{The vanishing-viscosity analysis as $\eps\down 0$ and $\nu>0$ is fixed}
\label{s:6-nu-fixed}
Throughout this section we  shall  keep the rate parameter $\nu>0$  \emph{fixed}. In order to signify this and simplify notation,
 we  shall  drop the dependence on $\nu$ of the viscous solutions and simply write 
 \[
 (\sft_\eps,\sfu_\eps,\sfz_\eps,\sfp_\eps) \quad \text{in place of} \quad  (\sft_{\eps,\nu},\sfu_{\eps,\nu},\sfz_{\eps,\nu},\sfp_{\eps,\nu}). 
 \]
Since the variables $u$ and $p$ relax to equilibrium and rate-independent evolution with the same 
rate with which $z$ relaxes to rate-independent behavior, 
a  $\Gamma$-convergence argument and the comparison with the 
general results from 
\cite{MRS14,MieRosBVMR} 
lead us to
expect that any 
parameterized curve $(\sft,\sfq) = (\sft,\sfu,\sfz,\sfp)$ arising  as a  limit point of the family 
$(\sft_\eps,\sfq_\eps)_\eps$ as 
$\eps \down 0$ shall satisfy the Energy-Dissipation (upper) estimate
    \begin{equation}
\label{EDue-lim}
   \calE_{\GGG \mu } (\sft(S),\sfq(S)) + \int_0^S
   \Mli{\sft(\tau)}{\sfq(\tau)}{\sft'(\tau)}{\sfq'(\tau)} \dd \tau   \le   \calE_{\GGG \mu } (\sft(0),\sfq(0)) +\int_{0}^{S} \partial_t \calE_\nu (\sft(\tau), \sfq(\tau)) \, \sft'(\tau) \dd \tau 
\end{equation}
   with $S = \lim_{\eps \down 0} S_\eps$
   and the functional 
   $\Mliname:  [0,T]  \times \bfQ \times [0,+\infty) \times \bfQ \to [0,+\infty] $ defined by (as usual, 
   here
   $q=(u,z,p)$)
   \begin{subequations}
\label{Mli-def}
\begin{align}
&
\Mli tq{t'}{q'}
: = \calR(z') + \calH(z,p') + \GGG \Mliredname(t,q,t',q') \,, \quad \text{where} 
\nonumber
\\
&
\label{Mli>0}
\begin{aligned}
& 
\text{if } t'>0\,,  \quad 
\GGG \Mliredname(t,q,t',q') := 
\begin{cases}
 0 &\text{if } 
\begin{cases}
-\mathrm{D}_u \calE_\mu (t,q)=0 \,,  \\
-\mathrm{D}_z  \calE_\mu  (t,q) \in \partial\calR(0) \,, \ \text{and} \\
-\mathrm{D}_p  \calE_\mu (t,q) \in \partial_\pi \calH(z,0)\,,
\end{cases}
 \\
 +\infty & \text{otherwise,}
\end{cases}
\end{aligned}
\\
& 
\label{Mli=0}
\begin{aligned}
&
\text{if } t'=0\,,  \quad 
\GGG \Mliredname(t,q,0,q') : = \DVito_\nu(q') \, \DVitos(t,q)\,,
\end{aligned}
\end{align}
\end{subequations}
(recall 
 \eqref{new-form-D-Dvito}  for the definition of the functionals $\DVito_\nu$ and $\DVitos$).
Observe that the product $\DVito_\nu(q') \, \DVitos(t,q)$
contains the term 
$ \tilded_{L^2(\Omega)} ({-}\mathrm{D}_z \calE (t,q),\partial\calR(0))$ which, in principle, need not to be finite at all
$(t,q) \in [0,T]\times \bfQ$ since, in general, we only have $\mathrm{D}_z \calE (t,q),\,  \partial\calR(0)  \subset \Hs(\Omega)^*$.
 Let us then clarify that
\begin{equation}
\label{fisima}
\text{if $\DVitos(t,q)=+\infty$ and $ \DVito_\nu(q')=0$,   in \eqref{Mli=0} we mean } \DVito_\nu(q') \, \DVitos(t,q): = +\infty.
 \end{equation} 
\par
 Following \cite{MRS14,MieRosBVMR}, we shall refer to 
  the functional
 $\Mliname$  from \eqref{Mli-def}  as \emph{vanishing-viscosity contact potential}. Observe that 
we keep on highlighting the dependence of  $\Mliname$  on the (fixed) parameters $\nu$ \GGG and $\mu$  for later use in Section~\ref{s:van-hard}.  
   \par
   Our definition of \emph{Balanced Viscosity} solution to the rate-independent  system with hardening 
    (\ref{mom-balance-intro}, \ref{stress-intro}, \ref{flow-rule-dam-intro}, \ref{viscous-bound-cond}, \ref{pl-hard-intro}) 
   features  \eqref{EDue-lim}  as a \emph{balance}, 
   satisfied along \emph{any} sub-interval of a given interval $[0,S]$.  Along the lines of  \cite{MRS14} we give the following 
   \begin{definition}
   \label{def:BV-solution-hardening}
   We say that a parameterized curve $(\sft,\sfq) = (\sft,\sfu,\sfz,\sfp) \in \AC ([0,S]; [0,T]{\times} \bfQ)$
   is a  \emph{(parameterized) Balanced Viscosity} ($\BV$, for short) solution to the rate-independent  system with hardening 
    (\ref{mom-balance-intro}, \ref{stress-intro}, \ref{flow-rule-dam-intro}, \ref{viscous-bound-cond}, \ref{pl-hard-intro}) 
    if \GGG $\sft \colon [0,S] \to [0,T]$ is nondecreasing and $(\sft, \sfq)$  fulfills the Energy-Dissipation balance
   \begin{equation}
\label{ED-balance}
 \begin{aligned}
 &
   \calE_{\GGG \mu }(\sft(s),\sfq(s)) + \int_{0}^{s}
   \Mli{\sft(\tau)}{\sfq(\tau)}{\sft'(\tau)}{\sfq'(\tau)} \dd \tau 
   =  \calE_{\GGG \mu } (\sft(0),\sfq(0)) +\int_{0}^{s} \partial_t \calE_{\GGG \mu } (\sft(\tau), \sfq(\tau)) \, \sft'(\tau) \dd \tau 
   \end{aligned}
\end{equation}
 for all  $ 0\leq s \leq S $. 
   We call a $\BV$ solution 
   $(\sft,\sfq)$
   \emph{non-degenerate} if, in addition, there holds 
   \begin{equation}
   \label{non-degenerate}
 \sft'(s)+\| \sfq'(s)\|_{\bfQ} =   \sft'(s)+ \| \sfu'(s)\|_{H^1(\Omega;\R^n)} +  \| \sfz'(s)\|_{\Hs(\Omega)} + \| \sfp'(s)\|_{L^2(\Omega;\MD)} >0 \quad \foraa\,   s\in (0,S)\,. 
   \end{equation}
   \end{definition}
%
\begin{remark}
\upshape
 We have defined $\BV$ solutions following the general setting where they are only required to be absolutely continuous in the reparametrized variable $s$.
  Nonetheless, up to a further reparametrization by arc-length, one obtains curves that are Lipschitz in $s$.
In fact, notice that in Theorem \ref{teo:rep-sol-exist} below we obtain the existence of a $\BV$ solution
$(\sft,\sfq) = (\sft,\sfu,\sfz,\sfp) \in W^{1,\infty} ([0,S]; [0,T]{\times} \bfQ)$.
The Lipschitz regularity   here  is a consequence of the normalization condition \eqref{normalization} used in the approximation.
\par We postpone a discussion of the non-degeneracy condition \eqref{non-degenerate} to
the upcoming \GGG Remark~\ref{rmk:nondeg}.
 
\end{remark}
As we  shall  see, the Energy-Dissipation balance \eqref{ED-balance} encodes all the information on the evolution of the
parameterized curve $(\sft,\sfq)$ and, in particular, on the onset of rate-dependent behavior
in the jump regime (i.e., when $\sft'=0$). While postponing further comments to
Remark \ref{rmk:mech-beh}, 
let us only record here the fact that the expression of \GGG $\Mliredname(t,q,0,q')$ 
shows that, at a jump, viscous behavior for the variables $u$, $z$, and $p$
emerges  ``in the same way'',  since the viscous terms related to each variable equally contribute to   \GGG $\Mliredname(t,q,0,q')$.  
 This aspect will be further explored in \GGG Remark~\ref{rmk:mech-beh}. 
\par
   The \textbf{main result of this \GGG sub section},
   Theorem\ \ref{teo:rep-sol-exist} ahead, shows that the parameterized  solutions $(\sft_\eps,\sfq_\eps)_\eps$  of the viscous system \eqref{RD-intro} converge to a $\BV$ solution to  the 
rate-independent  system with hardening 
    (\ref{mom-balance-intro}, \ref{stress-intro}, \ref{flow-rule-dam-intro}, \ref{pl-hard-intro}). Our proof will crucially rely \GGG on  the following characterization of $\BV$ solutions  that is in the same spirit as 
    \cite[Prop.\ 5.3]{MRS12}, \cite[Corollary 4.5]{MRS13}. 
%
\par
    \begin{proposition}
    \label{prop:charact-BV}
     For a parameterized curve $(\sft,\sfq) = (\sft,\sfu,\sfz,\sfp) \in \AC ([0,S]; [0,T]{\times} \bfQ)$ \GGG with $\sft \colon [0,S] \to [0,T]$  nondecreasing 
    the following properties are equivalent:
    \begin{enumerate}
    \item $(\sft,\sfq)$  is a  $\BV$ solution to  the  rate-independent  system with hardening;
    \item $(\sft,\sfq)$ fulfills   the  Energy-Dissipation upper estimate \eqref{EDue-lim};
\item  $(\sft,\sfq)$ fulfills   the \emph{contact condition}
\begin{equation}
\label{contact}
  \Mli{\sft(s)}{\sfq(s)}{\sft'(s)}{\sfq'(s)} = \pairing{}{\bfQ}{{-}\rmD_q \calE_{\GGG \mu }(\sft(s),\sfq(s))}{\sfq'(s)} \qquad \foraa\, s \in (0,S). 
\end{equation}
    \end{enumerate} 
    \end{proposition} 
  \noindent   The proof is based on the   following key chain rule-estimate
     \begin{lemma}
     \label{l:chain-rule}
     Along any parameterized curve
    \begin{equation}
    \label{desired-chain-rule}
    \begin{aligned}
    &
    (\sft,\sfq) \in \AC ([0,S]; [0,T]{\times} \bfQ)
\text{ s.t. }    \Mli{\sft}{\sfq}{\sft'}{\sfq'}<+\infty \ \aein\, (0,S), \text{ there holds} 
\\    
&
\begin{aligned}
-\frac{\dd}{\dd s} \calE_{\mu}(\sft(s),\sfq(s)) +\partial_t  \calE_{\mu}(\sft(s),\sfq(s)) &  =  \pairing{}{\bfQ}{{-}\rmD_q \calE_{\GGG \mu }(\sft(s),\sfq(s))}{\sfq'(s)}  \\ & \leq  \Mli{\sft(s)}{\sfq(s)}{\sft'(s)}{\sfq'(s)} \quad \foraa\, s \in (0,S)  \,.
\end{aligned}
\end{aligned}
    \end{equation}
     \end{lemma}
    \begin{proof}
The first equality in \eqref{desired-chain-rule}
 directly follows from the chain rule \eqref{chain-rule}. To deduce the second estimate, we start by observing that,
from $\Mli{\sft(s)}{\sfq(s)}{\sft'(s)}{\sfq'(s)}<+\infty $, we gather that
\[
\text{ if $\sft'(s)=0$ then} \quad \begin{cases}
-\mathrm{D}_u \calE_{\GGG \mu } (\sft(s),\sfq(s))=0 \,,  \\
-\mathrm{D}_z  \calE_{\GGG \mu }  (\sft(s),\sfq(s))\in \partial\calR(0) \,, \\
-\mathrm{D}_p  \calE_{\GGG \mu } (\sft(s),\sfq(s)) \in \partial_\pi \calH(z,0)\,,
\end{cases}
\]
while if $\sft'(s)>0$ then
\[
\|\sfz'(s)\|_{L^2} \tilded_{L^2(\Omega)} ({-}\mathrm{D}_z \calE_{\GGG \mu } (\sft(s),\sfq(s)),\partial\calR(0))
 \leq
  \DVito_\nu(\sfq'(s)) \, \DVitos(\sft(s), \sfq(s))\,. 
\]
 In each case, we have 
 \[
 \|\sfz'(s)\|_{L^2} \tilded_{L^2(\Omega)} ({-}\mathrm{D}_z \calE_{\GGG \mu } (\sft(s),\sfq(s)),\partial\calR(0)) <+\infty
 \]
whence $\tilded_{L^2(\Omega)} ({-}\mathrm{D}_z \calE_{\GGG \mu } (\sft(s),\sfq(s)),\partial\calR(0))<+\infty$ if $\sfz'(s) \neq 0$.  If   $\sfz'(s)=0$, taking into account  our convention
 \eqref{fisima} 
and 
 that    $\Mli{\sft(s)}{\sfq(s)}{\sft'(s)}{\sfq'(s)}<+\infty $   we again  get 
$\tilded_{L^2(\Omega)} ({-}\mathrm{D}_z \calE_{\GGG \mu } (\sft(s),\sfq(s)),\partial\calR(0))<+\infty$.
\par
 After this preliminary discussion, it is sufficient to observe that 
\[
\begin{aligned}
  & \pairing{}{\bfQ}{{-}\rmD_q \calE_{\mu}(\sft(s),\sfq(s))}{\sfq'(s)}  \leq   \| \sfu'(s)\|_{ H^1, \bbD }  \|{-}\mathrm{D}_u  \calE_{\GGG \mu }  (\sft(s),\sfq(s))\|_{( H^1, \bbD )^*} {+}  \|\sfz'(s)\|_{L^2} \tilded_{L^2} ({-}\mathrm{D}_z \calE_{\GGG \mu } (\sft(s),\sfq(s)),\partial\calR(0)) 
  \\
  & + \calR(\sfz'(s))
   + \|\sfp'(s)\|_{L^2}  d_{L^2}({-}\rmD_p \calE_{\GGG \mu }(\sft(s),\sfq(s)),\partial_\pi \calH(\sfz(s),0))+\calH(\sfz(s), \sfp'(s))
  \end{aligned}
 \]
(cf.\  \eqref{DVito-added}), in order to conclude \eqref{desired-chain-rule}. 
    \end{proof}
    We are now in a position to carry out the \textbf{proof of Proposition \ref{prop:charact-BV}}:
        Let us  suppose that $(\sft,\sfq) $ complies with  \eqref{EDue-lim}. Integrating \eqref{desired-chain-rule} in time gives the converse inequality and thus the desired  balance \eqref{ED-balance}.
    \par
    Clearly, combining the  \emph{contact condition} \eqref{contact}
 with the chain rule  \eqref{chain-rule} leads to  \eqref{ED-balance}. The converse implication is also true  thanks to inequality \eqref{desired-chain-rule}. This concludes the proof. \QED

 Adapting the arguments for \cite[Thm.\ 5.3]{MRS14} to the present context, we may now obtain
a characterization of $\BV$ solutions in terms of a system of subdifferential inclusions
that has the same structure as the  viscous system \eqref{RD-intro}, but where the viscous terms in the single equations
can be  ``activated''   only at jumps. 
    \begin{proposition}
    \label{prop:diff-charact-BV}
\begin{enumerate}
\item If $(\sft,\sfq) = (\sft,\sfu,\sfz,\sfp) \in \AC ([0,S]; [0,T]{\times} \bfQ)$ is a  $\BV$ solution to the rate-independent  system with hardening (\ref{mom-balance-intro}, \ref{stress-intro}, \ref{flow-rule-dam-intro}, \ref{viscous-bound-cond}, \ref{pl-hard-intro}),
then there exists a measurable function $\lambda : \GGG [0, S]  \to \GGG [0, 1] $ such that 
\begin{equation}\label{eq:lambda}
\sft'(s) \lambda(s) =0 \quad \foraa\ s\in(0,S)
\end{equation}
and $(\sft,\sfq)$ satisfies \foraa\ $s \in (0,S)$
the system of subdifferential inclusions 
\begin{subequations}\label{eq:diff-char}
\begin{align}
&\lambda(s)   \mathrm{D} \disv 2\nu(\sfu'(s)) + (1-\lambda(s)) \, \mathrm{D}_u \calE_{\GGG \mu } (\sft(s),\GGG \sfq(s)) = 0 && \text{in}\ H_\Dir^1(\Omega;\R^n)^* \,,
\\
&
(1{-}\lambda(s))\, \partial \mathcal{R}(\sfz'(s))+ \lambda(s)\, \mathrm{D}\mathcal{R}_2(\sfz'(s)) + (1{-}\lambda(s)) \, \mathrm{D}_z \calE_{\GGG \mu } (\sft(s),\GGG \sfq(s))  \ni 0 
 && \text{in } 
\Hs(\Omega)^* \,,
\\
&
(1{-}\lambda(s)) \, \partial_{\pi} \mathcal{H}(\sfz(s), \sfp'(s)) + \lambda(s)\,   \mathrm{D} \dish 2\nu(\sfp'(s))
+(1{-}\lambda(s)) \, \mathrm{D}_p \calE_{\GGG \mu } (\sft(s),\GGG \sfq(s))  \ni 0  && \text{in  } L^2(\Omega;\MD) \,,
\end{align}
\end{subequations}
 which is equivalent to 
\begin{subequations}\label{eq:diff-char-concrete}
  \begin{align*}
&
-\lambda(s)\, \mathrm{Div} \big( \nu  \mathbb{D}\sig{\sfu'(s)} \big) - (1{-}\lambda(s))\, \big( \mathrm{Div}( \serifsigma(s)) +  F(\sft(s)) \big) =0
&& \text{in } H_\Dir^1(\Omega;\R^n)^* ,
\\
&
(1{-}\lambda(s))\, \partial \mathcal{R}(\sfz'(s))+ \lambda(s)\, \sfz'(s) + (1{-}\lambda(s)) \left(\As(\sfz(s)) + W'(\sfz(s)) + \frac12 \mathbb{C}'(\sfz(s)) \sfe(s) : \sfe(s)\right) \ni 0 
 && \text{in } 
\Hs(\Omega)^* \,,
\\
&
(1{-}\lambda(s)) \, \partial_{\pi} H(\sfz(s), \sfp'(s)) + \lambda(s)\, \nu \sfp'(s)
+(1{-}\lambda(s)) \left( \nu \sfp(s) -  \big(\serifsigma(s)\big)_\dev  \right) \ni 0  && \text{a.e.\ in  } \Omega \,.
\end{align*}
\end{subequations}
\item Conversely, if $(\sft,\sfq)$ satisfies \eqref{eq:diff-char}, with $\lambda$ as in \eqref{eq:lambda}, and the map $s \mapsto \calE_{\GGG \mu } (\sft(s),\sfq(s))$ is absolutely continuous on $[0,S]$, then $(\sft,\sfq)$ is a $\BV$ solution.
\end{enumerate}
    \end{proposition}
    \begin{remark}
    \label{rmk:mech-beh}
    \upshape
    A few comments on the mechanical interpretation of system \eqref{eq:diff-char} are in order. Due to the  switching condition \eqref{eq:lambda}, the coefficient $\lambda$ can be   non-null  only  if $\sft'(s)=0$, namely  the system is jumping in the (slow) external time scale. When the system does not jump, the evolution of the variables $\sfz$ and $\sfp$ is rate-independent, and $\sfu$ ``follows them'' staying at elastic equilibrium. At a jump, the system may switch to a viscous regime where viscous dissipation intervenes in the evolution of the three variables $u,\,z,\, p$  modulated by the \emph{same} coefficient $\lambda$. This reflects the fact that, in the vanishing-viscosity approximation  $u,\,z,\, p$ relax to their limiting evolution  with the same rate.
    \end{remark}
    \begin{proof}
 The proof is a straightforward adaptation of  the argument for \cite[Thm.\ 5.1]{MRS14}. Thus, we shall only recapitulate it here, referring to \cite{MRS14} for all details.
The key point is to use that, by Prop.\ \ref{prop:charact-BV},  a parameterized curve  $(\sft,\sfq)$ is a $\BV$ solution if and only if it fulfills \eqref{contact}, namely
for almost all $s\in (0,S)$
\begin{equation}
\label{contact-set}
(\sft(s),\sfq(s),\sft'(s),\sfq'(s)) \in \Sigma  : = \{ (t,q,t',q') \in [0,T]\times \bfQ\times [0,+\infty)\times \bfQ\, : \  \Mli{t}{q}{t'}{q'} = \pairing{}{\bfQ}{{-}\rmD_q\calE_{\GGG \mu }(t,q)}{q'}   \}.
\end{equation}
Then,  \cite[Prop.\ 5.1]{MRS14} provides a characterization of the contact set $\Sigma$. 
 Such a characterization 
holds in our infinite-dimensional context as well,  and it allows us to describe $\Sigma$ 
as the union of two sets, that encompass elastic equilibrium for $u$ and rate-independent evolution for $(z,p)$ on the one hand, and viscous evolution for  all of the three variables on the other hand. Namely,
\begin{subequations}
\begin{equation}
\label{charact-contact-set}
\Sigma = \mathsf{E}_{\sfu} \mathsf{R}_{\sfz,\sfp} \cup \mathsf{V}_{\sfu,\sfz,\sfp}
\end{equation}
with  the sets 
\[
\begin{aligned}
&
\mathsf{E}_{\sfu} \mathsf{R}_{\sfz,\sfp} : = \{(t,q,t',q')\, : \ t'>0, \, \rmD_u \calE_{\GGG \mu }(t,q) =0,\, \partial \calR(z') +\rmD_z \calE_{\GGG \mu }(t,q) \ni 0, \,  \partial \calH(z,p') +\rmD_p \calE_{\GGG \mu }(t,q) \ni 0 \},
\\
&
 \mathsf{V}_{\sfu,\sfz,\sfp} : =\left \{(t,q,t',q')\, : \ t'=0 \text{ and } \exists\, \lambda \in [0,1]\, : \ \begin{cases}
 \lambda \rmD \disv2\nu (u') +(1{-}\lambda)\rmD_u \calE_{\GGG \mu }(t,q) =0,
 \\
 (1{-}\lambda) \partial\calR(z') +\lambda \rmD\calR_2(z') +(1{-}\lambda)\rmD_z \calE_{\GGG \mu }(t,q) \ni 0,
 \\
 (1{-}\lambda) \partial\calH(z,p') +\lambda \rmD\dish2\nu(p') +(1{-}\lambda)\rmD_p \calE_{\GGG \mu }(t,q) \ni 0
 \end{cases}
\right \}.
\end{aligned}
\]
\end{subequations}
Combining \eqref{contact-set} and \eqref{charact-contact-set}  leads to \eqref{eq:diff-char}. 
    \end{proof}
\par We conclude this section with 
our existence result for $\BV$ solutions, obtained as limits of a family 
$(\sft_\eps, \sfq_\eps)_\eps = (\sft_\eps,\sfu_\eps,\sfz_\eps,\sfp_\eps)_\eps$ of (reparameterized) viscous solutions to 
Problem \ref{prob:visc}.

 In order to properly state our convergence result,  
we recall that,  for 
$s_\eps: [0,T]\to [0,S_\eps]$ as in \eqref{def:arclength}, 
the sequence $(S_\eps)_\eps$  is bounded thanks to  \eqref{2009170054}. Moreover,   $S_\eps\ge T$ for every $\eps>0$. 
Hence, 
\begin{equation}
\label{subsequence}
\text{
there is  a sequence $\eps_k\down 0$ and $S>0$ such that $S_{\eps_k}\to S$.}
\end{equation}
If $S_\eps<S$, we extend $(\sft_\eps, \sfq_\eps)$  to $[0,S]$ by setting 
 $(\sft_{\eps}(s),\sfq_{\eps}(s))=(\sft_{\eps}(S_\eps){+}s{-}S_\eps,\sfq_{\eps}(S_\eps))$ for $s\in(S_\eps,S]$.
\par
\GGG  We are now in a position to show the existence of $\BV$ solution to  the 
rate-independent  system with hardening 
    (\ref{mom-balance-intro}, \ref{stress-intro}, \ref{flow-rule-dam-intro}, \ref{viscous-bound-cond}, \ref{pl-hard-intro}).
The proof is based on approximation by means of solutions to Problem \ref{prob:visc}.
The general scheme follows the steps of  \cite{MRS14,MieRosBVMR}. 
Some technical points, arising when dealing with the coupled plastic-damage system, 
are treated as in \cite[Theorem 5.4]{Crismale-Lazzaroni}, which  features the viscosity \emph{only} in the damage variable
and not in the plastic variable. 
 
\begin{theorem} 
\label{teo:rep-sol-exist}
Under the assumptions of \GGG Section~\ref{s:2} and
\eqref{EL-initial}, 
let $(\eps_k)_k$ be as in \eqref{subsequence}.
Then, there exist a (not relabeled) subsequence  $(\sft_{\eps_k},\sfq_{\eps_k})_k = (\sft_{\eps_k}, \sfu_{\eps_k}, \sfz_{\eps_k})$ and a Lipschitz curve
$(\sft,\sfq) = (\sft,\sfu,\sfz,\sfp) \in W^{1,\infty} ([0,S]; [0,T]{\times} \bfQ)$ such that
\begin{enumerate}
\item  for all $s\in[0,S]$  the following convergences hold as $k\to+\infty$  
\begin{equation}
\label{weak-converg-uzpe-rep}
\sft_{\eps_k}(s)\to\sft(s),  \  \,   \sfu_{\eps_k}(s) \weakto \sfu(s) \text{ in } H^1(\Omega;\R^n),
 \ \,   \sfz_{\eps_k}(s) \weakto \sfz(s) \text{ in } \Hs(\Omega),  \ \,  \sfp_{\eps_k}(s) \weakto \sfp(s) \text{ in } L^2(\Omega;\MD); 
\end{equation}
\item $(\sft,\sfq) $ is a  $\BV$ solution to the rate-independent  system with hardening according to Definition \ref{def:BV-solution-hardening}.
\end{enumerate}
\end{theorem}

\noindent 
 Arguments that are, by now, standard  (detailed in e.g.\ \cite{MRS12, MRS13, MRS14}) 
 would also allow us to prove,  \emph{a posteriori},  the convergence of the 
energy terms and of the Energy-Dissipation integrals in \eqref{rescaled-enid-eps-OLD} to their analogues in 
\eqref{ED-balance}; the same is true  for the other forthcoming convergence results, i.e.\ Theorems
\ref{teo:rep-sol-exist-2} and \ref{teo:exparBVsol}. However, to avoid 
overburdening the exposition we have preferred to overlook this point.  
\begin{proof} \GGG The proof is divided into 3 Steps. First, we find a limiting  parameterized curve    by compactness arguments,  then we deduce the finiteness of the vanishing-viscosity contact potential when $\sft'>0$, namely when there are no jumps in the fast time-scale, and eventually we prove 
the Energy-Dissipation upper estimate  \eqref{EDue-lim}. 
\par 
\noindent
\textbf{Step $1$: Compactness.} 
Let $(q_\eps)_\eps = (u_\eps,z_\eps,p_\eps)_\eps$ be a family of solutions to Problem \ref{prob:visc}. Let $s_\eps: [0,T]\to [0,S_\eps]$ be as in \eqref{def:arclength} and $(\sft_\eps,\sfq_\eps)=(\sft_\eps,\sfu_\eps, \sfz_\eps, \sfp_\eps)$ be as in \eqref{rescaled-solutions}. 
By \eqref{2009170054}, $S_\eps$ is uniformly bounded in $\eps$; moreover, $S_\eps\ge T$ for every $\eps$. Therefore, there is a sequence $\eps_k\to0^+$ and $S>0$ such that $S_{\eps_k}\to S$. Henceforth, 
 we shall write $(\sft_{k},\sfq_{k})$ in place of $(\sft_{\eps_k},\sfq_{\eps_k})$, and we  shall  not relabel subsequences.
%
\par
%
By \eqref{normalization}, the sequence $(\sft_{k},\sfq_{k})_k$ is equibounded in $W^{1,\infty}(0,S;[0,T]\times\bfQ)$. Therefore,  arguing as in Step 1 of the proof of Theorem \ref{th:1} above  (and in particular resorting to the compactness results from \cite{Simon87}),
we obtain a limit  curve  $(\sft,\sfq)$ such that (up to a subsequence, not relabeled) the following convergences hold  as  $k\to+\infty$
\begin{subequations}\label{limit-conv}
\begin{align}
 \sft_{k}\weaksto\sft &\quad \text{in } W^{1,\infty}(0,S;[0,T]) \,,\quad &&
 \sfu_{k}\weaksto\sfu  \quad \text{in } W^{1,\infty}(0,S;H^1(\Omega;\R^n)) \,,
 \\
 \sfz_{k}\weaksto\sfz &\quad \text{in } W^{1,\infty}(0,S;\Hs(\Omega)) \,, 
\quad &&\sfz_{k} \to \sfz  \quad \text{in }   \mathrm{C}^0([0,S]; \mathrm{C}^0(\overline\Omega)) \,, \label{190506}
\\
\sfp_{k}\weaksto\sfp & \quad\text{in } W^{1,\infty}(0,S;L^2(\Omega;\MD)) \,,
\end{align}
\end{subequations}
 as well as the pointwise convergences in \eqref{weak-converg-uzpe-rep}. 
\par
We now define
\begin{subequations}
\label{defso}
\begin{alignat}{3}
\label{def-so-1}
s_-(t)&:=\sup\{s\in [0,S]\,\colon\, \sft(s)<t\}\quad&&\text{for $t\in (0,T]$}\,,\\
\label{def-so-2}
s_+(t)&:=\inf\{s\in [0,S]\,\colon\, \sft(s)>t\}\quad&&\text{for $t\in [0,T)$}\,,
\end{alignat}
and $s_-(0):=0$, $s_+(T):=S$.
Then we have
\begin{equation*} 
s_-(t)\leq  \liminf_{k\to+\infty}  s_k(t)\leq   \limsup_{k\to+\infty}   s_k(t)\leq s_+(t)\quad \text{and}\quad \sft(s_-(t))=t=\sft(s_+(t)) \quad \text{for every } t\in[0,T] \,,
\end{equation*}
\begin{equation*}
s_-(\sft(s))\leq s\leq s_+(\sft(s)) \quad \text{for every } s\in [0,S] \,. 
\end{equation*}
Moreover the set
\begin{equation}\label{defso-3}
\calS:=\{t\in[0,T]\,\colon\,s_-(t)<s_+(t)\}
\end{equation} 
is at most countable.
Set
\begin{equation}
\label{def-so-4}
\calU:=\{s\in[0,S]\,\colon\,\sft \text{ is constant in a \GGG neighborhood  of } s \}\,,
\end{equation}
\end{subequations}
then 
\begin{equation*}
\calU=\bigcup_{t\in \calS}(s_-(t), s_+(t))\,.
\end{equation*}
 For \GGG future  convenience (see Step 3 below) we remark that the original functions $(u_k,z_k,p_k)$ satisfy, 
for every $t\in[0,T] \setminus \calS$,
%
\begin{subequations}\label{420}
\begin{align}
u_k(t)&\weakto\sfu(s_-(t))=\sfu(s_+(t)) \text{ in } H^1(\Omega;\R^n) \,,\label{555sol}\\
 z_k(t)  &\weakto\sfz(s_-(t))=\GGG \sfz(s_+(t))  \text{ in }\Hs(\Omega)\,,\label{558sol}\\
p_k(t)&\weakto\sfp(s_-(t))=\sfp(s_+(t)) \text{ in } L^2(\Omega;\MD) \,.\label{557sol}
\end{align}
\end{subequations}
\par
\noindent
\GGG \textbf{Step $2$:  Finiteness of  $\Mliredname(\sft(\tau),\sfq(\tau),\sft'(\tau),\sfq'(\tau))$ when $\sft'(\tau)>0$}. 
 We prove that 
\begin{equation} \label{finite}
\begin{aligned}
-\mathrm{D}_u \calE_{\GGG \mu } (\sft(\tau),\sfq(\tau))=0 \,,  \ 
-\mathrm{D}_z \calE_{\GGG \mu } (\sft(\tau),\sfq(\tau)) \in \partial\calR(0) \,, \ 
 & -\mathrm{D}_p \calE_{\GGG \mu } (\sft(\tau),\sfq(\tau)) \in \partial_\pi \calH(z,0)
 \\
 &
\quad
 \text{for a.a.}\ \tau\in A: =  \{\tau \in (0,S)\, : \  \sft'(\tau)>0 \} 
\,,
\end{aligned}
\end{equation}
i.e., the configuration is stable where $\sft$ grows.  This is equivalent to showing that $\Mli{\sft(\tau)}{\sfq(\tau)}{\sft'(\tau)}{\sfq'(\tau)}$ is finite for a.a.\ $\tau \in A$. 

 Preliminarily, we observe that 
\begin{equation}
\label{limsup}
  \limsup_{k\to+\infty}  \sft_k'(\tau)>0 \qquad \foraa\, \tau \in A. 
\end{equation}
This can be shown with the very same arguments as for the proof of \cite[(5.18)]{Crismale-Lazzaroni}. 

\GGG By \eqref{eqs:1106191955}  and convergences \eqref{weak-converg-uzpe-rep}
%
%
we have that
\begin{equation}
\label{lsc-Phi^nu}
0 \le \DVitos(\sft(\tau), \sfq(\tau)) \le   \liminf_{k\to+\infty}   \DVitos(\sft_k(\tau), \sfq_k(\tau)) \qquad \text{for all } \tau \in [0,S].
\end{equation}
\GGG Moreover, by
  \eqref{DVITOequalsSTAR}, written
 for $t= \sft_k(\tau)$ we obtain
 \begin{equation}
 \label{further-ingredient}
 \begin{aligned}
  \DVitos(\sft_k(\tau), \sfq_k(\tau))  =  \DVitos(\sft_k(\tau), q_k(\sft_k(\tau)))  
  &  =\eps_k \DVito_\nu (q_k'(\sft_k(\tau))) 
  \\
  &  =\eps_k  \sqrt{ \nu \| u_k'(\sft_k(\tau))\|^2_{ H^1, \bbD }   {+} 
\|z_k'(\sft_k(\tau))\|_{L^2}^2
{+} \nu\|p_k'(\sft_k(\tau))\|_{L^2}^2} 
\\
&= \frac{\eps_k}{\sft_k'(\tau)} \sqrt{ \nu \| \sfu_k'(\tau)\|^2_{ H^1, \bbD }   {+} 
\|\sfz_k'(\tau)\|_{L^2}^2
{+} \nu\|\sfp_k'(\tau)\|_{L^2}^2}
\\
& 
\leq  \frac{\eps_k}{\sft_k'(\tau)} \qquad \foraa\, \tau \in (0,S),
\end{aligned}
\end{equation}
where the last estimate follows from   the normalization condition  \eqref{normalization}  and since $\nu \leq 1$. 
Combining \eqref{limsup}, \eqref{lsc-Phi^nu}, and \eqref{further-ingredient},  we ultimately find
\[
0 \le \DVitos(\sft(\tau), \sfq(\tau)) \le   \liminf_{k\to+\infty}   \DVitos(\sft_k(\tau), \sfq_k(\tau))
=0 \qquad \foraa\, \tau \in A,
\]
which implies \eqref{finite}.

\GGG In particular, we obtain that 
$\Mli{\sft(\tau)}{\sfq(\tau)}{\sft'(\tau)}{\sfq'(\tau)}$ is finite, and equals $\calR(\sfz'(\tau)) + \calH(\sfz(\tau), \sfp'(\tau))$, for \GGG a.a.\  $\tau \in A$. 
\GGG Let us remark also that
\begin{equation}\label{1206191003}
[0,S]\ni \tau \mapsto \DVitos(\sft(\tau), \sfq(\tau)) \text{ is lower semicontinuous, }
\end{equation}
by \eqref{eqs:1106191955} and the fact that $(\sft, \sfq) \in W^{1,\infty}(0,S; [0,T] \times \bfQ)$.
In particular, the set
\begin{equation}\label{1206191009}
A^\circ := \{ \tau \in [0,S] \colon  \DVitos(\sft(\tau), \sfq(\tau)) >0\} \text{ is open and included in } [0,S] \setminus A\,.
\end{equation}

\par
\noindent
\textbf{Step $3$: \GGG The  Energy-Dissipation upper estimate.} 
By \eqref{rescaled-enid-eps} we have
\begin{equation}
\label{where-limit}
   \calE_{\GGG \mu } (\sft_k(S),\sfq_k(S)) + \int_{0}^{S} 
    \MVito(\sft_k(\tau),\sfq_k(\tau), \sfq_k'(\tau))  \dd \tau =   \calE_\mu(\sft_k(0),\sfq_k(0))  +\int_{0}^{S} \partial_t \calE_{\GGG \mu } (\sft_k(\tau), \sfq_k(\tau)) \, \sft_k'(\tau) \dd \tau \,.
\end{equation}
In order to obtain \eqref{EDue-lim}, we  shall  pass to the $\liminf$   in \eqref{where-limit},  using the lower semicontinuity of $\calE_{\GGG \mu }$
 and the previously proved \eqref{lsc-Phi^nu}. 
\GGG

We first prove the lower semicontinuity estimate
\begin{equation}\label{1206191015}
\begin{aligned}
 \int_{A^\circ}  \Mliredname(\sft(\tau), \sfq(\tau), \sft'(\tau), \sfq'(\tau))   \dd \tau &   
\leq \  \liminf_{k\to+\infty}   \int_{A^\circ}  \MredVitok(\sft_k(\tau), \sfq_k(\tau), \sfq_k'(\tau))  \dd \tau\,,
\end{aligned}
\end{equation}
where the set $A^\circ$ has been introduced in \eqref{1206191009}. By \eqref{0307191952}
\begin{equation}\label{1306191946}
  \MredVitok (\sft_k(\tau), \sfq_k(\tau), \sfq_k'(\tau))= \DVito_\nu(\sfq_{k}'(s)) \, \DVitos(\sft_k(\tau), \sfq_{k}(\tau))\,.
 \end{equation}
 Then, estimate \eqref{1206191015} follows from 
 Lemma~\ref{le:0107192150} in Appendix~B. 
  Indeed, we apply it  combining \eqref{lsc-Phi^nu} with  convergences \eqref{limit-conv},   which imply that 
 \[
 (\sqrt{\nu} \sfu_\varepsilon, \sfz_\varepsilon, \sqrt{\nu} \sfp_\varepsilon) \weakto (\sqrt{\nu} \sfu, \sfz, \sqrt{\nu} \sfp)\quad\text{in }W^{1,\infty}(0,S; \mathrm{Q})
\] 
(recall that $\nu>0$ is fixed).
 Hence, on the one hand we have that 
\[
 m_k(\tau):=  \DVito_\nu(\sfq_{k}'(\tau)) =  \sqrt{ \nu \| \sfu_{k}'(\tau)\|^2_{ H^1, \bbD }   {+} 
\|\sfz_{k}'(\tau)\|_{L^2}^2
{+} \nu\|\sfp_{k}'(\tau)\|_{L^2}^2}   \weaksto  m(\tau) \text{   in } L^\infty(0,S),
\]
and 
\[
 m(\tau) \geq \sqrt{\nu \|\sfu'(\tau)\|^2_{H^1, \mathrm{D}} + \| \sfz'(\tau)\|^2_{L^2} + \nu \|\sfp'(\tau)\|^2_{L^2}  } = \DVito_\nu(\sfq'(\tau)) 
 \qquad \foraa\, \tau \in (0,S). 
\]
 On the other hand, the sequence $h_k(\tau): = \DVitos(\sft_k(\tau), \sfq_{k}(\tau)) $ satisfies the first condition
in \eqref{hypLB1}.  
 Therefore, by Lemma \ref{le:0107192150} we have the desired estimate \eqref{1206191015}. 

Moreover, by \eqref{limit-conv} and Ioffe Theorem (cf.\ \cite[Thm.\ 21]{Valadier90}) it is not difficult to see that
\begin{equation}\label{1206191137}
\int_0^S \calR(\sfz'(\tau)) + \calH(\sfz(\tau), \sfp'(\tau)) \dd \tau \leq   \liminf_{k\to+\infty}    \int_0^S \calR(\sfz_\varepsilon'(\tau)) + \calH(\sfz_\varepsilon(\tau), \sfp_k'(\tau)) \dd \tau\,.
\end{equation}

As for the right-hand side, notice that
\begin{equation}\label{1206191144}
\int_{0}^{S} \partial_t \calE_{\GGG \mu } (\sft_k(\tau), \sfq_k(\tau)) \, \sft_k'(\tau) \dd \tau
= \int_{0}^{T} \partial_t \calE_{\GGG \mu } (\tau, q_k(\tau)) \dd \tau \,.
\end{equation}
By \eqref{420},
\[
\begin{aligned}
\partial_t \calE_{\GGG \mu }(\tau,q_k(\tau)) = & \int_\Omega \bbC(z) \, e_k(\tau) : \sig{w'(\tau)} \dd x  - \langle F'(\tau), u_k(\tau)+w(\tau) \rangle_{H^1(\Omega;\R^n)} - \langle F(\tau), w'(\tau) \rangle_{H^1(\Omega;\R^n)}  \\[.5em]
\to & \  
\partial_t \calE_{\GGG \mu }(\tau,\sfq(s_-(\tau))) \quad \text{for every } \tau \in[0,T] \setminus \calS \,.
\end{aligned}
\]
Hence by Dominated Convergence
\begin{equation}\label{1206191145}
\int_{0}^{T} \partial_t \calE_{\GGG \mu } (\tau, q_k(\tau)) \dd \tau \to
\int_{0}^{T} \partial_t \calE_{\GGG \mu }(\tau,\sfq(s_-(\tau))) \dd \tau 
=
\int_{0}^{S} \partial_t \calE_{\GGG \mu } (\sft(\tau), \sfq(\tau)) \, \sft'(\tau) \dd \tau \,,
\end{equation}
where we have used the fact that
\GGG $\sft'(s)=0$  for a.e.\ $s\in \calU$ and $s_-(\sft(s))=s$ for a.e.\ $s\in[0,S]\setminus \calU$  (see Step 1 above). 
\par
\GGG We now collect \eqref{finite} and  \eqref{1206191009}--\eqref{1206191145},
 to conclude the Energy-Dissipation upper estimate \eqref{EDue-lim}. By  the characterization 
provided by Proposition~\ref{prop:charact-BV}, we deduce that the curve $(\sft,\sfq)$ is a  $\BV$ solution. This concludes the proof. 
%
%
%
%
%
\end{proof}
\begin{remark}
\upshape
\label{rmk:nondeg}
It is an open problem to prove that the reparameterized viscous solutions converge to a \emph{non-degenerate}
(in the sense of \eqref{non-degenerate})
 $\BV$ solution. Nonetheless, 
 following \cite[Rmk.\ 2]{MRS09}
 any degenerate $\BV$ solution $(\sft,\sfq)$ can be reparameterized to a non-degenerate one $(\tilde\sft,\tilde\sfq)
 = (\tilde\sft,\tilde\sfu,\tilde\sfz,\tilde\sfp)
 $
 by setting
 \begin{equation}
 \label{deg2nondeg}
 \begin{aligned}
 &
 m: [0,S]\to [0,+\infty) && 
 m(s): = \int_0^s (\sft'(\tau){+} \|\sfq'(\tau)\|_{\bfQ})  \dd \tau,  &&  \tilde S: = m(S),  &&
 \\
 & r :[0,\tilde S] \to [0,S]  && r(\mu) : = \inf\{s\geq0\, : \ m(s) =\mu\}, && && 
 \\
 & \tilde\sft: [0,\tilde S]\to [0,T]  && \tilde\sft(\mu): = \sft(r(\mu)), &&  \tilde\sfq: [0,\tilde S]\to \bfQ  && \tilde\sfq(\mu): = \sfq(r(\mu)) \,.
 \end{aligned}
 \end{equation}
 With the very same calculations as in \cite[Rmk.\ 2]{MRS09}
 (cf.\ also \cite[Rmk.\ 5.2]{KRZ13}), one sees that 
 \[
 (\tilde\sft,\tilde\sfq) \in W^{1,\infty} (0,\tilde S; [0,T]{\times}\bfQ) \quad \text{with } \ \tilde\sft'(\mu)+\| \tilde\sfq'(\mu)\|_{\bfQ} \equiv 1 \text{ a.e. in } (0,\tilde S). 
 \]
 and that $(\tilde\sft,\tilde\sfq)$ is still a $\BV$ solution in the sense of Definition~\ref{def:BV-solution-hardening}. 
\end{remark}

\subsection{The vanishing-viscosity analysis as $\eps,\, \nu\down0$}
\label{s:6-nu-vanishes}
We now address the asymptotic analysis of \GGG Problem~\ref{prob:visc}  as \emph{both} the viscosity parameter $\eps$ 
\emph{and} the rate parameter $\nu$ tend to zero. Accordingly, throughout
this section we  shall  revert to the notation
$(\sft_{\eps,\nu}, \sfq_{\eps,\nu})_{\eps,\nu}$ for a family of reparameterized viscous solutions.
\par
Again, 
it is to be expected that any limit curve $(\sft,\sfq)$ of the family $(\sft_{\eps,\nu}, \sfq_{\eps,\nu})_{\eps,\nu}$ as $\eps,\, \nu \down 0$
shall satisfy the analogue of the Energy-Dissipation inequality \eqref{EDue-lim}, however featuring, in the present context, a \emph{different} \emph{vanishing-viscosity contact potential}
 that reflects the multi-rate character of the problem, and in particular the fact that $u$ and $p$ relax to equilibrium and rate-independent evolution, respectively, at a faster rate than $z$ relaxing to rate-independent evolution. 
 For consistency of notation, we shall 
denote this new contact potential $\Mlinamezero$. It will turn out  (in analogy with  the results from 
\cite{MRS14,MieRosBVMR}),   that $\Mlinamezero: [0,T]  \times \bfQ \times [0,+\infty) \times \bfQ \to [0,+\infty] $ is given \GGG by   
\[
\Mlizero tq{t'}{q'}=  \Mlieczero t u z p {t'}{u'}{z'}{p'} : = \calR(z') + \calH(z,p') + \Mliredzero t u z p {t'}{u'}{z'}{p'}\,
\]
with $\Mlirednamezero$ given by \eqref{Mli>0} if $t'>0$. Instead, in place of \eqref{Mli=0}
 we have 
\begin{subequations}
\label{Mil-sec6}
\begin{equation}
\label{Mlizero=0}
\begin{aligned}
\text{if } t'=0\,,  \quad 
 \Mliredname(t,q,t',q') : =
\begin{cases}
    \DVito(u', p') \, \DVitosred(t,q)   & \text{if } z'=0
\\
 \|z'\|_{L^2}\, \tilded_{L^2(\Omega)} ({-}\mathrm{D}_z \calE_{\GGG \mu } (t,q),\partial\calR(0)) & \text{if } \DVitosred(t,q)=0,
\\
+ \infty & \text{if } \|z'\|_{L^2} \DVitosred(t,q)>0
\end{cases}
\end{aligned}
\end{equation}
where we have used the \GGG notation
\begin{equation}\label{3006191243}
\begin{aligned}
&
  \DVito(u', p'): =   \sqrt{ \|u'\|_{H^1(\Omega;\R^n)}^2 {+} \|p'\|_{L^2(\Omega;\MD)}^2}, 
\\
&   \DVitosred(t,q)   :=  \sqrt{\|{-}\mathrm{D}_u  \calE_\mu  (t,q)\|^2_{( H^1, \bbD )^*}{+}  d_{L^2}({-}\mathrm{D}_p \calE_\mu (t,q),\partial_\pi \calH(z,0))^2}\,. 
\end{aligned}
\end{equation} 
 Again, in the case in which $z'=0$ and  $\tilded_{L^2(\Omega)} ({-}\mathrm{D}_z \calE_{\GGG \mu } (t,q),\partial\calR(0))= +\infty $, 
in \eqref{Mlizero=0}
we set  
\[
 \|z'\|_{L^2}\, \tilded_{L^2(\Omega)} ({-}\mathrm{D}_z \calE_{\GGG \mu } (t,q),\partial\calR(0)) : = +\infty\,.
 \]
%
\end{subequations}
\par
The multi-rate character of the vanishing-viscosity approximation  addressed in this case  is already apparent 
in the expression for  $ \Mlirednamezero(t,q,t',q') $ at $t'=0$. Indeed,
$ \Mlirednamezero(t,q,0,q') $ 
 is finite only either if $z'=0$ (i.e.\ $z$ is frozen), or
  if $\DVitosred(t,q) =0$, which entails that 
  $u$ is at equilibrium and $p$ fulfills the local stability condition $-\mathrm{D}_p \calE_{\GGG \mu } (t,q) \in \partial_{\pi} \calH(z,0)$, cf.\ \GGG Remark~\ref{rmk:mech-interp} later on for further comments.
\par
Accordingly, we give the following
  \begin{definition}
   \label{def:multi-rate-BV-solution-hardening}
   We call a parameterized curve $(\sft,\sfq) = (\sft,\sfu,\sfz,\sfp) \in \AC ([0,S]; [0,T]{\times} \bfQ)$
   a  \emph{(parameterized) Balanced Viscosity} solution to the \emph{multi-rate}   system with hardening 
    (\ref{mom-balance-intro}, \ref{stress-intro}, \ref{flow-rule-dam-intro}, \ref{viscous-bound-cond}, \ref{pl-hard-intro}) 
    if \GGG $\sft \colon [0,S] \to [0,T]$ is nondecreasing and $(\sft, \sfq)$  fulfills for all $0\leq s \leq S$  the Energy-Dissipation balance
   \begin{equation}
\label{ED-balance-zero}
   \calE_{\GGG \mu }(\sft(s),\sfq(s)) + \int_{0}^{s}
   \Mlizero{\sft(\tau)}{\sfq(\tau)}{\sft'(\tau)}{\sfq'(\tau)} \dd \tau = \calE_{\GGG \mu } (\sft(0),\sfq(0)) +\int_{0}^{s} \partial_t \calE_{\GGG \mu } (\sft(\tau), \sfq(\tau)) \, \sft'(\tau) \dd \tau \,.
\end{equation}
We say that $(\sft,\sfq)$ is \emph{non-degenerate} if it fulfills \eqref{non-degenerate}. 
   \end{definition}
The very analogue of Proposition~\ref{prop:charact-BV} holds for $\BV$ solutions to the multi-rate system as well, 
based on the chain-rule estimate 
\[
-\frac{\dd}{\dd s} \calE_{\GGG \mu }(\sft(s),\sfq(s)) +\partial_t  \calE_{\GGG \mu }(\sft(s),\sfq(s)) =  -\pairing{}{\bfQ}{{-}\rmD_q \calE_{\GGG \mu }(\sft(s),\sfq(s))}{\sfq'(s)}  \leq  \Mlizero{\sft(s)}{\sfq(s)}{\sft'(s)}{\sfq'(s)}
\]
for almost all $s\in (0,S)$, 
which can be shown along any  parameterized curve $(\sft,\sfq)  \in \AC ([0,S]; [0,T]{\times} \bfQ)$ 
such that  $\Mlizero{\sft(s}{\sfq(s)}{\sft'(s)}{\sfq'(s)}<\infty$ for almost all $s\in (0,S)$  
 by adapting the arguments for Lemma \ref{l:chain-rule}, \GGG see also Proposition~\ref{prop:charact}.  
\par
Likewise, we have a differential characterization for $\BV$ solutions in the sense of Definition~\ref{def:multi-rate-BV-solution-hardening} that has a structure analogous to  the characterization  from Proposition~\ref{prop:diff-charact-BV}.
    \begin{proposition}
    \label{prop:diff-charact-BV-zero}
    A parameterized curve 
$(\sft,\sfq) = (\sft,\sfu,\sfz,\sfp) \in \AC ([0,S]; [0,T]{\times} \bfQ)$ is a  $\BV$ solution to the \emph{multi-rate}  system with hardening
if and only if there exist two
 function $\lambda_{\sfu,\sfp}\,  \lambda_{\sfz}: [0; S] \to [0; 1]$ such that 
\begin{subequations}\label{eq:lambda-uzp}
\begin{align}
&
\label{switch-1}
\sft'(s) \lambda_{\sfu,\sfp}(s) =\sft'(s) \lambda_{\sfz}(s) =0  &&  \foraa\ s\in(0,S),
\\
& 
\label{switch-2}
 \lambda_{\sfu,\sfp}(s)(1{-} \lambda_{\sfz}(s))  = 0  &&  \foraa\ s\in(0,S), 
 \end{align}
\end{subequations}
and $(\sft,\sfq)$ satisfies \foraa\ $s \in (0,S)$
the system of subdifferential inclusions 
\begin{subequations}\label{eq:diff-char-zero}
\begin{align}
& \lambda_{\sfu,\sfp}(s)   \mathrm{D} \disvo 2(\sfu'(s))+ (1- \lambda_{\sfu,\sfp}(s)) \, \mathrm{D}_u \calE_{\GGG \mu } (\sft(s),\GGG \sfq(s) ) = 0 && \text{in}\ H_\Dir^1(\Omega;\R^n)^* \,,
\\
&
(1{-}\lambda_{\sfz}(s))\, \partial \mathcal{R}(\sfz'(s))+ \lambda_{\sfz}(s)\, \mathrm{D}\mathcal{R}_2(\sfz'(s)) + (1{-}\lambda_{\sfz}(s)) \, \mathrm{D}_z \calE_{\GGG \mu } (\sft(s),\GGG \sfq(s) )  \ni 0 
 && \text{in } 
\Hs(\Omega)^* \,,
\\
&
(1{-} \lambda_{\sfu,\sfp}(s)) \, \partial_{\pi} \mathcal{H}(\sfz(s), \sfp'(s)) +  \lambda_{\sfu,\sfp}(s)\,  \mathrm{D} \disho 2(\sfp'(s))
+(1{-} \lambda_{\sfu,\sfp}(s)) \, \mathrm{D}_p \calE_{\GGG \mu } (\sft(s), \GGG \sfq(s) )  \ni 0  && \text{in  } L^2(\Omega;\MD) \,.
\end{align}
\end{subequations}
    \end{proposition}
    Like the argument for Prop.\ \ref{prop:diff-charact-BV}, the \emph{proof} is again based on  the analysis of the structure 
    of the contact set associated with $\Mlinamezero$ (cf.\ \eqref{contact-set}), which in turn can be characterized by adapting the arguments from the proof of
    \cite[Prop.\ 5.1]{MRS14}.
    \begin{remark}
    \label{rmk:mech-interp}
    \upshape
    Along the lines of \cite[Rmk.\ 5.4]{MRS14}, we observe that system \eqref{eq:diff-char-zero}
reflects the fact that $\sfu$ and $ \sfp$ relax to equilibrium and rate-independent evolution,  respectively,   faster than $\sfz$. Indeed, at a jump
(corresponding to $\sft'=0$,  hence  the coefficients $\lambda_{\sfu,\sfp}$ and $\lambda_{\sfz}$ can be nonzero),
due to \eqref{switch-2}
 either  $\lambda_{\sfz}=1$, and then
$\sfz'=0$, or $\lambda_{\sfu,\sfp}=0$, which gives that $\sfu$ is at equilibrium and $\sfp$ satisfies the local stability condition $-\mathrm{D}_p \calE_{\GGG \mu } (t,q) \in \partial_{\pi} \calH(z,0)$. Namely, at a jump $\sfz$ cannot change until $\sfu$ has reached the equilibrium and $\sfp$ attained the stable set $\partial_{\pi} \calH(z,0)$. After that,
$\sfz$ may either evolve rate-independently (if $\lambda_\sfz=0$), or governed by viscosity  (if $\lambda_\sfz\in(0,1)$). 
    \end{remark}
    With our final result we prove the convergence of the reparameterized viscous solutions $(\sft_{\eps,\nu}, \sfq_{\eps,\nu})_{\eps,\nu}$ to a $\BV$ solution of the multi-rate system as \emph{both} $\eps$ and $\nu$ tend to zero. As observed right before the statement of Theorem~\ref{teo:rep-sol-exist}, we may suppose that the curves $(\sft_{\eps,\nu}, \sfq_{\eps,\nu})$ are defined in a  fixed   interval $(0,S)$. 
\begin{theorem} 
\label{teo:rep-sol-exist-2}
Under the assumptions of Section~\ref{s:2} and \eqref{EL-initial}, 
for all vanishing sequences  $(\eps_k)_k$ and    $(\nu_k)_k$ such that $S_{\eps_k,\nu_k} \to S$ there exist a (not relabeled) subsequence  $(\sft_{\eps_k, \nu_k},\sfq_{\eps_k, \nu_k})_k$ and a Lipschitz curve
$(\sft,\sfq) \in W^{1,\infty} ([0,S]; [0,T]{\times} \bfQ)$ such that convergences \eqref{weak-converg-uzpe-rep}
 hold  as $k\to+\infty$   and 
 $(\sft,\sfq) $ is a  $\BV$ solution to the \emph{multi-rate}  system with hardening according to Definition \ref{def:multi-rate-BV-solution-hardening}.
\end{theorem}
 \begin{proof} In the proof of this result, we  shall  mainly highlight the differences with respect to the 
  argument  for Theorem~\ref{teo:rep-sol-exist}, without repeating the analogous  passages.  
Hereafter, we shall consider sequences $\varepsilon_k$, $\nu_k\to 0$ and write $(\sft_k, \sfq_k)_k$
 in place of $(\sft_{\eps_k, \nu_k},\sfq_{\eps_k, \nu_k})_{\eps_k,\nu_k}$, and we shall not relabel subsequences. 
\paragraph{\bf Step $1$: Compactness.} As in the proof of Theorem~\ref{teo:rep-sol-exist}, we conclude that there exist a subsequence and $(\sft,\sfq) \in W^{1,\infty} ([0,S]; [0,T]{\times} \bfQ)$  such that the analogues of \eqref{limit-conv}--\eqref{420} hold. \GGG
\paragraph{\bf Step $2$: Finiteness of $\Mlirednamezero(\sft(\tau),\sfq(\tau),\sft'(\tau),\sfq'(\tau))$ when $\sft'(\tau)>0$.} 
As in Step~2 of Theorem~\ref{teo:rep-sol-exist},  we  introduce the set   $A:=\{\tau \in [0,S] \colon \sft'(\tau)>0\}$ and show that 
 $\Mlirednamezero(\sft(\tau),\sfq(\tau),\sft'(\tau),\sfq'(\tau))<+\infty$  for a.a.\ $t\in A$,  which yields the (local) stability condition \eqref{finite} for a.a.\ $t \in A$. 
 To do so, as in \eqref{limsup} we observe that  $\limsup_{k\to +\infty} \sft'_k(\tau)>0$.
We now use equality  \eqref{DVITOequalsSTAR}  at $r= \sft_k(\tau)$
and $q(r) = \sfq_k(\tau)$, and thus 
 we get
 \begin{equation}
 \label{1306192225}
 \begin{aligned}
  \DVitosk(\sft_k(\tau), \sfq_k(\tau))  =  \DVitosk(\sft_k(\tau), q_k(\sft_k(\tau)))  
  &  =\eps_k \sqrt{ \nu_k \| u_k'(\sft_k(\tau))\|^2_{ H^1, \bbD }   {+} 
\|z_k'(\sft_k(\tau))\|_{L^2}^2
{+} \nu_k\|p_k'(\sft_k(\tau))\|_{L^2}^2} 
\\
&= \frac{\eps_k}{\sft_k'(\tau)} \sqrt{ \nu_k \| \sfu_k'(\tau)\|^2_{ H^1, \bbD }   {+} 
\|\sfz_k'(\tau)\|_{L^2}^2
{+} \nu_k\|\sfp_k'(\tau)\|_{L^2}^2}
\\
& 
\leq  \frac{\eps_k}{\sft_k'(\tau)} \qquad \foraa\, \tau \in (0,S),
\end{aligned}
\end{equation}
where, again, the last estimate follows from  the normalization conditions \eqref{normalization}  and since $\nu \leq 1$. 
We observe that the right-hand side of \eqref{1306192225} goes to 0 as $k\to +\infty$.
\par
We now deduce  the local stability  \eqref{finite}  at almost all $s\in (0,S)$.   Indeed, if, say, $\|\mathrm{D}_u \calE_{\GGG \mu } (\sft(\tau),\sfq(\tau))\|_{(H^1, \mathbb{D})^*} > 0$, then we get by the semicontinuity inequality \eqref{1106191956} that $\liminf_{k\to +\infty}  \|\mathrm{D}_u \calE_{\GGG \mu } (\sft_k(\tau),\sfq_k(\tau))\|_{(H^1, \mathbb{D})^*} >0$. Recalling the definition of $\DVitosk$ from \eqref{0307191952}, and since $\nu_k \to 0$, this would give that $\liminf_{k\to +\infty} \DVitosk(\sft_k(\tau), \sfq_k(\tau)) = +\infty$, which contradicts \eqref{1306192225}. Thus $\mathrm{D}_u \calE_{ \mu } (\sft(\tau),\sfq(\tau))=0$. In the same way we get $-\mathrm{D}_p \calE_{\GGG \mu } (\sft(\tau),\sfq(\tau)) \in \partial_\pi \calH(z,0)$, while, if $-\mathrm{D}_z \calE_{\mu} (\sft(\tau),\sfq(\tau)) \notin \partial\calR(0)$, we would get  $\liminf_{k\to +\infty} \DVitosk(\sft_k(\tau), \sfq_k(\tau)) >0$, which still would contradict \eqref{1306192225}.
\par
Moreover, in view of  \eqref{eqs:1106191955} and of the regularity of $(\sft, \sfq)$, we have that  the sets 
\begin{equation}\label{1206191009'}
\begin{split}
B^\circ_\mu := \{ \tau \in [0,S] \colon   \DVitosred(\sft(\tau), \sfq(\tau)) >0\}\, \text{ and }\, C^\circ_\mu := \{ \tau \in [0,S] \colon  \tilded_{L^2(\Omega)} ({-}\mathrm{D}_z \calE_{\mu} (\sft(\tau),\sfq(\tau)),\partial\calR(0)) >0\}
\end{split}
\end{equation} 
are open and included in $[0,S] \setminus A$.

\paragraph{\bf Step $3$: The Energy-Dissipation upper estimate \eqref{ED-balance-zero}.} 
 By the analogue of Proposition \ref{prop:charact-BV}, in order to conclude that
$(\sft,\sfq)$ is a $\BV$ solution to the multi-rate system with hardening it is sufficient to obtain  \eqref{ED-balance-zero}  as an upper estimate $\leq$. With this aim, 
as in Step~3 of Theorem~\ref{teo:rep-sol-exist}, we start from the analogues of \eqref{where-limit} and \eqref{1306191946}.
 First of all,  it holds that for a.e.\ $\tau \in (0,S)$
 \begin{equation}\label{1506190857}
   \MredVitokk(\sft_k(\tau), \sfq_k(\tau), \sfq_k'(\tau))  \geq \frac{1}{\sqrt{\nu_k}} \|\sfz'_k(\tau)\|_{L^2} \, \DVitosred(\sft_k(\tau), \sfq_k(\tau))\,,
 \end{equation} 
recalling \eqref{0307191952} and \eqref{1306191946}.
Now, we may apply Lemma~\ref{le:0107192150}  with 
 the choices
 $I:=B^\circ$, $m_k=\|\sfz'_k\|_{L^2}$
 such that $m_k \weaksto m$ in $L^\infty (0,S)$ and 
  $m \geq \|\sfz'\|_{L^2}$ a.e.\ in $(0,S)$, 
  and with  $h_k:= \DVitosred(\sft_k, \sfq_k)$, $h:= \DVitosred(\sft, \sfq)$.
  Indeed, observe that
\begin{equation}
\label{citata-dop-sez7}
\liminf_{k\to+\infty} \DVitosred(\sft_k(\tau), \sfq_k(\tau)) \geq  \DVitosred(\sft(\tau), \sfq(\tau)) \qquad \text{for all } \tau \in [0,S],
\end{equation}

   thanks to \eqref{190506} and 
    the lower semicontinuity properties 
   \eqref{eqs:1106191955}. We thus obtain that
\begin{equation}\label{1506190906}
\int_{ B_\mu^\circ} \|\sfz'(\tau)\|_{L^2} \,  \DVitosred(\sft(\tau), \sfq(\tau)) \dd \tau \leq \liminf_{k\to \infty} \int_{ B_\mu^\circ} \|\sfz'_k(\tau)\|_{L^2}\,   \DVitosred(\sft_k(\tau), \sfq_k(\tau))\dd \tau\,.
\end{equation}
Since $\nu_k \to 0$, from \eqref{1506190857} and \eqref{1506190906} we deduce that $\sfz'(\tau)=0$ for a.e.\ $\tau \in B^\circ$, that is  
\begin{equation}\label{1506190910}
\sfz'(\tau) \,  \DVitosred(\sft(\tau), \sfq(\tau))=0 \quad \text{ for a.a.\ }\tau \in (0,S)\,.
\end{equation}
 In view of the definition \eqref{Mlizero=0} of $\Mlirednamezero$, \eqref{1506190910} yields
 that 
\begin{equation}\label{1506191644}
\Mlirednamezero(\sft(\tau), \sfq(\tau), \sft'(\tau), \sfq'(\tau)) = \DVito(\sfu'(\tau), \sfp'(\tau))\,  \DVitosred(\sft(\tau), \sfq(\tau))\qquad  \text{ a.e.\ in }B^\circ\,.
\end{equation}
\GGG
By \eqref{0307191952}, \eqref{1306191946},  and an easy algebraic calculation   we obtain that 
\begin{equation}\label{0207190031}
\MredVitokk(\sft_k(\tau), \sfq_k(\tau), \sfq_k'(\tau))
 \geq \DVito(\sfu'_k(\tau), \sfp'_k(\tau)) \,  \DVitosred(\sft_k(\tau), \sfq_k(\tau))\,.
\end{equation} 
Then, again by Lemma~\ref{le:0107192150}, applied thanks to \eqref{limit-conv} and \eqref{eqs:1106191955}, we deduce 
\begin{equation}\label{1506191716}
\begin{split}
\int_{B^\circ}  \Mlirednamezero(\sft(\tau), \sfq(\tau), \sft'(\tau), \sfq'(\tau)) \dd \tau &=\int_{B^\circ} \DVito(\sfu'(\tau), \sfp'(\tau)) \,   \DVitosred(\sft(\tau), \sfq(\tau))
 \dd \tau 
\\& \leq \liminf_{k\to +\infty}\int_{B^\circ} \MredVitokk(\sft_k(\tau), \sfq_k(\tau),  \sfq_k'(\tau)) \dd \tau\,.
\end{split}
\end{equation}

\GGG
Let us now consider the set   $C_\mu^\circ \setminus B_\mu^\circ$,   where $\tilded_{L^2(\Omega)} ({-}\mathrm{D}_z \calE_\mu (\sft(\tau),\sfq(\tau)),\partial\calR(0))>0$ with $ \DVitosred(\sft(\tau), \sfq(\tau))=0$, cf.\ \eqref{1206191009}.
Starting from \eqref{0307191952} and \eqref{1306191946}, we estimate
\begin{equation}\label{0207190035}
\MredVitokk(\sft_k(\tau), \sfq_k(\tau),  \sfq_k'(\tau)) \geq \|\sfz_k'(\tau)\|_{L^2} \, \tilded_{L^2(\Omega)} ({-}\mathrm{D}_z \calE_\mu (\sft_k(\tau),\sfq_k(\tau)),\partial\calR(0))\,.
\end{equation} 
We then employ Lemma~\ref{le:0107192150}  with  $I:= C^\circ_\mu \setminus B^\circ_\mu$, $m_k:=\|\sfz'_k\|_{L^2}$ 
such that $m_k\weaksto m \geq \|\sfz'\|_{L^2}$ in $L^\infty(0,S)$,
 $h_k:= \tilded_{L^2(\Omega)} ({-}\mathrm{D}_z \calE_\mu (\sft_k(\tau),\sfq_k(\tau)),\partial\calR(0)) $, $h:=\tilded_{L^2(\Omega)} ({-}\mathrm{D}_z \calE_\mu (\sft(\tau),\sfq(\tau)),\partial\calR(0)) $. Again, we obtain that $\liminf_{k\to+\infty} h_k(\tau) \geq h(\tau)$ for all $\tau\in [0,S]$
 by \eqref{limit-conv} and \eqref{1106191957}. Thus, with 
 Lemma~\ref{le:0107192150} 
  we  get
\begin{equation}\label{1506191723}
\begin{split}
\int_{(0,S) \setminus B^\circ_\mu}  & \Mlirednamezero(\sft(\tau), \sfq(\tau), \sft'(\tau), \sfq'(\tau)) \dd \tau = \int_{C^\circ_\mu \setminus B^\circ_\mu} \|\sfz'(\tau)\|_{L^2} \, \tilded_{L^2(\Omega)} ({-}\mathrm{D}_z \calE_\mu (\sft(\tau),\sfq(\tau)),\partial\calR(0)) \dd \tau 
\\ &
\leq \liminf_{k\to +\infty} \int_{C^\circ_\mu \setminus B^\circ_\mu} \|\sfz_k'(\tau)\|_{L^2} \, \tilded_{L^2(\Omega)} ({-}\mathrm{D}_z \calE_\mu (\sft_k(\tau),\sfq_k(\tau)),\partial\calR(0)) \dd \tau
\\ &
\leq \liminf_{k\to +\infty} \int_{C^\circ_\mu \setminus B^\circ_\mu} \MredVitokk(\sft_k(\tau), \sfq_k(\tau), \sfq_k'(\tau)) \dd \tau \,.
\end{split}
\end{equation}
\par
All in all,
collecting \eqref{1506191716} and \eqref{1506191723} we conclude
\begin{equation*}
\int_0^S  \Mlirednamezero(\sft(\tau), \sfq(\tau), \sft'(\tau), \sfq'(\tau)) \dd \tau \leq \liminf_{k\to +\infty} \int_0^S \MredVitokk(\sft_k(\tau), \sfq_k(\tau), \sfq_k'(\tau)) \dd \tau\,.
\end{equation*}
The  remainder   of the proof (namely the semicontinuity of the other terms in $\Mlinamezero$ and of the  driving
 energy $\calE_\mu$,  and  the continuity of the power term) follows as in Step~3 of Theorem~\ref{teo:rep-sol-exist}.
  In this way, we conclude the limit passage in the Energy-Dissipation balance \eqref{where-limit},
 obtaining the desired \eqref{ED-balance-zero} as an upper estimate $\leq$. 
The proof is then completed.
\end{proof}
\par

 \section{The vanishing-hardening limit}
  \label{s:van-hard}
  This final section addresses the limit passage in the viscous system \eqref{RD-intro} as the three parameters $\varepsilon$, $\nu$, $\mu$ vanishes simultaneously. 
  To deal with this asymptotic analysis we combine the approach to the  limit passage as viscosity vanishes
  developed in Section~\ref{s:van-visc}, with the technical, functional-analytic  tools related to the passage from plasticity with hardening to perfect plasticity. 
  Therefore,  in what follows 
  \begin{itemize}
  \item[-]
   first, we  shall 
   establish the setup for the limiting
  perfectly plastic model, recalling 
  results from \cite{DMDSMo06QEPL, FraGia2012}
  \item[-] secondly, 
  we   shall 
  introduce a suitable  ``Energy-Dissipation''  arclength reparameterization of viscous solutions; in combination with the Energy-Dissipation balance \eqref{rescaled-enid-eps-OLD},  the reparameterized solutions
   will  
  thus satisfy a suitable normalization condition whence the key estimates 
  stem, as well as the specific temporal and spatial regularity properties fixed in the notion of \emph{admissible parameterized curve}, cf.\ Definition~\ref{def:admparcur};
  \item[-] we 
   shall 
  properly define
   the vanishing-viscosity contact potential
  relevant for $\BV$ solutions to the perfectly plastic system,
    taking care of the technicalities
  related to the new functional setup. Hence, we  shall  proceed to the introduction of $\BV$ solutions in Definition~\ref{def:parBVsol-PP};
  \item[-] we 
   shall 
  address the properties of $\BV$ solutions and in particular characterize them in terms of an
  Energy-Dissipation \emph{upper estimate} in Proposition~\ref{prop:charact-BV-pp}. Such characterization 
  will  
  play a key role 
  in the proof of our existence result, Theorem~\ref{teo:exparBVsol} ahead. 
  \end{itemize}
   \par
Let us now first
   fix the setup for the perfectly plastic system. We mention in advance that the  space for the displacements 
   will be  $\BD(\Omega)$ and the space for the plastic strains will be   $\MbD$, i.e.\
      the space of  bounded Radon measures on  $\Omega\cup \Gdir$  with values in $\MD$; this reflects 
       the fact that, now,  the plastic strain $p$ is a measure that can
   concentrate on Lebesgue-negligible sets.
   \paragraph{\bf Setup adapted for perfect plasticity: the state space and  driving energy functional}
   The state space for the  perfectly plastic system with damage is 
   \begin{equation}
   \label{defQpp}
   \begin{aligned}
   \Qpp: = \{ q=(u,z,p) \in &\,\BD(\Omega)\times \Hs(\Omega) \times \MbD\, : \\& \, e: = \sig{u} -p \in \Lnn, 
   \ u \odot \mathrm{n} \,\hn + p =0 \text{ on }  \Gdir  \}
   \end{aligned}
   \end{equation}
   where $ \mathrm{n} $ denotes the normal vector to $\partial \Omega$ and $\odot$ the symmetrized tensorial product. 
   Observe that the condition $ u \odot \mathrm{n} \, \mathscr{H}^{n-1} + p =0 $ relaxes the homogeneous Dirichlet condition 
   $u=0 $ on  $\Gdir$. 
   \par
 \paragraph{\bf Setup adapted for perfect plasticity: the plastic dissipation potential}
  We extend the plastic dissipation potential $\mathcal{H}(z,\cdot)$ to the reference space 
   $\MbD$. 
 We define
 $\Hpp \colon \rmC^0(\overline{\Omega}; [0,1]) \times  \MbD  \to \mathbb{R}$ by
  \begin{equation}
\label{plast-diss-pot}
\Hpp(z,p):= \int_{ \Omega \cup \Gdir} H\biggl(z(x),\frac{\mathrm{d}p}{\mathrm{d}\mu}(x) \biggr)\,\mathrm{d}\mu(x)\,,
\end{equation}  
where \GGG  $H$ is defined in \eqref{0307191050},  $\mu \in \MbD$ is a positive  measure such that $ p \ll \mu $ and $\frac{\mathrm{d} p}{\mathrm{d}\mu}$ is the Radon-Nikod\'ym derivative of $p$ with respect to $\mu$; since $H(z(x),\cdot)$ is one-homogeneous, the definition is actually independent of $\mu$.
We refer to \cite{GofSer} for the theory of convex functions of measures.
By \cite[Proposition~2.37]{AmFuPa05FBVF},  the functional 
 $p \mapsto \Hpp(z,p)$ 
is convex and positively one-homogeneous for every   $z \in \rmC^0(\overline{\Omega};[0,1])$.
 In particular, 
$\Hpp(z,p_1+p_2)\leq \Hpp(z,p_1)+ \Hpp(z,p_2)$
 for every $z \in 
 \rmC^0(\overline{\Omega};[0,1])$ and $p_1, p_2 \in \MbD$. 
 Since 
  $|\frac{\mathrm{d}p}{\mathrm{d}|p|}(x)|=1$  
 for $|p|$-a.e.\ $x \in  \Omega \cup \Gdir $, \GGG by \eqref{propsH}  we have
 \begin{equation*}
r \|p\|_1\leq \Hpp(z, p) \leq R \|p \|_1\,,\label{Hco}
 \end{equation*}
where we denote by $\| \cdot \|_1$ the total variation of a measure (in the case of $p$ on $\Omega\cup \Gdir$), and
\begin{equation*}\label{modcont}
 0\leq \Hpp(z_2, p) - \Hpp(z_1,p) \leq C'_K\|z_1-z_2 \|_{ L^\infty} \|p \|_1 \quad
\text{for } 0\leq z_1\leq z_2\leq 1 \,.
 \end{equation*}
 Therefore, by Reshetnyak's lower semicontinuity Theorem, if $(z_k)_k$ and $(p_k)_k$ are sequences in $\rmC^0(\overline \Omega;[0,1])$ and $\MbD $ such that $z_k \rightarrow z$ in $\rmC^0(\overline\Omega)$ and $p_k \rightharpoonup p$ weakly$^*$ in $\MbD$, then
 \begin{equation*}
 \Hpp(z,p) \leq \liminf_{k \rightarrow +\infty} \Hpp(z_k, p_k)\,. \label{Hsci}
 \end{equation*}

\paragraph{\bf{Stress-strain duality.}}  
Let us recall the notion of stress-strain duality,  relying  on \cite{Kohn-Temam83}, \cite{DMDSMo06QEPL}, and the more recent extension to Lipschitz boundaries \cite{FraGia2012},
to which we refer for the properties mentioned below.
First of all, we recall the definition (in the sense of \cite{DMDSMo06QEPL}) of \emph{admissible displacement and strains} $A(w)$
 associated with a function   $w \in H^1(\R^n; \R^n)$, 
  that is
\begin{equation*}
\begin{split}
A(w):=\{(u,e,p) \in  \, & \BD(\Omega) \times L^2(\Omega;\Mnn) \times  \MbD  \colon \\
& \rmE(u)   =e+p \text{ in }\Omega,\, p=(w-u){\,\odot\,}\rmn\,\hn \text{ on }  \Gdir \}\,.
\end{split}
\end{equation*}
 We also recall  the \emph{space of admissible plastic strains}
\begin{equation*}
\begin{split}
\Pi(\Omega) := \{p\in  \MbD  \colon  \exists \, (u, w,e)\in \BD(\Omega)\times   H^1(\R^n;\R^n)  \times L^2(\Omega;\Mnn)\,
 \text{ s.t.}\, (u,e,p)\in A(w) \} \,.
\end{split}
\end{equation*}
We then define
\begin{equation*}
\Sigma(\Omega):= \{ \sigma \in \Lnn \colon  \mathrm{div}(\sigma)  \in L^n(\Omega;\R^n) \,, \ \sigma_\dev \in L^{\infty}(\Omega;\MD) \}
\end{equation*} 
and, for $\sigma \in \Sigma(\Omega)$ and $p \in \Pi(\Omega)$,  we set 
\begin{equation}\label{sD}
\langle [\sigma_\dev:p], \varphi\rangle:=-\int_\Omega \varphi\sigma\cdot (e{-} \rmE(w)  )\,\mathrm{d}x-\int_\Omega\sigma\cdot[(u{-}w)\odot \nabla \varphi]\,\mathrm{d}x -\int_\Omega \varphi \, ( \mathrm{div}(\sigma))  \cdot (u{-}w)\,\mathrm{d}x
\end{equation}
for every $\varphi \in \mathrm{C}^\infty_c(\R^n)$, where $u$ and $e$ are such that $(u,e,p)\in A(w)$;
the definition is indeed independent of $u$ and $e$.
%
If $\sigma \in \Sigma(\Omega)$ and $p \in \Pi(\Omega)$, then
$\sigma \in L^r(\Omega;\Mnn)$ for every $r < \infty$,
and $[\sigma_\dev:p]$ is a bounded Radon measure such that
$\|[\sigma_\dev:p]\|_1\leq \|\sigma_\dev\|_{ L^\infty}\|p\|_1$ in $\R^n$.
Considering the restriction of this measure to $ \Omega \cup \Gdir $, we also define
\[ \langle \sigma_\dev\, |\,p\rangle:=[\sigma_\dev:p]( \Omega \cup \Gdir )\,.\]
  By \eqref{sD}
  and taking into account that $u\in\BD(\Omega) \subset L^{\frac{n}{n-1}}(\Omega;\R^n)$, 
   if $[\sigma \rmn]\in L^\infty(\Gneu;\R^n)$ (recall \eqref{2809192054})  and \eqref{Omega-s2} holds,   then we have the  integration by parts
 formula
\begin{equation}\label{intparti}
\langle \sigma_\dev\,|\,p\rangle=-\langle\sigma, e- \rmE(w)  \rangle_{\Lnn}+  \langle -\diver\,\sigma, u-w \rangle_{L^{\frac{n}{n-1}}(\Omega; \R^n )}  + \langle [\sigma \rmn], u-w \rangle_{L^{1}(\Gneu; \R^n)}   
\end{equation}
for every  $\sigma \in \Sigma(\Omega)$ and $(u,e,p)\in A(w)$.
Thus,  defining for $\sigma \in \Sigma(\Omega)$  the  functional $-\Diver(\sigma) \in \BD(\Omega)^*$ via 
\begin{equation}\label{2307191723}
\langle -\Diver(\sigma), v \rangle_{\BD(\Omega)} :=\langle - \mathrm{div}(\sigma),   v \rangle_{L^{\frac{n}{n-1}}(\Omega; \R^n )}  + \langle [\sigma \rmn], v \rangle_ {L^{1}(\Gneu; \R^n)} 
\end{equation}
 for all   $v \in \BD(\Omega)$, 
we have that \eqref{intparti} reads as
\begin{equation}\label{2307191732}
\langle \sigma_\dev\,|\,p\rangle=-\langle\sigma, e- \rmE(w)  \rangle_{\Lnn}+\langle -\Diver(\sigma), u-w \rangle_{\BD(\Omega)}
\end{equation}
%
For $z \in \mathrm{C}^0(\ol\Omega)$ let 
\begin{equation}\label{2307192044}
\mathcal{K}_z(\Omega):= \{ \sigma  \in \Sigma(\Omega) \colon  \sigma_\dev(x) \in K(z(x)) \text{ for a.e.}\ x \in \Omega\} \,.
\end{equation}
Since the multifunction $z \in [0,1]\mapsto K(z)$ is continuous, 
from \cite[Proposition~3.9]{FraGia2012} (which holds also if $ \mathrm{div}(\sigma)$  is not identically 0) it follows that for every $\sigma \in \mathcal{K}_z(\Omega)$ 
\begin{equation}\label{Prop3.9}
H\biggl(z, \frac{\mathrm{d}p}{\mathrm{d} |p|}\biggr)|p| \geq [\sigma_\dev:p] \quad \text{as measures on }  \Omega \cup \Gdir  \,.
\end{equation}
 In particular, we have 
\begin{equation}\label{eq:carH}
\Hpp(z, p)\geq \sup_{\sigma  \in  \mathcal{K}_z(\Omega)}  \langle \sigma_\dev\, |\, p\rangle \,\qquad \text{for every $p \in \Pi(\Omega)$.} 
\end{equation}
 \begin{remark}\label{rem:3009191901}
\upshape
In \cite[Remark~2.9]{FraGia2012} the authors explain that in the presence of external forces one has to resort to the classic (deviatoric) stress-(plastic) strain duality, provided by \cite{Kohn-Temam83} and employed in several papers, e.g.\ \cite{DMDSMo06QEPL}, to put in duality $\varrho_D(t)$ and $p \in \Pi(\Omega)$. Such a duality requires one of the two following conditions, alternatively: either (1)  $\varrho \in \AC(0,T; \rmC^0(\ol \Omega; \MD))$ or (2) $\Omega$ globally of class $\rmC^2$.  The use of the Kohn-Temam duality seems to  be   needed 
to infer that   \eqref{eq:carH} holds as an equality,  which in turn   implies  the analogue of \eqref{controllo-SL} for $\Hpp$, $\calK_z(\Omega)$, $p \in \Mb$ in place of $\calH$, $\widetilde{\calK}_z(\Omega)$, $p\in L^1$. However, by our approximation procedure  via  plasticity with hardening, we  just need to use the coercivity condition \eqref{controllo-SL} in the \emph{a priori} estimates for the  solutions of the system  with plastic hardening (cf.\ \eqref{crucial-est-Vito}), together with \eqref{eq:carH} to pass to the limit. For this reason we do not assume any further regularity neither on $\varrho$ nor on $\Omega$.
\end{remark} 
 \noindent \textbf{The energy functional.}   The energy functional  $\calE_0$   driving the perfectly plastic system
   has an expression analogous to the functional
   $\calE_\mu$ \eqref{RIS-ene} for the system 
   with hardening  where   $\mu$ is formally set equal to $0$. Indeed,   it consists of the contributions of the  elastic energy, of  the potential energy for the damage variable, and of the time-dependent volume
   and surface forces.  Then
   $\calE_0: [0,T]\times \Qpp \to \R$ is  defined by 
\[
\begin{aligned}
\enen{0} (t,u,z,p)   :=  \calQ(z,e(t)) + \int_\Omega   W(z) \,\mathrm{d}x + \frac12 \ass(z,z)  - \langle F(t), u+w(t) \rangle_{\BD(\Omega)}\,, 
\end{aligned}
\]
where we have highlighted the elastic part 
$e(t)=\sig{u{+}w(t)}-p$ 
 of the strain tensor.
 Since $\varrho(t) $ from \eqref{2909191106} is in  $\Sigma(\Omega)$ and $F(t)=-\Diver(\varrho(t))$ for all $t\in [0,T]$ by \eqref{2809192200}, we may employ \eqref{2307191732} to rewrite $\calE_0$ as
(cf.\ \eqref{serve-for-later}) 
\begin{equation}
\label{serve-for-later-bis}
\begin{aligned}
&
\enen{0} (t,u,z,p)  
 = \calF_0 (t,z,e(t))   -  \langle \rho_\dev(t)\, |\, p \rangle  \qquad  \text{ with } 
 \\
 &
 \calF_0 (t,z,e)  : =   \calQ(z,e) + \int_\Omega   W(z)\,\mathrm{d}x + \frac12 \ass(z,z) -\int_\Omega \rho(t) (e  - \rmE(w(t)))  \dd x  -\langle F(t), w(t) \rangle_{\BD(\Omega)}\,. 
 \end{aligned}
\end{equation}
%
%
%
%
%
%
%
 \subsection*{\bf Energy-dissipation arclength reparameterization} 
  As already mentioned, we   shall obtain  Balanced Viscosity solutions 
  to the perfectly plastic system by taking 
  the joint vanishing-viscosity and hardening limit
 of   (reparameterized) viscous solutions to 
Problem \ref{prob:visc}. Thus,
  let   $(q^\mu_{\eps,\nu})_{\eps,\nu, \mu} = (u^\mu_{\eps,\nu},z^\mu_{\eps,\nu},p^\mu_{\eps,\nu})_{\eps,\nu,\mu}$ be a family of solutions to Problem~\ref{prob:visc}. We are going to reparameterize them by the 
 \emph{Energy-Dissipation arclength} %
 $ \widetilde{s}^\mu_{\eps,\nu} : [0,T]\to [0,\widetilde{S}_{\eps,\nu}^\mu]$ (with $\widetilde{S}_{\eps,\nu}^\mu: =  \widetilde{s}_{\eps,\nu}^\mu(T)$)
 defined by 
 \begin{equation}\label{2006191707}
 \begin{split}
 \widetilde{s}^\mu_{\eps,\nu}(t): = \int_0^t \Big( 1 & {+}  \sqrt{\mu}\| {u^{\mu}_{\eps,\nu}}'(\tau)\|_{H^1(\Omega;\R^n)} {+} \|{z^{\mu}_{\eps,\nu}}'(\tau)\|_{\Hs(\Omega)} {+} \|{p^{\mu}_{\eps,\nu}}'(\tau)\|_{L^1(\Omega;\MD)}  {+} \sqrt{\mu}\|{p^{\mu}_{\eps,\nu}}'(\tau)\|_{L^2(\Omega;\MD)} 
 \\&
 {+} \|{e^{\mu}_{\eps,\nu}}'(\tau)\|_{\Lnn} {+} \DVito_\nu ({u^{\mu}_{\eps,\nu}}'(\tau), {p^{\mu}_{\eps,\nu}}'(\tau))\, \DVitos(\tau, q^{\mu}_{\eps,\nu}(\tau))   \Big) \dd \tau \,,
\end{split} 
\end{equation}
with $\DVito_\nu$ and $\DVitos$ 
 from \eqref{new-form-D-Dvito}.  We shall comment on the choice of the 
arclength function $ \widetilde{s}^\mu_{\eps,\nu}$ below. 
\GGG By  estimates   \eqref{uniform-est-cont}  and \eqref{est-Menu}  we have that $\sup_{\varepsilon,\nu,\mu}   \widetilde{S}^\mu_{\eps,\nu} \leq C $.
As in \eqref{rescaled-solutions}, we set
\[
\begin{aligned}
 \sft^\mu_{\varepsilon,\nu}: =  (\widetilde{s}^\mu_{\eps,\nu})^{-1} , \qquad 
 \sfq^\mu_{\varepsilon,\nu}: =q^\mu_{\varepsilon,\nu}\circ  \sft^\mu_{\varepsilon,\nu}   = (\sfu^\mu_{\varepsilon,\nu}, \sfz^\mu_{\varepsilon,\nu},\sfp^\mu_{\varepsilon,\nu}),
 \qquad   \sfe^\mu_{\varepsilon,\nu}: = e^\mu_{\varepsilon,\nu}\circ  \sft^\mu_{\varepsilon,\nu},  
 \qquad
 \serifsigma^\mu_{\varepsilon,\nu} : =\sigma^\mu_{\varepsilon,\nu}\circ  \sft^\mu_{\varepsilon,\nu} 
 \end{aligned}
 \]
that we may assume defined on a fixed interval $[0,S]$, with $S:=\lim_{\varepsilon,\nu, \mu \down 0} \widetilde{S}^\mu_{\eps,\nu}$ (the limit is intended along a suitable subsequence). 
\par
The very same calculations as in Section~\ref{s:van-visc} show that the rescaled solutions $( \sft^\mu_{\varepsilon,\nu},  \sfq^\mu_{\varepsilon,\nu})_{\eps,\nu, \mu} $
 and the curves $(\sfe^\mu_{\varepsilon,\nu})_{\eps,\nu, \mu} $ 
 satisfy the parameterized Energy-Dissipation balance 
 \eqref{rescaled-enid-eps-OLD}
 as well as    the normalization condition 
\begin{equation}\label{2906191307}
\begin{split}
{\sft\men}'(s) &{+}  \sqrt{\mu}\| {\sfu^{\mu}_{\eps,\nu}}'(s)\|_{H^1(\Omega;\R^n)} {+} \|{\sfz^{\mu}_{\eps,\nu}}'(s)\|_{\Hs(\Omega)} {+} \sqrt{\mu}\|{\sfp^{\mu}_{\eps,\nu}}'(s)\|_{L^2(\Omega;\MD)} 
+
\|{\sfp^{\mu}_{\eps,\nu}}'(s)\|_{L^1(\Omega;\MD)} 
 \\&  
  {+} \|{\sfe^{\mu}_{\eps,\nu}}'(s)\|_{\Lnn} 
 {+} \DVito_\nu({u^{\mu}_{\eps,\nu}}'(s), {\sfp^{\mu}_{\eps,\nu}}'(s))\, \DVitos (\sft\men(s), \sfq^{\mu}_{\eps,\nu}(s))=1\qquad \text{ for a.e.\ } s\in (0, S\men)\,.
 \end{split}
\end{equation}
 The choice of  $ \widetilde{s}^\mu_{\eps,\nu}$ is precisely motivated by the 
need to ensure the validity of \eqref{2906191307}; 
in the lines below we are going to hint at the role of the term $\DVito_\nu\,  \DVitos $, while  that of the contributions modulated by $\sqrt{\mu} $ will be evident 
in the proof of the upcoming Theorem~\ref{teo:exparBVsol}. 
Let us also mention in advance that, in analogy with Section~\ref{s:van-visc}, we  shall  pass to the limit as $\eps,\, \nu,\, \mu \down 0$ in 
the energy balance 
 \begin{equation}\label{2906191258}
 \begin{split}
   \calE_\mu & (\sft^\mu_{\varepsilon,\nu}(S),\sfq^\mu_{\varepsilon,\nu}(S)) + \int_{0}^{S} 
   \MeVitoname(\sft^\mu_{\varepsilon,\nu}(\tau),\sfq^\mu_{\varepsilon,\nu}(\tau),{\sfq\men}'(\tau)) \dd \tau 
   \\& 
   =  \calE_\mu (0,\sfq_0) +\int_{0}^{S} \partial_t \calE_\mu (\sft\men(\tau), \sfq\men(\tau)) \, {\sft\men}'(\tau) \dd \tau \,.
   \end{split}
\end{equation}
 \subsection*{\bf The vanishing-viscosity contact potential for the perfectly plastic system}
 Clearly,  upon  taking the limit of the viscous system as the parameters $\eps,\, \nu,\, \mu \down 0$, we are in particular addressing the case in which 
 the viscosity in the momentum equation and in the plastic flow rule
 tends to zero with a higher rate than the viscosity in the damage flow rule. Therefore, the analysis carried out in Section~\ref{s:6-nu-vanishes} would lead us to expect, for the limiting system, a notion of $\BV$ solution featuring
 a  vanishing-viscosity potential 
  (that will be denoted by $\calM_{0,0}^0$ for consistency of notation), 
  with the same structure as that from  \eqref{Mil-sec6}, but associated with the driving energy $\calE_0$ for the perfectly plastic system. Specifically, one would envisage to deal with 
 the  quantity $\DVito^*(t,q): =\DVitosredzero(t,q)$ 
 encompassing   the   $(H^1)^*$-norm   of $\rmD_u \calE_0$, and the $L^2$-distance  of $\rmD_p \calE_0$
  from the stable set, cf.\ \eqref{3006191243}. However, such quantities 
 are no longer well defined for  the functional $\calE_0$, defined on $[0,T]\times \Qpp$ 
 (while the 
 $L^2$-distance of   $\rmD_z \calE_0$
  from the stable set still makes sense). 
  Therefore, in 
order to introduce the vanishing-viscosity potential $\oMliname$  for the perfectly plastic system, 
we first introduce suitable  ``surrogates''   of  the  $(H^1)^*$-norm   of $\rmD_u \calE_0$, and the $L^2$-distance from the stable set of $\rmD_p \calE_0$.
In accord with the representation formulae from Lemma \ref{lemma:sup},  we set,  for $\sigma(t)= \mathbb{C}(z) e(t)=\C(z) \rmE(u+w(t))-p$, 
\begin{subequations}
\label{surrogates}
 \begin{align}
 \label{surrogate-1}
 &
 \calS_u \calE_0(t, q) 
:= \sup_{\substack{\eta_u \in H_\Dir^1(\Omega;\R^n) \\ \|\eta_u\|_{( H^1, \bbD)}\le 1} } \langle -\mathrm{Div}(\sigma(t)) -F(t), \eta_u \rangle_{ H_\Dir^1(\Omega;\R^n) }
\,, 
\\
&
 \label{surrogate-2}
\mathcal{W}_p \calE_0(t,q):=\sup_{ \substack{\eta_p \in L^2(\Omega;\MD) \\  \|\eta_p\|_{L^2}\le1} } \left(  \langle \sigma_\dev(t), \eta_p \rangle_{L^2(\Omega;\MD)} -\calH(z,\eta_p)\right) \,.
\end{align}
We then set 
\begin{equation}
\label{right-Dstar}
\DVito^*(t,q): = \sqrt{( \calS_u \calE_0(t, q) )^2 +(\mathcal{W}_p \calE_0(t,q))^2} \,.
\end{equation}
\end{subequations}
 Notice that the above expressions are well-defined since
$e(t)$ and, a fortiori, $\sigma(t)$ are elements in $L^2(\Omega;\Mnn)$. 
%
%
\par
Thus, we are in a position to define the vanishing-viscosity contact potential  $\oMliname: [0,T]  \times \Qpp \times [0,+\infty) \times \Qpp \to [0,+\infty]$  via 
\begin{subequations}\label{eqs:3006190659}
\begin{equation}\label{2906191728}
\oMliname(t,q,t',q'):=   \calR(z') + \Hpp(z,p') + \oMliredname(t,q,t',q') 
\end{equation}
where for $q=(u,z,p)$ and $q'=(u',z',p')$ we have 
\begin{equation}
\label{2906191736}
\begin{aligned}
& 
\text{if } t'>0\,,  \quad 
\oMliredname(t,q,t',q'):= 
\begin{cases}
 0 &\text{if } 
\begin{cases}
\calS_u \calE_0 (t,q)=0 \,,  \\
\tilded_{L^2} ({-}\rmD_z \calE_0 (t,q),\partial\calR(0))=0
\,, \ \text{and} \\
\calW_p\calE_0(t,q)=0 \,,
\end{cases}
 \\
 +\infty & \text{otherwise,}
\end{cases}
\end{aligned}
\end{equation}
\begin{equation}
\label{oMlizero=0}
\begin{aligned}
\text{if } t'=0\,,  \quad 
 \oMliredname(t,q,t',q'): =
\begin{cases}
   \DVito(u',p') \, \DVito^*(t,q)  & \text{if } z'=0,\, 
\\
\|z'\|_{L^2}\tilded_{L^2(\Omega)} ({-}\mathrm{D}_z \calE_0 (t,q),\partial\calR(0)) & \text{if } \DVito^*(t,q)=0,\, 
\\
+\infty & \text{if } \|z'\|_{L^2} \DVito^*(t,q)>0\,,
\end{cases}
\end{aligned}
\end{equation}
\end{subequations}
 In particular, observe that, once again, the expression of
$ \oMliredname(t,q,t',q')$ for $t'>0$ enforces a  ``relaxed''  form of equilibrium for
$u$ with the condition $\calS_u \calE_0 (t,q)=0$, the local stability condition
$\tilded_{L^2} ({-}\rmD_z \calE_0 (t,q),\partial\calR(0))=0$ for $z$, and a  ``relaxed''   form of local stability for
$p$ via $\calW_p\calE_0(t,q)=0$, cf.\  Lemma \ref{l:Hencky-MGM} and Remark \ref{rem:2407190806} ahead. 
Recalling that $ \DVito(u',p') : = 
 \sqrt{ \|u'\|_{H^1}^2 {+} \|p'\|_{L^2}^2}$ (cf.\ \eqref{3006191243}),
the product $ \DVito(u',p') \, \DVito^*(t,q)$ is well defined as soon as 
$u' \in H^1(\Omega;\R^n)$ and $ p'\in L^2(\Omega;\MD)$; otherwise, 
we intend  
$
 \DVito(u',p') \, \DVito^*(t,q): = +\infty.
$
\paragraph{\bf Admissible parameterized curves}
 We are now in a position to  introduce the class of  parameterized curves 
enjoying the temporal and spatial integrability/``regularity'' properties
 of  the curves  that are limits of  reparameterized viscous solutions as $\varepsilon$, $\nu$, $\mu \down 0$.  
 Basically, such properties are motivated by the a priori estimates
that the rescaled viscous solutions inherit from the normalization condition \eqref{2906191307}. In particular, 
let us highlight that  \eqref{2906191307} provides a (uniform-in-time) bound 
 for  the quantity $ \DVito_\nu({u^{\mu}_{\eps,\nu}}', {\sfp^{\mu}_{\eps,\nu}}')\, \DVitos (\sft\men, \sfq^{\mu}_{\eps,\nu})$. 
 Recall that  $  \DVito_\nu({u^{\mu}_{\eps,\nu}}', {\sfp^{\mu}_{\eps,\nu}}')$ controls the $H^1$-norm of ${u^{\mu}_{\eps,\nu}}'$ and the $L^2$-norm of 
 ${p^{\mu}_{\eps,\nu}}'$. Therefore,  for the limiting parameterized curves $(\sft,\sfq) = (\sft,\sfu,\sfz,\sfp)$, 
 from such a bound one expects to
 infer,  ``away'' from the set where $\{\DVitos(\sft,\sfq) = 0\}$,  additional \emph{spatial regularity} for  $\sfu'$ and $\sfp'$
 in addition to that  provided by the estimate for  $\|{\sfu^{\mu}_{\eps,\nu}}'\|_{\mathrm{BD}}{+} \|{\sfp^{\mu}_{\eps,\nu}}'\|_{L^1} $. 
%
All of this is reflected in  the following definition, where we introduce the notion of \emph{admissibile parameterized curve} for the perfectly plastic system, in the spirit of \cite[Def.\ 4.1]{MRS13}. 
\begin{definition}\label{def:admparcur}
A curve $(\sft,\sfq)=(\sft, \sfu, \sfe, \sfp)\colon [0,S] \to [0,T] \times \Qpp$ is an \emph{admissibile parameterized curve for the perfectly plastic system} if $\sft \colon [0,S] \to [0,T]$ is nondecreasing and
\begin{subequations}\label{eqs:2906191823}
\begin{align}
&(\sft, \sfu, \sfz, \sfp) \in \AC\big([0,S]; [0,T] {\times} \BD(\Omega) {\times} \Hs(\Omega) \GGG {\times} \MbD \big)\,,\label{2906191813}\\
&  \sfe=\mathrm{E}(\sfu + w(\sft))-\sfp \in \AC([0,S]; \Lnn)\,,  
    \label{2906191814}\\
& (\sfu, \sfp) \in \AC_{\mathrm{loc}}(B^\circ; H^1(\Omega;\R^n){\times} L^2(\Omega; \MD))\,, \text{ where }B^\circ:=\{ s\in (0,S) \colon  \mathcal{D}^*(\sft(s), \sfq(s))>0\}\,,  \label{2906191815}\\
& 
 \label{2906191815+1}
\sft \text{ is constant in every connected component of }B^\circ\,.
\end{align}
\end{subequations}
We  shall  write $(\sft,\sfq) \in \mathscr{A}([0,S] ;[0,T] {\times} \Qpp)$. 
\end{definition}
 Let us point out that 
along an admissible curve  $s \mapsto (\sft(s),\sfq(s))$ we always have 
\[
 \DVito(\sfu'(s),\sfp'(s)) \, \DVito^*(\sft(s),\sfq(s))<\infty \qquad \text{ for a.a. } s\in  B^\circ\,. 
 \]  
 \subsection*{Balanced Viscosity solutions arising in the joint vanishing-viscosity and hardening limit
 and their properties}
 We are now in a position to give the following
  \begin{definition}\label{def:parBVsol-PP}
A curve $(\sft,\sfq)=(\sft, \sfu, \sfe, \sfp)\colon [0,S] \to [0,T] {\times} \Qpp$ is a \emph{(parameterized) Balanced Viscosity ($\BV$, for short) solution to the multi-rate system for perfect plasticity}
\eqref{RIS-intro}
 if 
\begin{itemize}
\item[-] $(\sft,\sfq)$
 is an admissible parameterized curve in the sense of Definition~\ref{def:admparcur};
 \item[-] 
 $(\sft,\sfq)$ fulfills the Energy-Dissipation balance 
\begin{equation}
\label{2906191838}
 \begin{aligned}
   \calE_0(\sft(s),\sfq(s)) + \int_{0}^{s}
   \oMliname(\sft(\tau),\sfq(\tau),\sft'(\tau),\sfq'(\tau)) \dd \tau   =  \calE_0 (\sft(0),\sfq(0)) +\int_{0}^{s} \partial_t \calE_0 (\sft(\tau), \sfq(\tau)) \, \sft'(\tau) \dd \tau  
   \end{aligned}
\end{equation}
for all   $0\leq s \leq S$.
\end{itemize}
We say that $(\sft,\sfq)$ is non-degenerate if it fulfills 
\[
\sft' + \|\sfz'\|_{\Hs(\Omega)} + \|\sfp'\|_{L^1(\Omega;\MD)} + \|\sfe'\|_{\Lnn} >0 \qquad \aein  (0,S)\,.
\]
\end{definition}
 \par
 As for $\BV$ solutions to the system with hardening, we  have 
 a characterization of $\BV$ solutions in terms of the upper Energy-Dissipation estimate  $\leq $ in 
 \eqref{2906191838}, cf.\ Proposition \ref{prop:charact-BV-pp} ahead. 
 Such characterization will rely on the chain-rule estimate in the forthcoming Lemma \ref{l:ch-rule-pp} that, in turn,
 hinges on the following technical result 
  that mimics  \cite[Proposition~3.5]{DMDSMo06QEPL}. 
 \begin{lemma}
  \label{l:Hencky-MGM}
 Suppose that $ \calS_u \calE_0(t, q) =\mathcal{W}_p \calE_0(t,q)=0$ at some
 $(t,q) \in [0,T]\times \Qpp$. Then,  for $\sigma(t)=\mathbb{C}(z) e(t)$, we have that 
 \begin{equation}\label{2307192048}
 \sigma(t) \in \calK_{z}(\Omega)\,,\quad - \diver\, \sigma(t)=f(t) \text{ a.e.\ in }\Omega\,,\quad [\sigma(t) \rmn] = g(t) \ \hn\text{-a.e.\ on }\Gneu\,.
 \end{equation} 
 \end{lemma}
%
\begin{proof}
  Since  $\calS_u \calE_0(t,q)=0$, we have that  $-\mathrm{Div}(\sigma(t))= F(t)$  in $ H_\Dir^1(\Omega; \R^n)^*$.  Recalling the form \eqref{body-force} of $F$, we get that   $-\mathrm{div}(\sigma(t))=f  \in L^n(\Omega; \R^n)$  a.e.\ in $\Omega$ and $[\sigma(t) \rmn]=g(t) \in L^\infty(\Gneu; \R^n)$.

Moreover, since $\calW_p \calE_0(t,q)=0$ and $H(z, \cdot)$ is positively 1-homogeneous, we get that for every $\eta_p \in L^2(\Omega;\MD)$
\begin{equation}\label{2307192111}
-\calH(z, -\eta_p) \leq \langle \sigma_\dev(t), \eta_p \rangle_{L^2} \leq \calH(z, \eta_p)\,,
\end{equation}
 (where $\langle \cdot, \cdot \rangle_{L^2}$ is short-hand for  the duality in $L^2(\Omega;\MD)$). 
Then we may argue as in the proof of 
\cite[Proposition~3.5]{DMDSMo06QEPL}:  in \eqref{2307192111} we choose the test function $\eta(x):= \mathrm{1}_{B}(x) \xi$, with 
$B \subset \Omega$ an arbitrary Borel set  and 
an arbitrary $\xi \in \MD$. In this way, 
 we get 
\begin{equation*}
-H(z(x), -\xi) \leq \sigma_\dev(t,x) \cdot \xi \leq H(z(x), \xi) \quad\text{for a.a.\ $x$ in }\Omega\,.
\end{equation*}
Then $\sigma_\dev(t, x) \in \partial_p H(z(x), 0)= K(z(x))$ for a.a.\ $x \in \Omega$, so that $\sigma(t) \in \calK_z(\Omega)$ and the proof is concluded.
\end{proof} 
\par
 \begin{remark}\label{rem:2407190806}
 \upshape
Conditions \eqref{2307192048}, expressed along $\BV$ solutions, correspond to the stability conditions in $u$ and $p$ $(ev1)$ and $(ev2)$ in the definition of the so-called \emph{rescaled quasistatic viscosity evolutions} in \cite[Definition~5.1]{Crismale-Lazzaroni}. Moreover, the identity $\tilded_{L^2} ({-}\rmD_z \calE_0 (t,q),\partial\calR(0))=0$ correspond to the \emph{Kuhn-Tucker inequality} $(ev3)$ therein. Notice that these three conditions hold in the set $\{ s \in (0,S) \colon \sft'(s) >0 \}$, cf.\ \eqref{2906191736}.
\end{remark} 
We are now in a position to prove the chain-rule estimate involving 
$  \oMliredname$.
 \begin{lemma}
 \label{l:ch-rule-pp}
     Along any \emph{admissible} parameterized curve 
    \begin{equation}
    \label{desired-chain-rule-pp}
    \begin{gathered}
(\sft,\sfq) \in \mathscr{A}([0,S] ;[0,T] {\times} \Qpp)
\text{ s.t. }    \oMli{\sft}{\sfq}{\sft'}{\sfq')}<+\infty \quad \aein\, (0,S),  \text{ we have that }
\\
s \mapsto   \calE_{0}(\sft(s),\sfq(s))  \text{ is absolutely continuous on } [0,S] \text{ and there holds}
\\    
-\frac{\dd}{\dd s} \calE_{0}(\sft(s),\sfq(s)) +\partial_t  \calE_{0}(\sft(s),\sfq(s))\,  \sft'(s) 
 \leq  \oMli{\sft(s)}{\sfq(s)}{\sft'(s)}{\sfq'(s)} \qquad \foraa\, s \in (0,S). 
\end{gathered}
    \end{equation}
 \end{lemma}
 \begin{proof}
 By the regularity of  admissible parameterized curves we easily deduce that the function  $s \mapsto   \calE_{0}(\sft(s),\sfq(s)) $ is absolutely continuous on $[0,S]$.  Its derivative is given by (cf.\ Lemma~\ref{l:Fr-diff-Enu})
\begin{equation*}
 \begin{aligned}
\frac{\dd}{\dd s}  \calE_{0}(\sft(s),\sfq(s))
=  \partial_t  \calE_0(\sft(s), \sfq(s))\, \sft'(s) {+} \langle \rmD_z \calE_0(\sft(s),\sfq(s)), \sfz'(s) \rangle_{H^m}  {-} \langle \serifsigma_\dev(s) \, |\, \sfp'(s) \rangle {-} \langle  \Diver(\serifsigma(s)) {+}  F(\sft(s)),  \sfu'(s) \rangle_{\BD}
\end{aligned} 
\end{equation*}
 for $\partial_t  \calE_0(\sft(s), \sfq(s))= \langle  \serifsigma(s), \rmE(w'(\sft(s))) \rangle_{L^2} -\langle F'(\sft(s)), \sfu(s)+w(\sft(s)) \rangle_{\BD} -\langle F(\sft(s)), w'(\sft(s)) \rangle_{\BD}$. 
 Therefore,  \eqref{desired-chain-rule-pp} follows if we prove that
\begin{equation}\label{2307192340}
 - \langle \rmD_z \calE_0(\sft(s),\sfq(s)), \sfz'(s) \rangle_{ \Hs} + \langle \serifsigma_\dev(s) \, |\, \sfp'(s) \rangle + \langle  \Diver(\serifsigma(s))  {+}  F(\sft(s)),    \sfu'(s) \rangle_{\BD} \leq \oMli{\sft(s)}{\sfq(s)}{\sft'(s)}{\sfq'(s)} 
\end{equation} 
\par
Let us then show \eqref{2307192340}. For a.e.\ $s \in (0,S)$ it holds that
\begin{equation}\label{2407190002}
- \langle \rmD_z \calE_0(\sft(s),\sfq(s)), \sfz'(s) \rangle_{H^m} \leq \calR(\sfz'(s)) + \|\sfz'(s)\|_{L^2} \, \tilded_{L^2(\Omega)} ({-}\mathrm{D}_z \calE_0 (\sft(s),\sfq(s)),\partial\calR(0))\,.
\end{equation} 
Let us estimate the two remaining terms, distinguishing the two cases of $\sft'(s)=0$ and $\sft'(s)>0$.
\par
If $\sft'(s)=0$
  and $s\in B^\circ$, since $\sfu'(s) \in H^1(\Omega; \R^n)$
and $\sfp'(s) \in L^2(\Omega; \MD)$ 
 a.e.\ in $B^\circ$ we have that  
 \begin{align}
\label{2307192352}
&
- \langle \Diver(\serifsigma(s)) + F(\sft(s)) , \sfu'(s) \rangle_{\BD}
 =- \langle  \mathrm{Div}(\serifsigma(s))  + F(\sft(s)) , \sfu'(s) \rangle_{H^1} 
 \leq \calS_u \calE_0(\sft(s), \sfq(s)) \, \|\sfu'(s)\|_{H^1}\,,
\\
&
\label{2307192349}
 \langle \serifsigma_\dev(s) \, | \, \sfp'(s) \rangle   - \Hpp(\sfz(s), \sfp'(s)) \leq \calW_p \calE_0(\sft(s), \sfq(s)) \, \|\sfp'(s)\|_{L^2}\,.
\end{align}
If $\sft'(s)=0$
  and $s\notin B^\circ$, then   $\calS_u \calE_0(\sft(s), \sfq(s))= \calW_p \calE_0(\sft(s), \sfq(s)) =0$, so that Lemma~\ref{l:Hencky-MGM} together with \eqref{eq:carH} imply that
\begin{equation}\label{2307192358}
-\Diver(\serifsigma(s))= F(\sft(s)) \ \text{in } \BD(\Omega)^*\,,\qquad \langle \serifsigma_\dev(s) \, | \, \sfp'(s) \rangle \leq \Hpp(\sfz(s), \sfp'(s))\,.
\end{equation}
If $\sft'(s)>0$,  again we have \eqref{2307192358}. 
\par
Collecting \eqref{2407190002}, \eqref{2307192352}, \eqref{2307192349}, \eqref{2307192358}, recalling the definition of $\oMliname$ \eqref{eqs:3006190659}, and employing the Cauchy-Schwartz inequality, we deduce \eqref{2307192340} and then conclude the proof. 
\end{proof} 
 \par
 As a straightforward  corollary of Lemma \ref{desired-chain-rule}, we   have the desired characterization of 
 $\BV$ solutions.
    \begin{proposition}
    \label{prop:charact-BV-pp}
     For an admissible 
      parameterized curve 
     $(\sft,\sfq) \in \mathscr{A}([0,S] ;[0,T] {\times} \Qpp)$  (in the sense of Definition \ref{def:admparcur}) 
         the following properties are equivalent:
    \begin{enumerate}
    \item $(\sft,\sfq)$  is a  $\BV$ solution to  the multi-rate     system 
    for perfect plasticity;
    \item $(\sft,\sfq)$ fulfills   the upper estimate $\leq$ in  \eqref{2906191838}.
    \end{enumerate} 
    \end{proposition} 
    We now give a lower semicontinuity result that will be useful in the proof of Theorem~\ref{teo:exparBVsol}.
  \begin{lemma}\label{le:2307191000}  
 Let $t_k \to t$ in $[0,T]$, $\mu_k \to 0$, $ (q_k)_k=(u_k, z_k, p_k) _k\subset \Qpp$ 
  such that  the following convergences hold as $k\to+\infty$:   $q_k \weakto   q=(u, z, p)$ in $\Qpp$,   $e(t_k)=\rmE(u_k+w(t_k))-p_k \to e(t)=\rmE(u+w(t))-p $ in $\Lnn$ and $\mu_k \, p_k \to 0$ in $L^2(\Omega;\MD)$. Then
\begin{subequations}\label{eqs:1106191955'}
\begin{align}
\calS_u \calE_0(t, q)  &\leq \liminf_{ k \to +\infty} \| \mathrm{D}_u \calE_{\mu_k} (t_k, q_k )\|_{( H^1, \bbD )^*} \,,\label{1106191956'} 
\\
\tilded_{L^2} ({-}\mathrm{D}_z \calE_0 (t,q ),\partial\calR(0)) & \leq \liminf_{ k \to +\infty} \tilded_{L^2} ({-}\mathrm{D}_z \calE_{\mu_k} (t_k,q_k ),\partial\calR(0))\,,\label{2307191134}
\\
\calW_p \calE_0(t, q) & \leq \liminf_{ k \to +\infty} \dLtwo ({-}\mathrm{D}_p \calE_{\mu_k} ( t_k,q_k ),\partial_\pi \calH( z_k,0))\,. \label{1106191958'}
\end{align}
\end{subequations}
\end{lemma}
\begin{proof}
It is immediate to see that, under the  assumed convergences, 
 setting $\sigma(t_k)=\C(z_k) e(t_k)$ and  $\sigma(t)=\C(z) e(t)$, for fixed  $\eta_u \in H^1(\Omega; \R^n)$ we have that
\[
\langle -\mathrm{Div}(\sigma(t_k)) -F(t_k), \eta_u \rangle_{H^1(\Omega;\R^n)} \to \langle -\mathrm{Div}(\sigma(t)) -F(t), \eta_u \rangle_{H^1(\Omega;\R^n)}\,.
\]
 Furthermore,  for fixed $\eta_p \in L^2(\Omega; \MD)$  there holds  
\[
\langle \sigma_\dev(t_k)- \mu_k \, p_k, \eta_p \rangle_{L^2(\Omega;\MD)} -\calH(z_k,\eta_p) \to \langle \sigma_\dev(t), \eta_p \rangle_{L^2(\Omega;\MD)} -\calH(z,\eta_p)\,.
\]
Passing to the supremums, we obtain \eqref{1106191956'} and \eqref{1106191958'}.
As for \eqref{2307191134}, this follows as in \eqref{1106191957} since one 
 only employs the convergence $z_k \weakto z$ in $H^m(\Omega)$,
  encompassed in the hypothesis $q_k\weakto q$ in $\Qpp$. 
\end{proof} 
    
 \subsection*{Existence of $\BV$ solutions to the multi-rate system for perfect plasticity}
We are now ready to state and prove our existence result for
$\BV$ solutions in the sense of Definition~\ref{def:parBVsol-PP}. 
In order to simplify  notation, we fix vanishing sequences $(\varepsilon_k)_k$, $(\nu_k)_k$, $(\mu_k)_k$, with $\nu_k \leq \mu_k$ and denote by
 $(\sft_k)_k$,  $(\sfq_k)_k$, $(\sfe_k)_k$, $(\serifsigma_k)_k$ 
 the sequences
  $(\sft^{\mu_k}_{\varepsilon_k,\nu_k})_k$,
   $(\sfq^{\mu_k}_{\varepsilon_k,\nu_k})_k$,
    $(\sfe^{\mu_k}_{\varepsilon_k,\nu_k})_k$,
   $(\serifsigma^{\mu_k}_{\varepsilon_k,\nu_k})_k$, respectively. 
\begin{theorem}\label{teo:exparBVsol}
Under the assumptions of Section~\ref{s:2}
 and 
 \eqref{EL-initial}
 for all  vanishing sequences
 $(\varepsilon_k)_k$, $(\nu_k)_k$, $(\mu_k)_k$ with $\nu_k \leq \mu_k$ for every $k\in \N$
there exist a (not relabeled) subsequence $(\sft_k,\sfq_k)$
and  a 
curve $ (\sft, \sfq)=(\sft, \sfu, \sfz, \sfp) \in \mathscr{A}([0,S];[0,T]{\times}\Qpp)$ such that   
\begin{enumerate}
\item for all $s\in[0,S]$ the following convergences hold as $k\to+\infty$
\begin{equation}\label{weak-converg-pp}
\begin{split}
\sft_{k}(s)\to\sft(s) &\,, \quad  \sfu_{k}(s) \weaksto \sfu(s) \text{ in } \BD(\Omega) \,, 
 \quad  \sfz_{k}(s) \weakto \sfz(s) \text{ in } \Hs(\Omega) \,,  \\
 \sfe_k(s) & \weakto \sfe(s) \text{ in } \Lnn\,,\quad  \sfp_{k}(s) \weakto \sfp(s) \text{ in } \MbD\,;
 \end{split}
\end{equation}
\item
 there exists $\overline{C}>0$ such that 
for a.e.\ $s\in (0, S)$ there holds
\begin{equation}\label{2906191858}
\begin{aligned}
&
{\sft}'(s) {+}  \| \sfu'(s)\|_{\BD(\Omega)} {+} \|\sfz'(s)\|_{\Hs(\Omega)} {+} \|\sfp'(s)\|_{\MbD} 
\\
&
 {+} \|\sfe'(s)\|_{\Lnn} 
  {+} \DVito(\sfu'(s), \sfp'(s))\, \DVito^* (\sft(s), \sfq(s))\leq \overline{C}\,;
 \end{aligned}
\end{equation}

 \item
 $(\sft,\sfq)$ is a Balanced Viscosity solution to the multi-rate system for perfect plasticity \eqref{RIS-intro}
  in the sense of Definition \ref{def:parBVsol-PP}. 
 \end{enumerate}
\end{theorem}
\begin{proof}
As done for Theorems
\ref{teo:rep-sol-exist} and
 \ref{teo:rep-sol-exist-2}, we divide the proof in 3 steps.
\paragraph{\bf Step $1$: Compactness.} Let $(\sft_k, \sfq_k)_k$ be a sequence as in the statement. 
It follows from 
 the normalization condition 
\eqref{2906191307}
 that  there exists $C>0$ such that for a.a.\ $s \in (0,S)$
\begin{equation}\label{2906191956}
\begin{aligned}
&
{\sft}_k'(s) {+}  \| \sfu_k'(s)\|_{W^{1,1}(\Omega; \R^n)} {+} \|\sfz_k'(s)\|_{\Hs(\Omega)} {+} \|\sfp_k'(s)\|_{L^1(\Omega;\MD)}
\\
&   {+} \sqrt{\mu_k} \, \|\sfp_k'(s)\|_{L^2(\Omega;\MD)}  
 {+} \|\sfe_k'(s)\|_{\Lnn} 
\leq C\,,
 \end{aligned}
\end{equation}
 where the estimate for $\| \sfu_k'\|_{W^{1,1}}$
 (recall that  $\sfu_k'  \in H^1(\Omega;\R^n)$),  
 ensues from the fact that $\sig{\sfu_k'} = \sfe_k'+ \sfp_k'$  is bounded in $L^1(\Omega;\Mnn)$   combined with Korn's inequality. 

Clearly,   in the estimates for $\sfu_k$ and $\sfp_k$ we may pass 
 from $W^{1,1}(\Omega;\R^n)$ and $L^1(\Omega; \MD)$ to the (duals of separable spaces) $\BD(\Omega)$ and $\MbD$. 
 Therefore, we are in a position to apply the compactness results from  \cite{Simon87}, to  get 
 that
 there exists $(\sft,\sfq) \in W^{1,\infty} ([0,S];[0,T]{\times}\BD(\Omega){\times}\Hs(\Omega){\times}\MbD)$,
 and $\sfe \in  W^{1,\infty}(0,S;\Lnn)$,
 such that,
  along a not relabeled subsequence,
  \begin{subequations}\label{2906191911}
\begin{align}
& \sft_{k}\weaksto\sft  \quad \text{in } W^{1,\infty}(0,S;[0,T]) \,, && \sfu_{k}\weaksto\sfu \quad \text{in } W^{1,\infty}(0,S;\BD(\Omega)) \,,
 \\
&  \sfz_{k}\weaksto\sfz \quad \text{in } W^{1,\infty}(0,S;\Hs(\Omega)) \,, &&
\sfz_{k} \to \sfz \quad \text{in }   \mathrm{C}^0([0,S]; \mathrm{C}^0(\overline\Omega)) \,,
\\
&\sfe_{k}\weaksto\sfe \quad \text{in } W^{1,\infty}(0,S;\Lnn)\,, &&
\sfp_{k}\weaksto\sfp \quad \text{in } W^{1,\infty}(0,S;\MbD) \,.
\end{align}
It can be checked that $\sfe = \sig{\sfu{+}w(\sft)} - \sfp$. 
In particular, the pointwise convergences \eqref{weak-converg-pp} hold.  Notice also that
\begin{equation}\label{2307191102}
\sqrt{\mu_k} \, \sfp_{k}\weaksto0 \quad \text{in } W^{1,\infty}(0,S;L^2(\Omega;\MD))\,,
\end{equation}
so that for every $s \in [0,S]$
\begin{equation}\label{2307191104}
\sqrt{\mu_k} \, \sfp_k(t) \weakto 0\quad\text{in }L^2(\Omega;\MD)\,.
\end{equation}
\end{subequations}
\par
Next, we introduce the functions $s^-$ and $s^+$ and the sets 
 $\mathcal{S}$, $\mathcal{U}$ in the same way as in  Step~1 in Theorem~\ref{teo:rep-sol-exist} (cf.\ \eqref{defso}); we readily deduce the  following  convergences for all $t\in[0,T]$
\begin{subequations}\label{eqs:3006190655}
\begin{align}
u_{\eps_k,\nu_k}^{\mu_k}(t)&\weaksto\sfu(s_-(t))=\sfu(s_+(t)) \text{ in } \BD(\Omega) \,,\label{3006190656}\\
z_{\eps_k,\nu_k}^{\mu_k}(t)&\weakto\sfz(s_-(t))=\GGG \sfz(s_+(t))  \text{ in }\Hs(\Omega)\,,\label{3006190657}\\
p_{\eps_k,\nu_k}^{\mu_k}(t)&\weaksto\sfp(s_-(t))=\sfp(s_+(t)) \text{ in } \MbD\,. \label{3006190658}
\end{align}
\end{subequations}
\paragraph{\bf Step~2: Finiteness of $\oMliredname(\sft(\tau), \sfq(\tau), \sft'(\tau), \sfq'(\tau))$ when $\sft'(\tau)>0$.} 
In view of the definition  \eqref{eqs:3006190659}  of $\oMliredname$,  the function  
$\tau \mapsto \oMliredname(\sft(\tau), \sfq(\tau), \sft'(\tau), \sfq'(\tau))$ is finite in 
the set $A: = \{ s\in [0,S]\,: \ \sft'(s)>0\}$
if and only if
\begin{equation}\label{3006190702}
\begin{aligned}
 \calS_u \calE_0(\sft(\tau), \sfq(\tau)) =0 \,, \qquad 
-\mathrm{D}_z \calE_0 (\sft(\tau),\sfq(\tau)) \in \partial\calR(0) \,, \qquad 
\mathcal{W}_p \calE_0(\sft(\tau),\sfq(\tau))=0
\quad
 \text{for a.a.}\ \tau\in A\,.
\end{aligned}
\end{equation} 
By \eqref{DVITOequalsSTAR} we obtain
 \begin{equation}\label{3006191143}
 \begin{aligned}
 \DVitoskk(\sft_k(\tau), \sfq_k(\tau))  =  \DVitoskk(\sft_k(\tau), q_k(\sft_k(\tau)))  
  &  =\eps_k \sqrt{ \nu_k \| u_k'(\sft_k(\tau))\|^2_{ H^1, \bbD }   {+} 
\|z_k'(\sft_k(\tau))\|_{L^2}^2
{+} \nu_k\|p_k'(\sft_k(\tau))\|_{L^2}^2} 
\\
&= \frac{\eps_k}{\sft_k'(\tau)} \sqrt{ \nu_k \| \sfu_k'(\tau)\|^2_{ H^1, \bbD }   {+} 
\|\sfz_k'(\tau)\|_{L^2}^2
{+} \nu_k\|\sfp_k'(\tau)\|_{L^2}^2}
\\
& 
\leq  \frac{\eps_k}{\sft_k'(\tau)} \qquad \foraa\, \tau \in (0,S),
\end{aligned}
\end{equation}
where in the last estimate we   exploited the normalization condition  \eqref{2906191307} and the fact that $\nu_k \leq \mu_k$.  
Moreover, one sees as in \eqref{limsup} that $\limsup_{k\to +\infty} \sft'_k(\tau)>0$ for a.e.\ $\tau \in A$. 
Since $\varepsilon_k$, $\nu_k \down 0$, 
 by Lemma~\ref{le:2307191000} (notice that its assumptions are satisfied by the convergences \eqref{weak-converg-pp} and \eqref{2307191104}, also recalling that $\mu_k \to 0$) and an argument analogous to that in Step~2 of Theorem~\ref{teo:rep-sol-exist-2} 
we deduce \eqref{3006190702}.\\
\textbf{Step $3$:  The  Energy-Dissipation upper estimate \eqref{2906191838}.} 
 In view of the characterization provided by  Proposition~\ref{prop:charact-BV-pp},  to conclude that $(\sft,\sfq)$ is a $\BV$ solution in the sense of Definition 
\ref{def:parBVsol-PP} it is sufficient to show that $(\sft,\sfq)$ is an admissible parameterized curve as in  Definition~\ref{def:admparcur},  and that 
it fulfills  \eqref{2906191838} as an upper estimate. 
First of all, we show that 
\begin{equation}\label{0107192336}
\sfz'(\tau) \, \DVito^*(\sft(\tau), \sfq(\tau))=0 \quad \text{ for a.a.\ }\tau \in (0,S)\,.
\end{equation}
This follows arguing similarly to what done 
in Step~2 of Theorem~\ref{teo:rep-sol-exist-2}. We start from 
\eqref{1506190857} and then observe that (cf.\ \eqref{citata-dop-sez7}) 
\begin{equation}
\label{0107192132}
\liminf_{k\to+\infty} \DVito^{*,\mu_k}(\sft_k(\tau), \sfq_k(\tau)) \geq  \DVito^*(\sft(\tau), \sfq(\tau)) \qquad \text{for all } \tau \in [0,S],
\end{equation}
due to  \eqref{weak-converg-pp}, \eqref{2307191104}, \eqref{1106191956'}, and \eqref{1106191958'}. 
Then, applying 
Lemma~\ref{le:0107192150} with the   analogous choices and  arguments as 
in the proofs of    Theorems~\ref{teo:rep-sol-exist}  and \ref{teo:rep-sol-exist-2},
 also relying on Lemma \ref{le:2307191000}    we conclude that 
\begin{equation}\label{1506190906-new}
\begin{aligned}
\int_{B^\circ} \|\sfz'(\tau)\|_{L^2} \,  \DVito^*(\sft(\tau), \sfq(\tau)) \dd \tau
& \leq
 \liminf_{k\to \infty} \int_{B^\circ} \|\sfz'_k(\tau)\|_{L^2}\,     \mathcal{D}^{*,\mu_k}(\sft_k(\tau), \sfq_k(\tau))\dd \tau
\\ &  \leq   \sqrt{\nu_k} \liminf_{k\to \infty}  \MredVitokk(\sft_k(\tau), \sfq_k(\tau), \sfq_k'(\tau)) \dd \tau=0\,,
 \end{aligned}
\end{equation}
with $B^\circ$ from \eqref{2906191815}. Then, \eqref{0107192336} ensues.
\par
In analogy with  \eqref{1206191009'}, we also introduce the set
%
%
\begin{equation}\label{0207190008}
\begin{split}
 C^\circ := \{ \tau \in [0,S] \colon  \tilded_{L^2(\Omega)} ({-}\mathrm{D}_z \calE_0 (\sft(\tau),\sfq(\tau)),\partial\calR(0)) >0\}\,.
\end{split}
\end{equation} 
By Lemma~\ref{le:0107192347} applied with the choices $X:=[0,1]{\times}[0,T]{\times} \Qpp$, 
$I:=[0,S]$, $B:=B^\circ$, $v_k(\tau):= (\mu_k, \sft_k(\tau), \sfq_k(\tau))$, $v(\tau):= (\mu, \sft(\tau), \sfq(\tau))$, 
 and with the function $f: X \to [0,+\infty]$ defined by 
\[
f (\mu, t, q): =\begin{cases}
\DVito^*_{\mu}(t, q) & \text{if }  \mu>0,
\\
\DVito^*(t,q) & \text{if } \mu =0;
\end{cases}
\]
indeed, thanks to Lemma \ref{le:2307191000} the function $f$ is weakly$^*$ lower semicontinuous on 
 $X:=[0,1]{\times}[0,T]{\times} \Qpp$.
Thus, we  obtain that 
%
for any compact subset  $K^{\circ}$
 of $B^\circ$ there exist $c>0$ and $\overline{k}\in \N$ such that
\begin{equation*}
\DVito^{*,\mu_k}(\sft_k(\tau), \sfq_k(\tau)) \geq c \qquad\text{for every } k\geq \overline{k}\,, \tau \in K^{\circ}\,.
\end{equation*}
By the normalization condition \eqref{2906191307} (recall the notation for $\sft_k$, $\sfq_k$) we obtain that for $k \geq \overline{k}$
\[
 \DVito(\sfu'_k(\tau), \sfp'_k(\tau))\leq \frac{1}{c} \qquad \foraa\, \tau \in K^{\circ},
\]
so that $\sfu_k$ and $\sfp_k$ are equi-Lipschitz in $K^{\circ}$ with values in $H^1(\Omega;\R^n)$ and $L^2(\Omega;\MD)$,
respectively. 
Therefore, we ultimately deduce
that $\sfu\in W^{1,\infty} (K^\circ;H^1(\Omega;\R^n))$ and that $\sfp \in W^{1,\infty} (K^\circ; L^2(\Omega;\MD))$. 
By the arbitrariness of $K^\circ$ we conclude that 
 $(\sfu, \sfp) \in 
  W^{1,\infty}_\mathrm{loc}(B^\circ; H^1(\Omega;\R^n){\times} L^2(\Omega; \MD))$,
and then $(\sft,\sfq)$ is an admissible parameterized curve in the sense of Definition~\ref{def:admparcur}. 
Moreover,  again for every $K^\circ \Subset B^\circ$ we have that 
the sequence $ (\DVito(\sfu'_k, \sfp'_k))_k$ converges weakly
to some $d$
 in $L^\infty(K^\circ)$,
 with $d \geq \DVito(\sfu',\sfp')$ a.e.\ in $K^\circ$. 
Then, we are again in a 
 position to apply Lemma~\ref{le:0107192150},  deducing (in view of the arbitrariness of $K^{\circ} \subset B^\circ$)
\begin{equation}\label{0207190044}
\begin{aligned}
\int_{B^\circ} \DVito(\sfu'(\tau), \sfp'(\tau))\, \DVito^*(\sft(\tau), \sfq(\tau)) \dd \tau &  \leq \liminf_{k\to +\infty} \int_{B^\circ} \DVito(\sfu'_k(\tau), \sfp'_k(\tau))\, \DVito^{*,\mu_k}(\sft_k(\tau), \sfq_k(\tau)) \dd \tau
\\
 &  \leq \liminf_{k\to +\infty} \int_{B^\circ} \MredVitokkk(\sft_k(\tau), \sfq_k(\tau), \sft_k'(\tau), \sfq_k'(\tau)) \dd \tau\,,
 \end{aligned}
\end{equation}
the last estimate due to \eqref{0207190031}. Then, estimate \eqref{2906191858} follows by lower semicontinuity arguments.
\par
Finally,  we repeat the arguments for \eqref{1506191723}, obtaining that 
\begin{equation}\label{0207190039}
\begin{split}
\int_{(0,S) \setminus B^\circ}  & \oMliredname(\sft(\tau), \sfq(\tau), \sft'(\tau), \sfq'(\tau)) \dd \tau = \int_{C^\circ \setminus B^\circ} \|\sfz'(\tau)\|_{L^2} \, \tilded_{L^2(\Omega)} ({-}\mathrm{D}_z \calE_0 (\sft(\tau),\sfq(\tau)),\partial\calR(0)) \dd \tau 
\\ &
\stackrel{(1)}{\leq} \liminf_{k\to \infty} \int_{C^\circ \setminus B^\circ} \|\sfz_k'(\tau)\|_{L^2} \, \tilded_{L^2(\Omega)} ({-}\mathrm{D}_z \calE_{\mu_k} (\sft_k(\tau),\sfq_k(\tau)),\partial\calR(0)) \dd \tau
\\ &
\leq \liminf_{k\to +\infty} \int_{C^\circ \setminus B^\circ} \MredVitokkk(\sft_k(\tau), \sfq_k(\tau),\sfq_k'(\tau)) \dd \tau \,.
\end{split}
\end{equation}
Observe that (1) follows from the very same argument as in the proof of 
Theorem~\ref{teo:rep-sol-exist-2},  now employing \eqref{2307191134} in place of \eqref{1106191957}. 
 \par
 It follows from 
 \eqref{0207190044} and  \eqref{0207190039}, also recalling the definition of $B^\circ$, that
\begin{equation*}
\int_0^S  \oMliredname(\sft(\tau), \sfq(\tau), \sft'(\tau), \sfq'(\tau)) \dd \tau \leq \liminf_{k\to +\infty} \int_0^S \MredVitokkk(\sft_k(\tau), \sfq_k(\tau),  \sfq_k'(\tau)) \dd \tau\,.
\end{equation*}
 Combining  the above lower semicontinuity estimate 
with the  limit passage in the terms with  driving energy and in the power term
(which is standard and goes as in Section~\ref{s:van-visc}), 
we succeed in taking the limit in the Energy-Dissipation inequality  \eqref{rescaled-enid-eps} to 
 conclude
the desired validity  of the Energy-Dissipation upper estimate $\leq$ in \eqref{2906191838}.
 This finishes
 the proof of Theorem~\ref{teo:exparBVsol}. 
\end{proof}

    \appendix
    \section{Discrete Gronwall-type Lemmas}
Here we collect, for the reader's convenience, 
the  discrete Gronwall-type  results that have been exploited in the proof of the a priori estimates from 
Proposition \ref{prop:energy-est-SL}. 
 \begin{lemma}
\label{l:discrG1}
Let $B,\, \tau>0$,   $N_\tau\in \N$,
$(a_k)_{k=1}^{N_\tau} ,\, (b_k)_{k=1}^{N_\tau}\subset [0,+\infty)$ fulfill
\[
a_k \leq B + \sum_{j=0}^{k-1}  a_j b_j  \quad \text{for all } k \in \{1,\ldots, N_\tau\}. 
\]
Then,  there holds
\begin{equation}
\label{discrG1}
a_k \leq 
B \exp \left( \sum_{j=0}^{k-1} b_j \right) \quad \text{for all } k \in \{1,\ldots, N_\tau\}. 
\end{equation}
\end{lemma}
 \begin{lemma}{\cite[Lemma 4.5]{RossiSavare06}}
\label{l:discrG1/2}
Let $N_\tau\in \N$ and $ b, \, \lambda,\,  \Lambda  \in (0,+\infty)$ fulfill $1-b \geq \tfrac 1\lambda>0$; let $(a_k)_{k=1}^{N_\tau} \subset [0,+\infty)$
satisfy
\[
a_k \leq  \Lambda  + b \sum_{j=1}^k a_j \qquad \text{for all } k \in \{1,\ldots, N_\tau\}. 
\]
Then,  there holds
\begin{equation}
\label{discrG1/2}
a_k\leq \lambda  \Lambda  \exp(\lambda b k) \quad \text{for all } k \in \{1,\ldots, N_\tau\}. 
\end{equation}
\end{lemma}
 
The following lemma   generalizes   \cite[Lemma~4.1]{KRZ13}; its proof is based on the calculations developed in 
for  \cite[Proposition~3.8]{Crismale-Lazzaroni}  (see also  \cite[Proposition~3.5]{ACO2018}); that it is why, we shall only partially carry out the argument, and we shall
refer to \cite{Crismale-Lazzaroni}   for more details.   
   \begin{lemma}
\label{l:discrG2}
   Let $\{a_k\}_{k=0}^{N_\tau}$, $\{M_k\}_{k=1}^{N_\tau}$, $\{r_k\}_{k=1}^{N_\tau}$,   $\{c_k\}_{k=0}^{N_\tau}$  $\rho$ and   $\eta$   be 
non-negative numbers, $\eps,\tau>0$ with 
 $\gamma:=\kappa_1\tau/\eps\leq 1$ for some $\kappa_1>0$   and  $N_\tau\in \N$, $N_\tau \tau=T$.  
Assume that  $a_0=0$,   $r_k \leq \kappa_2 a_k$  for some $\kappa_2>1$,   and that  for $1\leq k\leq N_\tau$ it holds
\begin{align}
\label{app_est1}
 a_k(a_k-a_{k-1}) + \gamma a_k^2 + \gamma M_k^2 \leq  \eta^2\gamma   \left(1{+} c^2_k {+}\tfrac{\delta_{1,k}}{\tau\eps} \rho^2\right) +  \gamma a_k r_k.
\end{align}
Then,  if $\gamma =\kappa_1 \tau/\varepsilon\leq 1/(2\kappa_2)$,  there exists a constant   $C=C(\eta,T)>0$   not depending on any of the 
other above quantities such that
\begin{align}
 \label{app-est-discrete}
 \sum_{k=1}^{N_\tau} \tau M_k \leq C\left(T  +  \rho + \sum_{k=1}^{N_\tau} \tau c_k^2  +   \sum_{k=1}^{N_\tau} \tau 
r_k\right).
\end{align} 
\end{lemma}
\begin{proof}
 For $2\leq k \leq N_\tau$, we can recast \eqref{app_est1} in  the same form as  \cite[inequality between (3.35) and (3.36)]{Crismale-Lazzaroni},  namely 
\begin{equation*}
2a_k' (a_k'-a_{k-1'})+2\zeta (a_k')^2 + (b_k')^2\leq (c_k')^2 +2 a_k' d_k'\,.
\end{equation*} 
For this, it is sufficient to 
replace $a_k$, $\gamma$, $\gamma M_k^2$, $\eta^2\gamma (1+c_k^2)$, and $r_k$
in \eqref{app_est1} (observe that $\delta_{1,k} =0$ for $k\in \{2,\ldots, N_\tau\}$), 
 by, respectively, $a_k'/\sqrt{2}$, $\zeta=\ol C \tau/\varepsilon$, $(b_k')^2$, $(c_k')^2$, and $d_k' / \sqrt{2}$, with  a universal constant $\ol C$. 
Following exactly the argument in \cite[Proposition~3.8]{Crismale-Lazzaroni} and then  rewriting  \cite[(3.41)]{Crismale-Lazzaroni} in the present setup,  we get that
\begin{equation}\label{2909192322}
\sum_{k=2}^{N_\tau} \tau M_k \leq C \Big(T + \varepsilon \, a_1 + \sum_{k=2}^{N_\tau} \tau c_k^2 + \sum_{k=2}^{N_\tau} \tau r_k\Big)\,.
\end{equation}
Let us now estimate $\tau M_1$ and $\varepsilon a_1$ by \eqref{app_est1} for $k=1$. 
Notice that, since $a_0=0$ and using  the  Cauchy Inequality $2 a_1 r_1 \leq a_1^2 +r_1^2$, we derive that
$
 M_1^2 \leq  \eta^2    \left(1{+} c^2_1 {+}\tfrac{ \rho^2}{\tau\eps}\right) +   r_1^2.
$
Multiplying by $\tau^2$, recalling $\tau < \varepsilon$,  and taking the square root we obtain, for a suitable $C$, that
\begin{equation}\label{2909192323}
\tau M_1 \leq C \tau (1+c_1 +r_1) + \varrho \leq C \tau (2+c_1^2 +r_1) + \varrho\,.
\end{equation}
We are then left to estimating $\varepsilon\, a_1$.  We again start from \eqref{app_est1} for $k=1$: recalling that $a_0=0$, we then have 
\begin{equation}
\label{take-up}
a_1^2 +\gamma a_1^2 +\gamma M_1^2 \leq \eta^2 \gamma \left( 1+c_1^2 +\frac1{\tau\eps}\rho^2 \right) +\gamma a_1r_1\,.
\end{equation}
Then, we use the conditions $r_k \leq  \kappa_2 a_k$ and   $k_1\tau/\varepsilon<1/(2 k_2)$  to get $\gamma a_1 \, r_1 \leq \frac{a_1^2}{2}$, which can be absorbed into the left-hand side of \eqref{take-up}. Multiplying by $\varepsilon^2$ we get
$
\varepsilon^2 a_1^2 \leq c^2\tau \varepsilon (1+c_1^2) + c^2 \varrho^2 $ for some $c>0$, so that 
\begin{equation}\label{2909192324}
\varepsilon a_1 \leq C\big( 1+ \sqrt{\varepsilon\, \tau c_1^2}  + \varrho \big)\leq C(2+\varepsilon\, \tau c_1^2 + \varrho) \,. 
\end{equation}
Collecting \eqref{2909192322}, \eqref{2909192323}, and \eqref{2909192324}, up to modifying $C$ we conclude \eqref{app-est-discrete}. 
\end{proof}

\section{Two abstract results}

\GGG We first recall an abstract lemma from \cite{MRS13} and \cite{MRS-MJM} (to which we refer for the proof).
\begin{lemma}\label{le:0107192150}
Let $I$ be a measurable subset of $\R$ and let $h_n$, $h$, $m_n$, $m \colon I \to [0,+\infty]$ be measurable functions for $n\in \N$ that satisfy
\begin{equation}
\label{hypLB1}
h(x) \leq \liminf_{n\to +\infty} h_n(x) \quad \text{for }\mathcal{L}^1\text{-a.e.\ }x \in I, \qquad m_n \weakto m \quad\text{in } L^1(I)\,.
\end{equation}
Then
\begin{equation*}
\int_I h(x) m(x) \dd x \leq \liminf_{n\to +\infty}\int_I h_n(x) m_n(x) \dd x\,.
\end{equation*}
\end{lemma}
Let us now consider a result that is applied in the proof of Theorem~\ref{teo:exparBVsol}.
\begin{lemma}\label{le:0107192347}
Let   $X=Y^*$,  for $Y$ a separable Banach space, 
$I =[a,b]\subset \R$, $f\colon X \to [0,+\infty]$ be weakly$^*$ lower semicontinuous, and 
let $(v_k)_k$ be a sequence of functions
$v_k \colon I \to X$ satisfying
\begin{equation}\label{0307190817}
\begin{aligned}
&
\exists\, C>0 \ \forall\, t, s \in I \, : \quad \|v_k(t)- v_k(s)\|_{X} \leq C |t-s|\,,
\\
&
 v_k(t) \weaksto v(t) \qquad \text{in } X \quad \text{for all } t \in I\,.
 \end{aligned}
\end{equation} 
Then, for every compact subset
$K \subset B:=\{  t \in I \colon f(v(t))>0\}$ there exist $c>0$ and $\overline{k}\in \N$ such that
\begin{equation}\label{0307190941}
f(v_k(t)) \geq c \qquad\text{ for every } k \geq \bar{k}, \, t\in K\,.
\end{equation}
\end{lemma}
\begin{proof}
By \eqref{0307190817} and the pointwise weak$^*$ convergence to $v$, we deduce that
\begin{equation}\label{0307190818}
\|v(t)- v(s)\|_{X} \leq C |t-s|\qquad\text{for every }t,s \in I,
\end{equation}
that the set $V: =\bigcup_{k} v_k(I) \cup v(I)$
 is bounded
in $X$,  and that 
\begin{equation}\label{0207191803}
 \lim_{k\to+\infty}\sup_{t\in I} d_{w^*}(v_k(t), v(t)) =0
\end{equation}
with $d_{w^*}$ the metric inducing the weak$^*$ topology on the bounded set  $V$
 (here we use the separability of $Y$, and refer to the compactness arguments by \cite{Simon87}).
\par
Let us now $K$ be as in the statement. By \eqref{0307190818} we get that $v(K)$ is compact in $X$, and since $K \subset B$ we have that $v(K) \subset \{ f >0 \} \doteq \{ f >0 \}$ the set $\{ x \in X \colon  f(x) >0\}$.
Then we can find an open set $A$ 
such that $v(K)\subset A \subset \overline{A} \subset \{f > 0\}$. We deduce, employing the lower semicontinuity of $f$, that
\begin{equation}\label{0307190940}
f(A)\subset [c, +\infty]\,,
\end{equation} 
for a suitable constant $c>0$.
Thanks to \eqref{0207191803} and the fact that $d_{w*}(v_1,v_2)\leq C^* \| v_1{-}v_2\|_X$, for a suitable $C^*$ and every $v_1$ and $v_2 \in V$, choosing $\varepsilon$ small enough we get that $v_k(K)\subset A$ for $k \geq \overline{k}$. Therefore, by \eqref{0307190940} we conclude \eqref{0307190941}.
%
%
%
%
\end{proof}

%


 \end{document}